\documentclass[12pt]{article}
\usepackage{amsmath, amsthm, amsfonts, amssymb, tikz,}
\usepackage{subcaption, verbatim, geometry}
\usetikzlibrary{positioning}
\usetikzlibrary{arrows}
\usetikzlibrary{decorations.markings}
\tikzset{
  big arrow/.style={
    decoration={markings,mark=at position 1 with {\arrow[scale=2,#1]{>}}}, postaction={decorate}, shorten >=0.4pt}, big arrow/.default=blue}\usepackage{pdflscape}

\definecolor{purple}{RGB}{138,43,226}
\definecolor{pink}{RGB}{255,182,193}

\usepackage{eufrak}
\usepackage{placeins}
\usepackage{cleveref}

\newtheorem{thm}{Theorem}
\newtheorem{prop}{Proposition}
\newtheorem{defi}{Definition}
\newtheorem{lemma}{Lemma}
\newcommand{\rh}[1]{\rho_{#1}}
\newcommand{\ts}{\otimes}
\newcommand{\bts}{\boxtimes}
\newcommand{\xii}{$\xi_i$}
\newcommand{\xiip}{$\xi_i'$}
\newcommand{\cf}{C^{flip}}
\newcommand{\fl}[1]{{#1}^{flip}}

\title{Elliptic Involution on Knot Complements}
\author{Yang (Chris) Xiu}

\begin{document}

\maketitle
\begin{abstract}
We show that Bordered Heegaard Floer invariant $\widehat{CFD}$ of a knot complement in $S^3$ is invariant under the elliptic involution on its boundary.
\end{abstract}

\section{Introduction}

Heegaard Floer homology is a 3-manifold invariant defined by Peter S. Ozsv\'{a}th and Zolt\'{a}n Szab\'{o} in \cite{HF}, which has proved to be powerful. A knot version is later defined independently by Jacob Rasmussen in \cite{JRKnot} and by Ozsv\'{a}th and Szab\'{o} in \cite{knotfloer}.

Bordered Floer homology developed by Lipshitz, Ozsv\'{a}th, and Thurston in \cite{bord} is a great tool to compute Heegaard Floer homology by decomposing a 3-manifold into smaller pieces with parametrized boundaries, computing the ``Bordered'' invariant on each piece, and finally taking an appropriate tensor product to recombine them. Specifically, the following theorem is proved in \cite{bord}.
\begin{thm}
Suppose that $F$ is a closed oriented surface and $Y_1$ and $Y_2$ are two 3-manifolds with parametrized boundary $F$ and $-F$, then 
\[
\widehat{CF}(Y_1\cup_FY_2) \cong \widehat{CFA}(Y_1)\tilde{\otimes}_{\mathcal{A}(F)} \widehat{CFD}(Y_2) \cong \widehat{CFA}(Y_1)\bts_{\mathcal{A}(F)} \widehat{CFD}(Y_2)
\]
\end{thm}

One natural question to ask is what happens to $\widehat{CFD}$ modules when the boundary parametrization of a bordered manifold changes. We prove the following result:
\begin{thm}\label{main}
Given a knot $K \subset S^3$, let $X$ be the knot complement with boundary parametrization $\phi: T^2 \to \partial X$. Let $h : T^2\to T^2$ be the elliptic involution on the torus. We have
\[
    \widehat{CFD} (X, \phi) \simeq \widehat{CFD} (X, \phi\circ h).
\]
\end{thm}

One application of this result is to study 3-manifold mutations. It trivially follows that
\begin{prop}
A mutation using an elliptic involution is not detected by $\widehat{HF}$ when either one of the two manifold with boundary is a knot complement.
\end{prop}

It is interesting to compare with the result in \cite{Clarkson}, where mutations using genus-2 hyperelliptic involution are studied. It is shown there that mutating by the genus-2 hyperelliptic involution can change the rank of the non-torsion summands of $\widehat{HF}$.

In this paper, we first introduce the necessary background in \cref{sec:background}, then we prove \cref{main} in \cref{sec:proof_total}, where two proofs are given. The first one in \cref{sec:proof} is simpler and more intuitive, but requires the knot to satisfy an extra mild condition. Although we are not aware of any knots that do not satisfy this condition, a general proof is given in \cref{sec:proofg}.


\section{Background}\label{sec:background}
\subsection{Knot Floer Homology}
We first review the setup of Knot Floer Homology, especially the aspects important to the purpose of this paper, following an overview in \cite{bord}. Suppose that a knot $K$ is specified by a doubly-pointed Heegaard diagram $\mathcal{H}=(\Sigma, \alpha, \beta, w, z)$, where the base points are $w$ and $z$. $CFK^-$ is a $\mathbb{Z}-$filtered chain complex generated over $\mathbb{F}_2[U]$ by $\mathfrak{S}_K$, which is the usual generators of Heegaard Floer homology from diagram $\mathcal{H}$, with the differential
\[
    \partial^-(x):= \sum_{y\in \mathfrak{S}_K} \sum_{\substack{B\in\tilde{\pi}_2(x,y)\\ \text{ind}(B)=1}}\#(\mathcal{M}^B(x,y))U^{n_w(B)}\cdot y.
\]

We summarize how Maslov grading and Alexander filtration interact with the differential.
If $B\in\tilde{\pi}(x,y)$, then
\begin{align}\label{masalex}
\begin{split}
A(x)-A(y)&=n_z(B)-n_w(B)\\
A(U\cdot x) &= A(x)-1\\
M(x)-M(y)&=\text{ind}(B)-2n_w(B)\\
M(U\cdot x) &= M(x)-2
\end{split}
\end{align}

We denote $\widehat{gCFK}(\mathcal{H},w,z)$ by $C$ and its summand in Alexander grading r by $C(r)$. Let $\partial^i:C(r) \to C(r+i)$ be the map counting holomorphic disks $\phi$ with $n_z(\phi)=-i$ and $n_w(\phi)=0$ if $i\leq 0$, and holomorphic disks $\phi$ with $n_z(\phi)=0$ and $n_w(\phi)=i$ if $i\geq 0$. Note $\partial^i$ drops (or raises) the Alexander grading by $r$. Also define $\partial_w := \sum_{i\geq 0}\partial^i$ and $\partial_z := \sum_{i\leq 0} \partial^i$.

$C$ has two filtrations, $C(\geq s):=\bigoplus_{r\geq s} C(r) $ and $C(\leq s) :=\bigoplus_{r\leq s} C(r)$, preserved by $\partial_w$ and $\partial_z$ respectively. We primarily use $(C(\leq s), \partial_w)$, defined to be the quotient of $(C, \partial_w)$ by subcomplex $(C(\geq s+1), \partial_w)$. Similarly, $(C(\geq s), \partial_z):= (C, \partial_z)/(C(\leq s-1), \partial_z)$.

\subsection{Bordered Theory}
The usage of bordered theory in this paper follows \cite{bord}, primarily Chapter 11.



\section{Proof of \cref{main} }\label{sec:proof_total}

The structure of this section is as follows. In \cref{sec:proof_idea} and \cref{sec:himodule}, we lay the ground work for proving \cref{main}, where we argue that \cref{flip} implies \cref{main}. In \cref{sec:proof}, we first prove \cref{flip} with an extra mild assumption, and then prove it in general in \cref{sec:proofg}.

\subsection{Main idea of the proof}\label{sec:proof_idea}

The proof of \cref{main} relies on the fact that the $\widehat{CFD}$ invariant of a knot complement can be extracted from $CFK^-$ of the knot in a purely algorithmic fashion. We include two theorems describing this process from \cite{bord}. Specifically, \cref{alg2} is the the general procedure, while \cref{alg} gives us a simpler algorithm by specifying bases of $CFK^-$. \cref{alg} will help us arrive at a simpler proof of \cref{main} at the expense of one extra assumption in this section. In \cref{sec:proofg}, we prove \cref{main} without the assumption using \cref{alg2}.

We interpret the result in terms of arrows for the convenience of our later proof.

\begin{thm}[11.35, A.11 \cite{bord}]\label{alg2}
Let $K\subset S^3$ be a knot with meridian $\mu$ and 0-framed longitude $\lambda$. Given a large enough positive integer $n$, $\widehat{CFD}(S^3\setminus nbd(K))$, with framing $-n$, denoted by $\widehat{CFD}$ for convenience, can be described by the following.
\[
\iota_0\widehat{CFD} := \bigoplus_{s\in \mathbb{Z}} V^0_s ~~~~~~~~ \text{and} ~~~~~~~~\iota_1\widehat{CFD} := \bigoplus_{s\in\mathbb{Z}+\frac{n+1}{2}} V^1_s,
\]
where
\[
V^0_s:=C(s)
\]
\[
V^1_s:=
  \begin{cases}
   C(\leq s+\frac{n-1}{2}) & \text{if } s \leq -\frac{n}{4} \\
   \mathbb{F}_2           & \text{if } |s| < \frac{n}{4}\\
   C(\geq s-\frac{n-1}{2}) & \text{if } s \geq \frac{n}{4}
  \end{cases}
\]
The differentials are described below:
\begin{itemize}
\item Arrows with only idempotents:
  \begin{itemize}
  \item within $V^0_s=C(s)$, they are $\partial^0$ from the knot complex; 
  \item within $V^1_s=C(\leq s+\frac{n-1}{2})$, when $s \leq -\frac{n}{4}$, they are the same as $\partial_w$;
  \item within $V^1_s=C(\geq s-\frac{n-1}{2})$, when $s \geq \frac{n}{4}$, they are the same as $\partial_z$;
  \item within $V^1_s=\mathbb{F}_2 $, when $|s| < \frac{n}{4}$, there are none.
  \end{itemize}
  
\item Arrows with $\rh{1}$:
  \begin{itemize}
  \item from $V^0_s=C(s)$ to $V^1_{s+\frac{n-1}{2}} = C(\geq s)$, they are the same as the inclusion of the subcomplex.
  \end{itemize}

\item Arrows with $\rh{2}$:
  \begin{itemize}
  \item from $V^1_s=C(\leq s+\frac{n-1}{2})$ to $V^0_{s+\frac{n+1}{2}} = C( s+\frac{n+1}{2})$, they are the same as the composition of maps: 
  \[
      \pi\circ\partial_w:C(\leq s+\frac{n-1}{2}) \to C(s+\frac{n+1}{2}),
  \]
  where $\pi:C\to C(s+\frac{n+1}{2})$ is the projection.
  \end{itemize}
  
  \item Arrows with $\rh{3}$:
  \begin{itemize}
  \item from $V^0_s=C(s)$ to $V^1_{s-\frac{n-1}{2}} = C(\leq s)$, they are the same as the inclusion of the subcomplex.
  \end{itemize}
  
  \item There are no arrows with $\rh{12}$.
  
  \item Arrows with $\rh{23}$:
  \begin{itemize}
  \item for $s< -\frac{n}{4}$, from $V^1_s=C(\leq s+\frac{n-1}{2})$ to $V^1_{s+1} = C(\leq s+\frac{n+1}{2})$, they are the same as the inclusion of the subcomplex;
  \item for $s\leq -\frac{n}{4}<s+1$, from $V^1_s=(C,\partial_w)$ to $V^1_{s+1}=\mathbb{F}_2$, they are the same as a chain map inducing an isomorphism in homology;
  \item for $|s|< \frac{n-2}{4}$, from $V^1_s=\mathbb{F}_2$ to $V^1_{s+1}=\mathbb{F}_2$, it is the unique isomorphism;
  \item for $s< \frac{n}{4}\leq s+1$, from $V^1_s=\mathbb{F}_2$ to $V^1_{s+1}=(C,\partial_z)$, they are the same as a chain map inducing an isomorphism in homology;
  \item for $s>\frac{n}{4}$, from $V^1_s=C(\geq s-\frac{n-1}{2})$ to $V^1_{s+1} = C(\geq s-\frac{n+1}{2})$, they are the same as the projection map.
  \end{itemize}
  
  \item Arrows with $\rh{123}$:
  \begin{itemize}
  \item from $V^0_s=C(s)$ to $V^1_{s+\frac{n+1}{2}}=$, they are the same as the composition of maps: 
  \[
     i \circ \partial_w^1:C(s) \to (C(\geq s+1),\partial_z),
  \]
  where $i:C(s+1)\to (C(\geq s+1),\partial_z)$ is the inclusion.
  \end{itemize}
\end{itemize}
\end{thm}

The following is the same procedure described in terms of bases of $CFK^-(K)$, which turns out to be more intuitive.

\begin{thm}[11.27, A.11 \cite{bord}]\label{alg}
Given a knot $K \subset S^3$, let $X$ be the knot complement with boundary parametrization $\phi: T^2 \to \partial X$ corresponding to an integral framing $n$. Let $CFK^-(K)$ be a reduced model for the knot Floer complex of $K$ admitting a basis $\{\xi_i\}$ which is simultaneously vertically and horizontally simplified. Suppose $\xi_v$ (respectively $\xi_h$) is the generator of $CFK^-(K)$ which has no in-coming or out-going vertical (respectively horizontal) arrows. Apply the following procedure to get $\widehat{CFD}(X, \phi)$:

$\iota_0\widehat{CFD}(X, \phi)$ is identified with $CFK^-(K)$ as an $\mathbb{F}_2$-module. Denote the corresponding generators of $\iota_0\widehat{CFD}(X, \phi)$ by $\xi_1', \xi_2', \dots$.

For each vertical arrow of length $\ell$ from $\xi_i$ to $\xi_{i+1}$, we associate a string of basis elements $\kappa_{1}^{i}, \kappa_{2}^{i}, \dots, \kappa_{\ell}^{i} $ for $\iota_1\widehat{CFD}(X, \phi)$ and differentials among them: 
\[
\xi_i'\xrightarrow{\rho_1}\kappa^{i}_{1}\xleftarrow{\rho_{23}}\cdots\xleftarrow{\rho_{23}}\kappa^{i}_{k}\xleftarrow{\rho_{23}}\kappa^{i}_{k+1}\xleftarrow{\rho_{23}}\cdots\xleftarrow{\rho_{23}}\kappa^{i}_{\ell}\xleftarrow{\rho_{123}}\xi_{i+1}'
\] 

For each horizontal arrow of length $\ell$ from $\xi_j$ to $\xi_{j+1}$, we associate a string of basis elements $\lambda_{1}^{j}, \lambda_{2}^{j}, \dots, \lambda_{\ell}^{j} $ for $\iota_1\widehat{CFD}(X, \phi)$ and differentials among them: 
\[
\xi_j'\xrightarrow{\rho_3}\lambda^{j}_{1}\xrightarrow{\rho_{23}}\cdots\xrightarrow{\rho_{23}}\lambda^{j}_{k}\xrightarrow{\rho_{23}}\lambda^{j}_{k+1}\xrightarrow{\rho_{23}}\cdots\xrightarrow{\rho_{23}}\lambda^{j}_{\ell}\xrightarrow{\rho_{2}}\xi_{j+1}'
\] 

We include one more string of generators and differentials called the unstable chain, depending on framing $n$. When $n<2\tau(K)$, it takes form:
\[
\xi_v'\xrightarrow{\rho_1}\mu_{1}\xleftarrow{\rho_{23}}\cdots\xleftarrow{\rho_{23}}\mu_{k}\xleftarrow{\rho_{23}}\mu_{k+1}\xleftarrow{\rho_{23}}\cdots\xleftarrow{\rho_{23}}\mu_{m}\xleftarrow{\rho_{3}}\xi_{h}',
\] 
where $m=2\tau(K)-n$. When $n=2\tau(K)$, it takes form: 
\[
\xi_v'\xrightarrow{\rho_{12}}\xi_{h}'.
\] 
When $n>2\tau(K)$, it takes form, 
\[
\xi_v'\xrightarrow{\rho_{123}}\mu_{1}\xrightarrow{\rho_{23}}\cdots\xrightarrow{\rho_{23}}\mu_{k}\xrightarrow{\rho_{23}}\mu_{k+1}\xrightarrow{\rho_{23}}\cdots\xrightarrow{\rho_{23}}\mu_{m}\xrightarrow{\rho_{2}}\xi_{h}',
\] 
where $m=n-2\tau(K)$.
\end{thm}

We then layout some basic definitions for our proof.

\begin{defi}
Denote the type $D$ module resulted from applying the procedures in \cref{alg2} or \cref{alg} to a complex $C$ by $KtD(C)$, standing for $CF\textbf{K}^-$ \textbf{t}o $\widehat{CF\textbf{D}}$. 
\end{defi}

\begin{defi}\label{flipdef}
Given a model $C$ for $CFK^-$ of a knot $K \subset S^3$ endowed with Alexander filtration $A$ and Maslov grading $M$, we defined the flipped complex $C^{flip}$ resulted from switching horizontal and vertical arrows of $C$ as follows:
\begin{itemize}
\item[] We take $\cf \cong C$ as a $\mathbb{F}_2[U]$-module, i.e. there is an isomorphism $C \to \cf$ taking a generator $x$ of $C$ to a generator $\fl{x}$ of $\cf$.
\item[] Define a Maslov grading on $\cf$ by $M(\fl{x}) = M(x)-2A(x)$
\item[] Define an Alexander filtration on $\cf$ by $A(\fl{x}) = -A(x)$.
\item[] For each arrow in $C$ of the form $x\to U^r y$ with $A(x)-A(y)=s$, we assign an arrow in $\cf$ of the form $\fl{x}\to U^{r+s}\fl{y}$.
\end{itemize}
\end{defi}

\begin{lemma}\label{fliplemma}
$\cf$ constructed in \cref{flipdef} is indeed a chain complex and all horizontal and vertical arrows are switched. Furthermore, if $C$ is computed directly from a Heegaard diagram $(\Sigma, \alpha, \beta, w, z)$, then $\cf$ is isomorphic to the knot Floer complex $C'$ computed from $(\Sigma, \alpha, \beta, z, w)$ up to overall Maslov grading and Alexander filtration shifts.
\end{lemma}
\begin{proof}
We begin by observing that arrow $x\to U^r y$ with $A(x)-A(y)=s$ drops Alexander filtration by $A(x)-A(U^r y) = r+s$ and $U$-filtration by $r$, by \cref{masalex}. The new arrow in $\cf$ of form $\fl{x}\to U^{r+s}\fl{y}$ drops Alexander filtration by $A(\fl{x})-A(U^{r+s}\fl{y})=-(A(x)-A(y))+r+s=r$, and the $U$-filtration by $r+s$. Graphically, the arrow in $C$ represented by a vector $(-r-s,-r)$ is now assigned an arrow in $\cf$ represented by $(-r,-r-s)$. For the Maslov grading, the new arrow drops it by \begin{align*}
\begin{split}
M(\fl{x})-M(U^{r+s}\fl{y})&=M(\fl{x})-M(\fl{y})+2r+2s\\
    &=(M(x)-M(y))-2(A(x)-A(y))+2r+2s\\
    &=M(x)-M(y)+2r=M(x)-M(U^ry)=1,
\end{split}
\end{align*}
which means the Maslov grading on $\cf$ is respected by the differential. $\partial^2 =0 $ is obvious as it is equivalent to there being even number of two-step paths from any given generator to any other generator in $\cf$. This is true for $C$ and graphically arrows in $\cf$ are just those in $C$ flipped across the diagonal.

Now suppose $C$ is computed directly from a Heegaard diagram $(\Sigma, \alpha, \beta, w, z)$. We denote the knot Floer complex computed from $(\Sigma, \alpha, \beta, z, w)$ by $C'$. Generators of $C'$, which we denote by $x', y',$ etc., are in one-to-one correspondence with generators $x,y,$ etc. of $C$, hence matching those of $\cf$. 

For each arrow in $C$ of form $x\to U^r y$ with $A(x)-A(y)=s$, represented by $B\in \tilde{\pi}_2(x,y)$, it follows from \cref{masalex} that $s=A(x)-A(y)=n_z(B)-n_w(B)$ and $M(x)-M(y)=\text{ind}(B)-2n_w(B)$. As $C'$ is computed from the same Heegaard diagram except with interchanged basepoints, the same holomorphic disk also contributes in $C'$, denoted by $B'\in \tilde{\pi}_2(x',y')$ satisfying $n_w(B')=n_z(B)$ and $n_z(B')=n_w(B)$. Now we have 
\begin{align}\label{tempeq1}
\begin{split}
A(x')-A(y')&=n_z(B')-n_w(B')=n_w(B)-n_z(B)=-(A(x)-A(y))\\
M(x')-M(y')&=\text{ind}(B')-2n_w(B')=\text{ind}(B)-2n_z(B)\\
           &=\text{ind}(B)-2n_w(B)-2(n_z(B)-n_w(B))\\
           &=M(x)-M(y)-2(A(x)-A(y))
\end{split}
\end{align}
Note $n_w(B')=n_z(B)=n_w(B)+s=r+s$, so there exists an arrow in $C'$ of form $x'\to U^{r+s}y'$ with $A(x')-A(y')=-s$, which agrees with our construction of $\cf$.

It is known that the values of $A(x')-A(y')$ and $M(x')-M(y')$ for all $x', y'$ and $B'\in \tilde{\pi}_2(x',y')$ completely determine the Maslov grading and Alexander filtration of $C'$ up to overall translations. So the fact that our construction of $A$ and $M$ on $\cf$ satisfies \cref{tempeq1} let us conclude that $\cf$ and $C'$ are indeed isomorphic.
\end{proof}

By Theorem 11 in \cite{bimodule}, $\widehat{CFD} (X, \phi\circ h) =\widehat{CFDA}(h) \boxtimes \widehat{CFD} (X, \phi)$. In order to prove \cref{main}, we need to compute $\widehat{CFDA}(h) \boxtimes \widehat{CFD} (X, \phi)$. To do this, we prove the following proposition, which implies \cref{main} immediately:
\begin{prop}\label{flip}
Given a knot $K \subset S^3$, let $X$ be the knot complement with boundary parametrization $\phi: T^2 \to \partial X$ corresponding to a integral framing $n\leq2\tau(K)-3$. Let $C$ be a reduced and horizontally simplified model for the knot Floer complex $CFK^-(K)$. Then $\widehat{CFDA}(h) \bts \widehat{CFD} (X, \phi)$ is isomorphic to the type $D$ module $KtD(C^{flip})$ resulted from applying the procedure in \cref{alg} to $C^{flip}$, i.e.
\[
\widehat{CFDA}(h) \bts \widehat{CFD} (X, \phi) \simeq KtD(C^{flip}).
\]
\end{prop}


\cref{flip} implies \cref{main} as follows.
\begin{proof}[Proof of \cref{main} assuming \cref{flip}]
We fix a doubly pointed Heegaard diagram $(\Sigma, \alpha, \beta, w, z)$ for knot $K\subset S^3$ and denote by $CFK^-(K)$ the knot Floer complex computed from the diagram. $(\Sigma, \alpha, \beta, z, w)$ is a doubly pointed Heegaard diagram for knot $-K\subset S^3$. By \cref{fliplemma}, switching basepoints has exactly the same effect as switching the horizontal and vertical arrows, i.e. $CFK^-(-K)=CFK^-(K)^{flip}$. It is known that $CFK^-(K)\simeq CFK^-(-K)$ from \cite{HF}, so $CFK^-(K)\simeq CFK^-(K)^{flip}$. Next, we find a reduced and horizontally simplified model $C$ of $CFK^-(K)$. The process of reducing and horizontally simplifying $CFK^-(K)$ can be flipped to reduce and vertically simplify $CFK^-(K)^{flip}$ to $C^{flip}$. As a result, $CFK^-(K) \simeq CFK^-(K)^{flip} \simeq C^{flip}$, which means $C^{flip}$ is another model for the knot Floer complex $CFK^-(K)$ of knot $K$, so \cref{alg} implies that the type $D$ module $KtD(C^{flip})$ that results from applying the procedure (with framing $n$) to $C^{flip}$ is homotopy equivalent to $\widehat{CFD}(X,\phi)$, i.e. $KtD(C^{flip})\simeq \widehat{CFD}(X,\phi)$. Hence for this small enough framing $n$, \cref{flip} implies \cref{main}. For a general framing, we obtain $\widehat{CFD}(X, \phi')$ for the desired framing by tensoring $\widehat{CFD}(X,\phi)$ with $\widehat{CFDA}(\tau_{\mu})$ repeatedly. \cref{main} is now proved by the fact $\widehat{CFDA}(h)\bts\widehat{CFDA}(\tau_{\mu})\cong\widehat{CFDA}(\tau_{\mu})\bts\widehat{CFDA}(h)$, as $h\circ\tau_{\mu} = \tau_{\mu}\circ h$ as topological maps.


\end{proof}

Now all there is left to do is to prove \cref{flip}.

\subsection{Type $DA$ Module $H$ for the elliptic involution}\label{sec:himodule}

The elliptic involution on a torus can be decomposed into $(\tau_{\mu}\circ \tau_{\lambda})^3$, where $\tau_{\mu}$ and $\tau_{\lambda}$ are Dehn twists along a meridian and a longitude, respectively \cite{MCG}. $\widehat{CFDA}(\tau_{\mu})$ is generated by $p, q, r$ and its non-trivial algebra actions are given as follows \cite{bord}:

\begin{equation*}
\begin{aligned}[c]
m_{0,1,1}(p,\rho_1)&=\rho_1\otimes q \\
m_{0,1,1}(p,\rho_{123})&=\rho_{123}\otimes q \\
m_{0,1,2}(p,\rho_{3},\rho_{23})&=\rho_{3}\otimes q\\
m_{0,1,1}(q,\rho_{23})&=\rho_{23}\otimes q \\
m_{0,1,1}(r,\rho_{3})&= q
\end{aligned}
\qquad
\begin{aligned}[c]
m_{0,1,1}(p,\rho_{12})&=\rho_{123}\otimes r \\
m_{0,1,2}(p,\rho_{3},\rho_{2})&=\rho_{3}\otimes r\\
m_{0,1,1}(q,\rho_{2})&=\rho_{23}\otimes r \\
m_{0,1,0}(r)&=\rho_{2}\otimes p
\end{aligned}
\end{equation*}

$\widehat{CFDA}(\tau_{\lambda})$ is generated by $p, q, s$ and its non-trivial algebra actions are given as follows \cite{bord}:

\begin{equation*}
\begin{aligned}[c]
m_{0,1,2}(q,\rh{2}, \rh{1})&=\rh{2}\otimes s \\
m_{0,1,2}(q,\rho_2, \rho_{123})&=\rho_{23}\otimes q \\
m_{0,1,1}(p,\rho_{12})&=\rho_{12}\otimes p\\
m_{0,1,1}(p,\rho_{3})&=\rho_{3}\otimes q \\
m_{0,1,1}(s,\rho_{2})&= p
\end{aligned}
\qquad
\begin{aligned}[c]
m_{0,1,2}(q,\rho_2,\rho_{12})&=\rho_{2}\otimes p \\
m_{0,1,1}(p,\rho_{1},\rho_{2})&=\rho_{12}\otimes s\\
m_{0,1,1}(p,\rho_{123})&=\rho_{123}\otimes q \\
m_{0,1,0}(s)&=\rho_{1}\otimes q\\
m_{0,1,1}(s,\rho_{23})&=\rho_{3}\otimes q 
\end{aligned}
\end{equation*}


By Theorem 12 in \cite{bimodule}, $H:=\widehat{CFDA}(h)$ can be computed as $\widehat{CFDA}(\tau_{\mu})\boxtimes \widehat{CFDA}(\tau_{\lambda}) \boxtimes \widehat{CFDA}(\tau_{\mu})\boxtimes \widehat{CFDA}(\tau_{\lambda}) \boxtimes \widehat{CFDA}(\tau_{\mu})\boxtimes \widehat{CFDA}(\tau_{\lambda})$.

We omit the computation of the tensor product, but instead describe the type $DA$ bimodule $H$ that results from canceling the following arrows in the following order:

$$p\otimes p\otimes p\otimes s\otimes r\otimes p\to p\otimes p\otimes p\otimes p\otimes p\otimes p, $$
$$p\otimes p\otimes p\otimes s\otimes r\otimes s\to p\otimes p\otimes p\otimes p\otimes p\otimes s, $$
$$p\otimes s\otimes r\otimes s\otimes q\otimes q\to p\otimes p\otimes p\otimes s\otimes q\otimes q, $$
$$p\otimes s\otimes r\otimes s\otimes r\otimes p\to p\otimes s\otimes r\otimes p\otimes p\otimes p, $$
$$p\otimes s\otimes r\otimes s\otimes r\otimes s\to p\otimes s\otimes r\otimes p\otimes p\otimes s, $$
$$q\otimes q\otimes r\otimes s\otimes r\otimes p\to q\otimes q\otimes r\otimes p\otimes p\otimes p, $$
$$q\otimes q\otimes r\otimes s\otimes r\otimes s\to q\otimes q\otimes r\otimes p\otimes p\otimes s, $$
$$r\otimes p\otimes p\otimes s\otimes r\otimes p\to r\otimes p\otimes p\otimes p\otimes p\otimes p, $$
$$r\otimes p\otimes p\otimes s\otimes r\otimes s\to r\otimes p\otimes p\otimes p\otimes p\otimes s, $$
$$r\otimes s\otimes q\otimes q\otimes r\otimes s\to q\otimes q\otimes r\otimes s\otimes q\otimes q, $$
$$r\otimes s\otimes r\otimes s\otimes q\otimes q\to r\otimes p\otimes p\otimes s\otimes q\otimes q, $$
$$r\otimes s\otimes r\otimes s\otimes r\otimes p\to r\otimes s\otimes r\otimes p\otimes p\otimes p, $$
$$r\otimes s\otimes r\otimes s\otimes r\otimes s\to r\otimes s\otimes r\otimes p\otimes p\otimes s. $$

For simplicity, we rename the remaining generators
$r\otimes s\otimes q\otimes q\otimes r\otimes p,
q\otimes q\otimes q\otimes q\otimes r\otimes p,
p\otimes s\otimes q\otimes q\otimes r\otimes p,
q\otimes q\otimes q\otimes q\otimes q\otimes q,
p\otimes s\otimes q\otimes q\otimes r\otimes s,
p\otimes s\otimes q\otimes q\otimes q\otimes q,
r\otimes s\otimes q\otimes q\otimes q\otimes q,
q\otimes q\otimes q\otimes q\otimes r\otimes s$ with $x_3, x_1, x_2, v, u, y_3, y_2, y_1$, respectively. Generators $x_1, x_2, x_3$ have idempotent $\iota_0$ and generators $u, v, y_1, y_2, y_3$ have idempotent $\iota_1$. Now $H$ can be described as in \cref{fig:he}.

Throughout this paper, we adapt the following color code for arrows: We use color blue, green, blue, red, and pink for $D-$side arrows 
with $\rh{1}, \rh{2}, \rh{3}, \rh{23},$ and $\rh{123}$ respectively, and for $A-side$ arrows with $A^\infty$ action of $\rh{1}, \rh{2}, \rh{3}, \rh{23},$ and $\rh{123}$ respectively. Finally, when tensoring arrows with the same color from both sides, we generally keep the color for the resulted arrow in the tensor product.

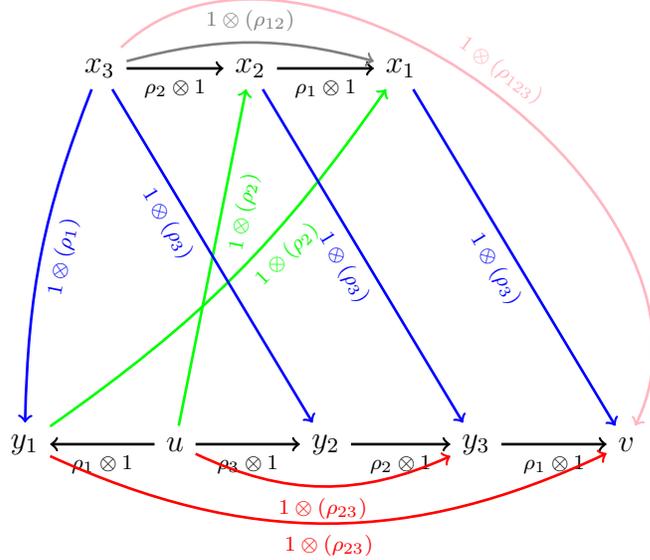
\begin{figure}[t]
\centering
\begin{tikzpicture}[node distance=2cm, line width=1pt, auto] 
  \node (x3) {$x_3$};
  \node (x2) [right of=x3] {$x_2$};
  \node (x1) [right of=x2] {$x_1$};
  \node (y1) [node distance=5cm, left of=x3, below of=x1] {$y_1$};
  \node (u) [right of=y1] {$u$};
  \node (y2) [right of=u] {$y_2$};
  \node (y3) [right of=y2] {$y_3$};
  \node (v) [right of=y3] {$v$};
  \draw[->, font=\scriptsize] (x3) to node [below] {$\rho_2\otimes 1$} (x2);  
  \draw[->, font=\scriptsize] (x2) to node [below] {$\rho_1\otimes 1$} (x1);  
  \draw[->, bend left=15, gray, font=\scriptsize] (x3) to node {$1\otimes (\rho_{12})$} (x1);
  \draw[->, font=\scriptsize] (u) to node [below] {$\rho_1\otimes 1$} (y1);
  \draw[->, font=\scriptsize] (u) to node [below] {$\rho_3\otimes 1$} (y2);
  \draw[->, font=\scriptsize] (y2) to node [below] {$\rho_2\otimes 1$} (y3);
  \draw[->, font=\scriptsize] (y3) to node [below] {$\rho_1\otimes 1$} (v);
  \draw[->, bend right=25, red, font=\scriptsize] (y1) to node [below] {$1\otimes (\rho_{23})$} (v);
  \draw[->, bend right=25, red, font=\scriptsize] (u) to node [below] {$1\otimes (\rho_{23})$} (y3);
  \draw[->, bend right=10, blue, font=\scriptsize] (x3) to node [below, sloped] {$1\otimes (\rho_{1})$} (y1);
  \draw[->, bend right=10, green, font=\scriptsize] (y1) to node [below right, sloped] {$1\otimes (\rho_{2})$} (x1);
  \draw[->, bend right=0, green, font=\scriptsize] (u) to node [below right, sloped] {$1\otimes (\rho_{2})$} (x2);
  \draw[->, bend right=0, blue, font=\scriptsize] (x3) to node [below left, sloped] {$1\otimes (\rho_{3})$} (y2);
  \draw[->, bend right=0, blue, font=\scriptsize] (x2) to node [below, sloped] {$1\otimes (\rho_{3})$} (y3);
  \draw[->, bend right=0, blue, font=\scriptsize] (x1) to node [below, sloped] {$1\otimes (\rho_{3})$} (v);
  \draw[->, bend left=80, pink, font=\scriptsize] (x3) to node [above, sloped] {$1\otimes (\rho_{123})$} (v);
\end{tikzpicture}
\caption{Arrows are color coded according to their $A^{\infty}$ actions.} \label{fig:he}
\end{figure}

\subsection{Proof of the Case of simplified $CFK^-$}\label{sec:proof}

We first prove \cref{flip} with an extra assumption. See \cref{sec:proofg} for the general proof. In this section, we assume that we have a model $C$ of $CFK^-(K)$ which is simultaneously horizontally and vertically simplified. We remark that even though there exist examples of filtered chain complexes that can not be simultaneously horizontally and vertically simplified, we are not aware of any examples that can be realized as $CFK^-$ of a knot.

Recall our notation, $H = \widehat{CFDA}(h)$, where $h$ is the elliptic involution and $KtD(C)$ (respectively, $KtD(C^{flip})$) is the type $D$ module resulted from applying the procedure in \cref{alg} to $C$(respectively, $C^{flip}$). 

\FloatBarrier
\subsubsection{A string of $\rh{23}$'s}\label{sec:string}

According to \cref{alg}, the middle part of each chain of arrows in $KtD(C)$ is a (possibly empty) string of $\rho_{23}$'s. We first investigate how $H$ acts on a string of $\rho_{23}$'s,
\[
\dots\xrightarrow{\rh{23}}\ast\xrightarrow{\rh{23}}\ast\xrightarrow{\rh{23}}\ast\xrightarrow{\rh{23}}\ast\dots.
\] 
The result of the tensor product is shown in \cref{fig:tensor23}, and after we cancel all the red arrows, the result is a string of $\rh{23}$'s going the opposite direction, shown in \cref{fig:string23cancel}. 

We observe that for each vertical arrow $\xi_i \to \xi_j$, the corresponding chain in $KtD(C)$ has a string of $\rh{23}$'s in the \emph{opposite} direction as the original arrow, where by \emph{opposite} we mean pointing toward $\xi_i'$ and away from $\xi_j'$. For each horizontal arrow $\xi_i \to \xi_j$ in $C^{flip}$, the corresponding chain in $KtD(C^{flip})$ has a string of $\rh{23}$'s in the \emph{same} direction as the original arrow, where by \emph{same} we mean pointing away from $\xi_i'$ and toward from $\xi_j'$. The box tensor product of $H$ with a string of $\rh{23}$'s gives exactly the correspondence between the ``middle'' parts of the two corresponding chains.

\begin{figure}[hbt]
\begin{subfigure}{.18\textwidth}
  \centering
  \begin{tikzpicture}[node distance=1.8cm, line width=1pt, auto] 
  \node (a) {};
  \node (b) [node distance=2cm, below of=a] {$\ast$};
  \node (c) [node distance=2cm, below of=b] {$\ast$};
  \node (d) [node distance=2cm, below of=c] {$\ast$};
  \node (e) [node distance=2cm, below of=d] {};
  \draw[->, font=\scriptsize] (a) to node [left] {$\rh{23}$} (b);
  \draw[->, font=\scriptsize] (b) to node [left] {$\rh{23}$} (c);
  \draw[->, font=\scriptsize] (c) to node [left] {$\rh{23}$} (d);
  \draw[->, font=\scriptsize] (d) to node [left] {$\rh{23}$} (e);
  \end{tikzpicture}
  \caption{A string of $\rho_{23}$'s}
  \label{fig:string23}
\end{subfigure}%
\begin{subfigure}{.6\textwidth}
  \centering
  \begin{tikzpicture}[node distance=1.6cm, line width=1pt, auto] 
  \node (y1) {};
  \node (u) [right of=y1] {};
  \node (y2) [right of=u] {};
  \node (y3) [right of=y2] {};
  \node (v) [right of=y3] {};
  \node (ay1) [node distance=2cm, below of=y1] {$y_1\ts\ast$};
  \node (au) [right of=ay1] {$u\ts\ast$};
  \node (ay2) [right of=au] {$y_2\ts\ast$};
  \node (ay3) [right of=ay2] {$y_3\ts\ast$};
  \node (av) [right of=ay3] {$v\ts\ast$};
  \draw[->, font=\scriptsize] (au) to node [below] {$\rho_1$} (ay1);
  \draw[->, font=\scriptsize] (au) to node [below] {$\rho_3$} (ay2);
  \draw[->, font=\scriptsize] (ay2) to node [below] {$\rho_2$} (ay3);
  \draw[->, font=\scriptsize] (ay3) to node [below] {$\rho_2$} (av);
  \node (by1) [node distance=2cm, below of=ay1] {$y_1\ts\ast$};
  \node (bu) [right of=by1] {$u\ts\ast$};
  \node (by2) [right of=bu] {$y_2\ts\ast$};
  \node (by3) [right of=by2] {$y_3\ts\ast$};
  \node (bv) [right of=by3] {$v\ts\ast$};
  \draw[->, font=\scriptsize] (bu) to node [below] {$\rho_1$} (by1);
  \draw[->, font=\scriptsize] (bu) to node [below] {$\rho_3$} (by2);
  \draw[->, font=\scriptsize] (by2) to node [below] {$\rho_2$} (by3);
  \draw[->, font=\scriptsize] (by3) to node [below] {$\rho_2$} (bv);
  \node (cy1) [node distance=2cm, below of=by1] {$y_1\ts\ast$};
  \node (cu) [right of=cy1] {$u\ts\ast$};
  \node (cy2) [right of=cu] {$y_2\ts\ast$};
  \node (cy3) [right of=cy2] {$y_3\ts\ast$};
  \node (cv) [right of=cy3] {$v\ts\ast$};
  \draw[->, font=\scriptsize] (cu) to node [below] {$\rho_1$} (cy1);
  \draw[->, font=\scriptsize] (cu) to node [below] {$\rho_3$} (cy2);
  \draw[->, font=\scriptsize] (cy2) to node [below] {$\rho_2$} (cy3);
  \draw[->, font=\scriptsize] (cy3) to node [below] {$\rho_2$} (cv);
   \node (dy1) [node distance=2cm, below of=cy1] {};
  \node (du) [right of=dy1] {};
  \node (dy2) [right of=du] {};
  \node (dy3) [right of=dy2] {};
  \node (dv) [right of=dy3] {};
  \draw[->, bend right=-5, red, font=\scriptsize] (av) to node [below, left] {$1$} (y1);
  \draw[->, bend right=5, red, font=\scriptsize] (ay3) to node [below, right] {$1$} (u);
    \draw[->, bend right=-5, red, font=\scriptsize] (bv) to node [below, left] {$1$} (ay1);
  \draw[->, bend right=5, red, font=\scriptsize] (by3) to node [below, right] {$1$} (au);
    \draw[->, bend right=-5, red, font=\scriptsize] (cv) to node [below, left] {$1$} (by1);
  \draw[->, bend right=5, red, font=\scriptsize] (cy3) to node [below, right] {$1$} (bu);
    \draw[->, bend right=-5, red, font=\scriptsize] (dv) to node [below, left] {$1$} (cy1);
  \draw[->, bend right=5, red, font=\scriptsize] (dy3) to node [below, right] {$1$} (cu);
  \end{tikzpicture}
  \caption{The tensor product of $H$ with a string of $\rh{23}$ arrows.} \label{fig:tensor23}
\end{subfigure}
\begin{subfigure}{.18\textwidth}
\centering
\begin{tikzpicture}[node distance=2cm, line width=1pt, auto] 
  \node (a) {};
  \node (b) [node distance=2cm, below of=a] {$y_2\ts\ast$};
  \node (c) [node distance=2cm, below of=b] {$y_2\ts\ast$};
  \node (d) [node distance=2cm, below of=c] {$y_2\ts\ast$};
  \node (e) [node distance=2cm, below of=d] {};
  \draw[->, font=\scriptsize, red] (b) to node [left] {$\rh{23}$} (a);
  \draw[->, font=\scriptsize, red] (c) to node [left] {$\rh{23}$} (b);
  \draw[->, font=\scriptsize, red] (d) to node [left] {$\rh{23}$} (c);
  \draw[->, font=\scriptsize, red] (e) to node [left] {$\rh{23}$} (d);
\end{tikzpicture}
\caption{After cancelling all arrows with 1.}
\label{fig:string23cancel}
\end{subfigure}
\caption{$H$ tensoring with a string of $\rh{23}$'s.}
\end{figure}
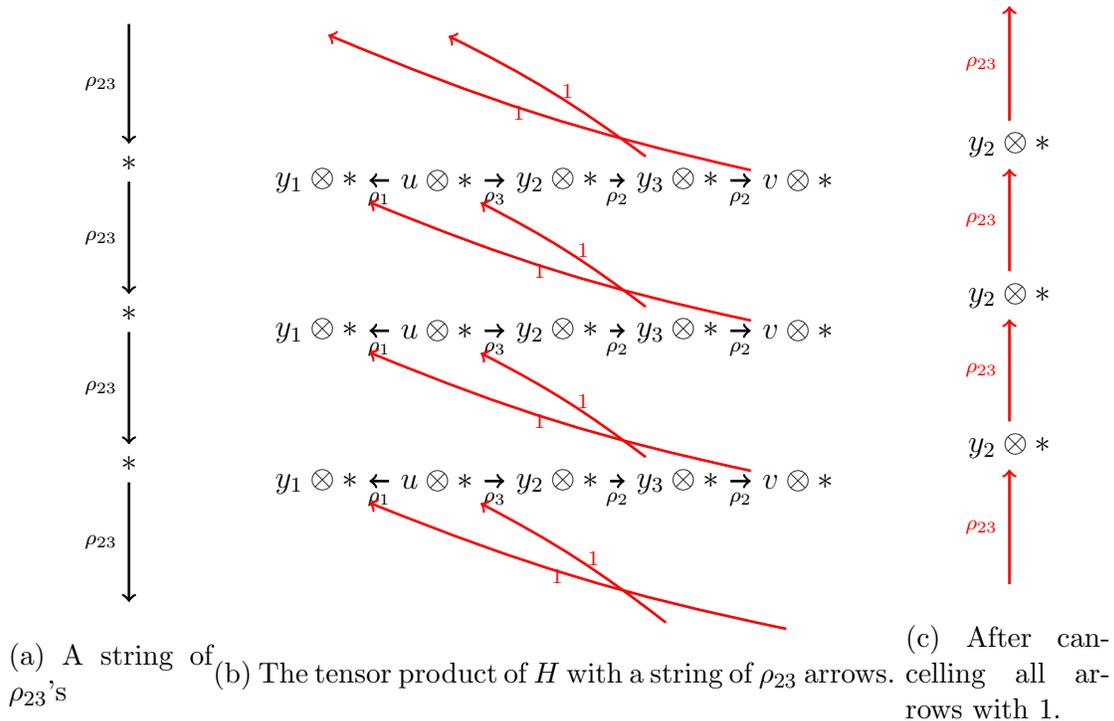

Now we investigate how the joints between these strings of $\rh{23}$'s are also ``flipped'' by tensoring with $H$.

Since $C$ is both horizontally and vertically simplified, at any generator $\xi_i\neq \xi_v \text{ or } \xi_h$, we have the following four kinds of joints in \cref{fig:joints}.

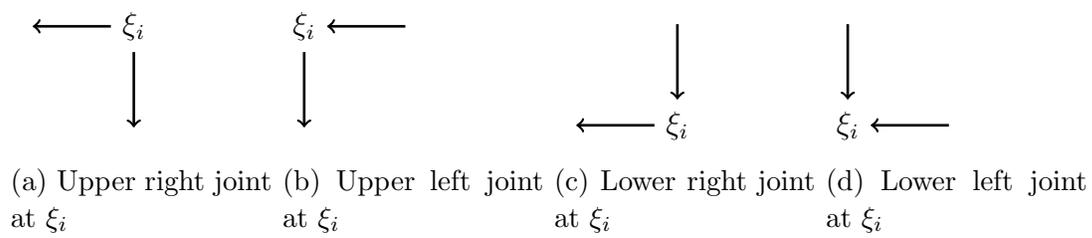
\begin{figure}[hb]
\begin{subfigure}{.23\textwidth}
\begin{tikzpicture}[node distance=1.5cm, line width=1pt, auto]
  \node (x) {\xii};
  \node (y) [below of =x] {};
  \node (z) [left of =x] {};
  \draw[->] (x) to node [above] {} (y);
  \draw[->] (x) to node [above] {} (z);
\end{tikzpicture}
\caption{Upper right joint at \xii}
\end{subfigure}
\begin{subfigure}{.23\textwidth}
\begin{tikzpicture}[node distance=1.5cm, line width=1pt, auto]
  \node (x) {\xii};
  \node (y) [below of =x] {};
  \node (z) [right of =x] {};
  \draw[->] (x) to node [above] {} (y);
  \draw[->] (z) to node [above] {} (x);
\end{tikzpicture}
\caption{Upper left joint at \xii}
\end{subfigure}
\begin{subfigure}{.23\textwidth}
\begin{tikzpicture}[node distance=1.5cm, line width=1pt, auto]
  \node (x) {\xii};
  \node (y) [above of =x] {};
  \node (z) [left of =x] {};
  \draw[->] (y) to node [above] {} (x);
  \draw[->] (x) to node [above] {} (z);
\end{tikzpicture}
\caption{Lower right joint at \xii}
\end{subfigure}
\begin{subfigure}{.23\textwidth}
\begin{tikzpicture}[node distance=1.5cm, line width=1pt, auto]
  \node (x) {\xii};
  \node (y) [above of =x] {};
  \node (z) [right of =x] {};
  \draw[->] (y) to node [above] {} (x);
  \draw[->] (z) to node [above] {} (x);
\end{tikzpicture}
\caption{Lower left joint at \xii}
\end{subfigure}
\caption{Four kinds of joints at $\xi_i \neq \xi_v$ or $\xi_h$}
\label{fig:joints}
\end{figure}

For each kind of joint at \xii, we argue the corresponding joint at \xiip~in $KtD(C)$ can be identified with a joint in $KtD(C^{flip})$ by tensoring with $H$. Our way of doing this is simply through cancellation in $H\bts KtD(C)$ and match the result to the corresponding part of $KtD(C^{flip})$. Then what is left to show $H\bts KtD(C) = KtD(C^{flip})$ is how these cancellation at different parts of $H\bts KtD(C)$ agree with each other.

We organize as follows the proof of that the ``local'' cancellations agree: if an arrow at \xii~in $C$ has length more than 1, then the chain of arrows at \xiip~in $KtD(C)$ is connected to a string of $\rh{23}$'s. We observe how the cancellation at this joint in $H\bts KtD(C)$ agree with the cancellation we perform for a string of $\rh{23}$'s, so the chain of $\rh{23}$'s along with the end point \xiip~is matched to a flipped chain in $KtD(C^{flip})$. If an arrow at \xii~in $C$ has length exactly 1, \xiip~in $KtD(C)$ is connected to another generator $\xi_j'$ through a short chain of two arrows. After we discuss the cancellation \emph{at both} of these joint, we show it agrees and gives the desired result when the two joint are put together. 

\FloatBarrier
\subsubsection{Upper right joint}\label{sec:urj}

For a generator \xii~with the upper right joint in \cref{fig:joints}, it appears in $KtD(C)$ as \cref{fig:j1a}. For either of the two arrows not immediately connecting to \xiip, we have two possibilities, corresponding to original arrows in $C$ having length 1 or having length more than 1. 

To be concise, the corresponding part of the box tensor product $H\bts KtD(C)$ for all four cases are shown in \cref{fig:j1b}, where all four cases share the same generators and arrows in blue. On the left side, the two red arrows (two green arrows, respectively) come from cases where the arrow out of $\lambda_1$ is $\rh{23}$ ($\rh{2}$, respectively). Similarly, on the bottom right, the two red arrows  (one pink arrow, respectively) come from cases where the arrow into $\kappa_1$ is $\rh{23}$ ($\rh{123}$, respectively).

Now cancel the three blue arrows that are pointing leftward and rearrange generators to arrive at \cref{fig:j1c}. Note the joint at $u\ts \kappa_1$ in $H\bts KtD(C)$, which is boxed in \cref{fig:j1c}, can now be identified with the corresponding part of $KtD(C^{flip})$. 

If either of the two possibilities is $\rh{23}$, we can continue to cancel the set of two red arrows in \cref{fig:j1c} in the same way we canceled arrows in the case of a string of $\rh{23}$'s, then arrive at \cref{fig:j1d}. In these cases, cancellation done here agree with those in the case of tensoring with a string of $\rh{23}$'s, so chains containing $\rh{23}$'s at $u\ts\kappa_1$ in $H\bts KtD(C)$ can be identified with chains in $KtD(C^{flip})$, compare with \cref{fig:j1a}. Arrows in cases where either arrow at \xii~have length 1 will be brought into shape after we deal with the joints \xii~is connected to.

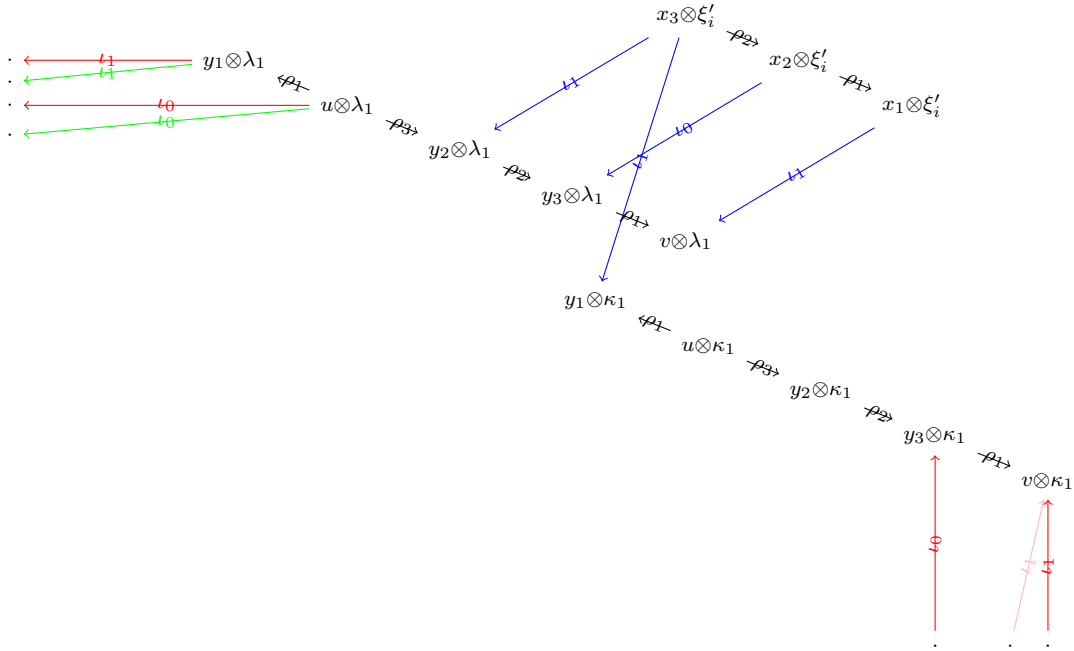
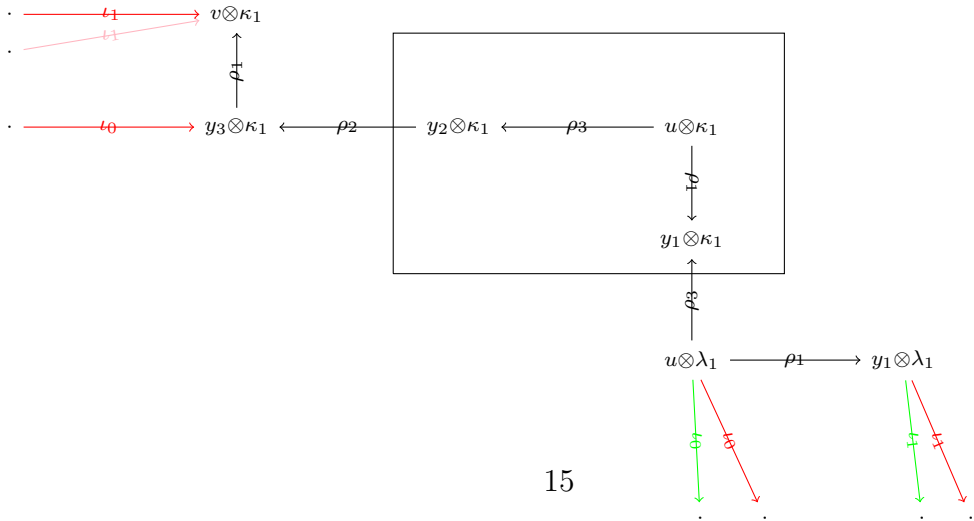
\begin{figure}[ht]
\centering
\begin{subfigure}{0.49\textwidth}
\centering
\begin{tikzpicture}[node distance=1.5cm, auto]
  \node (x) {\xiip};
  \node (k1) [below of =x] {$\kappa_1$};
  \node (l1) [left of =x] {$\lambda_1$};
  \node (k2) [below of =k1] {};
  \node (l2) [left of =l1] {};
  \draw[->] (x) to node [above] {$\rho_3$} (l1);
  \draw[->] (x) to node [right] {$\rho_1$} (k1);
  \draw[->] (l1) to node [above] {$\rho_{23}$} (l2);
  \draw[->] (k2) to node [right] {$\rho_{123}$} (k1);
  \path (l1) to node[label=below:$\rho_2$] {} (l2);
  \path (k2) to node[label=below:$\rho_{23}$] {} (k1);
\end{tikzpicture}
\caption{Upper right joint at \xii}\label{fig:j1a}
\end{subfigure}

\begin{subfigure}{0.99\textwidth}
\begin{tikzpicture}[node distance=2cm]
\path[font = \scriptsize]
(-7.74, 2.06) node(y3lam1l1) {$\cdot$}
(6.06, -2.94) node(vkap1) {$v$$\otimes$$\kappa_{1}$}
(0.05, -0.54) node(y1kap1) {$y_1$$\otimes$$\kappa_{1}$}
(3.05, -1.74) node(y2kap1) {$y_2$$\otimes$$\kappa_{1}$}
(4.56, -5.14) node(ukap1d1) {$\cdot$}
(1.55, -1.14) node(ukap1) {$u$$\otimes$$\kappa_{1}$}
(-0.25, 0.86) node(y3lam1) {$y_3$$\otimes$$\lambda_{1}$}
(-4.75, 2.66) node(y1lam1) {$y_1$$\otimes$$\lambda_{1}$}
(-3.25, 2.06) node(ulam1) {$u$$\otimes$$\lambda_{1}$}
(4.25, 2.06) node(x1xii) {$x_1$$\otimes$$\xi_{i}^\prime$}
(5.56, -5.14) node(x3kap1d2) {$\cdot$}
(4.56, -2.34) node(y3kap1) {$y_3$$\otimes$$\kappa_{1}$}
(-1.75, 1.47) node(y2lam1) {$y_2$$\otimes$$\lambda_{1}$}
(-7.74, 1.66) node(x2lam1l2) {$\cdot$}
(-7.74, 2.66) node(vlam1l1) {$\cdot$}
(2.75, 2.66) node(x2xii) {$x_2$$\otimes$$\xi_{i}^\prime$}
(-7.74, 2.37) node(x1lam1l2) {$\cdot$}
(6.06, -5.14) node(y1kap1d1) {$\cdot$}
(1.25, 0.26) node(vlam1) {$v$$\otimes$$\lambda_{1}$}
(1.25, 3.26) node(x3xii) {$x_3$$\otimes$$\xi_{i}^\prime$}
;
\draw[->, bend right = 0, black, font=\scriptsize] (ukap1) to node [sloped] {$\rho_{1}$} (y1kap1);
\draw[->, bend right = 0, black, font=\scriptsize] (ukap1) to node [sloped] {$\rho_{3}$} (y2kap1);
\draw[->, bend right = 0, red, font=\scriptsize] (ukap1d1) to node [sloped] {$\iota_0$} (y3kap1);
\draw[->, bend right = 0, black, font=\scriptsize] (ulam1) to node [sloped] {$\rho_{1}$} (y1lam1);
\draw[->, bend right = 0, black, font=\scriptsize] (ulam1) to node [sloped] {$\rho_{3}$} (y2lam1);
\draw[->, bend right = 0, red, font=\scriptsize] (ulam1) to node [sloped] {$\iota_0$} (y3lam1l1);
\draw[->, bend right = 0, green, font=\scriptsize] (ulam1) to node [sloped] {$\iota_0$} (x2lam1l2);
\draw[->, bend right = 0, blue, font=\scriptsize] (x1xii) to node [sloped] {$\iota_1$} (vlam1);
\draw[->, bend right = 0, black, font=\scriptsize] (x2xii) to node [sloped] {$\rho_{1}$} (x1xii);
\draw[->, bend right = 0, blue, font=\scriptsize] (x2xii) to node [sloped] {$\iota_0$} (y3lam1);
\draw[->, bend right = 0, pink, font=\scriptsize] (x3kap1d2) to node [sloped] {$\iota_1$} (vkap1);
\draw[->, bend right = 0, blue, font=\scriptsize] (x3xii) to node [sloped] {$\iota_1$} (y1kap1);
\draw[->, bend right = 0, black, font=\scriptsize] (x3xii) to node [sloped] {$\rho_{2}$} (x2xii);
\draw[->, bend right = 0, blue, font=\scriptsize] (x3xii) to node [sloped] {$\iota_1$} (y2lam1);
\draw[->, bend right = 0, red, font=\scriptsize] (y1kap1d1) to node [sloped] {$\iota_1$} (vkap1);
\draw[->, bend right = 0, red, font=\scriptsize] (y1lam1) to node [sloped] {$\iota_1$} (vlam1l1);
\draw[->, bend right = 0, green, font=\scriptsize] (y1lam1) to node [sloped] {$\iota_1$} (x1lam1l2);
\draw[->, bend right = 0, black, font=\scriptsize] (y2kap1) to node [sloped] {$\rho_{2}$} (y3kap1);
\draw[->, bend right = 0, black, font=\scriptsize] (y2lam1) to node [sloped] {$\rho_{2}$} (y3lam1);
\draw[->, bend right = 0, black, font=\scriptsize] (y3kap1) to node [sloped] {$\rho_{1}$} (vkap1);
\draw[->, bend right = 0, black, font=\scriptsize] (y3lam1) to node [sloped] {$\rho_{1}$} (vlam1);
\end{tikzpicture}
\caption{$H$ box tensor product with the joint.}\label{fig:j1b}
\end{subfigure}

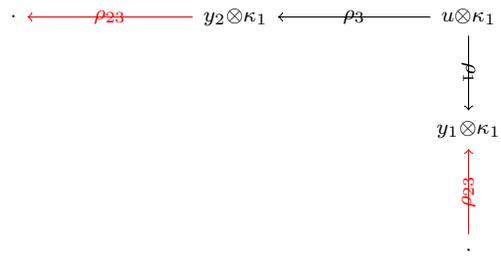
\begin{subfigure}{0.99\textwidth}
\begin{tikzpicture}[node distance=2cm]
\path[font = \scriptsize]
(2.27, -1.35) node(ulam1) {$u$$\otimes$$\lambda_{1}$}
(-3.78, 3.25) node(vkap1) {$v$$\otimes$$\kappa_{1}$}
(3.24, -3.45) node(y3lam1l1) {$\cdot$}
(-6.8, 3.25) node(y1kap1d1) {$\cdot$}
(5.33, -3.45) node(x1lam1l2) {$\cdot$}
(-6.8, 1.75) node(ukap1d1) {$\cdot$}
(-0.83, 1.75) node(y2kap1) {$y_2$$\otimes$$\kappa_{1}$}
(2.27, 1.75) node(ukap1) {$u$$\otimes$$\kappa_{1}$}
(-6.8, 2.75) node(x3kap1d2) {$\cdot$}
(5.08, -1.35) node(y1lam1) {$y_1$$\otimes$$\lambda_{1}$}
(5.98, -3.45) node(vlam1l1) {$\cdot$}
(2.38, -3.45) node(x2lam1l2) {$\cdot$}
(2.27, 0.25) node(y1kap1) {$y_1$$\otimes$$\kappa_{1}$}
(-3.78, 1.75) node(y3kap1) {$y_3$$\otimes$$\kappa_{1}$}
;
\draw[->, bend right = 0, black, font=\scriptsize] (ukap1) to node [sloped] {$\rho_{1}$} (y1kap1);
\draw[->, bend right = 0, black, font=\scriptsize] (ukap1) to node [sloped] {$\rho_{3}$} (y2kap1);
\draw[->, bend right = 0, red, font=\scriptsize] (ukap1d1) to node [sloped] {$\iota_0$} (y3kap1);
\draw[->, bend right = 0, black, font=\scriptsize] (ulam1) to node [sloped] {$\rho_{1}$} (y1lam1);
\draw[->, bend right = 0, red, font=\scriptsize] (ulam1) to node [sloped] {$\iota_0$} (y3lam1l1);
\draw[->, bend right = 0, green, font=\scriptsize] (ulam1) to node [sloped] {$\iota_0$} (x2lam1l2);
\draw[->, bend right = 0, black, font=\scriptsize] (ulam1) to node [sloped] {$\rho_{3}$} (y1kap1);
\draw[->, bend right = 0, pink, font=\scriptsize] (x3kap1d2) to node [sloped] {$\iota_1$} (vkap1);
\draw[->, bend right = 0, red, font=\scriptsize] (y1kap1d1) to node [sloped] {$\iota_1$} (vkap1);
\draw[->, bend right = 0, red, font=\scriptsize] (y1lam1) to node [sloped] {$\iota_1$} (vlam1l1);
\draw[->, bend right = 0, green, font=\scriptsize] (y1lam1) to node [sloped] {$\iota_1$} (x1lam1l2);
\draw[->, bend right = 0, black, font=\scriptsize] (y2kap1) to node [sloped] {$\rho_{2}$} (y3kap1);
\draw[->, bend right = 0, black, font=\scriptsize] (y3kap1) to node [sloped] {$\rho_{1}$} (vkap1);
\draw (-1.7,-0.2) -- (3.5,-0.2) -- (3.5,3) -- (-1.7,3) -- (-1.7,-0.2);
\end{tikzpicture}
\caption{After canceling the three blue arrows pointing left}\label{fig:j1c}
\end{subfigure}

\caption{Upper right joint}
\end{figure}

\begin{figure}[h]
\centering
\begin{subfigure}{0.48\textwidth}
\begin{tikzpicture}[node distance=2cm]
\path[font = \scriptsize]
(2.27, -1.35) node(ulam1) {$\cdot$}
(-0.83, 1.75) node(y2kap1) {$y_2$$\otimes$$\kappa_{1}$}
(2.27, 1.75) node(ukap1) {$u$$\otimes$$\kappa_{1}$}
(2.27, 0.25) node(y1kap1) {$y_1$$\otimes$$\kappa_{1}$}
(-3.78, 1.75) node(y3kap1) {$\cdot$}
;
\draw[->, bend right = 0, black, font=\scriptsize] (ukap1) to node [sloped] {$\rho_{1}$} (y1kap1);
\draw[->, bend right = 0, black, font=\scriptsize] (ukap1) to node [sloped] {$\rho_{3}$} (y2kap1);
\draw[->, bend right = 0, red, font=\scriptsize] (ulam1) to node [sloped] {$\rho_{23}$} (y1kap1);
\draw[->, bend right = 0, red, font=\scriptsize] (y2kap1) to node [sloped] {$\rho_{23}$} (y3kap1);
\end{tikzpicture}
\caption{Cases where either arrow at \xiip~have length more than 1.}\label{fig:j1d}
\end{subfigure}
\caption{Upper right joint}
\end{figure}

\FloatBarrier
\subsubsection{Upper left joint}\label{sec:ulj}

For a generator \xii~with the upper left joint in \cref{fig:joints}, it appears in $KtD(C)$ as \cref{fig:j2a}. For either of the two arrows not immediately connecting to \xiip, we have two possibilities, corresponding to original arrows in $C$ having length 1 or having length more than 1. To be concise, the corresponding part of the box tensor product $H\bts KtD(C)$ for all four cases are shown in \cref{fig:j2b}, where all four cases share the same generators and arrows in the middle. On the right side, the two red arrows (three blue arrows, respectively) come from cases where in \cref{fig:j2a} the arrow going into $\lambda_\ell$ has a $\rh{23}$ ($\rh{3}$, respectively). Similarly, on bottom left, the two red arrows (one pink arrow, respectively) come from cases where the arrow going into $\kappa_1$ has a $\rh{23}$ ($\rh{123}$, respectively).

Then we cancel the blue arrow and the two green arrows in the middle and arrive at \cref{fig:j2c}. Note the joint at $u\ts \kappa_1$, which is boxed in \cref{fig:j2c}, can now be matched exactly to the corresponding part of $KtD(C^{flip})$, which would look like \cref{fig:j3a}. 

If either of the two possibilities is $\rh{23}$, we observe that cancellation of either set of two red arrows in \cref{fig:j2c} agree with the cancellations for a string of $\rh{23}$'s in \cref{sec:string}, which shows chains containing $\rh{23}$'s at $u\ts\kappa_1$ in $H\bts KtD(C)$ (\cref{fig:j2d}) exactly match corresponding chains in $KtD(C^{flip})$, again see \cref{fig:j3a}.

If the horizontal arrow at $\lambda_\ell$ has a $\rh{3}$ coming in, see \cref{fig:j2a}, meaning original arrow out of \xii~has length 1, the $\rh{3}$ must come from another $\xi_j'$. $\xi_j'$ belongs to either a upper right joint or a lower right joint. We defer the latter case to the lower right joint section, \cref{sec:lrj}. In the former case, we combine the tensor product here and the tensor product \cref{fig:j1b} in \cref{sec:urj}. Now $\lambda_\ell=\lambda_1$ and the green arrows in \cref{fig:j2b} and \cref{fig:j1b} are identified, so are the three horizontal blue arrows in \cref{fig:j2b} and \cref{fig:j1b}. Reviewing the cancellation in this section and the upper right joint section, we find they combine together without conflicting and result in \cref{fig:j2e}, which is what the corresponding part in $KtD(C^{flip})$ look like. 
 
\begin{figure}[ht]

\begin{subfigure}{0.99\textwidth}
\centering
\begin{tikzpicture}[node distance=1.5cm]
\path[font = \scriptsize]
(2.57, 1.43) node(lamlr1) {}
(-1.43, -2.57) node(kap1d2) {}
(0.57, 1.43) node(laml) {$\lambda_{l}$}
(-1.43, 1.43) node(xii) {$\xi_{i}^\prime$}
(-1.43, -2.57) node(kap1d1) {}
(-1.43, -0.57) node(kap1) {$\kappa_{1}$}
(2.57, 1.43) node(lamlr2) {}
;
\draw[->, bend right = 0, black, font=\scriptsize] (kap1d1) to node [left] {$\rho_{23}$} (kap1);
\draw[->, bend right = 0, black, font=\scriptsize] (kap1d2) to node [right] {$\rho_{123}$} (kap1);
\draw[->, bend right = 0, black, font=\scriptsize] (laml) to node [above] {$\rho_{2}$} (xii);
\draw[->, bend right = 0, black, font=\scriptsize] (lamlr1) to node [above] {$\rho_{23}$} (laml);
\draw[->, bend right = 0, black, font=\scriptsize] (lamlr2) to node [below] {$\rho_{3}$} (laml);
\draw[->, bend right = 0, black, font=\scriptsize] (xii) to node [left] {$\rho_{1}$} (kap1);
\end{tikzpicture}
\caption{Upper left joint at \xiip}\label{fig:j2a}
\end{subfigure}

\begin{subfigure}{0.99\textwidth}
\centerline{
\begin{tikzpicture}[node distance=2cm]
\path[font = \scriptsize]
(-4.38, 3.56) node(x1xii) {$x_1$$\otimes$$\xi_{i}^\prime$}
(9.42, 1.76) node(y2lamlr1) {$\cdot$}
(4.62, 2.96) node(vlaml) {$v$$\otimes$$\lambda_{l}$}
(10.92, 2.36) node(y3lamlr1) {$\cdot$}
(-5.88, -3.04) node(x2kap1d2) {$\cdot$}
(7.92, 2.36) node(x2lamlr2) {$\cdot$}
(-4.38, -3.64) node(y2kap1d1) {$\cdot$}
(-5.88, -4.24) node(ukap1d1) {$\cdot$}
(-7.38, -3.64) node(x3kap1d2) {$\cdot$}
(-5.88, -1.84) node(ukap1) {$u$$\otimes$$\kappa_{1}$}
(3.12, 2.36) node(y3laml) {$y_3$$\otimes$$\lambda_{l}$}
(7.92, 1.16) node(ulamlr1) {$\cdot$}
(6.42, 1.76) node(x3lamlr2) {$\cdot$}
(-4.38, -2.44) node(x1kap1d2) {$\cdot$}
(-1.38, -0.04) node(vkap1) {$v$$\otimes$$\kappa_{1}$}
(-5.88, 2.96) node(x2xii) {$x_2$$\otimes$$\xi_{i}^\prime$}
(-7.38, -4.84) node(y1kap1d1) {$\cdot$}
(-1.38, -2.44) node(vkap1d1) {$\cdot$}
(6.42, 0.56) node(y1lamlr1) {$\cdot$}
(12.42, 2.96) node(vlamlr1) {$\cdot$}
(-2.88, -0.64) node(y3kap1) {$y_3$$\otimes$$\kappa_{1}$}
(-4.38, -1.24) node(y2kap1) {$y_2$$\otimes$$\kappa_{1}$}
(-2.88, -3.04) node(y3kap1d1) {$\cdot$}
(-7.38, 2.36) node(x3xii) {$x_3$$\otimes$$\xi_{i}^\prime$}
(9.42, 2.96) node(x1lamlr2) {$\cdot$}
(-7.38, -2.44) node(y1kap1) {$y_1$$\otimes$$\kappa_{1}$}
(1.62, 1.76) node(y2laml) {$y_2$$\otimes$$\lambda_{l}$}
(-1.38, 0.56) node(y1laml) {$y_1$$\otimes$$\lambda_{l}$}
(0.12, 1.16) node(ulaml) {$u$$\otimes$$\lambda_{l}$}
;
\draw[->, bend right = 0, black, font=\scriptsize] (ukap1) to node [sloped] {$\rho_{1}$} (y1kap1);
\draw[->, bend right = 0, black, font=\scriptsize] (ukap1) to node [sloped] {$\rho_{3}$} (y2kap1);
\draw[->, bend right = 15.0, red, font=\scriptsize] (ukap1d1) to node [sloped] {$\iota_0$} (y3kap1);
\draw[->, bend right = 0, black, font=\scriptsize] (ulaml) to node [sloped] {$\rho_{1}$} (y1laml);
\draw[->, bend right = 0, black, font=\scriptsize] (ulaml) to node [sloped] {$\rho_{3}$} (y2laml);
\draw[->, bend right = 0, green, font=\scriptsize] (ulaml) to node [sloped] {$\iota_0$} (x2xii);
\draw[->, bend right = 16.76, red, font=\scriptsize] (ulamlr1) to node [sloped] {$\iota_0$} (y3laml);
\draw[->, bend right = 0, blue, font=\scriptsize] (x1lamlr2) to node [sloped] {$\iota_1$} (vlaml);
\draw[->, bend right = 0, blue, font=\scriptsize] (x2lamlr2) to node [sloped] {$\iota_0$} (y3laml);
\draw[->, bend right = 0, black, font=\scriptsize] (x2xii) to node [sloped] {$\rho_{1}$} (x1xii);
\draw[->, bend right = 16.47, pink, font=\scriptsize] (x3kap1d2) to node [sloped] {$\iota_1$} (vkap1);
\draw[->, bend right = 0, blue, font=\scriptsize] (x3lamlr2) to node [sloped] {$\iota_1$} (y2laml);
\draw[->, bend right = 0, blue, font=\scriptsize] (x3xii) to node [sloped] {$\iota_1$} (y1kap1);
\draw[->, bend right = 0, black, font=\scriptsize] (x3xii) to node [sloped] {$\rho_{2}$} (x2xii);
\draw[->, bend right = 15.0, red, font=\scriptsize] (y1kap1d1) to node [sloped] {$\iota_1$} (vkap1);
\draw[->, bend right = 0, green, font=\scriptsize] (y1laml) to node [sloped] {$\iota_1$} (x1xii);
\draw[->, bend right = 0, red, font=\scriptsize] (y1lamlr1) to node [sloped] {$\iota_1$} (vlaml);
\draw[->, bend right = 0, black, font=\scriptsize] (y2kap1) to node [sloped] {$\rho_{2}$} (y3kap1);
\draw[->, bend right = 0, black, font=\scriptsize] (y2laml) to node [sloped] {$\rho_{2}$} (y3laml);
\draw[->, bend right = 0, black, font=\scriptsize] (y3kap1) to node [sloped] {$\rho_{1}$} (vkap1);
\draw[->, bend right = 0, black, font=\scriptsize] (y3laml) to node [sloped] {$\rho_{1}$} (vlaml);
\end{tikzpicture}
}
\caption{Tensor product with $H$.}\label{fig:j2b}
\end{subfigure}

\caption{Upper left joint.}
\end{figure}
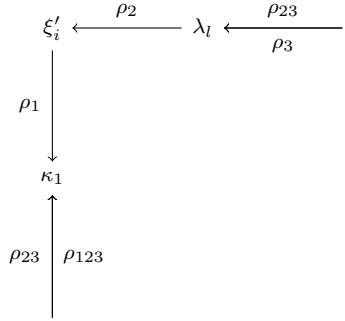
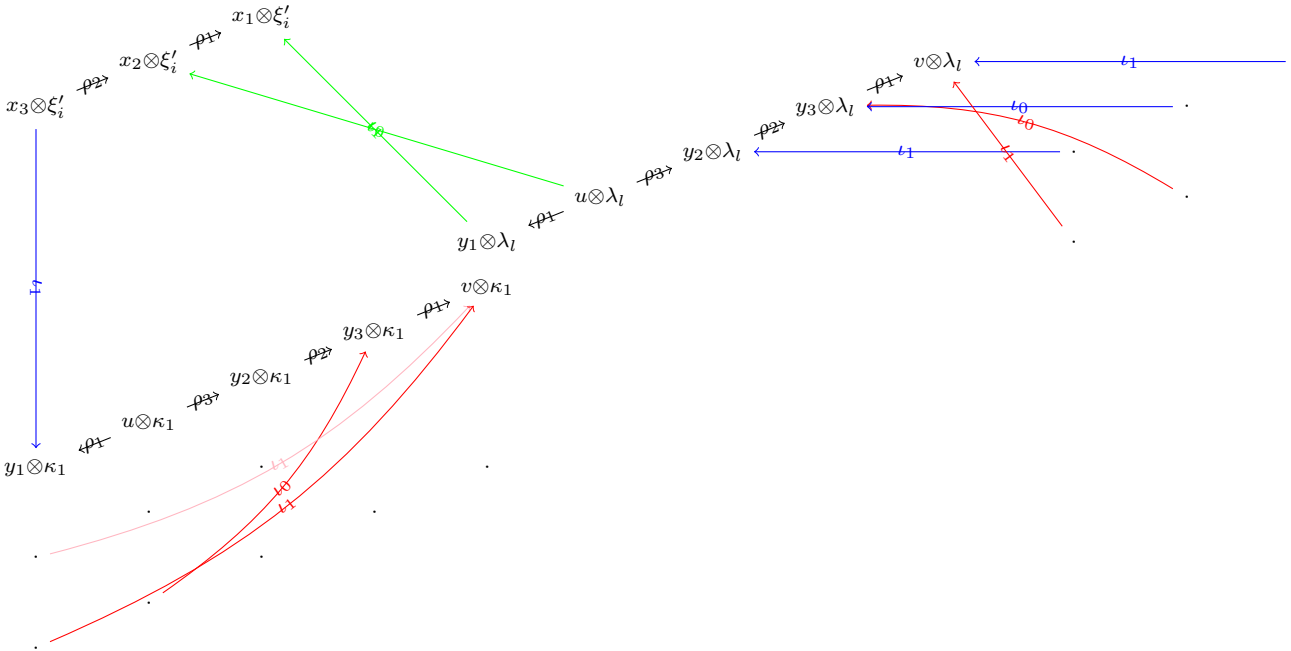

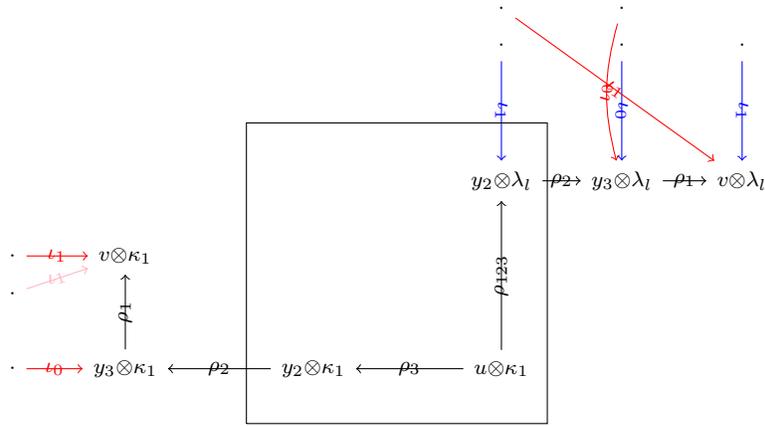
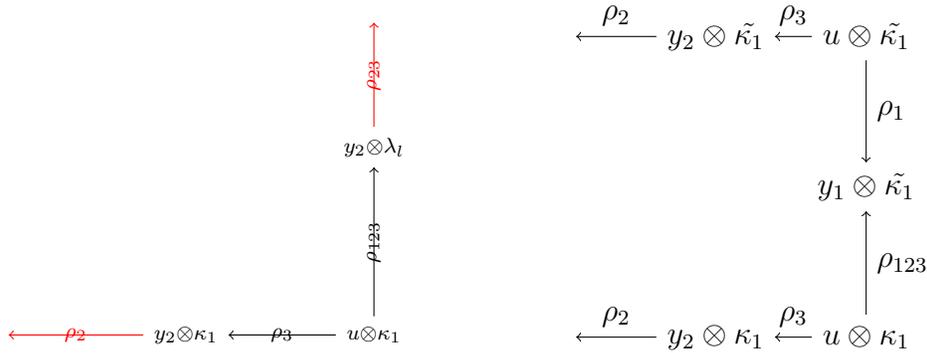
\begin{figure}

\begin{subfigure}{0.99\textwidth}
\centerline{
\begin{tikzpicture}[node distance=2cm]
\path[font = \scriptsize]
(-5.11, -2.27) node(ukap1d1) {$\cdot$}
(1.39, 2.03) node(x3lamlr2) {$\cdot$}
(2.99, 2.03) node(x2lamlr2) {$\cdot$}
(-5.11, -1.27) node(x3kap1d2) {$\cdot$}
(4.59, 2.03) node(x1lamlr2) {$\cdot$}
(2.99, 0.23) node(y3laml) {$y_3$$\otimes$$\lambda_{l}$}
(1.39, -2.27) node(ukap1) {$u$$\otimes$$\kappa_{1}$}
(1.39, 0.23) node(y2laml) {$y_2$$\otimes$$\lambda_{l}$}
(2.99, 2.53) node(ulamlr1) {$\cdot$}
(-1.11, -2.27) node(y2kap1) {$y_2$$\otimes$$\kappa_{1}$}
(1.39, 2.53) node(y1lamlr1) {$\cdot$}
(4.59, 0.23) node(vlaml) {$v$$\otimes$$\lambda_{l}$}
(-3.61, -2.27) node(y3kap1) {$y_3$$\otimes$$\kappa_{1}$}
(-5.11, -0.77) node(y1kap1d1) {$\cdot$}
(-3.61, -0.77) node(vkap1) {$v$$\otimes$$\kappa_{1}$}
;
\draw[->, bend right = 0, black, font=\scriptsize] (ukap1) to node [sloped] {$\rho_{3}$} (y2kap1);
\draw[->, bend right = 0, black, font=\scriptsize] (ukap1) to node [sloped] {$\rho_{123}$} (y2laml);
\draw[->, bend right = 0, red, font=\scriptsize] (ukap1d1) to node [sloped] {$\iota_0$} (y3kap1);
\draw[->, bend right = 15.0, red, font=\scriptsize] (ulamlr1) to node [sloped] {$\iota_0$} (y3laml);
\draw[->, bend right = 0, blue, font=\scriptsize] (x1lamlr2) to node [sloped] {$\iota_1$} (vlaml);
\draw[->, bend right = 0, blue, font=\scriptsize] (x2lamlr2) to node [sloped] {$\iota_0$} (y3laml);
\draw[->, bend right = 0, pink, font=\scriptsize] (x3kap1d2) to node [sloped] {$\iota_1$} (vkap1);
\draw[->, bend right = 0, blue, font=\scriptsize] (x3lamlr2) to node [sloped] {$\iota_1$} (y2laml);
\draw[->, bend right = 0, red, font=\scriptsize] (y1kap1d1) to node [sloped] {$\iota_1$} (vkap1);
\draw[->, bend right = 0, red, font=\scriptsize] (y1lamlr1) to node [sloped] {$\iota_1$} (vlaml);
\draw[->, bend right = 0, black, font=\scriptsize] (y2kap1) to node [sloped] {$\rho_{2}$} (y3kap1);
\draw[->, bend right = 0, black, font=\scriptsize] (y2laml) to node [sloped] {$\rho_{2}$} (y3laml);
\draw[->, bend right = 0, black, font=\scriptsize] (y3kap1) to node [sloped] {$\rho_{1}$} (vkap1);
\draw[->, bend right = 0, black, font=\scriptsize] (y3laml) to node [sloped] {$\rho_{1}$} (vlaml);
\draw (-2,-3) -- (2,-3) -- (2,1) -- (-2,1) -- (-2,-3);
\end{tikzpicture}
}
\caption{After cancellation.}\label{fig:j2c}
\end{subfigure}

\begin{subfigure}{0.49\textwidth}
\begin{tikzpicture}[node distance=2cm]
\path[font = \scriptsize]
(-3.5, -1.36) node(y3kap1) {  }
(1.5, -1.36) node(ukap1) {$u$$\otimes$$\kappa_{1}$}
(1.5, 2.94) node(x3lamlr2) {  }
(1.5, 1.14) node(y2laml) {$y_2$$\otimes$$\lambda_{l}$}
(-1.0, -1.36) node(y2kap1) {$y_2$$\otimes$$\kappa_{1}$}
;
\draw[->, bend right = 0, black, font=\scriptsize] (ukap1) to node [sloped] {$\rho_{3}$} (y2kap1);
\draw[->, bend right = 0, black, font=\scriptsize] (ukap1) to node [sloped] {$\rho_{123}$} (y2laml);
\draw[->, bend right = 0, red, font=\scriptsize] (y2kap1) to node [sloped] {$\rho_{2}$} (y3kap1);
\draw[->, bend right = 0, red, font=\scriptsize] (y2laml) to node [sloped] {$\rho_{23}$} (x3lamlr2);
\end{tikzpicture}
\caption{Cases where either arrow at \xiip~have length more than 1}\label{fig:j2d}
\end{subfigure}
\begin{subfigure}{0.49\textwidth}
\begin{tikzpicture}[node distance=2cm, auto]
  \node (uk1) {$u\ts\kappa_1$};
  \node (y2k1) [left of =uk1] {$y_2\ts\kappa_1$};
  \node (nonamed) [left of =y2k1] {};
  \node (y1k1) [above of =uk1] {$y_1\ts\tilde{\kappa_1}$};
  \node (uk1u) [above of =y1k1] {$u\ts\tilde{\kappa_1}$};
  \node (y2k1u) [left of =uk1u] {$y_2\ts\tilde{\kappa_1}$};
  \node (nonameu) [left of =y2k1u] {};
  \draw[->] (uk1) to node [right] {$\rho_{123}$} (y1k1);
  \draw[->] (uk1) to node [above] {$\rho_3$} (y2k1);
  \draw[->] (y2k1) to node [above] {$\rho_{2}$} (nonamed);
  \draw[->] (uk1u) to node [right] {$\rho_{1}$} (y1k1);
  \draw[->] (uk1u) to node [above] {$\rho_{3}$} (y2k1u);
  \draw[->] (y2k1u) to node [above] {$\rho_{2}$} (nonameu);
\end{tikzpicture}
\caption{Upper left joint connected to a upper right joint.}\label{fig:j2e}
\end{subfigure}

\caption{Upper left joint.}
\end{figure}

\FloatBarrier
\subsubsection{Lower right joint}\label{sec:lrj}

For a generator \xii~with the lower right joint in \cref{fig:joints}, it appears in $KtD(C)$ as \cref{fig:j3a}. Again, for either of the two arrows not immediately connecting to \xiip, we have two possibilities, corresponding to original arrows in $C$ having length 1 or having length more than 1, and the corresponding part of the box tensor product $H\bts KtD(C)$ for all four cases are shown in \cref{fig:j3b}, where all four cases share the same generators and arrows in the middle. On the top right, the two red arrows (one blue arrow, respectively) come from cases where in \cref{fig:j3a} the arrow going out of $\kappa_\ell$ (into $\kappa_\ell$, respectively) has a $\rh{23}$ ($\rh{1}$, respectively). Similarly, on the left side, the two red arrows (two green arrows, respectively) come from cases where the arrow going into $\lambda_1$ has a $\rh{23}$ ($\rh{2}$, respectively).

Then we cancel the three blue arrows and the pink arrow in the middle and arrive at \cref{fig:j3c}. Note the joint at $y_3\ts \kappa_\ell$, which is boxed in \cref{fig:j3c}, can now be identified with the corresponding part of $KtD(C^{flip})$, which would look like the joint in \cref{fig:j2a}. 

If either of the two possibilities in \cref{fig:j3a} is $\rh{23}$, as above, we observe that cancellation of either set of two red arrows in \cref{fig:j3c} agree with the cancellations we performed for a string of $\rh{23}$'s in \cref{sec:string}. This shows chains containing $\rh{23}$'s at $y_3\ts\kappa_\ell$ in $H\bts KtD(C)$ (\cref{fig:j3d}) exactly match corresponding chains in $KtD(C^{flip})$, again see \cref{fig:j2a}.

Now we deal with the possibilities other than $\rh{23}$ in \cref{fig:j3a}. Suppose arrows at \xii~takes form as in \cref{fig:jlr}.

\begin{figure}[h]
\centering
\begin{tikzpicture}[node distance=1.5cm, auto]
  \centering
  \node (x) {\xii};
  \node (y) [above of =x] {$\xi_j$};
  \node (z) [left of =x] {$\xi_k$};
  \draw[->] (y) to node [above] {} (x);
  \draw[->] (x) to node [above] {} (z);
\end{tikzpicture}
\caption{Lower right joint at \xii}\label{fig:jlr}
\end{figure}
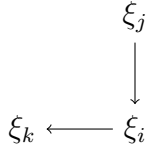

If the vertical arrow $\xi_j \to \xi_i$, has length 1 in $C$, the joint at $\xi_j$ could be either a upper right joint or a upper left joint. We argue that the latter case can not happen. If it did happen, denote the generator connecting to $\xi_j$ through a horizontal arrow by $\xi_r$, then $\partial^2 \xi_r$ has a term $U\xi_i$. Because the $C$ is simultaneously horizontally and vertically simplified, there can not be any arrow going into \xii~besides $\xi_j\to \xi_i$, leaving $\partial^2 \xi_r$ nonzero.

For the case where $\xi_j$ belongs to a upper right joint, we combine the corresponding parts of $H\bts KtD(C)$ in \cref{fig:j1b} and \cref{fig:j3b}. Now generators has $\kappa_\ell$ are identified those with $\kappa_1$. The downward point blue arrow in \cref{fig:j1b} is identified with the blue arrow on top of \cref{fig:j3b}, so are the pink arrow in each of the figures. The cancellation we performed for each type of joints can be both applied and result in \cref{fig:j3e}. It can be identified with the corresponding part in $KtD(C^{flip})$.   

If the horizontal arrow $\xi_i \to \xi_k$ has length 1 in $C$, the joint at $\xi_k$ could be either a upper left joint or a lower left joint. We postpone the latter case to the lower left joint section, \cref{sec:llj}. For the former case, we argue that it could not happen. $\partial^2 \xi_j$ has term $U\xi_k$, but there can not be any arrow going into $\xi_k$ besides $\xi_i \to \xi_k$, leaving $\partial^2 \xi_j$ nonzero.

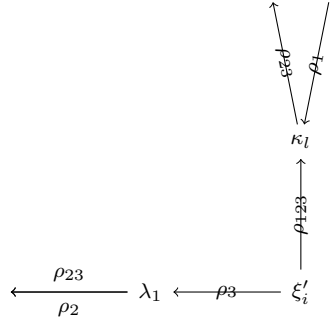
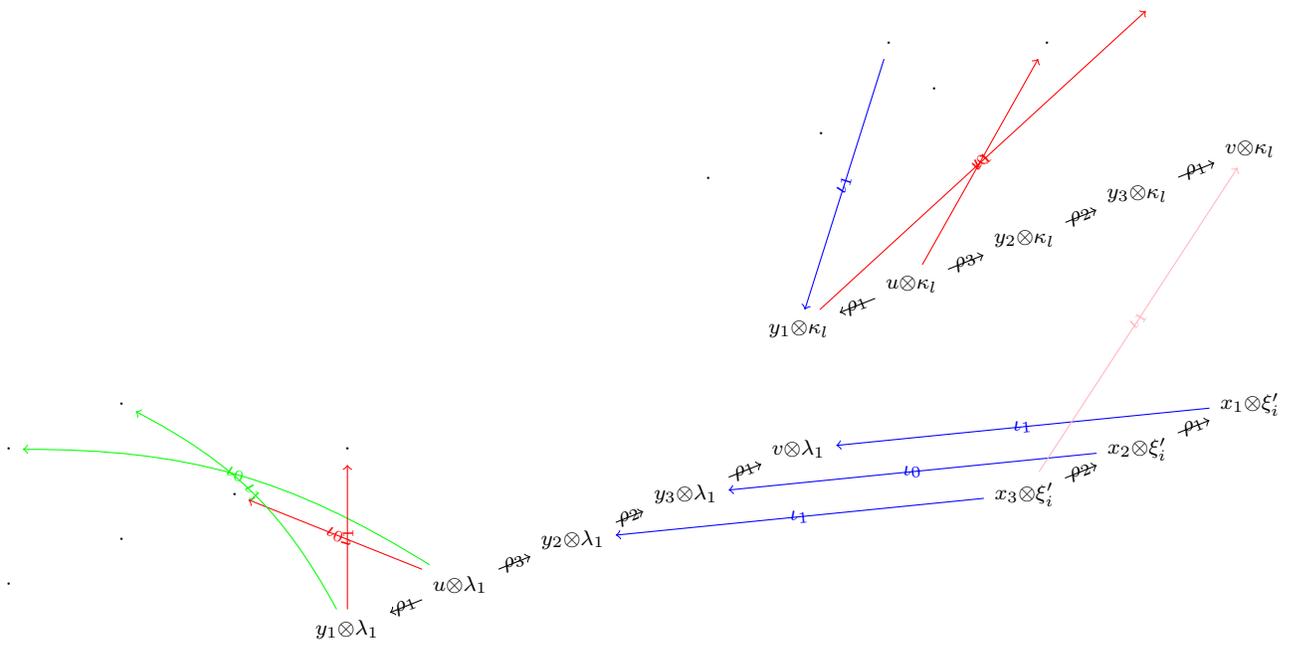
\begin{figure}[ht]

\begin{subfigure}{0.99\textwidth}
\centerline{
\begin{tikzpicture}[node distance=2cm]
\path[font = \scriptsize]
(-2.57, -1.43) node(lam1l1) {}
(1.83, 2.57) node(kaplu2) {}
(-0.57, -1.43) node(lam1) {$\lambda_{1}$}
(-2.57, -1.43) node(lam1l2) {}
(1.03, 2.57) node(kaplu1) {}
(1.43, 0.57) node(kapl) {$\kappa_{l}$}
(1.43, -1.43) node(xii) {$\xi_{i}^\prime$}
;
\draw[->, bend right = 0, black, font=\scriptsize] (kapl) to node [sloped] {$\rho_{23}$} (kaplu1);
\draw[->, bend right = 0, black, font=\scriptsize] (kaplu2) to node [sloped] {$\rho_{1}$} (kapl);
\draw[->, bend right = 0, black, font=\scriptsize] (lam1) to node [above] {$\rho_{23}$} (lam1l1);
\draw[->, bend right = 0, black, font=\scriptsize] (lam1) to node [below] {$\rho_{2}$} (lam1l2);
\draw[->, bend right = 0, black, font=\scriptsize] (xii) to node [sloped] {$\rho_{3}$} (lam1);
\draw[->, bend right = 0, black, font=\scriptsize] (xii) to node [sloped] {$\rho_{123}$} (kapl);
\end{tikzpicture}
}
\caption{Lower right joint at \xiip}\label{fig:j3a}
\end{subfigure}

\begin{subfigure}{0.99\textwidth}
\centerline{
\begin{tikzpicture}[node distance=2cm]
\path[font = \scriptsize]
(-7.42, -0.92) node(x1lam1l2) {$\cdot$}
(4.58, 1.28) node(y2kapl) {$y_2$$\otimes$$\kappa_{l}$}
(4.58, -2.12) node(x3xii) {$x_3$$\otimes$$\xi_{i}^\prime$}
(6.08, 1.88) node(y3kapl) {$y_3$$\otimes$$\kappa_{l}$}
(-8.92, -3.32) node(ulam1l1) {$\cdot$}
(2.78, 3.88) node(x3kaplu2) {$\cdot$}
(0.08, -2.12) node(y3lam1) {$y_3$$\otimes$$\lambda_{1}$}
(3.08, 0.68) node(ukapl) {$u$$\otimes$$\kappa_{l}$}
(4.28, 4.48) node(x2kaplu2) {$\cdot$}
(-5.92, -2.12) node(y3lam1l1) {$\cdot$}
(-7.42, -2.72) node(y2lam1l1) {$\cdot$}
(7.58, 2.48) node(vkapl) {$v$$\otimes$$\kappa_{l}$}
(4.88, 3.88) node(y3kaplu1) {$\cdot$}
(-10.42, -3.92) node(y1lam1l1) {$\cdot$}
(1.58, 0.08) node(y1kapl) {$y_1$$\otimes$$\kappa_{l}$}
(5.78, 5.08) node(x1kaplu2) {$\cdot$}
(6.08, -1.52) node(x2xii) {$x_2$$\otimes$$\xi_{i}^\prime$}
(-1.42, -2.72) node(y2lam1) {$y_2$$\otimes$$\lambda_{1}$}
(-2.92, -3.32) node(ulam1) {$u$$\otimes$$\lambda_{1}$}
(6.38, 4.48) node(vkaplu1) {$\cdot$}
(-4.42, -3.92) node(y1lam1) {$y_1$$\otimes$$\lambda_{1}$}
(3.38, 3.28) node(y2kaplu1) {$\cdot$}
(-4.42, -1.52) node(vlam1l1) {$\cdot$}
(-8.92, -1.52) node(x2lam1l2) {$\cdot$}
(-10.42, -2.12) node(x3lam1l2) {$\cdot$}
(0.38, 2.08) node(y1kaplu1) {$\cdot$}
(1.58, -1.52) node(vlam1) {$v$$\otimes$$\lambda_{1}$}
(1.88, 2.68) node(ukaplu1) {$\cdot$}
(7.58, -0.92) node(x1xii) {$x_1$$\otimes$$\xi_{i}^\prime$}
;
\draw[->, bend right = 0, black, font=\scriptsize] (ukapl) to node [sloped] {$\rho_{1}$} (y1kapl);
\draw[->, bend right = 0, black, font=\scriptsize] (ukapl) to node [sloped] {$\rho_{3}$} (y2kapl);
\draw[->, bend right = 0, red, font=\scriptsize] (ukapl) to node [sloped] {$\iota_0$} (y3kaplu1);
\draw[->, bend right = 0, black, font=\scriptsize] (ulam1) to node [sloped] {$\rho_{1}$} (y1lam1);
\draw[->, bend right = 0, black, font=\scriptsize] (ulam1) to node [sloped] {$\rho_{3}$} (y2lam1);
\draw[->, bend right = 0, red, font=\scriptsize] (ulam1) to node [sloped] {$\iota_0$} (y3lam1l1);
\draw[->, bend right = 16.83, green, font=\scriptsize] (ulam1) to node [sloped] {$\iota_0$} (x2lam1l2);
\draw[->, bend right = 0.97, blue, font=\scriptsize] (x1xii) to node [sloped] {$\iota_1$} (vlam1);
\draw[->, bend right = 0, black, font=\scriptsize] (x2xii) to node [sloped] {$\rho_{1}$} (x1xii);
\draw[->, bend right = 0.97, blue, font=\scriptsize] (x2xii) to node [sloped] {$\iota_0$} (y3lam1);
\draw[->, bend right = 0, blue, font=\scriptsize] (x3kaplu2) to node [sloped] {$\iota_1$} (y1kapl);
\draw[->, bend right = 0, pink, font=\scriptsize] (x3xii) to node [sloped] {$\iota_1$} (vkapl);
\draw[->, bend right = 0, black, font=\scriptsize] (x3xii) to node [sloped] {$\rho_{2}$} (x2xii);
\draw[->, bend right = 0.97, blue, font=\scriptsize] (x3xii) to node [sloped] {$\iota_1$} (y2lam1);
\draw[->, bend right = 0, red, font=\scriptsize] (y1kapl) to node [sloped] {$\iota_1$} (vkaplu1);
\draw[->, bend right = 0, red, font=\scriptsize] (y1lam1) to node [sloped] {$\iota_1$} (vlam1l1);
\draw[->, bend right = 17.0, green, font=\scriptsize] (y1lam1) to node [sloped] {$\iota_1$} (x1lam1l2);
\draw[->, bend right = 0, black, font=\scriptsize] (y2kapl) to node [sloped] {$\rho_{2}$} (y3kapl);
\draw[->, bend right = 0, black, font=\scriptsize] (y2lam1) to node [sloped] {$\rho_{2}$} (y3lam1);
\draw[->, bend right = 0, black, font=\scriptsize] (y3kapl) to node [sloped] {$\rho_{1}$} (vkapl);
\draw[->, bend right = 0, black, font=\scriptsize] (y3lam1) to node [sloped] {$\rho_{1}$} (vlam1);
\end{tikzpicture}
}
\caption{Tensor product with $H$.}\label{fig:j3b}
\end{subfigure}

\caption{Lower right joint}
\end{figure}

\begin{figure}
\begin{subfigure}{0.99\textwidth}
\centerline{
\begin{tikzpicture}[node distance=2cm]
\path[font = \scriptsize]
(3.64, 0.89) node(vkaplu1) {$\cdot$}
(-2.36, -3.11) node(y3lam1l1) {$\cdot$}
(-0.36, 2.89) node(y2kapl) {$y_2$$\otimes$$\kappa_{l}$}
(3.64, 0.39) node(x3kaplu2) {$\cdot$}
(-2.86, -3.11) node(x2lam1l2) {$\cdot$}
(-0.36, -1.11) node(y1lam1) {$y_1$$\otimes$$\lambda_{1}$}
(-2.36, 0.89) node(vkapl) {$v$$\otimes$$\kappa_{l}$}
(1.64, 0.89) node(y1kapl) {$y_1$$\otimes$$\kappa_{l}$}
(3.64, 2.89) node(y3kaplu1) {$\cdot$}
(-0.36, -3.11) node(vlam1l1) {$\cdot$}
(-2.36, 2.89) node(y3kapl) {$y_3$$\otimes$$\kappa_{l}$}
(-0.86, -3.11) node(x1lam1l2) {$\cdot$}
(-2.36, -1.11) node(ulam1) {$u$$\otimes$$\lambda_{1}$}
(1.64, 2.89) node(ukapl) {$u$$\otimes$$\kappa_{l}$}
;
\draw[->, bend right = 0, black, font=\scriptsize] (ukapl) to node [sloped] {$\rho_{1}$} (y1kapl);
\draw[->, bend right = 0, black, font=\scriptsize] (ukapl) to node [sloped] {$\rho_{3}$} (y2kapl);
\draw[->, bend right = 0, red, font=\scriptsize] (ukapl) to node [sloped] {$\iota_0$} (y3kaplu1);
\draw[->, bend right = 0, black, font=\scriptsize] (ulam1) to node [sloped] {$\rho_{1}$} (y1lam1);
\draw[->, bend right = 0, red, font=\scriptsize] (ulam1) to node [sloped] {$\iota_0$} (y3lam1l1);
\draw[->, bend right = 0, green, font=\scriptsize] (ulam1) to node [sloped] {$\iota_0$} (x2lam1l2);
\draw[->, bend right = 0, black, font=\scriptsize] (ulam1) to node [sloped] {$\rho_{3}$} (vkapl);
\draw[->, bend right = 0, blue, font=\scriptsize] (x3kaplu2) to node [sloped] {$\iota_1$} (y1kapl);
\draw[->, bend right = 0, red, font=\scriptsize] (y1kapl) to node [sloped] {$\iota_1$} (vkaplu1);
\draw[->, bend right = 0, red, font=\scriptsize] (y1lam1) to node [sloped] {$\iota_1$} (vlam1l1);
\draw[->, bend right = 0, green, font=\scriptsize] (y1lam1) to node [sloped] {$\iota_1$} (x1lam1l2);
\draw[->, bend right = 0, black, font=\scriptsize] (y2kapl) to node [sloped] {$\rho_{2}$} (y3kapl);
\draw[->, bend right = 0, black, font=\scriptsize] (y3kapl) to node [sloped] {$\rho_{1}$} (vkapl);
\draw (-3.5,0) -- (1,0) -- (1,4) -- (-3.5,4) -- (-3.5,0);
\end{tikzpicture}
}
\caption{After cancellation.}\label{fig:j3c}
\end{subfigure}

\begin{subfigure}{0.48\textwidth}
\begin{tikzpicture}[node distance=2cm, auto]
  \node (y3kl) {$y_3\ts\kappa_l$};
  \node (y2kl) [right of =y3kl] {$y_2\ts\kappa_l$};
  \node (nonamer) [right of =y2kl] {};
  \node (vkl) [below of =y3kl] {$v\ts\kappa_l$};
  \node (nonamed) [below of =vkl] {};
  \draw[->] (y3kl) to node [right] {$\rho_{1}$} (vkl);
  \draw[->, red] (nonamed) to node [right] {$\rho_{23}$} (vkl);
  \draw[->] (y2kl) to node [below] {$\rho_{2}$} (y3kl);
  \draw[->, red] (nonamer) to node [below] {$\rho_{23}$} (y2kl);
\end{tikzpicture}
\caption{Cases where either arrow at \xiip~have length more than 1}\label{fig:j3d}
\end{subfigure}
\begin{subfigure}{0.48\textwidth}
\begin{tikzpicture}[node distance=2cm, auto]
  \node (y3kl) {$y_3\ts\kappa_1$};
  \node (y2kl) [right of =y3kl] {$y_2\ts\kappa_1$};
  \node (uk1) [right of =y2kl] {$u\ts\kappa_1$};
  \node (vkl) [below of =y3kl] {$v\ts\kappa_1$};
  \node (y1k1) [below of =uk1] {$y_1\ts{\kappa_1}$};
  \node (nonamed2) [below of =y1k1] {};
  \node (nonamed) [below of =vkl] {};
  \draw[->] (y3kl) to node [right] {$\rho_{1}$} (vkl);
  \draw[->] (nonamed) to node [right] {$\rho_{3}$} (vkl);
  \draw[->] (y2kl) to node [above] {$\rho_{2}$} (y3kl);
  \draw[->] (uk1) to node [above] {$\rho_{3}$} (y2kl);
  \draw[->] (uk1) to node [right] {$\rho_{1}$} (y1k1);
  \draw[->] (nonamed2) to node [right] {$\rho_{3}$} (y1k1);
\end{tikzpicture}
\caption{A lower right joint connected to a upper right joint}\label{fig:j3e}
\end{subfigure}

\caption{Lower right joint}
\end{figure}
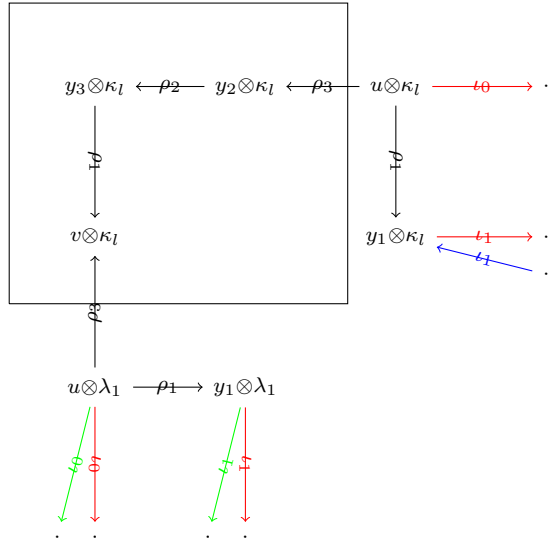
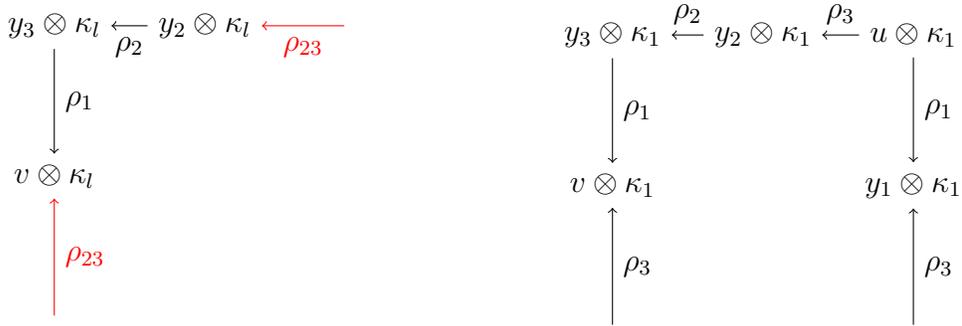

\FloatBarrier
We include a summary of all eight possible ways two joints could overlap, see \cref{fig:overlappingjoints}.

\begin{figure}[h]
\begin{subfigure}{.32\textwidth}
\centering
\begin{tikzpicture}[node distance=1.5cm, auto]
  \node (x) {};
  \node (y) [above of =x] {};
  \node (z) [right of =y] {};
  \node (w) [right of =z] {};
  \node (t) [below of =w] {};
  \draw[->] (y) to node [above] {} (x);
  \draw[->] (z) to node [above] {} (y);
  \draw[-] (w) to node [above] {} (z);
  \draw[->] (w) to node [above] {} (t);
\end{tikzpicture}
\caption{\cref{sec:ulj}}
\end{subfigure}
\begin{subfigure}{.32\textwidth}
\centering
\begin{tikzpicture}[node distance=1.5cm, auto]
  \node (x) {};
  \node (y) [right of =x] {};
  \node (z) [below of =y] {};
  \node (w) [below of =z] {};
  \node (t) [left of =w] {};
  \draw[->] (y) to node [above] {} (x);
  \draw[-] (y) to node [above] {} (z);
  \draw[->] (z) to node [above] {} (w);
  \draw[->] (w) to node [above] {} (t);
\end{tikzpicture}
\caption{\cref{sec:lrj}}
\end{subfigure}
\begin{subfigure}{.32\textwidth}
\centering
\begin{tikzpicture}[node distance=1.5cm, auto]
  \node (x) {};
  \node (y) [left of =x] {};
  \node (z) [below of =y] {};
  \node (w) [below of =z] {};
  \node (t) [left of =w] {};
  \draw[->] (x) to node [above] {} (y);
  \draw[-] (y) to node [above] {} (z);
  \draw[->] (z) to node [above] {} (w);
  \draw[->] (w) to node [above] {} (t);
\end{tikzpicture}
\caption{Ruled out in \cref{sec:lrj}}
\end{subfigure}

\begin{subfigure}{.32\textwidth}
\centering
\begin{tikzpicture}[node distance=1.5cm, auto]
  \node (x) {};
  \node (y) [above of =x] {};
  \node (z) [right of =y] {};
  \node (w) [right of =z] {};
  \node (t) [above of =w] {};
  \draw[->] (y) to node [above] {} (x);
  \draw[->] (z) to node [above] {} (y);
  \draw[-] (w) to node [above] {} (z);
  \draw[->] (t) to node [above] {} (w);
\end{tikzpicture}
\caption{Ruled out in \cref{sec:lrj}}
\end{subfigure}
\begin{subfigure}{.32\textwidth}
\centering
\begin{tikzpicture}[node distance=1.5cm, auto]
  \node (x) {};
  \node (y) [below of =x] {};
  \node (z) [right of =y] {};
  \node (w) [right of =z] {};
  \node (t) [below of =w] {};
  \draw[->] (x) to node [above] {} (y);
  \draw[->] (z) to node [above] {} (y);
  \draw[-] (w) to node [above] {} (z);
  \draw[->] (w) to node [above] {} (t);
\end{tikzpicture}
\caption{\cref{sec:llj}}\label{fig:overlapj5}
\end{subfigure}
\begin{subfigure}{.32\textwidth}
\centering
\begin{tikzpicture}[node distance=1.5cm, auto]
  \node (x) {};
  \node (y) [below of =x] {};
  \node (z) [right of =y] {};
  \node (w) [right of =z] {};
  \node (t) [above of =w] {};
  \draw[->] (x) to node [above] {} (y);
  \draw[->] (z) to node [above] {} (y);
  \draw[-] (w) to node [above] {} (z);
  \draw[->] (t) to node [above] {} (w);
\end{tikzpicture}
\caption{\cref{sec:llj}}\label{fig:overlapj6}
\end{subfigure}

\begin{subfigure}{.32\textwidth}
\centering
\begin{tikzpicture}[node distance=1.5cm, auto]
  \node (x) {};
  \node (y) [left of =x] {};
  \node (z) [below of =y] {};
  \node (w) [below of =z] {};
  \node (t) [right of =w] {};
  \draw[->] (x) to node [above] {} (y);
  \draw[-] (y) to node [above] {} (z);
  \draw[->] (z) to node [above] {} (w);
  \draw[->] (t) to node [above] {} (w);
\end{tikzpicture}
\caption{\cref{sec:llj}}\label{fig:overlapj7}
\end{subfigure}
\begin{subfigure}{.32\textwidth}
\centering
\begin{tikzpicture}[node distance=1.5cm, auto]
  \node (x) {};
  \node (y) [right of =x] {};
  \node (z) [below of =y] {};
  \node (w) [below of =z] {};
  \node (t) [right of =w] {};
  \draw[->] (y) to node [above] {} (x);
  \draw[-] (y) to node [above] {} (z);
  \draw[->] (z) to node [above] {} (w);
  \draw[->] (t) to node [above] {} (w);
\end{tikzpicture}
\caption{\cref{sec:llj}}\label{fig:overlapj8}
\end{subfigure}

\caption{A summary of eight possible overlapping joints}
\label{fig:overlappingjoints}
\end{figure}
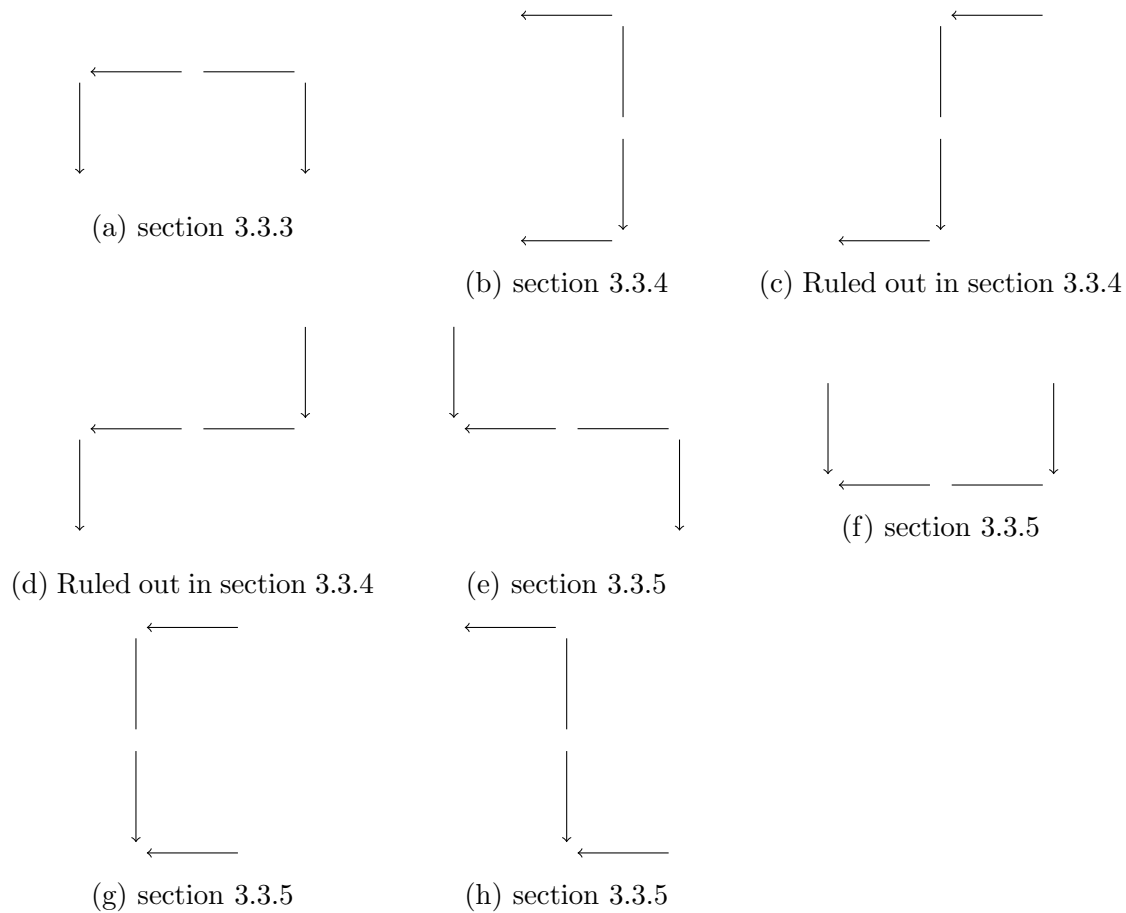

\FloatBarrier
\subsubsection{Lower left joint}\label{sec:llj}

For a generator \xii~with the lower left joint in \cref{fig:joints}, it appears in $KtD(C)$ as \cref{fig:j4a}. Again, for either of the two arrows not immediately connecting to \xiip, we have two possibilities, corresponding to original arrows in $C$ having length 1 or having length more than 1, and the corresponding part of the box tensor product $H\bts KtD(C)$ for all four cases are shown in \cref{fig:j4b}, where all four cases share the same generators and arrows in the middle. On the top, the two red arrows (one blue arrow, respectively) come from cases where in \cref{fig:j4a} the arrow going out of $\kappa_\ell$ (into $\kappa_\ell$, respectively) has a $\rh{23}$ ($\rh{1}$, respectively). Similarly, on the right side, the two red arrows (three blue arrows, respectively) come from cases where in \cref{fig:j4a} the arrow going into $\lambda_\ell$ has a $\rh{23}$ ($\rh{3}$, respectively). 

Then we cancel the two green arrows and the pink arrow in the middle and arrive at \cref{fig:j4c}. Note the joint at $y_3\ts \kappa_\ell$, which is boxed in \cref{fig:j4c}, can now be identified with the corresponding part of $KtD(C^{flip})$, which would look like the joint itself. 

If either of the two possibilities in \cref{fig:j4a} is $\rh{23}$, as above, we observe that cancellation of either set of two red arrows in \cref{fig:j4c} agree with the cancellations we performed for a string of $\rh{23}$'s in \cref{sec:string}. This shows chains containing $\rh{23}$'s at $y_3\ts\kappa_\ell$ in $H\bts KtD(C)$ (\cref{fig:j4d}) exactly match corresponding chains in $KtD(C^{flip})$, which would again look like the joint itself.

Now we deal with the four cases of overlapping joints involving a lower left joint, see the last four diagrams of \cref{fig:overlappingjoints}. First, for the case in \cref{fig:overlapj5}, we combine \cref{fig:j4b} with \cref{fig:j1b}. All generators involving $\lambda_\ell$ are identified with those with $\lambda_1$, and the two sets of two green arrows and two sets of three blue arrows are respectively identified. We perform the same two sets of cancellations and arrive at \cref{fig:j4e}. It can be identified with the corresponding part in $KtD(C^{flip})$.

The remaining three cases in \cref{fig:overlapj6}, \cref{fig:overlapj7}, and \cref{fig:overlapj8} can be proved in a similar fashion: we combine two tensor products of $H$ with corresponding joints shown in this and above sections, then identify overlapping generators and arrows, and finally observe that cancellations done separately for each joint can be combined without conflicting. We provide the results in \cref{fig:j4f}, \cref{fig:j4g}, and \cref{fig:j4h}, and omit further details. 

\begin{figure}[ht]

\begin{subfigure}{0.99\textwidth}
\centerline{
\begin{tikzpicture}[node distance=2cm]
\path[font = \scriptsize]
(-1.03, 2.57) node(kaplu2) {}
(0.57, -1.43) node(laml) {$\lambda_{l}$}
(2.57, -1.43) node(lamlr2) {}
(2.57, -1.43) node(lamlr1) {}
(-1.83, 2.57) node(kaplu1) {}
(-1.43, -1.43) node(xii) {$\xi_{i}^\prime$}
(-1.43, 0.57) node(kapl) {$\kappa_{l}$}
;
\draw[->, bend right = 0, black, font=\scriptsize] (kapl) to node [sloped] {$\rho_{23}$} (kaplu2);
\draw[->, bend right = 0, black, font=\scriptsize] (kaplu1) to node [sloped] {$\rho_{1}$} (kapl);
\draw[->, bend right = 0, black, font=\scriptsize] (laml) to node [sloped] {$\rho_{2}$} (xii);
\draw[->, bend right = 0, black, font=\scriptsize] (lamlr1) to node [above] {$\rho_{23}$} (laml);
\draw[->, bend right = 0, black, font=\scriptsize] (lamlr2) to node [below] {$\rho_{3}$} (laml);
\draw[->, bend right = 0, black, font=\scriptsize] (xii) to node [sloped] {$\rho_{123}$} (kapl);
\end{tikzpicture}
}
\caption{Lower left joint at \xiip}\label{fig:j4a}
\end{subfigure}

\begin{subfigure}{0.99\textwidth}
\centerline{
\begin{tikzpicture}[node distance=2cm]
\path[font = \scriptsize]
(-4.78, 1.34) node(y2kapl) [outer sep=-3pt]{$y_2$$\otimes$$\kappa_{l}$}
(0.97, -2.16) node(ulaml) [outer sep=-3pt]{$u$$\otimes$$\lambda_{l}$}
(9.22, -2.06) node(y2lamlr1) [outer sep=-3pt]{$\cdot$}
(-2.13, 2.84) node(y3kaplu2) [outer sep=-3pt]{$\cdot$}
(9.22, -3.16) node(x1lamlr2) [outer sep=-3pt]{$\cdot$}
(-0.88, 2.34) node(vkaplu2) [outer sep=-3pt]{$\cdot$}
(2.22, -2.66) node(y2laml) [outer sep=-3pt]{$y_2$$\otimes$$\lambda_{l}$}
(10.47, -2.56) node(y3lamlr1) [outer sep=-3pt]{$\cdot$}
(-3.53, 0.84) node(y3kapl) [outer sep=-3pt]{$y_3$$\otimes$$\kappa_{l}$}
(6.72, -1.06) node(y1lamlr1) [outer sep=-3pt]{$\cdot$}
(-4.78, -2.16) node(x1xii) [outer sep=-3pt]{$x_1$$\otimes$$\xi_{i}^\prime$}
(3.47, -3.16) node(y3laml) [outer sep=-3pt]{$y_3$$\otimes$$\lambda_{l}$}
(-2.28, 0.34) node(vkapl) [outer sep=-3pt]{$v$$\otimes$$\kappa_{l}$}
(-8.68, 4.84) node(x3kaplu1) [outer sep=-3pt]{$\cdot$}
(-7.43, 4.34) node(x2kaplu1) [outer sep=-3pt]{$\cdot$}
(-6.18, 3.84) node(x1kaplu1) [outer sep=-3pt]{$\cdot$}
(-6.03, -1.66) node(x2xii) [outer sep=-3pt]{$x_2$$\otimes$$\xi_{i}^\prime$}
(-4.63, 3.84) node(ukaplu2) [outer sep=-3pt]{$\cdot$}
(-7.28, -1.16) node(x3xii) [outer sep=-3pt]{$x_3$$\otimes$$\xi_{i}^\prime$}
(7.97, -1.56) node(ulamlr1) [outer sep=-3pt]{$\cdot$}
(7.97, -2.66) node(x2lamlr2) [outer sep=-3pt]{$\cdot$}
(-5.88, 4.34) node(y1kaplu2) [outer sep=-3pt]{$\cdot$}
(6.72, -2.16) node(x3lamlr2) [outer sep=-3pt]{$\cdot$}
(-7.28, 2.34) node(y1kapl) [outer sep=-3pt]{$y_1$$\otimes$$\kappa_{l}$}
(4.72, -3.66) node(vlaml) [outer sep=-3pt]{$v$$\otimes$$\lambda_{l}$}
(-3.38, 3.34) node(y2kaplu2) [outer sep=-3pt]{$\cdot$}
(-0.28, -1.66) node(y1laml) [outer sep=-3pt]{$y_1$$\otimes$$\lambda_{l}$}
(-6.03, 1.84) node(ukapl) [outer sep=-3pt]{$u$$\otimes$$\kappa_{l}$}
(11.72, -3.06) node(vlamlr1) [outer sep=-3pt]{$\cdot$}
;
\draw[->, bend right = 0, black, font=\scriptsize] (ukapl) to node [sloped] {$\rho_{1}$} (y1kapl);
\draw[->, bend right = 0, black, font=\scriptsize] (ukapl) to node [sloped] {$\rho_{3}$} (y2kapl);
\draw[->, bend right = 0, red, font=\scriptsize] (ukapl) to node [sloped] {$\iota_0$} (y3kaplu2);
\draw[->, bend right = 0, black, font=\scriptsize] (ulaml) to node [sloped] {$\rho_{1}$} (y1laml);
\draw[->, bend right = 0, black, font=\scriptsize] (ulaml) to node [sloped] {$\rho_{3}$} (y2laml);
\draw[->, bend right = 4.34, green, font=\scriptsize] (ulaml) to node [sloped] {$\iota_0$} (x2xii);
\draw[->, bend right = 16.23, red, font=\scriptsize] (ulamlr1) to node [sloped] {$\iota_0$} (y3laml);
\draw[->, bend right = 0, blue, font=\scriptsize] (x1lamlr2) to node [sloped] {$\iota_1$} (vlaml);
\draw[->, bend right = 0, blue, font=\scriptsize] (x2lamlr2) to node [sloped] {$\iota_0$} (y3laml);
\draw[->, bend right = 0, black, font=\scriptsize] (x2xii) to node [sloped] {$\rho_{1}$} (x1xii);
\draw[->, bend right = 0, blue, font=\scriptsize] (x3kaplu1) to node [sloped] {$\iota_1$} (y1kapl);
\draw[->, bend right = 0, blue, font=\scriptsize] (x3lamlr2) to node [sloped] {$\iota_1$} (y2laml);
\draw[->, bend right = 0, pink, font=\scriptsize] (x3xii) to node [sloped] {$\iota_1$} (vkapl);
\draw[->, bend right = 0, black, font=\scriptsize] (x3xii) to node [sloped] {$\rho_{2}$} (x2xii);
\draw[->, bend right = 0, red, font=\scriptsize] (y1kapl) to node [sloped] {$\iota_1$} (vkaplu2);
\draw[->, bend right = 0, green, font=\scriptsize] (y1laml) to node [sloped] {$\iota_1$} (x1xii);
\draw[->, bend right = 0, red, font=\scriptsize] (y1lamlr1) to node [sloped] {$\iota_1$} (vlaml);
\draw[->, bend right = 0, black, font=\scriptsize] (y2kapl) to node [sloped] {$\rho_{2}$} (y3kapl);
\draw[->, bend right = 0, black, font=\scriptsize] (y2laml) to node [sloped] {$\rho_{2}$} (y3laml);
\draw[->, bend right = 0, black, font=\scriptsize] (y3kapl) to node [sloped] {$\rho_{1}$} (vkapl);
\draw[->, bend right = 0, black, font=\scriptsize] (y3laml) to node [sloped] {$\rho_{1}$} (vlaml);
\end{tikzpicture}
}
\caption{Tensor product with $H$.}\label{fig:j4b}
\end{subfigure}

\begin{subfigure}{0.99\textwidth}
\centerline{
\begin{tikzpicture}[node distance=2cm]
\path[font = \scriptsize]
(4.6, -1.13) node(y3kaplu2) [outer sep=-3pt]{$\cdot$}
(3.6, -1.13) node(ukapl) [outer sep=-3pt]{$u$$\otimes$$\kappa_{l}$}
(4.6, -3.63) node(x3kaplu1) [outer sep=-3pt]{$\cdot$}
(-0.4, 0.87) node(y2laml) [outer sep=-3pt]{$y_2$$\otimes$$\lambda_{l}$}
(-0.4, 2.37) node(x3lamlr2) [outer sep=-3pt]{$\cdot$}
(1.6, -1.13) node(y2kapl) [outer sep=-3pt]{$y_2$$\otimes$$\kappa_{l}$}
(-4.9, 2.37) node(y1lamlr1) [outer sep=-3pt]{$\cdot$}
(3.6, -3.13) node(y1kapl) [outer sep=-3pt]{$y_1$$\otimes$$\kappa_{l}$}
(-2.4, 0.87) node(y3laml) [outer sep=-3pt]{$y_3$$\otimes$$\lambda_{l}$}
(-2.4, 2.37) node(x2lamlr2) [outer sep=-3pt]{$\cdot$}
(-4.4, 2.37) node(x1lamlr2) [outer sep=-3pt]{$\cdot$}
(4.6, -3.13) node(vkaplu2) [outer sep=-3pt]{$\cdot$}
(-2.9, 2.37) node(ulamlr1) [outer sep=-3pt]{$\cdot$}
(-0.4, -1.13) node(y3kapl) [outer sep=-3pt]{$y_3$$\otimes$$\kappa_{l}$}
(-4.4, 0.87) node(vlaml) [outer sep=-3pt]{$v$$\otimes$$\lambda_{l}$}
;
\draw[->, bend right = 0, black, font=\scriptsize] (ukapl) to node [sloped] {$\rho_{1}$} (y1kapl);
\draw[->, bend right = 0, black, font=\scriptsize] (ukapl) to node [sloped] {$\rho_{3}$} (y2kapl);
\draw[->, bend right = 0, red, font=\scriptsize] (ukapl) to node [sloped] {$\iota_0$} (y3kaplu2);
\draw[->, bend right = 0, red, font=\scriptsize] (ulamlr1) to node [sloped] {$\iota_0$} (y3laml);
\draw[->, bend right = 0, blue, font=\scriptsize] (x1lamlr2) to node [sloped] {$\iota_1$} (vlaml);
\draw[->, bend right = 0, blue, font=\scriptsize] (x2lamlr2) to node [sloped] {$\iota_0$} (y3laml);
\draw[->, bend right = 0, blue, font=\scriptsize] (x3kaplu1) to node [sloped] {$\iota_1$} (y1kapl);
\draw[->, bend right = 0, blue, font=\scriptsize] (x3lamlr2) to node [sloped] {$\iota_1$} (y2laml);
\draw[->, bend right = 0, red, font=\scriptsize] (y1kapl) to node [sloped] {$\iota_1$} (vkaplu2);
\draw[->, bend right = 0, red, font=\scriptsize] (y1lamlr1) to node [sloped] {$\iota_1$} (vlaml);
\draw[->, bend right = 0, black, font=\scriptsize] (y2kapl) to node [sloped] {$\rho_{2}$} (y3kapl);
\draw[->, bend right = 0, black, font=\scriptsize] (y2laml) to node [sloped] {$\rho_{2}$} (y3laml);
\draw[->, bend right = 0, black, font=\scriptsize] (y3kapl) to node [sloped] {$\rho_{123}$} (y2laml);
\draw[->, bend right = 0, black, font=\scriptsize] (y3laml) to node [sloped] {$\rho_{1}$} (vlaml);
\draw (-1.3,-1.8) -- (3,-1.8) -- (3,1.5) -- (-1.3,1.5) -- (-1.3,-1.8);
\end{tikzpicture}
}
\caption{After cancellation.}\label{fig:j4c}
\end{subfigure}

\caption{Lower left joint}
\end{figure}
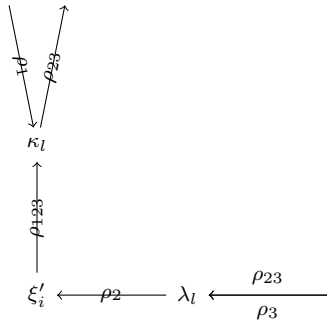
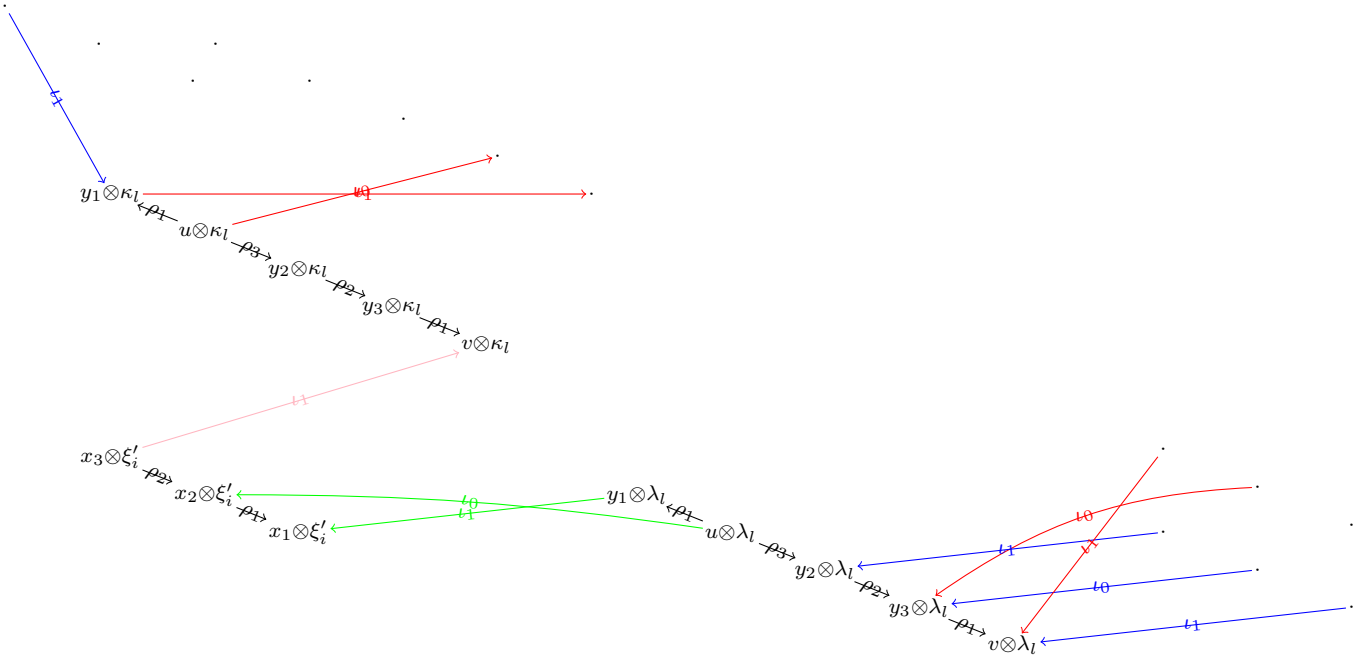
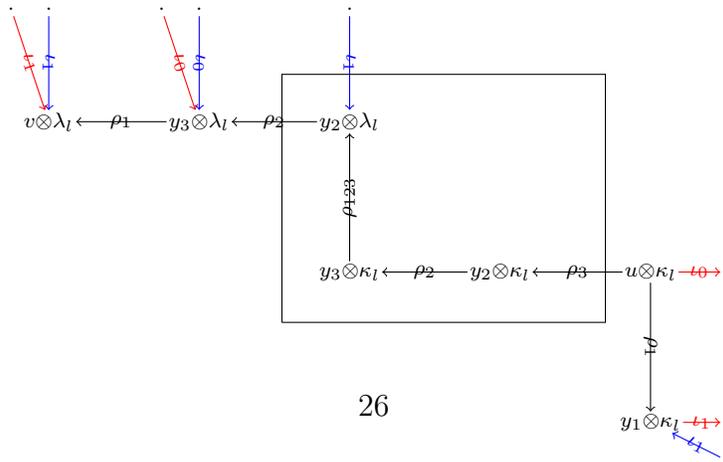

\begin{figure}

\begin{subfigure}{0.99\textwidth}
\centerline{
\begin{tikzpicture}[node distance=1.8cm, auto]
  \node (y3kl) {$y_3\ts\kappa_l$};
  \node (y2kl) [right of =y3kl] {$y_2\ts\kappa_l$};
  \node (nonamer) [right of =y2kl] {};
  \node (y2ll) [above of =y3kl] {$y_2\ts\lambda_l$};
  \node (nonameu) [above of =y2ll] {};
  \draw[->] (y3kl) to node [right] {$\rho_{123}$} (y2ll);
  \draw[->, red] (y2ll) to node [right] {$\rho_{23}$} (nonameu);
  \draw[->] (y2kl) to node [below] {$\rho_{2}$} (y3kl);
  \draw[->, red] (nonamer) to node [below] {$\rho_{23}$} (y2kl);
\end{tikzpicture}
}
\caption{Cases where either arrow at \xiip~have length more than 1}\label{fig:j4d}
\end{subfigure}

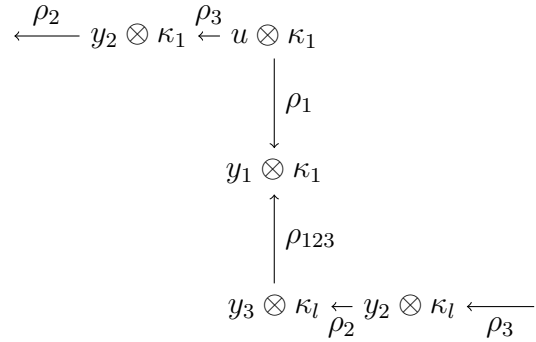
\begin{subfigure}{0.99\textwidth}
\centerline{
\begin{tikzpicture}[node distance=1.8cm, auto]
  \node (y3kl) {$y_3\ts\kappa_l$};
  \node (y2kl) [right of =y3kl] {$y_2\ts\kappa_l$};
  \node (nonamer) [right of =y2kl] {};
  \node (y1k1) [above of =y3kl] {$y_1\ts\kappa_1$};
  \node (uk1) [above of =y1k1] {$u\ts{\kappa_1}$};
  \node (y2k1) [left of =uk1] {$y_2\ts{\kappa_1}$};
  \node (nonamel) [left of =y2k1] {};
  \draw[->] (y3kl) to node [right] {$\rho_{123}$} (y1k1);
  \draw[->] (uk1) to node [right] {$\rho_{1}$} (y1k1);
  \draw[->] (y2kl) to node [below] {$\rho_{2}$} (y3kl);
  \draw[->] (nonamer) to node [below] {$\rho_{3}$} (y2kl);
  \draw[->] (uk1) to node [above] {$\rho_{3}$} (y2k1);
  \draw[->] (y2k1) to node [above] {$\rho_{2}$} (nonamel);
\end{tikzpicture}
}
\caption{A lower left joint connected to a upper right joint}\label{fig:j4e}
\end{subfigure}

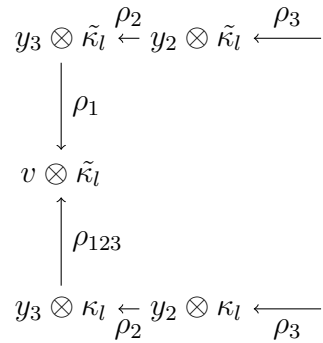
\begin{subfigure}{0.99\textwidth}
\centerline{
\begin{tikzpicture}[node distance=1.8cm, auto]
  \node (y3kl) {$y_3\ts\kappa_l$};
  \node (y2kl) [right of =y3kl] {$y_2\ts\kappa_l$};
  \node (nonamer) [right of =y2kl] {};
  \node (vkl) [above of =y3kl] {$v\ts\tilde{\kappa_l}$};
  \node (y3klt) [above of =vkl] {$y_3\ts\tilde{\kappa_l}$};
  \node (y2klt) [right of =y3klt] {$y_2\ts\tilde{\kappa_l}$};
  \node (nonameu) [right of =y2klt] {};
  \draw[->] (y3kl) to node [right] {$\rho_{123}$} (vkl);
  \draw[->] (y3klt) to node [right] {$\rho_{1}$} (vkl);
  \draw[->] (y2kl) to node [below] {$\rho_{2}$} (y3kl);
  \draw[->] (nonamer) to node [below] {$\rho_{3}$} (y2kl);
  \draw[->] (nonameu) to node [above] {$\rho_{3}$} (y2klt);
  \draw[->] (y2klt) to node [above] {$\rho_{2}$} (y3klt);
\end{tikzpicture}
}
\caption{A lower left joint connected to a lower right joint. $\tilde{\kappa_\ell}$ refers to the $\kappa_\ell$ from the lower right joint.}\label{fig:j4f}
\end{subfigure}

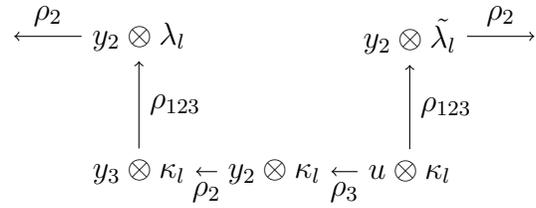
\caption{Lower left joint}
\end{figure}

\begin{figure}

\begin{subfigure}{0.99\textwidth}
\centerline{
\begin{tikzpicture}[node distance=1.8cm, auto]
  \node (y3kl) {$y_3\ts\kappa_l$};
  \node (y2kl) [right of =y3kl] {$y_2\ts\kappa_l$};
  \node (ukl) [right of =y2kl] {$u\ts\kappa_l$};
  \node (y2ll) [above of =y3kl] {$y_2\ts\lambda_l$};
  \node (y2llt) [above of =ukl] {$y_2\ts\tilde{\lambda_l}$};
  \node (nonamer) [right of =y2llt] {};
  \node (nonamel) [left of =y2ll] {};
  \draw[->] (y3kl) to node [right] {$\rho_{123}$} (y2ll);
  \draw[->] (y2ll) to node [above] {$\rho_{2}$} (nonamel);
  \draw[->] (y2kl) to node [below] {$\rho_{2}$} (y3kl);
  \draw[->] (ukl) to node [below] {$\rho_{3}$} (y2kl);
  \draw[->] (ukl) to node [right] {$\rho_{123}$} (y2llt);
  \draw[->] (y2llt) to node [above] {$\rho_{2}$} (nonamer);
\end{tikzpicture}
}
\caption{A lower left joint connected to a upper left joint. $\tilde{\lambda_\ell}$ refers to the $\lambda_\ell$ from the upper left joint.}\label{fig:j4g}
\end{subfigure}

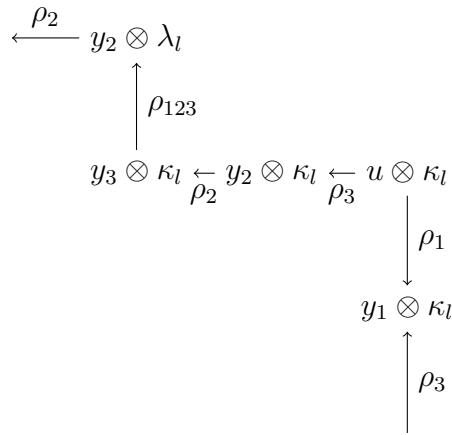
\begin{subfigure}{0.99\textwidth}
\centerline{
\begin{tikzpicture}[node distance=1.8cm, auto]
  \node (y3kl) {$y_3\ts\kappa_l$};
  \node (y2kl) [right of =y3kl] {$y_2\ts\kappa_l$};
  \node (ukl) [right of =y2kl] {$u\ts\kappa_l$};
  \node (y2ll) [above of =y3kl] {$y_2\ts\lambda_l$};
  \node (y1kl) [below of =ukl] {$y_1\ts{\kappa_l}$};
  \node (nonamer) [below of =y1kl] {};
  \node (nonamel) [left of =y2ll] {};
  \draw[->] (y3kl) to node [right] {$\rho_{123}$} (y2ll);
  \draw[->] (y2ll) to node [above] {$\rho_{2}$} (nonamel);
  \draw[->] (y2kl) to node [below] {$\rho_{2}$} (y3kl);
  \draw[->] (ukl) to node [below] {$\rho_{3}$} (y2kl);
  \draw[->] (ukl) to node [right] {$\rho_{1}$} (y1kl);
  \draw[->] (nonamer) to node [right] {$\rho_{3}$} (y1kl);
\end{tikzpicture}
}
\caption{A lower left joint connected to a upper right joint in a different way}\label{fig:j4h}
\end{subfigure}

\caption{Lower left joint}
\end{figure}

\FloatBarrier
\subsubsection{The unstable chain}

So far we have avoided $\xi_v$, $\xi_h$, and the unstable chain between them in our discussion. We now argue that the proof for the unstable chain can be reduced to the cases discussed above.

Recall when $n<2\tau(K)$, the unstable chain takes form:
\[
\xi_v'\xrightarrow{\rho_1}\mu_{1}\xleftarrow{\rho_{23}}\cdots\xleftarrow{\rho_{23}}\mu_{k}\xleftarrow{\rho_{23}}\mu_{k+1}\xleftarrow{\rho_{23}}\cdots\xleftarrow{\rho_{23}}\mu_{m}\xleftarrow{\rho_{3}}\xi_{h}',
\] 
where $m=2\tau(K)-n\geq 3$, by assumption. 

If $\xi_v = \xi_h$, the unstable chain appears as in \cref{fig:unstable1}. Now, $H$ acts on the string of $\rh{23}$'s by flipping them into the opposite direction. At generator $\xi_v'$, the arrows are exactly like a upper right joint in \cref{fig:j1a} in \cref{sec:urj}, so the arrows are also flipped. The proof there carries over to show that $H$ tensoring with the unstable chain in $KtD(C)$ matches the unstable chain in $KtD(C^{flip})$.

Now suppose $\xi_v \neq \xi_h$. We observe that the unstable chain is the end of chains for vertical arrows spliced with the end of chains for horizontal arrows. For this reason, we can interpret arrows at $\xi_v'$ and $\xi_h'$ as joints we have already discussed.

For example, arrows at $\xi_v'$ match exactly to a upper left joint as in \cref{fig:j2a} if $\xi_v$ has an incoming arrow, and a upper right joint as in \cref{fig:j1a} if $\xi_v$ has an outgoing arrow, see \cref{fig:unstable2}.

Similarly, arrows at $\xi_h'$ match exactly to a lower right joint as in \cref{fig:j3a} if $\xi_h$ has an incoming arrow, and a upper right joint as in \cref{fig:j1a} if $\xi_h$ has an outgoing arrow, see \cref{fig:unstable3}.

Now the proof is easy. In the case where $\xi_v$ has an incoming horizontal arrow, the corresponding generator $\xi_v^{flip}$ in $C^{flip}$ is a generator with no horizontal arrow and an incoming vertical arrow, so its corresponding generator $\xi_v^{flip\prime}$ in $KtD(C^{flip})$ has a lower right joint. Now $H$ acts by matching the upper left joint at $\xi_v'$ to the lower right joint at $\xi^{flip\prime}_v$ as discussed in \cref{sec:ulj}. The case where $\xi_h$ has an incoming vertical arrow works backwards using the $H$ action discussed in \cref{sec:lrj}.

Similarly, the correspondence between the case where $\xi_v$ has an outgoing horizontal arrow and the case where $\xi_h$ has an outgoing vertical arrow can be shown with the $H$ action discussed in \cref{sec:urj}.

Now we have finished the proof.

\begin{figure}[ht]

\begin{subfigure}{0.99\textwidth}
\centerline{
\begin{tikzpicture}[node distance=1.8cm, auto]
  \node (xii) {$\xi_v' = \xi_h'$};
  \node (l1) [left of =xii] {};
  \node (l2) [left of =l1] {};
  \node (l3) [below left of =l2] {};
  \node (d1) [below of =xii] {};
  \node (d2) [below left of =d1] {};
  \draw[->] (xii) to node [above] {$\rho_{3}$} (l1);
  \draw[->] (l1) to node [above] {$\rho_{23}$} (l2);
  \draw[->] (l2) to node [above] {$\rho_{23}$} (l3);
  \draw[->] (xii) to node [right] {$\rho_{1}$} (d1);
  \draw[->] (d2) to node [right] {$\rho_{23}$} (d1);
  \draw[->, dashed, bend right = 30] (l3) to node [right] {} (d2);
\end{tikzpicture}
}
\caption{Unstable chain in the case where $\xi_v = \xi_h$}\label{fig:unstable1}
\end{subfigure}

\begin{subfigure}{0.99\textwidth}
\centerline{
\begin{tikzpicture}[node distance=1.5cm, auto]
  \node (xii) {$\xi_v'$};
  \node (r1) [right of =xii] {};
  \node (r2) [right of =r1] {};
  \node (d1) [below of =xii] {};
  \node (d2) [below of =d1] {};
  \node (d3) [below right of =d2] {};
  \node (d35) [below right of =d3] {};
  \node (d4) [right of =d35] {};
  \draw[->] (r1) to node [above] {$\rho_{2}$} (xii);
  \draw[->] (r2) to node [above] {$\rho_{23}$} (r1);
  \draw[->] (r2) to node [below] {$\rho_{3}$} (r1);
  \draw[->] (xii) to node [right] {$\rho_{1}$} (d1);
  \draw[->] (d2) to node [right] {$\rho_{23}$} (d1);
  \draw[->] (d3) to node [right] {$\rho_{23}$} (d2);
  \draw[->, dashed, bend left = 30] (d4) to node [right] {} (d3);
  \node (middle) [right of =r2] {};
  \node (r2r) [right of =middle] {};
  \node (r1r) [right of =r2r] {};
  \node (xiir) [right of =r1r] {$\xi_v'$};
  \node (d1r) [below of =xiir] {};
  \node (d2r) [below of =d1r] {};
  \node (d3r) [below right of =d2r] {};
  \node (d35r) [below right of =d3r] {};
  \node (d4r) [right of =d35r] {};
  \draw[->] (xiir) to node [above] {$\rho_{3}$} (r1r);
  \draw[->] (r1r) to node [above] {$\rho_{23}$} (r2r);
  \draw[->] (r1r) to node [below] {$\rho_{2}$} (r2r);
  \draw[->] (xiir) to node [right] {$\rho_{1}$} (d1r);
  \draw[->] (d2r) to node [right] {$\rho_{23}$} (d1r);
  \draw[->] (d3r) to node [right] {$\rho_{23}$} (d2r);
  \draw[->, dashed, bend left = 30] (d4r) to node [right] {} (d3r);
\end{tikzpicture}
}
\caption{Joints at $\xi_v'$}\label{fig:unstable2}
\end{subfigure}

\begin{subfigure}{0.99\textwidth}
\centerline{
\begin{tikzpicture}[node distance=1.5cm, auto]
  \node (xii) at (0, 0) {$\xi_h'$};
  \node (u1) [above of =xii] {};
  \node (u2) at (-0.3, 3) {};
  \node (u22) at (0.3, 3) {};
  \node (l1) [left=of xii] {};
  \node (l2) [left=of l1] {};
  \node (l3) [left=of l2] {};
  \node (l4) [left=of l3] {};
  \draw[->] (xii) to node [above] {$\rho_{3}$} (l1);
  \draw[->] (l1) to node [above] {$\rho_{23}$} (l2);
  \draw[->] (l2) to node [above] {$\rho_{23}$} (l3);
  \draw[->, dashed] (l3) to node [above] {$\rho_{23}$} (l4);
  \draw[->] (xii) to node [right] {$\rho_{123}$} (u1);
  \draw[->] (u1) to node [right] {$\rho_{23}$} (u22);
  \draw[->] (u2) to node [left] {$\rho_{1}$} (u1);
  \node (middle) [right=of u22] {};
  \node (l4r) [right=of middle] {};
  \node (l3r) [right=of l4r] {};
  \node (l2r) [right=of l3r] {};
  \node (l1r) [right=of l2r] {};
  \node (xiir) [right=of l1r] {$\xi_h'$};
  \node (d1r) [below of =xiir] {};
  \node (d2r) [below of =d1r] {};
  \draw[->] (xiir) to node [above] {$\rho_{3}$} (l1r);
  \draw[->] (l1r) to node [above] {$\rho_{23}$} (l2r);
  \draw[->] (l2r) to node [above] {$\rho_{23}$} (l3r);
  \draw[->, dashed] (l3r) to node [above] {$\rho_{23}$} (l4r);
  \draw[->] (xiir) to node [right] {$\rho_{1}$} (d1r);
  \draw[->] (d2r) to node [right] {$\rho_{23}$} (d1r);
  \draw[->] (d2r) to node [left] {$\rho_{123}$} (d1r);
\end{tikzpicture}
}
\caption{Joints at $\xi_h'$}\label{fig:unstable3}
\end{subfigure}

\caption{}
\end{figure}

\FloatBarrier
\subsection{Proof of the General Case}\label{sec:proofg}
Now we prove \cref{flip} without the extra assumption of $C$ being simultaneously horizontally and vertically simplified. Instead we only assume it is reduced and horizontally simplified as in \cref{flip}. We use \cref{alg2}, the base-free version of the algorithm from $\widehat{CFD}$ to $CFK^-$. 

First, we give an example of \cref{alg2}. Suppose we have a knot Floer complex $C$ as in \cref{fig:alg2exCFK}.

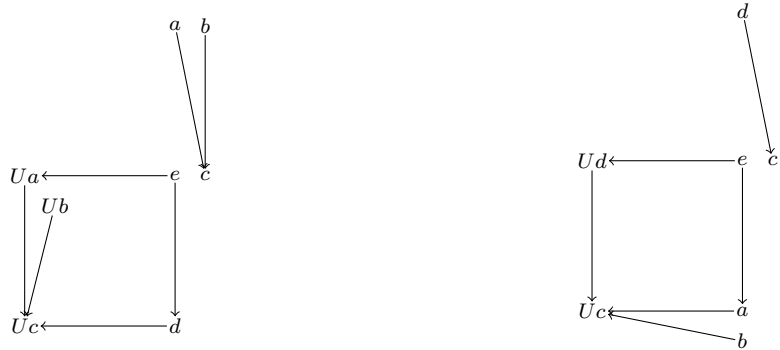
\begin{figure}[ht]

\begin{subfigure}{0.49\textwidth}
\centering
\begin{tikzpicture}[node distance=2cm]
\path[font = \scriptsize]
(-1.4, -1.95) node(uc) [outer sep=-3pt]{$Uc$}
(-1.4, 0.05) node(ua) [outer sep=-3pt]{$Ua$}
(0.6, 0.05) node(e) [outer sep=-3pt]{$e$}
(1.0, 0.05) node(c) [outer sep=-3pt]{$c$}
(0.6, -1.95) node(d) [outer sep=-3pt]{$d$}
(0.6, 2.05) node(a) [outer sep=-3pt]{$a$}
(-1.0, -0.35) node(ub) [outer sep=-3pt]{$Ub$}
(1.0, 2.05) node(b) [outer sep=-3pt]{$b$}
;
\draw[->, bend right = 0, black, font=\scriptsize] (a) to node [sloped] {} (c);
\draw[->, bend right = 0, black, font=\scriptsize] (b) to node [sloped] {} (c);
\draw[->, bend right = 0, black, font=\scriptsize] (d) to node [sloped] {} (uc);
\draw[->, bend right = 0, black, font=\scriptsize] (e) to node [sloped] {} (d);
\draw[->, bend right = 0, black, font=\scriptsize] (e) to node [sloped] {} (ua);
\draw[->, bend right = 0, black, font=\scriptsize] (ua) to node [sloped] {} (uc);
\draw[->, bend right = 0, black, font=\scriptsize] (ub) to node [sloped] {} (uc);
\end{tikzpicture}
\caption{Generators $a$ and $b$ have Alexander filtration level 1. Alexander filtration level of generators $c$ and $e$ is 0, and that of $d$ is -1.}\label{fig:alg2exCFK}
\end{subfigure}
\begin{subfigure}{0.49\textwidth}
\centering
\begin{tikzpicture}[node distance=2cm]
\path[font = \scriptsize]
(0.51, 0.63) node(e) [outer sep=-3pt]{$e$}
(0.51, -1.77) node(b) [outer sep=-3pt]{$b$}
(0.51, -1.37) node(a) [outer sep=-3pt]{$a$}
(0.91, 0.63) node(c) [outer sep=-3pt]{$c$}
(-1.49, -1.37) node(uc) [outer sep=-3pt]{$Uc$}
(0.51, 2.63) node(d) [outer sep=-3pt]{$d$}
(-1.49, 0.63) node(ud) [outer sep=-3pt]{$Ud$}
;
\draw[->, bend right = 0, black, font=\scriptsize] (a) to node [sloped] {} (uc);
\draw[->, bend right = 0, black, font=\scriptsize] (b) to node [sloped] {} (uc);
\draw[->, bend right = 0, black, font=\scriptsize] (d) to node [sloped] {} (c);
\draw[->, bend right = 0, black, font=\scriptsize] (e) to node [sloped] {} (ud);
\draw[->, bend right = 0, black, font=\scriptsize] (e) to node [sloped] {} (a);
\draw[->, bend right = 0, black, font=\scriptsize] (ud) to node [sloped] {} (uc);
\end{tikzpicture}
\caption{The Alexander filtration levels for $a,b,c,d$ and $e$ are -1, -1, 0, 1, and 0, respectively. }\label{fig:alg2exCFKf}
\end{subfigure}

\caption{}
\end{figure}

Applying \cref{alg2} with $n=7$ yields $KtD(C)$ as in \cref{fig:alg2exCFD}.

\clearpage
\newgeometry{margin=2cm}
\begin{landscape}
\begin{figure}
\centering
\begin{tikzpicture}[node distance=2cm]
\path[font = \scriptsize]
(4.64, -5.04) node(Alex$V^1_{2}$) [outer sep=-1pt]{$V^1_{2}$}
(4.76, -4.14) node(d_n2) [outer sep=-1pt]{$d$}
(-5.99, 2.16) node(a_2) [outer sep=-1pt]{$a$}
(9.76, 1.86) node(b_n4) [outer sep=-1pt]{$b$}
(6.51, 2.16) node(a_n3) [outer sep=-1pt]{$a$}
(-0.24, 9.36) node(e_0) [outer sep=-1pt]{$e$}
(6.51, -0.84) node(c_n3) [outer sep=-1pt]{$c$}
(6.01, 9.36) node(b_n1) [outer sep=-1pt]{$b$}
(4.76, -1.14) node(e_n2) [outer sep=-1pt]{$e$}
(-7.74, -1.14) node(e_3) [outer sep=-1pt]{$e$}
(-0.99, 9.66) node(c_0) [outer sep=-1pt]{$c$}
(-10.24, -4.14) node(d_4) [outer sep=-1pt]{$d$}
(2.26, -1.14) node(dot_n1) [outer sep=-1pt]{}
(4.01, -0.84) node(c_n2) [outer sep=-1pt]{$c$}
(-6.49, 9.36) node(d_1) [outer sep=-1pt]{$d$}
(-8.49, -0.84) node(c_3) [outer sep=-1pt]{$c$}
(-5.36, -5.04) node(Alex$V^1_{-2}$) [outer sep=-1pt]{$V^1_{-2}$}
(7.26, 1.86) node(b_n3) [outer sep=-1pt]{$b$}
(-2.74, -1.14) node(dot_1) [outer sep=-1pt]{}
(-0.24, -1.14) node(dot_0) [outer sep=-1pt]{}
(5.26, 9.66) node(a_n1) [outer sep=-1pt]{$a$}
(4.76, 1.86) node(b_n2) [outer sep=-1pt]{$b$}
(-6.61, 10.86) node(Alex$V^0_{-1}$) [outer sep=-1pt]{$V^0_{-1}$}
(-5.99, -0.84) node(c_2) [outer sep=-1pt]{$c$}
(-0.36, 10.86) node(Alex$V^0_{0}$) [outer sep=-1pt]{$V^0_{0}$}
(7.14, -5.04) node(Alex$V^1_{3}$) [outer sep=-1pt]{$V^1_{3}$}
(9.01, 2.16) node(a_n4) [outer sep=-1pt]{$a$}
(4.01, 2.16) node(a_n2) [outer sep=-1pt]{$a$}
(-7.86, -5.04) node(Alex$V^1_{-3}$) [outer sep=-1pt]{$V^1_{-3}$}
(2.14, -5.04) node(Alex$V^1_{1}$) [outer sep=-1pt]{$V^1_{1}$}
(-10.36, -5.04) node(Alex$V^1_{-4}$) [outer sep=-1pt]{$V^1_{-4}$}
(5.89, 10.86) node(Alex$V^0_{1}$) [outer sep=-1pt]{$V^0_{1}$}
(-5.24, 1.86) node(b_2) [outer sep=-1pt]{$b$}
(-5.24, -1.14) node(e_2) [outer sep=-1pt]{$e$}
(9.64, -5.04) node(Alex$V^1_{4}$) [outer sep=-1pt]{$V^1_{4}$}
(-7.74, -4.14) node(d_3) [outer sep=-1pt]{$d$}
(-2.86, -5.04) node(Alex$V^1_{-1}$) [outer sep=-1pt]{$V^1_{-1}$}
(-0.36, -5.04) node(Alex$V^1_{0}$) [outer sep=-1pt]{$V^1_{0}$}
(7.26, -1.14) node(e_n3) [outer sep=-1pt]{$e$}
(-5.24, -4.14) node(d_2) [outer sep=-1pt]{$d$}
;
\draw[->, bend right = -5.06, blue, big arrow=blue, font=\scriptsize](a_n1) to node [sloped] {$\rho_{3}$} (a_2);
\draw[->, bend right = -3.88, blue, big arrow=blue, font=\scriptsize](a_n1) to node [sloped] {$\rho_{1}$} (a_n4);
\draw[->, bend right = 0, black, big arrow=black, font=\scriptsize](a_n2) to node [sloped] {$\iota_1$} (c_n2);
\draw[->, bend right = 0, red, big arrow=red, font=\scriptsize](a_n2) to node [sloped] {$\rho_{23}$} (a_n3);
\draw[->, bend right = 0, black, big arrow=black, font=\scriptsize](a_n3) to node [sloped] {$\iota_1$} (c_n3);
\draw[->, bend right = 0, red, big arrow=red, font=\scriptsize](a_n3) to node [sloped] {$\rho_{23}$} (a_n4);
\draw[->, bend right = 0, red, big arrow=red, font=\scriptsize](b_2) to node [sloped] {$\rho_{23}$} (dot_1);
\draw[->, bend right = -5.06, blue, big arrow=blue, font=\scriptsize](b_n1) to node [sloped] {$\rho_{3}$} (b_2);
\draw[->, bend right = -3.88, blue, big arrow=blue, font=\scriptsize](b_n1) to node [sloped] {$\rho_{1}$} (b_n4);
\draw[->, bend right = 0, black, big arrow=black, font=\scriptsize](b_n2) to node [sloped] {$\iota_1$} (c_n2);
\draw[->, bend right = 0, red, big arrow=red, font=\scriptsize](b_n2) to node [sloped] {$\rho_{23}$} (b_n3);
\draw[->, bend right = 0, black, big arrow=black, font=\scriptsize](b_n3) to node [sloped] {$\iota_1$} (c_n3);
\draw[->, bend right = 0, red, big arrow=red, font=\scriptsize](b_n3) to node [sloped] {$\rho_{23}$} (b_n4);
\draw[->, bend right = -7.2, blue, big arrow=blue, font=\scriptsize](c_0) to node [sloped] {$\rho_{3}$} (c_3);
\draw[->, bend right = -8.1, blue, big arrow=blue, font=\scriptsize](c_0) to node [sloped] {$\rho_{1}$} (c_n3);
\draw[->, bend right = 0, red, big arrow=red, font=\scriptsize](c_3) to node [sloped] {$\rho_{23}$} (c_2);
\draw[->, bend right = 0, red, big arrow=red, font=\scriptsize](c_n2) to node [sloped] {$\rho_{23}$} (c_n3);
\draw[->, bend right = -4.42, blue, big arrow=blue, font=\scriptsize](d_1) to node [sloped] {$\rho_{3}$} (d_4);
\draw[->, bend right = -9.0, blue, big arrow=blue, font=\scriptsize](d_1) to node [sloped] {$\rho_{1}$} (d_n2);
\draw[->, bend right = -5.04, pink, big arrow=pink, font=\scriptsize](d_1) to node [sloped] {$\rho_{123}$} (c_n3);
\draw[->, bend right = 0, black, big arrow=black, font=\scriptsize](d_2) to node [sloped] {$\iota_1$} (c_2);
\draw[->, bend right = 0, black, big arrow=black, font=\scriptsize](d_3) to node [sloped] {$\iota_1$} (c_3);
\draw[->, bend right = 0, red, big arrow=red, font=\scriptsize](d_3) to node [sloped] {$\rho_{23}$} (d_2);
\draw[->, bend right = 0, red, big arrow=red, font=\scriptsize](d_4) to node [sloped] {$\rho_{23}$} (d_3);
\draw[->, bend right = -8.89, green, big arrow=green, font=\scriptsize](d_4) to node [sloped] {$\rho_{2}$} (c_0);
\draw[->, bend right = 0, red, big arrow=red, font=\scriptsize](dot_0) to node [sloped] {$\rho_{23}$} (dot_n1);
\draw[->, bend right = 0, red, big arrow=red, font=\scriptsize](dot_1) to node [sloped] {$\rho_{23}$} (dot_0);
\draw[->, bend right = 0, red, big arrow=red, font=\scriptsize](dot_n1) to node [sloped] {$\rho_{23}$} (b_n2);
\draw[->, bend right = 0, red, big arrow=red, font=\scriptsize](dot_n1) to node [sloped] {$\rho_{23}$} (a_n2);
\draw[->, bend right = -7.2, blue, big arrow=blue, font=\scriptsize](e_0) to node [sloped] {$\rho_{3}$} (e_3);
\draw[->, bend right = -7.2, blue, big arrow=blue, font=\scriptsize](e_0) to node [sloped] {$\rho_{1}$} (e_n3);
\draw[->, bend right = 0, pink, big arrow=pink, font=\scriptsize](e_0) to node [sloped] {$\rho_{123}$} (a_n4);
\draw[->, bend right = 0, black, big arrow=black, font=\scriptsize](e_2) to node [sloped] {$\iota_1$} (a_2);
\draw[->, bend right = 0, red, big arrow=red, font=\scriptsize](e_3) to node [sloped] {$\rho_{23}$} (e_2);
\draw[->, bend right = -5.64, green, big arrow=green, font=\scriptsize](e_3) to node [sloped] {$\rho_{2}$} (a_n1);
\draw[->, bend right = 0, black, big arrow=black, font=\scriptsize](e_n2) to node [sloped] {$\iota_1$} (d_n2);
\draw[->, bend right = 0, red, big arrow=red, font=\scriptsize](e_n2) to node [sloped] {$\rho_{23}$} (e_n3);
\end{tikzpicture}
\caption{}\label{fig:alg2exCFD}
\end{figure}
\end{landscape}
\clearpage
\restoregeometry

If we apply the procedures in \cref{alg2} to $C^{flip}$, which is shown in \cref{fig:alg2exCFKf}, we get the complex $KtD(C^{flip})$ shown in \cref{fig:alg2exCFDf}. We now rearrange the generators by rotating both $V^0$ and $V^1$ $180$ degrees to arrive at \cref{fig:alg2exCFDff}.

\clearpage
\newgeometry{margin=2cm}
\begin{landscape}
\begin{figure}
\centering
\begin{tikzpicture}[node distance=2cm]
\path[font = \scriptsize]
(-7.26, -5.05) node(Alex$V^1_{-3}$) [outer sep=-1pt]{$V^1_{-3}$}
(0.24, 9.95) node(c_0) [outer sep=-1pt]{$c$}
(-9.01, -3.85) node(b_n4) [outer sep=-1pt]{$b$}
(8.49, -0.85) node(e_3) [outer sep=-1pt]{$e$}
(-7.26, -3.55) node(a_n3) [outer sep=-1pt]{$a$}
(0.99, 9.65) node(e_0) [outer sep=-1pt]{$e$}
(-7.26, -0.55) node(c_n3) [outer sep=-1pt]{$c$}
(-5.26, 9.65) node(b_n1) [outer sep=-1pt]{$b$}
(-4.01, -0.85) node(e_n2) [outer sep=-1pt]{$e$}
(5.99, -3.85) node(b_2) [outer sep=-1pt]{$b$}
(5.24, -3.55) node(a_2) [outer sep=-1pt]{$a$}
(10.24, 2.45) node(d_4) [outer sep=-1pt]{$d$}
(0.24, -5.05) node(Alex$V^1_{0}$) [outer sep=-1pt]{$V^1_{0}$}
(-4.76, -0.55) node(c_n2) [outer sep=-1pt]{$c$}
(-4.76, -5.05) node(Alex$V^1_{-2}$) [outer sep=-1pt]{$V^1_{-2}$}
(-6.51, -3.85) node(b_n3) [outer sep=-1pt]{$b$}
(2.74, -0.55) node(dot_1) [outer sep=-1pt]{}
(0.24, -0.55) node(dot_0) [outer sep=-1pt]{}
(-6.01, 9.95) node(a_n1) [outer sep=-1pt]{$a$}
(0.24, 11.45) node(Alex$V^0_{0}$) [outer sep=-1pt]{$V^0_{0}$}
(5.24, -5.05) node(Alex$V^1_{2}$) [outer sep=-1pt]{$V^1_{2}$}
(-6.01, 11.45) node(Alex$V^0_{-1}$) [outer sep=-1pt]{$V^0_{-1}$}
(5.24, -0.55) node(c_2) [outer sep=-1pt]{$c$}
(6.49, 9.95) node(d_1) [outer sep=-1pt]{$d$}
(7.74, -5.05) node(Alex$V^1_{3}$) [outer sep=-1pt]{$V^1_{3}$}
(-9.76, -3.55) node(a_n4) [outer sep=-1pt]{$a$}
(-4.76, -3.55) node(a_n2) [outer sep=-1pt]{$a$}
(-4.76, 2.45) node(d_n2) [outer sep=-1pt]{$d$}
(2.74, -5.05) node(Alex$V^1_{1}$) [outer sep=-1pt]{$V^1_{1}$}
(10.24, -5.05) node(Alex$V^1_{4}$) [outer sep=-1pt]{$V^1_{4}$}
(-9.76, -5.05) node(Alex$V^1_{-4}$) [outer sep=-1pt]{$V^1_{-4}$}
(6.49, 11.45) node(Alex$V^0_{1}$) [outer sep=-1pt]{$V^0_{1}$}
(5.99, -0.85) node(e_2) [outer sep=-1pt]{$e$}
(-6.51, -0.85) node(e_n3) [outer sep=-1pt]{$e$}
(5.24, 2.45) node(d_2) [outer sep=-1pt]{$d$}
(7.74, 2.45) node(d_3) [outer sep=-1pt]{$d$}
(-2.26, -5.05) node(Alex$V^1_{-1}$) [outer sep=-1pt]{$V^1_{-1}$}
(7.74, -0.55) node(c_3) [outer sep=-1pt]{$c$}
(-4.01, -3.85) node(b_n2) [outer sep=-1pt]{$b$}
(-2.26, -0.55) node(dot_n1) [outer sep=-1pt]{}
;
\draw[->, bend right = -9.0, blue, big arrow=blue, font=\scriptsize](a_n1) to node [sloped] {$\rho_{1}$} (a_2);
\draw[->, bend right = -4.42, blue, big arrow=blue, font=\scriptsize](a_n1) to node [sloped] {$\rho_{3}$} (a_n4);
\draw[->, bend right = -8.0, pink, big arrow=pink, font=\scriptsize](a_n1) to node [sloped] {$\rho_{123}$} (c_3);
\draw[->, bend right = 0, black, big arrow=black, font=\scriptsize](a_n2) to node [sloped] {$\iota_1$} (c_n2);
\draw[->, bend right = 0, black, big arrow=black, font=\scriptsize](a_n3) to node [sloped] {$\iota_1$} (c_n3);
\draw[->, bend right = 0.6, red, big arrow=red, font=\scriptsize](a_n3) to node [sloped] {$\rho_{23}$} (a_n2);
\draw[->, bend right = -7.94, green, big arrow=green, font=\scriptsize](a_n4) to node [sloped] {$\rho_{2}$} (c_0);
\draw[->, bend right = 0.6, red, big arrow=red, font=\scriptsize](a_n4) to node [sloped] {$\rho_{23}$} (a_n3);
\draw[->, bend right = -7.25, blue, big arrow=blue, font=\scriptsize](b_n1) to node [sloped] {$\rho_{1}$} (b_2);
\draw[->, bend right = -4.42, blue, big arrow=blue, font=\scriptsize](b_n1) to node [sloped] {$\rho_{3}$} (b_n4);
\draw[->, bend right = -5.04, pink, big arrow=pink, font=\scriptsize](b_n1) to node [sloped] {$\rho_{123}$} (c_3);
\draw[->, bend right = 0, black, big arrow=black, font=\scriptsize](b_n2) to node [sloped] {$\iota_1$} (c_n2);
\draw[->, bend right = 0, red, big arrow=red, font=\scriptsize](b_n2) to node [sloped] {$\rho_{23}$} (dot_n1);
\draw[->, bend right = 0, black, big arrow=black, font=\scriptsize](b_n3) to node [sloped] {$\iota_1$} (c_n3);
\draw[->, bend right = 0.6, red, big arrow=red, font=\scriptsize](b_n3) to node [sloped] {$\rho_{23}$} (b_n2);
\draw[->, bend right = -8.89, green, big arrow=green, font=\scriptsize](b_n4) to node [sloped] {$\rho_{2}$} (c_0);
\draw[->, bend right = 0.6, red, big arrow=red, font=\scriptsize](b_n4) to node [sloped] {$\rho_{23}$} (b_n3);
\draw[->, bend right = -7.2, blue, big arrow=blue, font=\scriptsize](c_0) to node [sloped] {$\rho_{1}$} (c_3);
\draw[->, bend right = -7.2, blue, big arrow=blue, font=\scriptsize](c_0) to node [sloped] {$\rho_{3}$} (c_n3);
\draw[->, bend right = 0.6, red, big arrow=red, font=\scriptsize](c_2) to node [sloped] {$\rho_{23}$} (c_3);
\draw[->, bend right = 0.6, red, big arrow=red, font=\scriptsize](c_n3) to node [sloped] {$\rho_{23}$} (c_n2);
\draw[->, bend right = 0, blue, big arrow=blue, font=\scriptsize](d_1) to node [sloped] {$\rho_{1}$} (d_4);
\draw[->, bend right = 0, blue, big arrow=blue, font=\scriptsize](d_1) to node [sloped] {$\rho_{3}$} (d_n2);
\draw[->, bend right = 0, black, big arrow=black, font=\scriptsize](d_2) to node [sloped] {$\iota_1$} (c_2);
\draw[->, bend right = 0, red, big arrow=red, font=\scriptsize](d_2) to node [sloped] {$\rho_{23}$} (d_3);
\draw[->, bend right = 0, black, big arrow=black, font=\scriptsize](d_3) to node [sloped] {$\iota_1$} (c_3);
\draw[->, bend right = 0, red, big arrow=red, font=\scriptsize](d_3) to node [sloped] {$\rho_{23}$} (d_4);
\draw[->, bend right = 0, red, big arrow=red, font=\scriptsize](dot_0) to node [sloped] {$\rho_{23}$} (dot_1);
\draw[->, bend right = -3.41, red, big arrow=red, font=\scriptsize](dot_1) to node [sloped] {$\rho_{23}$} (b_2);
\draw[->, bend right = 0, red, big arrow=red, font=\scriptsize](dot_n1) to node [sloped] {$\rho_{23}$} (dot_0);
\draw[->, bend right = -6.3, blue, big arrow=blue, font=\scriptsize](e_0) to node [sloped] {$\rho_{1}$} (e_3);
\draw[->, bend right = -5.94, blue, big arrow=blue, font=\scriptsize](e_0) to node [sloped] {$\rho_{3}$} (e_n3);
\draw[->, bend right = 1.48, pink, big arrow=pink, font=\scriptsize](e_0) to node [sloped] {$\rho_{123}$} (d_4);
\draw[->, bend right = 0, black, big arrow=black, font=\scriptsize](e_2) to node [sloped] {$\iota_1$} (a_2);
\draw[->, bend right = 0.6, red, big arrow=red, font=\scriptsize](e_2) to node [sloped] {$\rho_{23}$} (e_3);
\draw[->, bend right = 0, black, big arrow=black, font=\scriptsize](e_n2) to node [sloped] {$\iota_1$} (d_n2);
\draw[->, bend right = -4.8, green, big arrow=green, font=\scriptsize](e_n3) to node [sloped] {$\rho_{2}$} (d_1);
\draw[->, bend right = 0.6, red, big arrow=red, font=\scriptsize](e_n3) to node [sloped] {$\rho_{23}$} (e_n2);
\end{tikzpicture}
\caption{}\label{fig:alg2exCFDf}
\end{figure}
\end{landscape}
\clearpage
\restoregeometry

\clearpage
\newgeometry{margin=2cm}
\begin{landscape}
\begin{figure}
\centering
\begin{tikzpicture}[node distance=2cm]
\path[font = \scriptsize]
(7.26, -5.04) node(Alex$V^1_{-3}$) [outer sep=-1pt]{$V^1_{-3}$}
(-0.99, 9.66) node(c_0) [outer sep=-1pt]{$c$}
(9.76, 1.86) node(b_n4) [outer sep=-1pt]{$b$}
(-7.74, -1.14) node(e_3) [outer sep=-1pt]{$e$}
(6.51, 2.16) node(a_n3) [outer sep=-1pt]{$a$}
(-0.24, 9.36) node(e_0) [outer sep=-1pt]{$e$}
(6.51, -0.84) node(c_n3) [outer sep=-1pt]{$c$}
(6.01, 9.36) node(b_n1) [outer sep=-1pt]{$b$}
(4.76, -1.14) node(e_n2) [outer sep=-1pt]{$e$}
(-5.24, 1.86) node(b_2) [outer sep=-1pt]{$b$}
(-5.99, 2.16) node(a_2) [outer sep=-1pt]{$a$}
(-10.24, -4.14) node(d_4) [outer sep=-1pt]{$d$}
(-0.24, -5.04) node(Alex$V^1_{0}$) [outer sep=-1pt]{$V^1_{0}$}
(4.01, -0.84) node(c_n2) [outer sep=-1pt]{$c$}
(4.76, -5.04) node(Alex$V^1_{-2}$) [outer sep=-1pt]{$V^1_{-2}$}
(7.26, 1.86) node(b_n3) [outer sep=-1pt]{$b$}
(-2.74, -1.14) node(dot_1) [outer sep=-1pt]{}
(-0.24, -1.14) node(dot_0) [outer sep=-1pt]{}
(5.26, 9.66) node(a_n1) [outer sep=-1pt]{$a$}
(-0.24, 10.86) node(Alex$V^0_{0}$) [outer sep=-1pt]{$V^0_{0}$}
(-5.24, -5.04) node(Alex$V^1_{2}$) [outer sep=-1pt]{$V^1_{2}$}
(6.01, 10.86) node(Alex$V^0_{-1}$) [outer sep=-1pt]{$V^0_{-1}$}
(-5.99, -0.84) node(c_2) [outer sep=-1pt]{$c$}
(-6.49, 9.36) node(d_1) [outer sep=-1pt]{$d$}
(-7.74, -5.04) node(Alex$V^1_{3}$) [outer sep=-1pt]{$V^1_{3}$}
(9.01, 2.16) node(a_n4) [outer sep=-1pt]{$a$}
(4.01, 2.16) node(a_n2) [outer sep=-1pt]{$a$}
(4.76, -4.14) node(d_n2) [outer sep=-1pt]{$d$}
(-2.74, -5.04) node(Alex$V^1_{1}$) [outer sep=-1pt]{$V^1_{1}$}
(-10.24, -5.04) node(Alex$V^1_{4}$) [outer sep=-1pt]{$V^1_{4}$}
(9.76, -5.04) node(Alex$V^1_{-4}$) [outer sep=-1pt]{$V^1_{-4}$}
(-6.49, 10.86) node(Alex$V^0_{1}$) [outer sep=-1pt]{$V^0_{1}$}
(-5.24, -1.14) node(e_2) [outer sep=-1pt]{$e$}
(7.26, -1.14) node(e_n3) [outer sep=-1pt]{$e$}
(-5.24, -4.14) node(d_2) [outer sep=-1pt]{$d$}
(-7.74, -4.14) node(d_3) [outer sep=-1pt]{$d$}
(2.26, -5.04) node(Alex$V^1_{-1}$) [outer sep=-1pt]{$V^1_{-1}$}
(-8.49, -0.84) node(c_3) [outer sep=-1pt]{$c$}
(4.76, 1.86) node(b_n2) [outer sep=-1pt]{$b$}
(2.26, -1.14) node(dot_n1) [outer sep=-1pt]{}
;
\draw[->, bend right = -5.06, blue, big arrow=blue, font=\scriptsize](a_n1) to node [sloped] {$\rho_{1}$} (a_2);
\draw[->, bend right = -3.88, blue, big arrow=blue, font=\scriptsize](a_n1) to node [sloped] {$\rho_{3}$} (a_n4);
\draw[->, bend right = -8.2, pink, big arrow=pink, font=\scriptsize](a_n1) to node [sloped] {$\rho_{123}$} (c_3);
\draw[->, bend right = 0, black, big arrow=black, font=\scriptsize](a_n2) to node [sloped] {$\iota_1$} (c_n2);
\draw[->, bend right = 0, black, big arrow=black, font=\scriptsize](a_n3) to node [sloped] {$\iota_1$} (c_n3);
\draw[->, bend right = 0.6, red, big arrow=red, font=\scriptsize](a_n3) to node [sloped] {$\rho_{23}$} (a_n2);
\draw[->, bend right = -7.66, green, big arrow=green, font=\scriptsize](a_n4) to node [sloped] {$\rho_{2}$} (c_0);
\draw[->, bend right = 0.6, red, big arrow=red, font=\scriptsize](a_n4) to node [sloped] {$\rho_{23}$} (a_n3);
\draw[->, bend right = -5.06, blue, big arrow=blue, font=\scriptsize](b_n1) to node [sloped] {$\rho_{1}$} (b_2);
\draw[->, bend right = -3.88, blue, big arrow=blue, font=\scriptsize](b_n1) to node [sloped] {$\rho_{3}$} (b_n4);
\draw[->, bend right = -7.47, pink, big arrow=pink, font=\scriptsize](b_n1) to node [sloped] {$\rho_{123}$} (c_3);
\draw[->, bend right = 0, black, big arrow=black, font=\scriptsize](b_n2) to node [sloped] {$\iota_1$} (c_n2);
\draw[->, bend right = 0, red, big arrow=red, font=\scriptsize](b_n2) to node [sloped] {$\rho_{23}$} (dot_n1);
\draw[->, bend right = 0, black, big arrow=black, font=\scriptsize](b_n3) to node [sloped] {$\iota_1$} (c_n3);
\draw[->, bend right = 0.6, red, big arrow=red, font=\scriptsize](b_n3) to node [sloped] {$\rho_{23}$} (b_n2);
\draw[->, bend right = -7.81, green, big arrow=green, font=\scriptsize](b_n4) to node [sloped] {$\rho_{2}$} (c_0);
\draw[->, bend right = 0.6, red, big arrow=red, font=\scriptsize](b_n4) to node [sloped] {$\rho_{23}$} (b_n3);
\draw[->, bend right = -7.2, blue, big arrow=blue, font=\scriptsize](c_0) to node [sloped] {$\rho_{1}$} (c_3);
\draw[->, bend right = -8.1, blue, big arrow=blue, font=\scriptsize](c_0) to node [sloped] {$\rho_{3}$} (c_n3);
\draw[->, bend right = 0.6, red, big arrow=red, font=\scriptsize](c_2) to node [sloped] {$\rho_{23}$} (c_3);
\draw[->, bend right = 0.6, red, big arrow=red, font=\scriptsize](c_n3) to node [sloped] {$\rho_{23}$} (c_n2);
\draw[->, bend right = -4.42, blue, big arrow=blue, font=\scriptsize](d_1) to node [sloped] {$\rho_{1}$} (d_4);
\draw[->, bend right = -9.0, blue, big arrow=blue, font=\scriptsize](d_1) to node [sloped] {$\rho_{3}$} (d_n2);
\draw[->, bend right = 0, black, big arrow=black, font=\scriptsize](d_2) to node [sloped] {$\iota_1$} (c_2);
\draw[->, bend right = 0, red, big arrow=red, font=\scriptsize](d_2) to node [sloped] {$\rho_{23}$} (d_3);
\draw[->, bend right = 0, black, big arrow=black, font=\scriptsize](d_3) to node [sloped] {$\iota_1$} (c_3);
\draw[->, bend right = 0, red, big arrow=red, font=\scriptsize](d_3) to node [sloped] {$\rho_{23}$} (d_4);
\draw[->, bend right = 0, red, big arrow=red, font=\scriptsize](dot_0) to node [sloped] {$\rho_{23}$} (dot_1);
\draw[->, bend right = 0, red, big arrow=red, font=\scriptsize](dot_1) to node [sloped] {$\rho_{23}$} (b_2);
\draw[->, bend right = 0, red, big arrow=red, font=\scriptsize](dot_n1) to node [sloped] {$\rho_{23}$} (dot_0);
\draw[->, bend right = -7.2, blue, big arrow=blue, font=\scriptsize](e_0) to node [sloped] {$\rho_{1}$} (e_3);
\draw[->, bend right = -7.2, blue, big arrow=blue, font=\scriptsize](e_0) to node [sloped] {$\rho_{3}$} (e_n3);
\draw[->, bend right = -7.94, pink, big arrow=pink, font=\scriptsize](e_0) to node [sloped] {$\rho_{123}$} (d_4);
\draw[->, bend right = 0, black, big arrow=black, font=\scriptsize](e_2) to node [sloped] {$\iota_1$} (a_2);
\draw[->, bend right = 0.6, red, big arrow=red, font=\scriptsize](e_2) to node [sloped] {$\rho_{23}$} (e_3);
\draw[->, bend right = 0, black, big arrow=black, font=\scriptsize](e_n2) to node [sloped] {$\iota_1$} (d_n2);
\draw[->, bend right = -8.0, green, big arrow=green, font=\scriptsize](e_n3) to node [sloped] {$\rho_{2}$} (d_1);
\draw[->, bend right = 0.6, red, big arrow=red, font=\scriptsize](e_n3) to node [sloped] {$\rho_{23}$} (e_n2);
\end{tikzpicture}
\caption{}\label{fig:alg2exCFDff}
\end{figure}
\end{landscape}
\clearpage
\restoregeometry

Comparing the result with $KtD(C)$, we see the generators can be completely identified between them. So are the differentials without non-trivial algebra elements. All arrows with $\rh{23}$'s have their directions reversed, except those between the central string of $\rho_{23}$'s and the two full copies of $C$ immediately next to the string. $\rh{1}$'s and $\rh{3}$'s switch sides. With careful choices of $n$ and the Alexander filtrations on $C^{flip}$, we can proof this is in general true.

We first set the convention to display $KtD$'s as we did in the example above: $V^0$ on top of $V^1$, with smaller $s$ on the left. Within $V^1_{s}$, we arrange generators vertically according to Alexander level and we call them ``in the same column''. On the horizontal direction, put generators corresponding to the same generator in $C$ on the same horizontal line. They are denoted by the same letters in the diagrams. Hence they are connected by horizontal arrows with $\rho_{23}$'s. We say they are on the same ``row''. 

We state here the identifications between various sub/quotient complexes of $C$ and $\fl{C}$:

\[
(C(\leq s), \partial_w) \cong (\fl{C}(\geq -s), \partial_z)
\]
\[
(C(\geq s), \partial_z) \cong (\fl{C}(\leq -s), \partial_w)
\]

\begin{lemma}\label{localchange}
Let $h$ be the maximal Alexander filtration level and $\ell$ be the minimal Alexander filtration level of $C$. Let $t=\textit{max}\{h,-\ell\}$. We choose $n=4t+3$ for both $KtD(C)$ and $KtD(\fl{C})$ from now on. Note that with the choice of $n$, we simply have $V^1_{s}=\mathbb{F}_2, |s|\leq t$, $V^1_{t+1}=(C,\partial_z)$, and $V^1_{-t-1}=(C,\partial_w)$. We call the latter two (and other $V_s^1$ such that $V^1_{s\geq t+1}=(C,\partial_z)$ and $V^1_{s \leq -t-1}=(C,\partial_w)$) the ``full copies''.

Now $KtD(\fl{C})$ can be directly constructed from $KtD(C)$ as follows. We describe this construction graphically to facilitate the proof later.
\begin{itemize}
\item Keep all the generators in $KtD(C)$.
\item Relabel $V^0_s$ with $V^0_{-s}$ and $V^1_s$ with $V^1_{-s}$ for all $s$, but still arrange them graphically the same way, i.e. $s$ decreases now from left to right.
\item Reverse the direction of all $\rh{23}$'s, except those between the central sequence of $\rho_{23}$'s and the two full copies on either side of the sequence. 
\item Replace the rest of $\rho_{23}$'s (those between the sequence and the two full copies) with isomorphism-inducing chain maps in the opposite direction.
\item At every generator in $V^0$, switch the arrows $\rh{1}$ and $\rh{3}$ coming out.
\item Remove all existing $\rh{2}$ and $\rh{123}$.
\item For $s<-t$, $V^1_{s}=C(\geq -s-\frac{n-1}{2})$ is missing $C(-s-1-\frac{n-1}{2})$ if compared to the column on its left $V^1_{s+1}=C(\geq -s-1-\frac{n-1}{2})$. If there exists an downward arrow $x\to y$ in $V^1_{s+1}$, where $y\in C(-s-1-\frac{n-1}{2})$, then add an arrow $\rh{2}$ from $x\in V^1_{s}$ to $y \in V^0_{s+\frac{n+1}{2}}$.
\item For every generator $x \in V^0$, if there exists an (downward) outgoing arrow $x\to y$ of length 1 in $V^1_{-t}$, then add an arrow $\rh{123}$ from $x \in V^0$ to the unique generator $y$ at the leftmost position of its row.\end{itemize}
\end{lemma}

\begin{proof}
This proof is nothing more than applying the procedure in \cref{alg2} to $\fl{C}$ and comparing with $KtD(C)$.

We call the D-module constructed this way $X$ and argue $X$ can be identified with $KtD(\fl{C})$.

Per our construction, $\forall s$, ($V^0_s$ of $X$) $\cong$ ($V^0_{-s}$ of $KtD(C)$) $\cong$ $C(-s)$ $\cong$ $\fl{C}(s)$ $\cong$  ($V^0_{s}$ of $KtD(\fl{C}))$, as $\mathbb{Z}$-filtered chain complexes over $\mathbb{F}_2$. 

$\forall s\leq -\frac{n}{4}=-t-0.75$, ($V^1_s$ of $X$) $\cong$ ($V^1_{-s}$ of $KtD(C)$) $\cong$ $(C(\geq -s-\frac{n-1}{2}=-s-2t-1), \partial_z)$ $\cong$ $(\fl{C}(\leq s+2t+1), \partial_w)$ $\cong$  ($V^1_{s}$ of $KtD(\fl{C}))$, as $\mathbb{Z}$-filtered chain complexes over $\mathbb{F}_2$.

$\forall |s|\leq \frac{n}{4}=t+0.75$, ($V^1_s$ of $X$) $\cong$ $\mathbb{F}_2$ $\cong$  ($V^1_{s}$ of $KtD(\fl{C}))$

$\forall s\geq \frac{n}{4}=t+0.75$, ($V^1_s$ of $X$) $\cong$ ($V^1_{-s}$ of $KtD(C)$) $\cong$ $(C(\leq -s+\frac{n-1}{2}=-s+2t+1), \partial_w)$ $\cong$ $(\fl{C}(\geq s-2t-1), \partial_z)$ $\cong$  ($V^1_{s}$ of $KtD(\fl{C}))$, as $\mathbb{Z}$-filtered chain complexes over $\mathbb{F}_2$.

Now we look at arrows with nontrivial algebra elements on them. Fist, arrows with $\rh{1}$ and $\rh{3}.$ Per our construction, $D_1: V^0_s\to V^1_{s+\frac{n-1}{2}} of KtD(C)$ become $D_3: V^0_{-s}\to V^1_{-s-\frac{n-1}{2}} of X$ and $D_3: V^0_s\to V^1_{s-\frac{n-1}{2}} of KtD(C)$ become $D_1: V^0_{-s}\to V^1_{-s+\frac{n-1}{2}} of X$. So
\[
(D_1: V^0_{-s}\to V^1_{-s+\frac{n-1}{2}} of X) \cong (D_3: V^0_{s}=C(s)\to V^1_{s-\frac{n-1}{2}}=C(\leq s) of KtD(C)) 
\]
\[
\cong (i: C(s)\to C(\leq s)) \cong (i: \fl{C}(-s) \to \fl{C}(\geq -s))
\]
\[
\cong (D_1: V^0_{-s}=\fl{C}(-s)\to V^1_{-s+\frac{n-1}{2}}=\fl{C}(\geq -s) of KtC(\fl{C}))
\]
Similarly,
\[
(D_3: V^0_{-s}\to V^1_{-s-\frac{n-1}{2}} of X) \cong (D_1: V^0_s\to V^1_{s+\frac{n-1}{2}} of KtD(C))
\]
\[
(i: C(s)\to C(\geq s)) \cong (i: \fl{C}(-s)\to \fl{C}(\leq -s))
\]
\[
(D_3: V^0_{-s}=\fl{C}(-s)\to V^1_{-s-\frac{n-1}{2}}=\fl{C}(\leq -s))
\]
The newly added $\rh{2}$'s are exactly done to match those in $KtD(\fl{C}).$ They are:
\begin{align*}
(D_2: V^1_{-s} = C(\geq -s-\frac{n-1}{2}) \to^{\partial_z} C \to^{\pi}  V^0_{-s-\frac{n+1}{2}} = C(-s-\frac{n+1}{2}) in X) \\
\cong (\fl{C}(\leq s+\frac{n-1}{2})\to^{\partial_w} \fl{C} \to^{\pi} \fl{C}(s+\frac{n+1}{2})),
\end{align*}
which matches $D_2$ in $KtD(\fl{C})$.

Similar logic goes for the newly added $\rh{123}$'s. They are:
\begin{align*}
(D_{123}: V^0_{-s} = C(s) \to^{\partial_z} C(s-1) \to V^1_{-s+\frac{n+1}{2}}= C(\leq s-1) in X)\\
\cong (\fl{C}(-s) \to^{\partial_w} C(-s+1) \to C(\geq -s+1))
\cong (D_{123} in KtD(\fl{C}))
\end{align*}

$D_{23}$: The new $\rh{23}$'s are: for $s>\frac{n}{4},$
\begin{align*}
D_{23}: V^1_{s} = C(\leq -s+\frac{n-1}{2}) \to V^1_{s+1} = C(\leq -s-1+\frac{n-1}{2})\\
\cong \fl{C}(\geq s-\frac{n-1}{2})\to \fl{C}(\geq s-\frac{n+1}{2}),
\end{align*}
which matches those in $KtD(\fl{C}).$
The rest of the $\rh{2}$'s similarly match.

\end{proof}

\FloatBarrier
Now that we have proved the differences between $KtD(C)$ and $KtD(C^{flip})$ are fairly ``local'', we proceed to prove that tensoring with $H$ and canceling carefully transform $KtD(C)$ into $KtD(C^{flip})$.

Each generator in $V^0\subset KtD(C)$ corresponds to three generators in $KtD(C)\bts H$ and each one in $V^1\subset KtD(C)$ corresponds to five generators in $KtD(C)\bts H$, see \cref{fig:tensorpiecei01}. An arrow with $\rh{1}$ ($\rh{3}, \rh{2}, \rh{23}, \rh{123}$, respectively) corresponds to arrows in $KtD(C)\bts H$ in \cref{fig:tensorpiecer1} (\cref{fig:tensorpiecer3}, \cref{fig:tensorpiecer2}, \cref{fig:tensorpiecer23}, \cref{fig:tensorpiecer123} respectively.)

\begin{figure}[ht]
\begin{subfigure}{0.49\textwidth}
\centerline{
\begin{tikzpicture}[node distance=2cm]
\path[font = \scriptsize]
(2.61, 2.2) node(x1u) [outer sep=-1pt]{$x_1$}
(4.71, -1.0) node(vd) [outer sep=-1pt]{$v$}
(0.01, 2.2) node(x3u) [outer sep=-1pt]{$x_3$}
(-0.49, -1.0) node(y1d) [outer sep=-1pt]{$y_1$}
(1.31, 2.2) node(x2u) [outer sep=-1pt]{$x_2$}
(0.81, -1.0) node(ud) [outer sep=-1pt]{$u$}
(3.41, -1.0) node(y3d) [outer sep=-1pt]{$y_3$}
(2.11, -1.0) node(y2d) [outer sep=-1pt]{$y_2$}
;
\draw[->, bend right = 0, black, , font=\scriptsize](ud) to node [sloped] {$\rho_{1}$} (y1d);
\draw[->, bend right = 0, black, , font=\scriptsize](ud) to node [sloped] {$\rho_{3}$} (y2d);
\draw[->, bend right = 0, black, , font=\scriptsize](x2u) to node [sloped] {$\rho_{1}$} (x1u);
\draw[->, bend right = 0, black, , font=\scriptsize](x3u) to node [sloped] {$\rho_{2}$} (x2u);
\draw[->, bend right = 0, black, , font=\scriptsize](y2d) to node [sloped] {$\rho_{2}$} (y3d);
\draw[->, bend right = 0, black, , font=\scriptsize](y3d) to node [sloped] {$\rho_{1}$} (vd);
\end{tikzpicture}
}
\caption{Generators in $KtD(C)\bts H$}\label{fig:tensorpiecei01}
\end{subfigure}
\begin{subfigure}{0.49\textwidth}
\centerline{
\begin{tikzpicture}[node distance=2cm]
\path[font = \scriptsize]
(3.41, -1.0) node(y3d) [outer sep=-1pt]{$y_3$}
(2.61, 2.2) node(x1u) [outer sep=-1pt]{$x_1$}
(0.81, -1.0) node(ud) [outer sep=-1pt]{$u$}
(4.71, -1.0) node(vd) [outer sep=-1pt]{$v$}
(-0.49, -1.0) node(y1d) [outer sep=-1pt]{$y_1$}
(2.11, -1.0) node(y2d) [outer sep=-1pt]{$y_2$}
(1.31, 2.2) node(x2u) [outer sep=-1pt]{$x_2$}
(0.01, 2.2) node(x3u) [outer sep=-1pt]{$x_3$}
;
\draw[->, bend right = 0, black, , font=\scriptsize](ud) to node [sloped] {$\rho_{1}$} (y1d);
\draw[->, bend right = 0, black, , font=\scriptsize](ud) to node [sloped] {$\rho_{3}$} (y2d);
\draw[->, bend right = 0, black, , font=\scriptsize](x2u) to node [sloped] {$\rho_{1}$} (x1u);
\draw[->, bend right = 0, blue, , font=\scriptsize](x3u) to node [sloped] {$\iota_1$} (y1d);
\draw[->, bend right = 0, black, , font=\scriptsize](x3u) to node [sloped] {$\rho_{2}$} (x2u);
\draw[->, bend right = 0, black, , font=\scriptsize](y2d) to node [sloped] {$\rho_{2}$} (y3d);
\draw[->, bend right = 0, black, , font=\scriptsize](y3d) to node [sloped] {$\rho_{1}$} (vd);
\end{tikzpicture}
}
\caption{$H$ box tensor product with an arrow with $\rh{1}$.}\label{fig:tensorpiecer1}
\end{subfigure}

\begin{subfigure}{0.49\textwidth}
\centerline{
\begin{tikzpicture}[node distance=2cm]
\path[font = \scriptsize]
(0.01, 2.2) node(x3u) [outer sep=-1pt]{$x_3$}
(3.41, -1.0) node(y3d) [outer sep=-1pt]{$y_3$}
(1.31, 2.2) node(x2u) [outer sep=-1pt]{$x_2$}
(2.11, -1.0) node(y2d) [outer sep=-1pt]{$y_2$}
(-0.49, -1.0) node(y1d) [outer sep=-1pt]{$y_1$}
(4.71, -1.0) node(vd) [outer sep=-1pt]{$v$}
(2.61, 2.2) node(x1u) [outer sep=-1pt]{$x_1$}
(0.81, -1.0) node(ud) [outer sep=-1pt]{$u$}
;
\draw[->, bend right = 0, black, , font=\scriptsize](ud) to node [sloped] {$\rho_{1}$} (y1d);
\draw[->, bend right = 0, black, , font=\scriptsize](ud) to node [sloped] {$\rho_{3}$} (y2d);
\draw[->, bend right = 0, blue, , font=\scriptsize](x1u) to node [sloped] {$\iota_1$} (vd);
\draw[->, bend right = 0, black, , font=\scriptsize](x2u) to node [sloped] {$\rho_{1}$} (x1u);
\draw[->, bend right = 0, blue, , font=\scriptsize](x2u) to node [sloped] {$\iota_0$} (y3d);
\draw[->, bend right = 0, black, , font=\scriptsize](x3u) to node [sloped] {$\rho_{2}$} (x2u);
\draw[->, bend right = 0, blue, , font=\scriptsize](x3u) to node [sloped] {$\iota_1$} (y2d);
\draw[->, bend right = 0, black, , font=\scriptsize](y2d) to node [sloped] {$\rho_{2}$} (y3d);
\draw[->, bend right = 0, black, , font=\scriptsize](y3d) to node [sloped] {$\rho_{1}$} (vd);
\end{tikzpicture}
}
\caption{$H$ box tensor product with an arrow with $\rh{3}$.}\label{fig:tensorpiecer3}
\end{subfigure}
\begin{subfigure}{0.49\textwidth}
\centerline{
\begin{tikzpicture}[node distance=2cm]
\path[font = \scriptsize]
(2.61, 2.2) node(x1u) [outer sep=-1pt]{$x_1$}
(1.31, 2.2) node(x2u) [outer sep=-1pt]{$x_2$}
(-0.49, -1.0) node(y1d) [outer sep=-1pt]{$y_1$}
(3.41, -1.0) node(y3d) [outer sep=-1pt]{$y_3$}
(4.71, -1.0) node(vd) [outer sep=-1pt]{$v$}
(0.01, 2.2) node(x3u) [outer sep=-1pt]{$x_3$}
(2.11, -1.0) node(y2d) [outer sep=-1pt]{$y_2$}
(0.81, -1.0) node(ud) [outer sep=-1pt]{$u$}
;
\draw[->, bend right = 0, black, , font=\scriptsize](ud) to node [sloped] {$\rho_{1}$} (y1d);
\draw[->, bend right = 0, black, , font=\scriptsize](ud) to node [sloped] {$\rho_{3}$} (y2d);
\draw[->, bend right = 0, green, , font=\scriptsize](ud) to node [sloped] {$\iota_0$} (x2u);
\draw[->, bend right = 0, black, , font=\scriptsize](x2u) to node [sloped] {$\rho_{1}$} (x1u);
\draw[->, bend right = 0, black, , font=\scriptsize](x3u) to node [sloped] {$\rho_{2}$} (x2u);
\draw[->, bend right = 20, green, , font=\scriptsize](y1d) to node [sloped] {$\iota_1$} (x1u);
\draw[->, bend right = 0, black, , font=\scriptsize](y2d) to node [sloped] {$\rho_{2}$} (y3d);
\draw[->, bend right = 0, black, , font=\scriptsize](y3d) to node [sloped] {$\rho_{1}$} (vd);
\end{tikzpicture}
}
\caption{$H$ box tensor product with an arrow with $\rh{2}$.}\label{fig:tensorpiecer2}
\end{subfigure}

\begin{subfigure}{0.49\textwidth}
\centerline{
\begin{tikzpicture}[node distance=2cm]
\path[font = \scriptsize]
(2.11, 1.8) node(y2u) [outer sep=-1pt]{$y_2$}
(0.81, -1.0) node(ud) [outer sep=-1pt]{$u$}
(-0.49, 1.8) node(y1u) [outer sep=-1pt]{$y_1$}
(3.41, 1.8) node(y3u) [outer sep=-1pt]{$y_3$}
(-0.49, -1.0) node(y1d) [outer sep=-1pt]{$y_1$}
(4.71, 1.8) node(vu) [outer sep=-1pt]{$v$}
(0.81, 1.8) node(uu) [outer sep=-1pt]{$u$}
(3.41, -1.0) node(y3d) [outer sep=-1pt]{$y_3$}
(2.11, -1.0) node(y2d) [outer sep=-1pt]{$y_2$}
(4.71, -1.0) node(vd) [outer sep=-1pt]{$v$}
;
\draw[->, bend right = 0, black, , font=\scriptsize](ud) to node [sloped] {$\rho_{1}$} (y1d);
\draw[->, bend right = 0, black, , font=\scriptsize](ud) to node [sloped] {$\rho_{3}$} (y2d);
\draw[->, bend right = 0, black, , font=\scriptsize](uu) to node [sloped] {$\rho_{1}$} (y1u);
\draw[->, bend right = 0, black, , font=\scriptsize](uu) to node [sloped] {$\rho_{3}$} (y2u);
\draw[->, bend right = 0, red, , font=\scriptsize](uu) to node [sloped] {$\iota_0$} (y3d);
\draw[->, bend right = 10, red, , font=\scriptsize](y1u) to node [sloped] {$\iota_1$} (vd);
\draw[->, bend right = 0, black, , font=\scriptsize](y2d) to node [sloped] {$\rho_{2}$} (y3d);
\draw[->, bend right = 0, black, , font=\scriptsize](y2u) to node [sloped] {$\rho_{2}$} (y3u);
\draw[->, bend right = 0, black, , font=\scriptsize](y3d) to node [sloped] {$\rho_{1}$} (vd);
\draw[->, bend right = 0, black, , font=\scriptsize](y3u) to node [sloped] {$\rho_{1}$} (vu);
\end{tikzpicture}
}
\caption{$H$ box tensor product with an arrow with $\rh{23}$.}\label{fig:tensorpiecer23}
\end{subfigure}
\begin{subfigure}{0.49\textwidth}
\centerline{
\begin{tikzpicture}[node distance=2cm]
\path[font = \scriptsize]
(2.61, 2.2) node(x1u) [outer sep=-1pt]{$x_1$}
(1.31, 2.2) node(x2u) [outer sep=-1pt]{$x_2$}
(3.41, -1.0) node(y3d) [outer sep=-1pt]{$y_3$}
(2.11, -1.0) node(y2d) [outer sep=-1pt]{$y_2$}
(-0.49, -1.0) node(y1d) [outer sep=-1pt]{$y_1$}
(4.71, -1.0) node(vd) [outer sep=-1pt]{$v$}
(0.81, -1.0) node(ud) [outer sep=-1pt]{$u$}
(0.01, 2.2) node(x3u) [outer sep=-1pt]{$x_3$}
;
\draw[->, bend right = 0, black, , font=\scriptsize](ud) to node [sloped] {$\rho_{1}$} (y1d);
\draw[->, bend right = 0, black, , font=\scriptsize](ud) to node [sloped] {$\rho_{3}$} (y2d);
\draw[->, bend right = 0, black, , font=\scriptsize](x2u) to node [sloped] {$\rho_{1}$} (x1u);
\draw[->, bend right = 0, pink, , font=\scriptsize](x3u) to node [sloped] {$\iota_1$} (vd);
\draw[->, bend right = 0, black, , font=\scriptsize](x3u) to node [sloped] {$\rho_{2}$} (x2u);
\draw[->, bend right = 0, black, , font=\scriptsize](y2d) to node [sloped] {$\rho_{2}$} (y3d);
\draw[->, bend right = 0, black, , font=\scriptsize](y3d) to node [sloped] {$\rho_{1}$} (vd);
\end{tikzpicture}
}
\caption{$H$ box tensor product with an arrow with $\rh{123}$.}\label{fig:tensorpiecer123}
\end{subfigure}

\FloatBarrier
\caption{Tensoring for generators and arrows}
\end{figure}
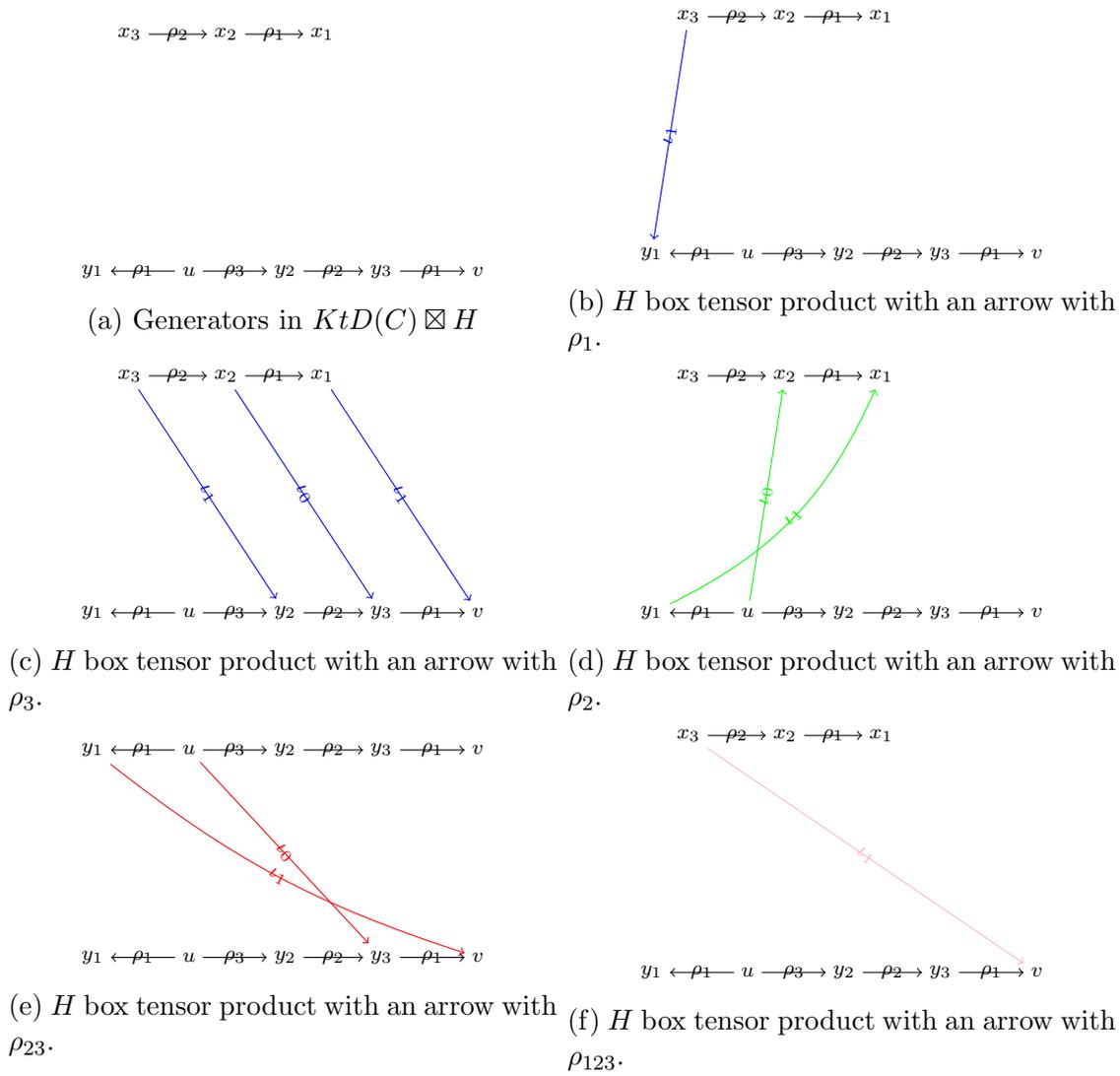

Now we can finally take the tensor $KtD(C)\bts H$ and perform cancellation. We first state that cancellation will be done in the order discussed below. 

\FloatBarrier
\subsubsection{Right-hand side middle part of rows in $V^1$}\label{sec:rightmiddle}

We start from the ``right'' side of $V^1$. For any row of generator $x$ connected by a string of $\rh{23}$'s in $\oplus_{s>t}V^1_{s}$, we first look at the middle part, i.e., excluding the leftmost one in the full copy and rightmost one where there might be $\rh{1}, \rh{123}$ coming in, see \cref{fig:g1a}. Black arrows into and out of $x$'s represent the potential and possibly multiple arrows of $\partial_z$ within this column. 

It looks like \cref{fig:g1b} in the tensor product. Each string of $y_1, u, u_2, y_3, v$ corresponds to a single generator before. Black arrows coming in and out of them represent the tensor result of potential arrows of $\partial_z$.

We first cancel the red arrows from $y_1$ to $v$, arriving at \cref{fig:g1c}. We keep the arrows resulted from cancellation in shape of zigzags to make it easier to see how they were generated.

Then canceling red arrows from $u$ to $y_3$ yields \cref{fig:g1d}. Note the new $\rh{23}$'s in the opposite direction between the $y_2$'s. They and the $\rh{23}$'s between them match the generators in $KtD(\fl{C})$ in the way described by \cref{localchange}. In general, throughout the proof, $y_2$'s in the tensor product are the generators that will survive and match to generators in $KtD(\fl{C}).$ We repeat the process above for all such rows in the ``right'' side. In this process, some undesirable arrows appear as a side effect and we argue that such arrows will not survive at the end. For example, the (potential) arrow (marked by $\#$ in \cref{fig:g1d}) coming from a $v_i$ above in the middle column and going to a $y_1$ below in the left column. When we follow this process in the row containing the $v_i$, we will cancel a red arrow going into $v_i$, meaning any arrow coming out of $v_i$ will be discarded. Similar logic goes to the rest of the undesirable arrows; they either come out $v$'s and $y_3$'s, which are the target of arrows to be cancelled, or enter $y_1$'s and $u$'s, which are the origin of arrows to be cancelled. One exception will be when the $u$'s below are at rightmost positions of their row, where they won't have red arrowing coming out to be cancelled. We remark that, in such cases, only the (possibly multiple) arrow marked by $*$ could survive, while the others will be gone because of their origins being targets of arrows to be cancelled. Note in each row, such a surviving arrow can only appear once for each downward arrow in $\partial_z$. See \cref{fig:g1e} for the final product. We emphasize that we perform cancellation described here across all applicable places first before moving on to other parts of the module. 

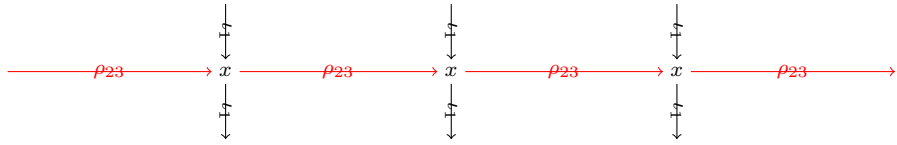
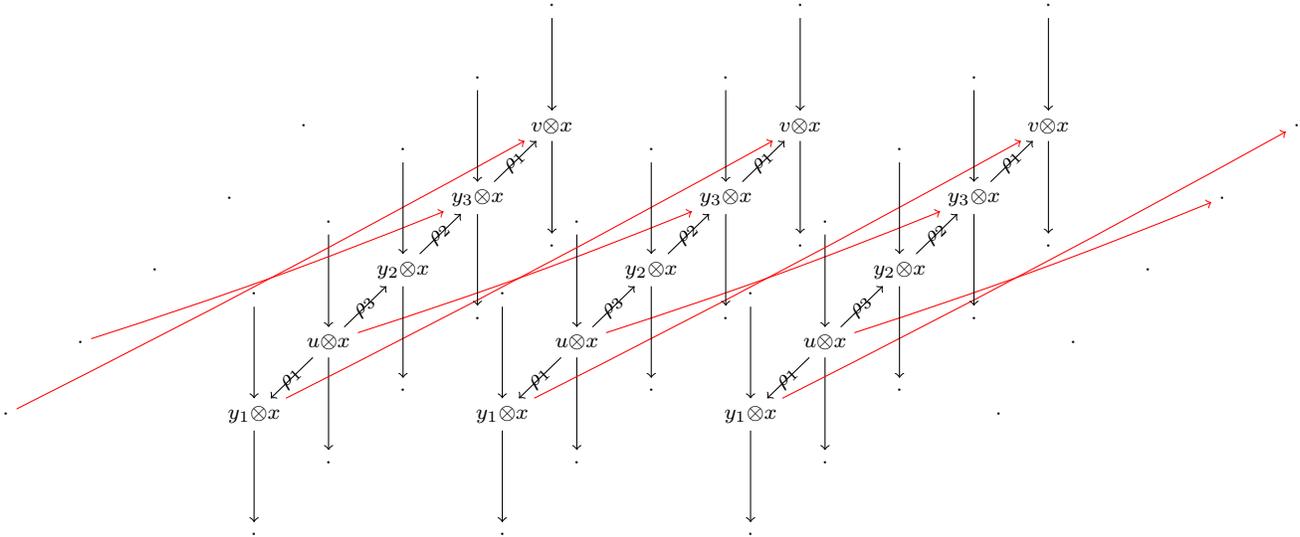
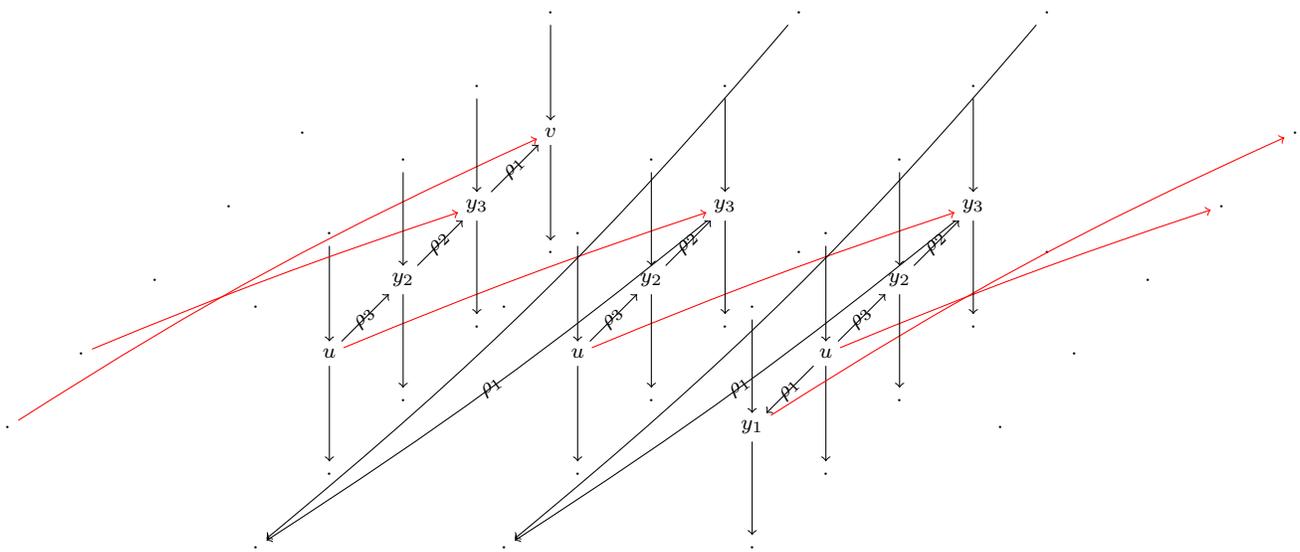
\begin{figure}

\begin{subfigure}{0.99\textwidth}
\centerline{
\begin{tikzpicture}[node distance=2cm]
\path[font = \scriptsize]
(6.0, -1.0) node(rnodown) [outer sep=-1pt]{}
(3.0, 1.0) node(g3up) [outer sep=-1pt]{}
(6.0, 0.0) node(rno) [outer sep=-1pt]{}
(0.0, -1.0) node(g2down) [outer sep=-1pt]{}
(-3.0, 0.0) node(g1) [outer sep=-1pt]{$x$}
(0.0, 0.0) node(g2) [outer sep=-1pt]{$x$}
(-3.0, 1.0) node(g1up) [outer sep=-1pt]{}
(6.0, 1.0) node(rnoup) [outer sep=-1pt]{}
(0.0, 1.0) node(g2up) [outer sep=-1pt]{}
(-6.0, 1.0) node(lnoup) [outer sep=-1pt]{}
(3.0, -1.0) node(g3down) [outer sep=-1pt]{}
(-6.0, 0.0) node(lno) [outer sep=-1pt]{}
(-6.0, -1.0) node(lnodown) [outer sep=-1pt]{}
(-3.0, -1.0) node(g1down) [outer sep=-1pt]{}
(3.0, 0.0) node(g3) [outer sep=-1pt]{$x$}
;
\draw[->, bend right = 0, red, , font=\scriptsize](g1) to node [sloped] {$\rho_{23}$} (g2);
\draw[->, bend right = 0, black, , font=\scriptsize](g1) to node [sloped] {$\iota_1$} (g1down);
\draw[->, bend right = 0, black, , font=\scriptsize](g1up) to node [sloped] {$\iota_1$} (g1);
\draw[->, bend right = 0, red, , font=\scriptsize](g2) to node [sloped] {$\rho_{23}$} (g3);
\draw[->, bend right = 0, black, , font=\scriptsize](g2) to node [sloped] {$\iota_1$} (g2down);
\draw[->, bend right = 0, black, , font=\scriptsize](g2up) to node [sloped] {$\iota_1$} (g2);
\draw[->, bend right = 0, red, , font=\scriptsize](g3) to node [sloped] {$\rho_{23}$} (rno);
\draw[->, bend right = 0, black, , font=\scriptsize](g3) to node [sloped] {$\iota_1$} (g3down);
\draw[->, bend right = 0, black, , font=\scriptsize](g3up) to node [sloped] {$\iota_1$} (g3);
\draw[->, bend right = 0, red, , font=\scriptsize](lno) to node [sloped] {$\rho_{23}$} (g1);
\end{tikzpicture}
}
\caption{Middle part of a row.}\label{fig:g1a}
\end{subfigure}

\begin{subfigure}{0.99\textwidth}
\centerline{
\begin{tikzpicture}[node distance=2cm]
\path[font = \scriptsize]
(0.97, -1.91) node(ug2down) [outer sep=-1pt]{$\cdot$}
(5.26, 2.25) node(y2g3up) [outer sep=-1pt]{$\cdot$}
(7.57, -0.31) node(urno) [outer sep=-1pt]{$\cdot$}
(2.95, 1.61) node(y3g2) [outer sep=-1pt]{$y_3$$\otimes$$x$}
(-3.32, -1.27) node(y1g1) [outer sep=-1pt]{$y_1$$\otimes$$x$}
(3.28, -1.27) node(y1g3) [outer sep=-1pt]{$y_1$$\otimes$$x$}
(-0.35, 1.61) node(y3g1) [outer sep=-1pt]{$y_3$$\otimes$$x$}
(6.58, -1.27) node(y1rno) [outer sep=-1pt]{$\cdot$}
(-3.65, 1.61) node(y3lno) [outer sep=-1pt]{$\cdot$}
(2.95, 0.01) node(y3g2down) [outer sep=-1pt]{$\cdot$}
(-1.34, 2.25) node(y2g1up) [outer sep=-1pt]{$\cdot$}
(-1.34, -0.95) node(y2g1down) [outer sep=-1pt]{$\cdot$}
(0.64, 4.17) node(vg1up) [outer sep=-1pt]{$\cdot$}
(-2.66, 2.57) node(vlno) [outer sep=-1pt]{$\cdot$}
(1.96, -0.95) node(y2g2down) [outer sep=-1pt]{$\cdot$}
(-4.64, 0.65) node(y2lno) [outer sep=-1pt]{$\cdot$}
(-0.02, -2.87) node(y1g2down) [outer sep=-1pt]{$\cdot$}
(3.28, 0.33) node(y1g3up) [outer sep=-1pt]{$\cdot$}
(3.28, -2.87) node(y1g3down) [outer sep=-1pt]{$\cdot$}
(7.24, 0.97) node(vg3down) [outer sep=-1pt]{$\cdot$}
(10.54, 2.57) node(vrno) [outer sep=-1pt]{$\cdot$}
(4.27, -0.31) node(ug3) [outer sep=-1pt]{$u$$\otimes$$x$}
(5.26, 0.65) node(y2g3) [outer sep=-1pt]{$y_2$$\otimes$$x$}
(-0.02, 0.33) node(y1g2up) [outer sep=-1pt]{$\cdot$}
(3.94, 2.57) node(vg2) [outer sep=-1pt]{$v$$\otimes$$x$}
(-3.32, 0.33) node(y1g1up) [outer sep=-1pt]{$\cdot$}
(-5.63, -0.31) node(ulno) [outer sep=-1pt]{$\cdot$}
(4.27, 1.29) node(ug3up) [outer sep=-1pt]{$\cdot$}
(-6.62, -1.27) node(y1lno) [outer sep=-1pt]{$\cdot$}
(2.95, 3.21) node(y3g2up) [outer sep=-1pt]{$\cdot$}
(0.64, 2.57) node(vg1) [outer sep=-1pt]{$v$$\otimes$$x$}
(-0.02, -1.27) node(y1g2) [outer sep=-1pt]{$y_1$$\otimes$$x$}
(7.24, 4.17) node(vg3up) [outer sep=-1pt]{$\cdot$}
(-0.35, 3.21) node(y3g1up) [outer sep=-1pt]{$\cdot$}
(-1.34, 0.65) node(y2g1) [outer sep=-1pt]{$y_2$$\otimes$$x$}
(4.27, -1.91) node(ug3down) [outer sep=-1pt]{$\cdot$}
(-3.32, -2.87) node(y1g1down) [outer sep=-1pt]{$\cdot$}
(6.25, 0.01) node(y3g3down) [outer sep=-1pt]{$\cdot$}
(3.94, 4.17) node(vg2up) [outer sep=-1pt]{$\cdot$}
(9.55, 1.61) node(y3rno) [outer sep=-1pt]{$\cdot$}
(-2.33, -0.31) node(ug1) [outer sep=-1pt]{$u$$\otimes$$x$}
(0.97, -0.31) node(ug2) [outer sep=-1pt]{$u$$\otimes$$x$}
(8.56, 0.65) node(y2rno) [outer sep=-1pt]{$\cdot$}
(0.64, 0.97) node(vg1down) [outer sep=-1pt]{$\cdot$}
(1.96, 0.65) node(y2g2) [outer sep=-1pt]{$y_2$$\otimes$$x$}
(6.25, 1.61) node(y3g3) [outer sep=-1pt]{$y_3$$\otimes$$x$}
(-2.33, 1.29) node(ug1up) [outer sep=-1pt]{$\cdot$}
(3.94, 0.97) node(vg2down) [outer sep=-1pt]{$\cdot$}
(5.26, -0.95) node(y2g3down) [outer sep=-1pt]{$\cdot$}
(-2.33, -1.91) node(ug1down) [outer sep=-1pt]{$\cdot$}
(0.97, 1.29) node(ug2up) [outer sep=-1pt]{$\cdot$}
(-0.35, 0.01) node(y3g1down) [outer sep=-1pt]{$\cdot$}
(1.96, 2.25) node(y2g2up) [outer sep=-1pt]{$\cdot$}
(6.25, 3.21) node(y3g3up) [outer sep=-1pt]{$\cdot$}
(7.24, 2.57) node(vg3) [outer sep=-1pt]{$v$$\otimes$$x$}
;
\draw[->, bend right = 0, black, , font=\scriptsize](ug1) to node [sloped] {$\rho_{1}$} (y1g1);
\draw[->, bend right = 0, black, , font=\scriptsize](ug1) to node [sloped] {$\rho_{3}$} (y2g1);
\draw[->, bend right = 2.68, red, , font=\scriptsize](ug1) to node [sloped] {} (y3g2);
\draw[->, bend right = 0, black, , font=\scriptsize](ug1) to node [sloped] {} (ug1down);
\draw[->, bend right = 0, black, , font=\scriptsize](ug1up) to node [sloped] {} (ug1);
\draw[->, bend right = 0, black, , font=\scriptsize](ug2) to node [sloped] {$\rho_{1}$} (y1g2);
\draw[->, bend right = 0, black, , font=\scriptsize](ug2) to node [sloped] {$\rho_{3}$} (y2g2);
\draw[->, bend right = 2.68, red, , font=\scriptsize](ug2) to node [sloped] {} (y3g3);
\draw[->, bend right = 0, black, , font=\scriptsize](ug2) to node [sloped] {} (ug2down);
\draw[->, bend right = 0, black, , font=\scriptsize](ug2up) to node [sloped] {} (ug2);
\draw[->, bend right = 0, black, , font=\scriptsize](ug3) to node [sloped] {$\rho_{1}$} (y1g3);
\draw[->, bend right = 0, black, , font=\scriptsize](ug3) to node [sloped] {$\rho_{3}$} (y2g3);
\draw[->, bend right = 2.68, red, , font=\scriptsize](ug3) to node [sloped] {} (y3rno);
\draw[->, bend right = 0, black, , font=\scriptsize](ug3) to node [sloped] {} (ug3down);
\draw[->, bend right = 0, black, , font=\scriptsize](ug3up) to node [sloped] {} (ug3);
\draw[->, bend right = 2.68, red, , font=\scriptsize](ulno) to node [sloped] {} (y3g1);
\draw[->, bend right = 0, black, , font=\scriptsize](vg1) to node [sloped] {} (vg1down);
\draw[->, bend right = 0, black, , font=\scriptsize](vg1up) to node [sloped] {} (vg1);
\draw[->, bend right = 0, black, , font=\scriptsize](vg2) to node [sloped] {} (vg2down);
\draw[->, bend right = 0, black, , font=\scriptsize](vg2up) to node [sloped] {} (vg2);
\draw[->, bend right = 0, black, , font=\scriptsize](vg3) to node [sloped] {} (vg3down);
\draw[->, bend right = 0, black, , font=\scriptsize](vg3up) to node [sloped] {} (vg3);
\draw[->, bend right = 1.25, red, , font=\scriptsize](y1g1) to node [sloped] {} (vg2);
\draw[->, bend right = 0, black, , font=\scriptsize](y1g1) to node [sloped] {} (y1g1down);
\draw[->, bend right = 0, black, , font=\scriptsize](y1g1up) to node [sloped] {} (y1g1);
\draw[->, bend right = 1.25, red, , font=\scriptsize](y1g2) to node [sloped] {} (vg3);
\draw[->, bend right = 0, black, , font=\scriptsize](y1g2) to node [sloped] {} (y1g2down);
\draw[->, bend right = 0, black, , font=\scriptsize](y1g2up) to node [sloped] {} (y1g2);
\draw[->, bend right = 1.25, red, , font=\scriptsize](y1g3) to node [sloped] {} (vrno);
\draw[->, bend right = 0, black, , font=\scriptsize](y1g3) to node [sloped] {} (y1g3down);
\draw[->, bend right = 0, black, , font=\scriptsize](y1g3up) to node [sloped] {} (y1g3);
\draw[->, bend right = 1.25, red, , font=\scriptsize](y1lno) to node [sloped] {} (vg1);
\draw[->, bend right = 0, black, , font=\scriptsize](y2g1) to node [sloped] {$\rho_{2}$} (y3g1);
\draw[->, bend right = 0, black, , font=\scriptsize](y2g1) to node [sloped] {} (y2g1down);
\draw[->, bend right = 0, black, , font=\scriptsize](y2g1up) to node [sloped] {} (y2g1);
\draw[->, bend right = 0, black, , font=\scriptsize](y2g2) to node [sloped] {$\rho_{2}$} (y3g2);
\draw[->, bend right = 0, black, , font=\scriptsize](y2g2) to node [sloped] {} (y2g2down);
\draw[->, bend right = 0, black, , font=\scriptsize](y2g2up) to node [sloped] {} (y2g2);
\draw[->, bend right = 0, black, , font=\scriptsize](y2g3) to node [sloped] {$\rho_{2}$} (y3g3);
\draw[->, bend right = 0, black, , font=\scriptsize](y2g3) to node [sloped] {} (y2g3down);
\draw[->, bend right = 0, black, , font=\scriptsize](y2g3up) to node [sloped] {} (y2g3);
\draw[->, bend right = 0, black, , font=\scriptsize](y3g1) to node [sloped] {$\rho_{1}$} (vg1);
\draw[->, bend right = 0, black, , font=\scriptsize](y3g1) to node [sloped] {} (y3g1down);
\draw[->, bend right = 0, black, , font=\scriptsize](y3g1up) to node [sloped] {} (y3g1);
\draw[->, bend right = 0, black, , font=\scriptsize](y3g2) to node [sloped] {$\rho_{1}$} (vg2);
\draw[->, bend right = 0, black, , font=\scriptsize](y3g2) to node [sloped] {} (y3g2down);
\draw[->, bend right = 0, black, , font=\scriptsize](y3g2up) to node [sloped] {} (y3g2);
\draw[->, bend right = 0, black, , font=\scriptsize](y3g3) to node [sloped] {$\rho_{1}$} (vg3);
\draw[->, bend right = 0, black, , font=\scriptsize](y3g3) to node [sloped] {} (y3g3down);
\draw[->, bend right = 0, black, , font=\scriptsize](y3g3up) to node [sloped] {} (y3g3);
\end{tikzpicture}
}
\caption{}\label{fig:g1b}
\end{subfigure}

\begin{subfigure}{0.99\textwidth}
\centerline{
\begin{tikzpicture}[node distance=2cm]
\path[font = \scriptsize]
(-0.38, 0.01) node(y3g1down) [outer sep=-1pt]{$\cdot$}
(6.22, 1.61) node(y3g3) [outer sep=-1pt]{$y_3$}
(-3.32, 0.27) node(y1g1up) [outer sep=-1pt]{$\cdot$}
(0.6, 4.19) node(vg1up) [outer sep=-1pt]{$\cdot$}
(-1.36, 2.23) node(y2g1up) [outer sep=-1pt]{$\cdot$}
(7.56, -0.35) node(urno) [outer sep=-1pt]{$\cdot$}
(4.26, -1.95) node(ug3down) [outer sep=-1pt]{$\cdot$}
(5.24, 0.63) node(y2g3) [outer sep=-1pt]{$y_2$}
(-2.7, 2.59) node(vlno) [outer sep=-1pt]{$\cdot$}
(6.58, -1.33) node(y1rno) [outer sep=-1pt]{$\cdot$}
(-0.38, 1.61) node(y3g1) [outer sep=-1pt]{$y_3$}
(-4.66, 0.63) node(y2lno) [outer sep=-1pt]{$\cdot$}
(-2.34, 1.25) node(ug1up) [outer sep=-1pt]{$\cdot$}
(7.2, 0.99) node(vg3down) [outer sep=-1pt]{$\cdot$}
(10.5, 2.59) node(vrno) [outer sep=-1pt]{$\cdot$}
(-0.02, 0.27) node(y1g2up) [outer sep=-1pt]{$\cdot$}
(-3.32, -2.93) node(y1g1down) [outer sep=-1pt]{$\cdot$}
(-2.34, -1.95) node(ug1down) [outer sep=-1pt]{$\cdot$}
(1.94, -0.97) node(y2g2down) [outer sep=-1pt]{$\cdot$}
(3.28, 0.27) node(y1g3up) [outer sep=-1pt]{$\cdot$}
(2.92, 0.01) node(y3g2down) [outer sep=-1pt]{$\cdot$}
(-2.34, -0.35) node(ug1) [outer sep=-1pt]{$u$}
(1.94, 2.23) node(y2g2up) [outer sep=-1pt]{$\cdot$}
(7.2, 4.19) node(vg3up) [outer sep=-1pt]{$\cdot$}
(5.24, 2.23) node(y2g3up) [outer sep=-1pt]{$\cdot$}
(6.22, 0.01) node(y3g3down) [outer sep=-1pt]{$\cdot$}
(0.96, -0.35) node(ug2) [outer sep=-1pt]{$u$}
(-0.02, -2.93) node(y1g2down) [outer sep=-1pt]{$\cdot$}
(2.92, 3.21) node(y3g2up) [outer sep=-1pt]{$\cdot$}
(9.52, 1.61) node(y3rno) [outer sep=-1pt]{$\cdot$}
(-5.64, -0.35) node(ulno) [outer sep=-1pt]{$\cdot$}
(-3.68, 1.61) node(y3lno) [outer sep=-1pt]{$\cdot$}
(4.26, -0.35) node(ug3) [outer sep=-1pt]{$u$}
(3.28, -2.93) node(y1g3down) [outer sep=-1pt]{$\cdot$}
(0.96, 1.25) node(ug2up) [outer sep=-1pt]{$\cdot$}
(8.54, 0.63) node(y2rno) [outer sep=-1pt]{$\cdot$}
(0.6, 2.59) node(vg1) [outer sep=-1pt]{$v$}
(-6.62, -1.33) node(y1lno) [outer sep=-1pt]{$\cdot$}
(0.96, -1.95) node(ug2down) [outer sep=-1pt]{$\cdot$}
(3.9, 0.99) node(vg2down) [outer sep=-1pt]{$\cdot$}
(1.94, 0.63) node(y2g2) [outer sep=-1pt]{$y_2$}
(-1.36, -0.97) node(y2g1down) [outer sep=-1pt]{$\cdot$}
(0.6, 0.99) node(vg1down) [outer sep=-1pt]{$\cdot$}
(6.22, 3.21) node(y3g3up) [outer sep=-1pt]{$\cdot$}
(5.24, -0.97) node(y2g3down) [outer sep=-1pt]{$\cdot$}
(-0.38, 3.21) node(y3g1up) [outer sep=-1pt]{$\cdot$}
(-1.36, 0.63) node(y2g1) [outer sep=-1pt]{$y_2$}
(2.92, 1.61) node(y3g2) [outer sep=-1pt]{$y_3$}
(3.28, -1.33) node(y1g3) [outer sep=-1pt]{$y_1$}
(3.9, 4.19) node(vg2up) [outer sep=-1pt]{$\cdot$}
(4.26, 1.25) node(ug3up) [outer sep=-1pt]{$\cdot$}
;
\draw[->, bend right = 0, black, , font=\scriptsize](ug1) to node [sloped] {$\rho_{3}$} (y2g1);
\draw[->, bend right = -2.76, red, , font=\scriptsize](ug1) to node [sloped] {} (y3g2);
\draw[->, bend right = 0, black, , font=\scriptsize](ug1) to node [sloped] {} (ug1down);
\draw[->, bend right = 0, black, , font=\scriptsize](ug1up) to node [sloped] {} (ug1);
\draw[->, bend right = 0, black, , font=\scriptsize](ug2) to node [sloped] {$\rho_{3}$} (y2g2);
\draw[->, bend right = -2.76, red, , font=\scriptsize](ug2) to node [sloped] {} (y3g3);
\draw[->, bend right = 0, black, , font=\scriptsize](ug2) to node [sloped] {} (ug2down);
\draw[->, bend right = 0, black, , font=\scriptsize](ug2up) to node [sloped] {} (ug2);
\draw[->, bend right = 0, black, , font=\scriptsize](ug3) to node [sloped] {$\rho_{1}$} (y1g3);
\draw[->, bend right = 0, black, , font=\scriptsize](ug3) to node [sloped] {$\rho_{3}$} (y2g3);
\draw[->, bend right = -2.76, red, , font=\scriptsize](ug3) to node [sloped] {} (y3rno);
\draw[->, bend right = 0, black, , font=\scriptsize](ug3) to node [sloped] {} (ug3down);
\draw[->, bend right = 0, black, , font=\scriptsize](ug3up) to node [sloped] {} (ug3);
\draw[->, bend right = -2.76, red, , font=\scriptsize](ulno) to node [sloped] {} (y3g1);
\draw[->, bend right = 0, black, , font=\scriptsize](vg1) to node [sloped] {} (vg1down);
\draw[->, bend right = 0, black, , font=\scriptsize](vg1up) to node [sloped] {} (vg1);
\draw[->, bend right = -5.93, black, , font=\scriptsize](vg2up) to node [sloped] {} (y1g1down);
\draw[->, bend right = -5.93, black, , font=\scriptsize](vg3up) to node [sloped] {} (y1g2down);
\draw[->, bend right = -4.38, red, , font=\scriptsize](y1g3) to node [sloped] {} (vrno);
\draw[->, bend right = 0, black, , font=\scriptsize](y1g3) to node [sloped] {} (y1g3down);
\draw[->, bend right = 0, black, , font=\scriptsize](y1g3up) to node [sloped] {} (y1g3);
\draw[->, bend right = -4.38, red, , font=\scriptsize](y1lno) to node [sloped] {} (vg1);
\draw[->, bend right = 0, black, , font=\scriptsize](y2g1) to node [sloped] {$\rho_{2}$} (y3g1);
\draw[->, bend right = 0, black, , font=\scriptsize](y2g1) to node [sloped] {} (y2g1down);
\draw[->, bend right = 0, black, , font=\scriptsize](y2g1up) to node [sloped] {} (y2g1);
\draw[->, bend right = 0, black, , font=\scriptsize](y2g2) to node [sloped] {$\rho_{2}$} (y3g2);
\draw[->, bend right = 0, black, , font=\scriptsize](y2g2) to node [sloped] {} (y2g2down);
\draw[->, bend right = 0, black, , font=\scriptsize](y2g2up) to node [sloped] {} (y2g2);
\draw[->, bend right = 0, black, , font=\scriptsize](y2g3) to node [sloped] {$\rho_{2}$} (y3g3);
\draw[->, bend right = 0, black, , font=\scriptsize](y2g3) to node [sloped] {} (y2g3down);
\draw[->, bend right = 0, black, , font=\scriptsize](y2g3up) to node [sloped] {} (y2g3);
\draw[->, bend right = 0, black, , font=\scriptsize](y3g1) to node [sloped] {$\rho_{1}$} (vg1);
\draw[->, bend right = 0, black, , font=\scriptsize](y3g1) to node [sloped] {} (y3g1down);
\draw[->, bend right = 0, black, , font=\scriptsize](y3g1up) to node [sloped] {} (y3g1);
\draw[->, bend right = 0, black, , font=\scriptsize](y3g2) to node [sloped] {} (y3g2down);
\draw[->, bend right = -3.77, black, , font=\scriptsize](y3g2) to node [sloped] {$\rho_{1}$} (y1g1down);
\draw[->, bend right = 0, black, , font=\scriptsize](y3g2up) to node [sloped] {} (y3g2);
\draw[->, bend right = 0, black, , font=\scriptsize](y3g3) to node [sloped] {} (y3g3down);
\draw[->, bend right = -3.77, black, , font=\scriptsize](y3g3) to node [sloped] {$\rho_{1}$} (y1g2down);
\draw[->, bend right = 0, black, , font=\scriptsize](y3g3up) to node [sloped] {} (y3g3);
\end{tikzpicture}
}
\caption{}\label{fig:g1c}
\end{subfigure}

\caption{Middle part of a row.}
\end{figure}

\begin{figure}

\begin{subfigure}{0.99\textwidth}
\centerline{
\begin{tikzpicture}[node distance=2cm]
\path[font = \scriptsize]
(-0.52, 4.48) node(vg1up) [outer sep=-1pt]{$\cdot$}
(-2.82, 0.97) node(ug1up) [outer sep=-1pt]{$\cdot$}
(-2.06, 0.54) node(y2g1) [outer sep=-1pt]{$y_2$}
(-2.06, 2.14) node(y2g1up) [outer sep=-1pt]{$\cdot$}
(2.64, 3.31) node(y3g2up) [outer sep=-1pt]{$\cdot$}
(-2.06, -1.06) node(y2g1down) [outer sep=-1pt]{$\cdot$}
(1.1, -2.23) node(ug2down) [outer sep=-1pt]{$\cdot$}
(5.8, -1.06) node(y2g3down) [outer sep=-1pt]{$\cdot$}
(-1.29, 3.31) node(y3g1up) [outer sep=-1pt]{$\cdot$}
(-5.98, 0.54) node(y2lno) [outer sep=-1pt]{$\cdot$}
(2.64, 0.11) node(y3g2down) [outer sep=-1pt]{$\cdot$}
(1.87, 0.54) node(y2g2) [outer sep=-1pt]{$y_2$}
(0.33, -3.4) node(y1g2down) [outer sep=-1pt]{$\cdot$}
(-7.52, -1.8) node(y1lno) [outer sep=-1pt]{$\cdot$}
(10.49, 1.71) node(y3rno) [outer sep=-1pt]{$\cdot$}
(1.87, -1.06) node(y2g2down) [outer sep=-1pt]{$\cdot$}
(5.8, 0.54) node(y2g3) [outer sep=-1pt]{$y_2$}
(0.33, -0.2) node(y1g2up) [outer sep=-1pt]{$\cdot$}
(8.96, -0.63) node(urno) [outer sep=-1pt]{$\cdot$}
(-3.59, -3.4) node(y1g1down) [outer sep=-1pt]{$\cdot$}
(7.33, 1.28) node(vg3down) [outer sep=-1pt]{$\cdot$}
(9.72, 0.54) node(y2rno) [outer sep=-1pt]{$\cdot$}
(7.33, 4.48) node(vg3up) [outer sep=-1pt]{$\cdot$}
(-1.29, 1.71) node(y3g1) [outer sep=-1pt]{$y_3$}
(11.26, 2.88) node(vrno) [outer sep=-1pt]{$\cdot$}
(5.03, -0.63) node(ug3) [outer sep=-1pt]{$u$}
(-0.52, 1.28) node(vg1down) [outer sep=-1pt]{$\cdot$}
(3.41, 1.28) node(vg2down) [outer sep=-1pt]{$\cdot$}
(6.57, 3.31) node(y3g3up) [outer sep=-1pt]{$\cdot$}
(-5.22, 1.71) node(y3lno) [outer sep=-1pt]{$\cdot$}
(-0.52, 2.88) node(vg1) [outer sep=-1pt]{$v$}
(1.1, 0.97) node(ug2up) [outer sep=-1pt]{$\cdot$}
(-1.29, 0.11) node(y3g1down) [outer sep=-1pt]{$\cdot$}
(-4.45, 2.88) node(vlno) [outer sep=-1pt]{$\cdot$}
(5.8, 2.14) node(y2g3up) [outer sep=-1pt]{$\cdot$}
(8.19, -1.8) node(y1rno) [outer sep=-1pt]{$\cdot$}
(5.03, 0.97) node(ug3up) [outer sep=-1pt]{$\cdot$}
(4.26, -3.4) node(y1g3down) [outer sep=-1pt]{$\cdot$}
(3.41, 4.48) node(vg2up) [outer sep=-1pt]{$\cdot$}
(-6.75, -0.63) node(ulno) [outer sep=-1pt]{$\cdot$}
(4.26, -0.2) node(y1g3up) [outer sep=-1pt]{$\cdot$}
(6.57, 0.11) node(y3g3down) [outer sep=-1pt]{$\cdot$}
(4.26, -1.8) node(y1g3) [outer sep=-1pt]{$y_1$}
(1.87, 2.14) node(y2g2up) [outer sep=-1pt]{$\cdot$}
(5.03, -2.23) node(ug3down) [outer sep=-1pt]{$\cdot$}
(-2.82, -2.23) node(ug1down) [outer sep=-1pt]{$\cdot$}
(-3.59, -0.2) node(y1g1up) [outer sep=-1pt]{$\cdot$}
;
\draw[->, bend right = 0, black, , font=\scriptsize](ug3) to node [sloped] {$\rho_{1}$} (y1g3);
\draw[->, bend right = 0, black, , font=\scriptsize](ug3) to node [sloped] {$\rho_{3}$} (y2g3);
\draw[->, bend right = -8.0, red, , font=\scriptsize](ug3) to node [sloped] {} (y3rno);
\draw[->, bend right = 0, black, , font=\scriptsize](ug3) to node [sloped] {} (ug3down);
\draw[->, bend right = 0, black, , font=\scriptsize](ug3up) to node [sloped] {} (ug3);
\draw[->, bend right = -8.0, red, , font=\scriptsize](ulno) to node [sloped] {} (y3g1);
\draw[->, bend right = 0, black, , font=\scriptsize](vg1) to node [sloped] {} (vg1down);
\draw[->, bend right = 0, black, , font=\scriptsize](vg1up) to node [sloped] {} (vg1);
\draw[->, bend right = -7.46, black, , font=\scriptsize](vg2up) to node [sloped] {$\#$~~~~~~~~~~~~~~~} (y1g1down);
\draw[->, bend right = -7.46, black, , font=\scriptsize](vg3up) to node [sloped] {} (y1g2down);
\draw[->, bend right = -6.72, red, , font=\scriptsize](y1g3) to node [sloped] {} (vrno);
\draw[->, bend right = 0, black, , font=\scriptsize](y1g3) to node [sloped] {} (y1g3down);
\draw[->, bend right = 0, black, , font=\scriptsize](y1g3up) to node [sloped] {} (y1g3);
\draw[->, bend right = -6.72, red, , font=\scriptsize](y1lno) to node [sloped] {} (vg1);
\draw[->, bend right = 0, black, , font=\scriptsize](y2g1) to node [sloped] {$\rho_{2}$} (y3g1);
\draw[->, bend right = 0, black, , font=\scriptsize](y2g1) to node [sloped] {} (y2g1down);
\draw[->, bend right = 0, black, , font=\scriptsize](y2g1up) to node [sloped] {} (y2g1);
\draw[->, bend right = 0, black, , font=\scriptsize](y2g2) to node [sloped] {} (y2g2down);
\draw[->, bend right = 0, red, , font=\scriptsize](y2g2) to node [sloped] {$\rho_{23}$} (y2g1);
\draw[->, bend right = -6.88, black, , font=\scriptsize](y2g2) to node [sloped] {$\rho_{2}$~~~~~*} (ug1down);
\draw[->, bend right = 0, black, , font=\scriptsize](y2g2up) to node [sloped] {} (y2g2);
\draw[->, bend right = 0, black, , font=\scriptsize](y2g3) to node [sloped] {} (y2g3down);
\draw[->, bend right = 0, red, , font=\scriptsize](y2g3) to node [sloped] {$\rho_{23}$} (y2g2);
\draw[->, bend right = -6.88, black, , font=\scriptsize](y2g3) to node [sloped] {$\rho_{2}$} (ug2down);
\draw[->, bend right = 0, black, , font=\scriptsize](y2g3up) to node [sloped] {} (y2g3);
\draw[->, bend right = 0, black, , font=\scriptsize](y3g1) to node [sloped] {$\rho_{1}$} (vg1);
\draw[->, bend right = 0, black, , font=\scriptsize](y3g1) to node [sloped] {} (y3g1down);
\draw[->, bend right = 0, black, , font=\scriptsize](y3g1up) to node [sloped] {} (y3g1);
\draw[->, bend right = -6.88, black, , font=\scriptsize](y3g2up) to node [sloped] {$\rho_{3}$} (y2g1);
\draw[->, bend right = -6.18, black, , font=\scriptsize](y3g2up) to node [sloped] {} (ug1down);
\draw[->, bend right = -6.88, black, , font=\scriptsize](y3g3up) to node [sloped] {$\rho_{3}$} (y2g2);
\draw[->, bend right = -6.18, black, , font=\scriptsize](y3g3up) to node [sloped] {} (ug2down);
\end{tikzpicture}
}
\caption{}\label{fig:g1d}
\end{subfigure}

\begin{subfigure}{0.99\textwidth}
\centerline{
\begin{tikzpicture}[node distance=2cm]
\path[font = \scriptsize]
(0.2, 0.0) node(g2) [outer sep=-1pt]{$y_2$}
(6.2, 0.0) node(rno) [outer sep=-1pt]{}
(3.2, -1.0) node(g3down) [outer sep=-1pt]{}
(0.2, 1.0) node(g2up) [outer sep=-1pt]{}
(-3.4, -1.2) node(g1downu) [outer sep=-1pt]{$u$}
(-5.8, 0.0) node(lno) [outer sep=-1pt]{}
(-2.8, 1.0) node(g1up) [outer sep=-1pt]{}
(-2.8, -1.0) node(g1down) [outer sep=-1pt]{}
(3.2, 1.0) node(g3up) [outer sep=-1pt]{}
(3.2, 0.0) node(g3) [outer sep=-1pt]{$y_2$}
(-2.8, 0.0) node(g1) [outer sep=-1pt]{$y_2$}
(0.2, -1.0) node(g2down) [outer sep=-1pt]{}
;
\draw[->, bend right = 0, red, , font=\scriptsize](g1) to node [sloped] {$\rho_{23}$} (lno);
\draw[->, bend right = 0, black, , font=\scriptsize](g1) to node [sloped] {} (g1down);
\draw[->, bend right = 0, black, , font=\scriptsize](g1up) to node [sloped] {} (g1);
\draw[->, bend right = 0, red, , font=\scriptsize](g2) to node [sloped] {$\rho_{23}$} (g1);
\draw[->, bend right = 0, black, , font=\scriptsize](g2) to node [sloped] {} (g2down);
\draw[->, bend right = 9.0, black, , font=\scriptsize](g2) to node [sloped] {*~~~~~$\rho_{2}$} (g1downu);
\draw[->, bend right = 0, black, , font=\scriptsize](g2up) to node [sloped] {} (g2);
\draw[->, bend right = 0, red, , font=\scriptsize](g3) to node [sloped] {$\rho_{23}$} (g2);
\draw[->, bend right = 0, black, , font=\scriptsize](g3) to node [sloped] {} (g3down);
\draw[->, bend right = 0, black, , font=\scriptsize](g3up) to node [sloped] {} (g3);
\draw[->, bend right = 0, red, , font=\scriptsize](rno) to node [sloped] {$\rho_{23}$} (g3);
\end{tikzpicture}
}
\caption{}\label{fig:g1e}
\end{subfigure}

\caption{Middle part of a row.}
\end{figure}
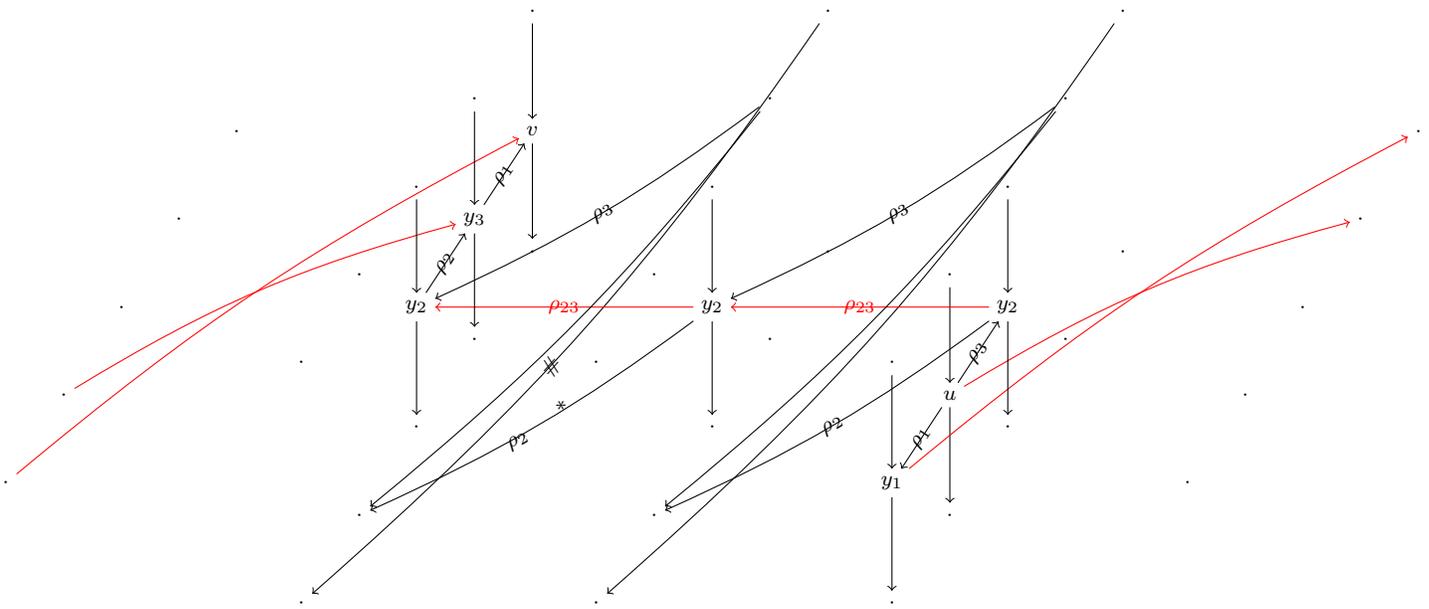
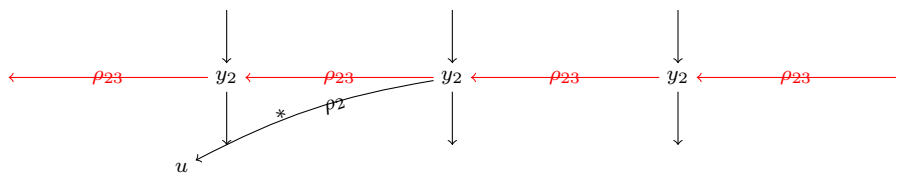

\FloatBarrier
\subsubsection{Right-hand side rightmost end of rows in $V^1$}\label{sec:rightend}

We look at the rightmost positions of rows, see \cref{fig:g2a}. The blue, green, and pink arrows represent $\rh{1}$, $\rh{3}$, and potential and possibly multiple $\rh{2}$ and $\rh{123}$. It corresponds to \cref{fig:g2b} in the tensor product. Again the black arrows are potential arrows from tensoring with $\partial_z$. Note that there is only one connecting to the right string of generators. There are not any going out because this is the right most position and only $y_2$ has black arrow(s) coming in because we have performed the cancellation described in \cref{sec:rightmiddle} on rows above. Regarding the left set of generators, there are only three of them and one incoming black arrow for the same reason. The black arrows marked by $*$ are the potential side effect of cancellation performed to rows above, see \cref{sec:rightmiddle}. 

Next we cancel the two red arrows in the middle successively, see \cref{fig:g2c} and \cref{fig:g2d}. Now $\partial_z$ into $y_2$'s and the $\rh{23}$ among them match those in $KtD(\fl{C}).$ As for the messy additional arrows, we keep them in mind and bring them into shape later. We make the observation that the potential side effect $\rh{2}$ arrows going to $u$'s and the pink arrow going to a $y_1$ in row below (all marked by $**$ in \cref{fig:g2d}) only exist if they enter $u$'s in the rightmost end of their rows. In the case of the $**\dagger$ arrows, they exist only if $x$ has a length one $\partial_z$ arrow. Even though they enter another right-most position in another row, we clearly see that they do not interfere with the procedure in this section applied there. So we are safe to say that we can perform these cancelations at all right-most ends of rows. An exception is when a row has only one generator, which can only happen at the bottom row. We discuss that case in \cref{sec:rightfull}.

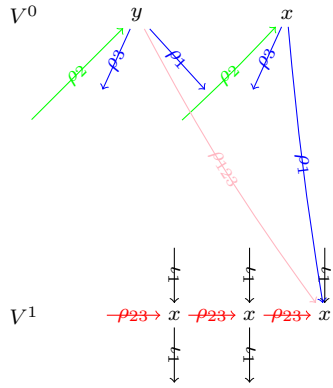
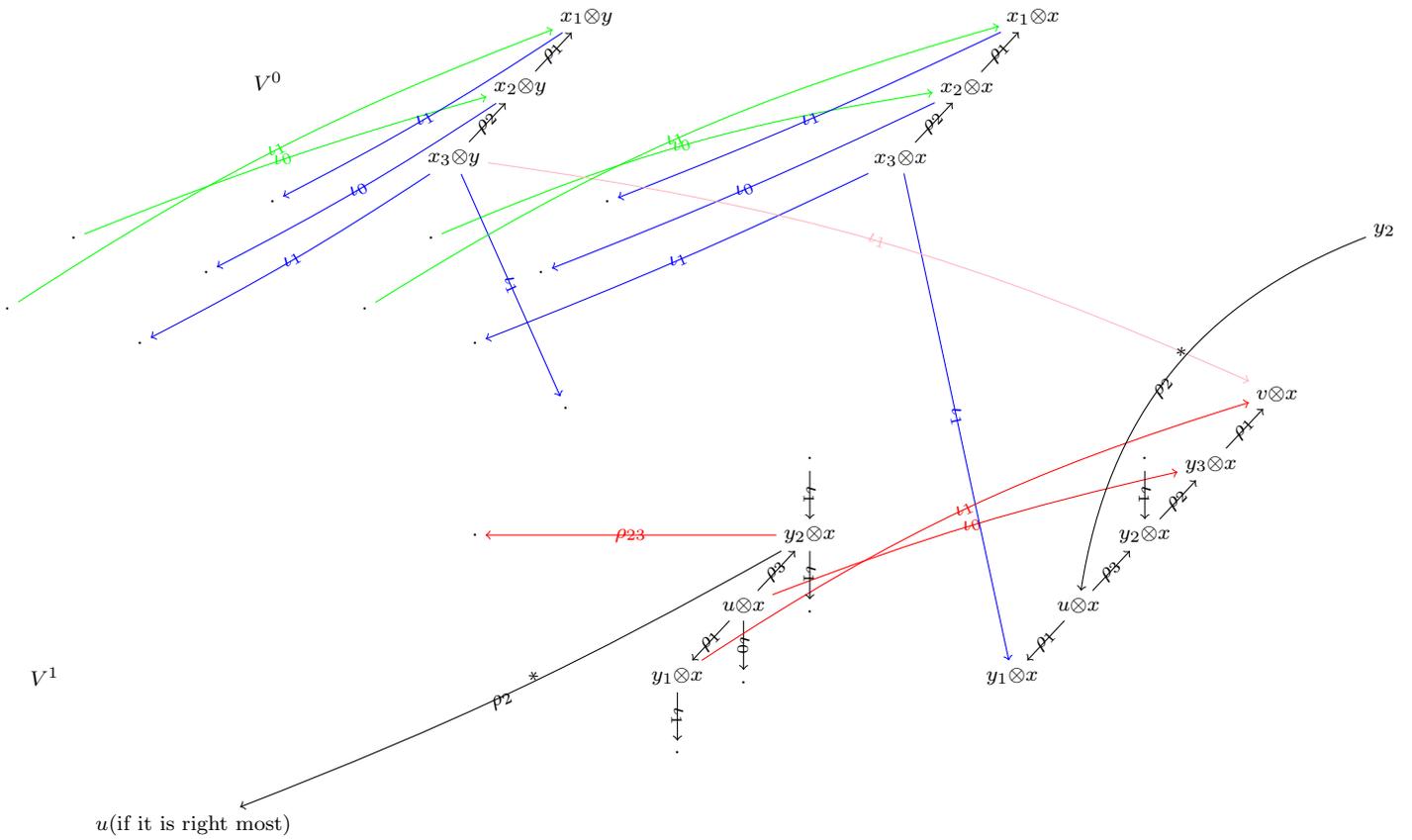
\begin{figure}

\begin{subfigure}{0.99\textwidth}
\centerline{
\begin{tikzpicture}[node distance=2cm]
\path[font = \scriptsize]
(-3.0, 0.0) node(v1) [outer sep=-1pt]{$V^1$}
(0.5, 4.0) node(xup) [outer sep=-1pt]{$x$}
(-3.0, 2.5) node(no1) [outer sep=-1pt]{}
(-2.0, 2.9) node(no2) [outer sep=-1pt]{}
(-2.0, 0.0) node(lno) [outer sep=-1pt]{}
(-0.0, 2.9) node(no5) [outer sep=-1pt]{}
(0.0, 1.0) node(g2up) [outer sep=-1pt]{}
(-1.0, 0.0) node(g1) [outer sep=-1pt]{$x$}
(1.0, 1.0) node(g3up) [outer sep=-1pt]{}
(-0.5, 2.9) node(no4) [outer sep=-1pt]{}
(-1.5, 4.0) node(yup) [outer sep=-1pt]{$y$}
(-3.0, 4.0) node(v0) [outer sep=-1pt]{$V^0$}
(0.0, 0.0) node(g2) [outer sep=-1pt]{$x$}
(0.0, -1.0) node(g2down) [outer sep=-1pt]{}
(1.0, 0.0) node(g3) [outer sep=-1pt]{$x$}
(-1.0, 1.0) node(g1up) [outer sep=-1pt]{}
(-1.0, -1.0) node(g1down) [outer sep=-1pt]{}
(-1.0, 2.5) node(no3) [outer sep=-1pt]{}
;
\draw[->, bend right = 0, red, , font=\scriptsize](g1) to node [sloped] {$\rho_{23}$} (g2);
\draw[->, bend right = 0, black, , font=\scriptsize](g1) to node [sloped] {$\iota_1$} (g1down);
\draw[->, bend right = 0, black, , font=\scriptsize](g1up) to node [sloped] {$\iota_1$} (g1);
\draw[->, bend right = 0, red, , font=\scriptsize](g2) to node [sloped] {$\rho_{23}$} (g3);
\draw[->, bend right = 0, black, , font=\scriptsize](g2) to node [sloped] {$\iota_1$} (g2down);
\draw[->, bend right = 0, black, , font=\scriptsize](g2up) to node [sloped] {$\iota_1$} (g2);
\draw[->, bend right = 0, black, , font=\scriptsize](g3up) to node [sloped] {$\iota_1$} (g3);
\draw[->, bend right = 0, red, , font=\scriptsize](lno) to node [sloped] {$\rho_{23}$} (g1);
\draw[->, bend right = 0, green, , font=\scriptsize](no1) to node [sloped] {$\rho_{2}$} (yup);
\draw[->, bend right = 2.67, green, , font=\scriptsize](no3) to node [sloped] {$\rho_{2}$} (xup);
\draw[->, bend right = 0, blue, , font=\scriptsize](xup) to node [sloped] {$\rho_{3}$} (no5);
\draw[->, bend right = 2.46, blue, , font=\scriptsize](xup) to node [sloped] {$\rho_{1}$} (g3);
\draw[->, bend right = 0, blue, , font=\scriptsize](yup) to node [sloped] {$\rho_{3}$} (no2);
\draw[->, bend right = 0, blue, , font=\scriptsize](yup) to node [sloped] {$\rho_{1}$} (no4);
\draw[->, bend right = 4.49, pink, , font=\scriptsize](yup) to node [sloped] {$\rho_{123}$} (g3);
\end{tikzpicture}
}
\caption{}\label{fig:g2a}
\end{subfigure}

\begin{subfigure}{0.99\textwidth}
\centerline{
\begin{tikzpicture}[node distance=2cm]
\path[font = \scriptsize]
(-8.0, -2.0) node(ubottomleft) [outer sep=-1pt]{$u$(if it is right most)}
(-4.81, 5.89) node(uno3) [outer sep=-1pt]{$\cdot$}
(-9.61, 5.89) node(uno1) [outer sep=-1pt]{$\cdot$}
(-7.0, 8.0) node(v0) [outer sep=-1pt]{$V^0$}
(-3.61, 7.92) node(x2yup) [outer sep=-1pt]{$x_2$$\otimes$$y$}
(-0.61, 0.95) node(ug2) [outer sep=-1pt]{$u$$\otimes$$x$}
(1.5, 6.97) node(x3xup) [outer sep=-1pt]{$x_3$$\otimes$$x$}
(-10.5, 4.94) node(y1no1) [outer sep=-1pt]{$\cdot$}
(-2.44, 6.38) node(vno5) [outer sep=-1pt]{$\cdot$}
(-3.33, 5.43) node(y3no5) [outer sep=-1pt]{$\cdot$}
(-4.5, 6.97) node(x3yup) [outer sep=-1pt]{$x_3$$\otimes$$y$}
(-4.22, 1.9) node(y2lno) [outer sep=-1pt]{$\cdot$}
(0.28, 0.87) node(y2g2down) [outer sep=-1pt]{$\cdot$}
(5.67, 2.85) node(y3g3) [outer sep=-1pt]{$y_3$$\otimes$$x$}
(-2.72, 8.87) node(x1yup) [outer sep=-1pt]{$x_1$$\otimes$$y$}
(3.89, 0.95) node(ug3) [outer sep=-1pt]{$u$$\otimes$$x$}
(0.28, 1.9) node(y2g2) [outer sep=-1pt]{$y_2$$\otimes$$x$}
(6.56, 3.8) node(vg3) [outer sep=-1pt]{$v$$\otimes$$x$}
(2.39, 7.92) node(x2xup) [outer sep=-1pt]{$x_2$$\otimes$$x$}
(-4.22, 4.48) node(y2no5) [outer sep=-1pt]{$\cdot$}
(4.78, 1.9) node(y2g3) [outer sep=-1pt]{$y_2$$\otimes$$x$}
(3.0, 0.0) node(y1g3) [outer sep=-1pt]{$y_1$$\otimes$$x$}
(0.28, 2.93) node(y2g2up) [outer sep=-1pt]{$\cdot$}
(4.78, 2.93) node(y2g3up) [outer sep=-1pt]{$\cdot$}
(-5.7, 4.94) node(y1no3) [outer sep=-1pt]{$\cdot$}
(8.0, 6.0) node(y2topright) [outer sep=-1pt]{$y_2$}
(-0.61, -0.08) node(ug2down) [outer sep=-1pt]{$\cdot$}
(-6.94, 6.38) node(vno2) [outer sep=-1pt]{$\cdot$}
(-1.5, -1.03) node(y1g2down) [outer sep=-1pt]{$\cdot$}
(-8.72, 4.48) node(y2no2) [outer sep=-1pt]{$\cdot$}
(-7.83, 5.43) node(y3no2) [outer sep=-1pt]{$\cdot$}
(-3.0, 3.6) node(y1no4) [outer sep=-1pt]{$\cdot$}
(3.28, 8.87) node(x1xup) [outer sep=-1pt]{$x_1$$\otimes$$x$}
(-1.5, 0.0) node(y1g2) [outer sep=-1pt]{$y_1$$\otimes$$x$}
(-10.0, 0.0) node(v1) [outer sep=-1pt]{$V^1$}
;
\draw[->, bend right = 0, black, , font=\scriptsize](ug2) to node [sloped] {$\rho_{1}$} (y1g2);
\draw[->, bend right = 0, black, , font=\scriptsize](ug2) to node [sloped] {$\rho_{3}$} (y2g2);
\draw[->, bend right = -4.93, red, , font=\scriptsize](ug2) to node [sloped] {$\iota_0$} (y3g3);
\draw[->, bend right = 0, black, , font=\scriptsize](ug2) to node [sloped] {$\iota_0$} (ug2down);
\draw[->, bend right = 0, black, , font=\scriptsize](ug3) to node [sloped] {$\rho_{1}$} (y1g3);
\draw[->, bend right = 0, black, , font=\scriptsize](ug3) to node [sloped] {$\rho_{3}$} (y2g3);
\draw[->, bend right = -3.96, green, , font=\scriptsize](uno1) to node [sloped] {$\iota_0$} (x2yup);
\draw[->, bend right = -7.09, green, , font=\scriptsize](uno3) to node [sloped] {$\iota_0$} (x2xup);
\draw[->, bend right = -2.4, blue, , font=\scriptsize](x1xup) to node [sloped] {$\iota_1$} (vno5);
\draw[->, bend right = -3.05, blue, , font=\scriptsize](x1yup) to node [sloped] {$\iota_1$} (vno2);
\draw[->, bend right = 0, black, , font=\scriptsize](x2xup) to node [sloped] {$\rho_{1}$} (x1xup);
\draw[->, bend right = -2.4, blue, , font=\scriptsize](x2xup) to node [sloped] {$\iota_0$} (y3no5);
\draw[->, bend right = 0, black, , font=\scriptsize](x2yup) to node [sloped] {$\rho_{1}$} (x1yup);
\draw[->, bend right = -3.05, blue, , font=\scriptsize](x2yup) to node [sloped] {$\iota_0$} (y3no2);
\draw[->, bend right = 0, blue, , font=\scriptsize](x3xup) to node [sloped] {$\iota_1$} (y1g3);
\draw[->, bend right = 0, black, , font=\scriptsize](x3xup) to node [sloped] {$\rho_{2}$} (x2xup);
\draw[->, bend right = -2.4, blue, , font=\scriptsize](x3xup) to node [sloped] {$\iota_1$} (y2no5);
\draw[->, bend right = -8.98, pink, , font=\scriptsize](x3yup) to node [sloped] {$\iota_1$} (vg3);
\draw[->, bend right = 0, blue, , font=\scriptsize](x3yup) to node [sloped] {$\iota_1$} (y1no4);
\draw[->, bend right = 0, black, , font=\scriptsize](x3yup) to node [sloped] {$\rho_{2}$} (x2yup);
\draw[->, bend right = -3.05, blue, , font=\scriptsize](x3yup) to node [sloped] {$\iota_1$} (y2no2);
\draw[->, bend right = -8.75, red, , font=\scriptsize](y1g2) to node [sloped] {$\iota_1$} (vg3);
\draw[->, bend right = 0, black, , font=\scriptsize](y1g2) to node [sloped] {$\iota_1$} (y1g2down);
\draw[->, bend right = -6.0, green, , font=\scriptsize](y1no1) to node [sloped] {$\iota_1$} (x1yup);
\draw[->, bend right = -8.96, green, , font=\scriptsize](y1no3) to node [sloped] {$\iota_1$} (x1xup);
\draw[->, bend right = 0, black, , font=\scriptsize](y2g2) to node [sloped] {$\iota_1$} (y2g2down);
\draw[->, bend right = 0, red, , font=\scriptsize](y2g2) to node [sloped] {$\rho_{23}$} (y2lno);
\draw[->, bend right = -4.8, black, , font=\scriptsize](y2g2) to node [sloped] {$\rho_{2}$~~~*} (ubottomleft);
\draw[->, bend right = 0, black, , font=\scriptsize](y2g2up) to node [sloped] {$\iota_1$} (y2g2);
\draw[->, bend right = 0, black, , font=\scriptsize](y2g3) to node [sloped] {$\rho_{2}$} (y3g3);
\draw[->, bend right = 0, black, , font=\scriptsize](y2g3up) to node [sloped] {$\iota_1$} (y2g3);
\draw[->, bend right = 30, black, , font=\scriptsize](y2topright) to node [sloped] {$\rho_{2}$~~~*} (ug3);
\draw[->, bend right = 0, black, , font=\scriptsize](y3g3) to node [sloped] {$\rho_{1}$} (vg3);
\end{tikzpicture}
}
\caption{}\label{fig:g2b}
\end{subfigure}

\caption{Rightmost position of a row}
\end{figure}

\begin{figure}

\begin{subfigure}{0.99\textwidth}
\centerline{
\begin{tikzpicture}[node distance=2cm]
\path[font = \scriptsize]
(-4.22, 1.52) node(y2lno) [outer sep=-1pt]{$\cdot$}
(4.78, 2.34) node(y2g3up) [outer sep=-1pt]{$\cdot$}
(1.5, 5.58) node(x3xup) [outer sep=-1pt]{$x_3$$\otimes$$x$}
(-0.61, -0.06) node(ug2down) [outer sep=-1pt]{$\cdot$}
(3.0, 0.0) node(y1g3) [outer sep=-1pt]{$y_1$$\otimes$$x$}
(-3.33, 4.34) node(y3no5) [outer sep=-1pt]{$\cdot$}
(-4.81, 4.72) node(uno3) [outer sep=-1pt]{$\cdot$}
(0.28, 2.34) node(y2g2up) [outer sep=-1pt]{$\cdot$}
(-4.5, 5.58) node(x3yup) [outer sep=-1pt]{$x_3$$\otimes$$y$}
(-6.94, 5.1) node(vno2) [outer sep=-1pt]{$\cdot$}
(5.67, 2.28) node(y3g3) [outer sep=-1pt]{$y_3$$\otimes$$x$}
(0.28, 0.7) node(y2g2down) [outer sep=-1pt]{$\cdot$}
(-3.61, 6.34) node(x2yup) [outer sep=-1pt]{$x_2$$\otimes$$y$}
(3.89, 0.76) node(ug3) [outer sep=-1pt]{$u$$\otimes$$x$}
(-8.0, -1.6) node(ubottomleft) [outer sep=-1pt]{$u$}
(-0.61, 0.76) node(ug2) [outer sep=-1pt]{$u$$\otimes$$x$}
(8.0, 4.8) node(y2topright) [outer sep=-1pt]{$y_2$}
(-3.0, 2.88) node(y1no4) [outer sep=-1pt]{$\cdot$}
(-2.44, 5.1) node(vno5) [outer sep=-1pt]{$\cdot$}
(-10.5, 3.96) node(y1no1) [outer sep=-1pt]{$\cdot$}
(0.28, 1.52) node(y2g2) [outer sep=-1pt]{$y_2$$\otimes$$x$}
(-4.22, 3.58) node(y2no5) [outer sep=-1pt]{$\cdot$}
(-9.61, 4.72) node(uno1) [outer sep=-1pt]{$\cdot$}
(-5.7, 3.96) node(y1no3) [outer sep=-1pt]{$\cdot$}
(-7.0, 6.4) node(v0) [outer sep=-1pt]{$V^0$}
(3.28, 7.1) node(x1xup) [outer sep=-1pt]{$x_1$$\otimes$$x$}
(4.78, 1.52) node(y2g3) [outer sep=-1pt]{$y_2$$\otimes$$x$}
(-10.0, 0.0) node(v1) [outer sep=-1pt]{$V^1$}
(-1.5, -0.82) node(y1g2down) [outer sep=-1pt]{$y_1$}
(-7.83, 4.34) node(y3no2) [outer sep=-1pt]{$\cdot$}
(-2.72, 7.1) node(x1yup) [outer sep=-1pt]{$x_1$$\otimes$$y$}
(2.39, 6.34) node(x2xup) [outer sep=-1pt]{$x_2$$\otimes$$x$}
(-8.72, 3.58) node(y2no2) [outer sep=-1pt]{$\cdot$}
;
\draw[->, bend right = 0, black, , font=\scriptsize](ug2) to node [sloped] {$\rho_{3}$} (y2g2);
\draw[->, bend right = -5.64, red, , font=\scriptsize](ug2) to node [sloped] {$\iota_0$} (y3g3);
\draw[->, bend right = 0, black, , font=\scriptsize](ug2) to node [sloped] {$\iota_0$} (ug2down);
\draw[->, bend right = 0, black, , font=\scriptsize](ug3) to node [sloped] {$\rho_{1}$} (y1g3);
\draw[->, bend right = 0, black, , font=\scriptsize](ug3) to node [sloped] {$\rho_{3}$} (y2g3);
\draw[->, bend right = -4.82, green, , font=\scriptsize](uno1) to node [sloped] {$\iota_0$} (x2yup);
\draw[->, bend right = -7.43, green, , font=\scriptsize](uno3) to node [sloped] {$\iota_0$} (x2xup);
\draw[->, bend right = -3.4, blue, , font=\scriptsize](x1xup) to node [sloped] {$\iota_1$} (vno5);
\draw[->, bend right = -3.75, blue, , font=\scriptsize](x1yup) to node [sloped] {$\iota_1$} (vno2);
\draw[->, bend right = 0, black, , font=\scriptsize](x2xup) to node [sloped] {$\rho_{1}$} (x1xup);
\draw[->, bend right = -3.4, blue, , font=\scriptsize](x2xup) to node [sloped] {$\iota_0$} (y3no5);
\draw[->, bend right = 0, black, , font=\scriptsize](x2yup) to node [sloped] {$\rho_{1}$} (x1yup);
\draw[->, bend right = -3.75, blue, , font=\scriptsize](x2yup) to node [sloped] {$\iota_0$} (y3no2);
\draw[->, bend right = 0, blue, , font=\scriptsize](x3xup) to node [sloped] {$\iota_1$} (y1g3);
\draw[->, bend right = 0, black, , font=\scriptsize](x3xup) to node [sloped] {$\rho_{2}$} (x2xup);
\draw[->, bend right = -3.4, blue, , font=\scriptsize](x3xup) to node [sloped] {$\iota_1$} (y2no5);
\draw[->, bend right = 0, blue, , font=\scriptsize](x3yup) to node [sloped] {$\iota_1$} (y1no4);
\draw[->, bend right = 0, black, , font=\scriptsize](x3yup) to node [sloped] {$\rho_{2}$} (x2yup);
\draw[->, bend right = -3.75, blue, , font=\scriptsize](x3yup) to node [sloped] {$\iota_1$} (y2no2);
\draw[->, bend right = -20, pink, , font=\scriptsize](x3yup) to node [sloped] {$\iota_1$} (y1g2down);
\draw[->, bend right = -6.41, green, , font=\scriptsize](y1no1) to node [sloped] {$\iota_1$} (x1yup);
\draw[->, bend right = -8.97, green, , font=\scriptsize](y1no3) to node [sloped] {$\iota_1$} (x1xup);
\draw[->, bend right = 0, black, , font=\scriptsize](y2g2) to node [sloped] {$\iota_1$} (y2g2down);
\draw[->, bend right = 0, red, , font=\scriptsize](y2g2) to node [sloped] {$\rho_{23}$} (y2lno);
\draw[->, bend right = -5.41, black, , font=\scriptsize](y2g2) to node [sloped] {$\rho_{2}$} (ubottomleft);
\draw[->, bend right = 0, black, , font=\scriptsize](y2g2up) to node [sloped] {$\iota_1$} (y2g2);
\draw[->, bend right = 0, black, , font=\scriptsize](y2g3) to node [sloped] {$\rho_{2}$} (y3g3);
\draw[->, bend right = 0, black, , font=\scriptsize](y2g3up) to node [sloped] {$\iota_1$} (y2g3);
\draw[->, bend right = 30, black, , font=\scriptsize](y2topright) to node [sloped] {$\rho_{2}$} (ug3);
\draw[->, bend right = 10, black, , font=\scriptsize](y3g3) to node [sloped] {$\rho_{1}$} (y1g2down);
\end{tikzpicture}
}
\caption{}\label{fig:g2c}
\end{subfigure}

\begin{subfigure}{0.99\textwidth}
\centerline{
\begin{tikzpicture}[node distance=2cm]
\path[font = \scriptsize]
(-10.0, 0.0) node(v1) [outer sep=-1pt]{$V^1$}
(1.5, 5.58) node(x3xup) [outer sep=-1pt]{$x_3$$\otimes$$x$}
(3.28, 7.1) node(x1xup) [outer sep=-1pt]{$x_1$$\otimes$$x$}
(-2.44, 5.1) node(vno5) [outer sep=-1pt]{$\cdot$}
(-8.0, -1.6) node(ubottomleft) [outer sep=-1pt]{$u$}
(-7.83, 4.34) node(y3no2) [outer sep=-1pt]{$\cdot$}
(3.89, 0.76) node(ug3) [outer sep=-1pt]{$u$$\otimes$$x$}
(-6.94, 5.1) node(vno2) [outer sep=-1pt]{$\cdot$}
(-4.22, 3.58) node(y2no5) [outer sep=-1pt]{$\cdot$}
(0.28, 2.34) node(y2g2up) [outer sep=-1pt]{$\cdot$}
(-10.5, 3.96) node(y1no1) [outer sep=-1pt]{$\cdot$}
(-4.81, 4.72) node(uno3) [outer sep=-1pt]{$\cdot$}
(0.28, 1.52) node(y2g2) [outer sep=-1pt]{$y_2$$\otimes$$x$}
(-4.5, 5.58) node(x3yup) [outer sep=-1pt]{$x_3$$\otimes$$y$}
(2.39, 6.34) node(x2xup) [outer sep=-1pt]{$x_2$$\otimes$$x$}
(-1.5, -0.82) node(y1g2down) [outer sep=-1pt]{$y_1$}
(4.78, 2.34) node(y2g3up) [outer sep=-1pt]{$\cdot$}
(4.78, 1.52) node(y2g3) [outer sep=-1pt]{$y_2$$\otimes$$x$}
(-3.61, 6.34) node(x2yup) [outer sep=-1pt]{$x_2$$\otimes$$y$}
(-4.22, 1.52) node(y2lno) [outer sep=-1pt]{$\cdot$}
(8.0, 4.8) node(y2topright) [outer sep=-1pt]{$y_2$}
(-0.61, -0.06) node(ug2down) [outer sep=-1pt]{$u$}
(-3.0, 2.88) node(y1no4) [outer sep=-1pt]{$\cdot$}
(-2.72, 7.1) node(x1yup) [outer sep=-1pt]{$x_1$$\otimes$$y$}
(-5.7, 3.96) node(y1no3) [outer sep=-1pt]{$\cdot$}
(-7.0, 6.4) node(v0) [outer sep=-1pt]{$V^0$}
(3.0, 0.0) node(y1g3) [outer sep=-1pt]{$y_1$$\otimes$$x$}
(-9.61, 4.72) node(uno1) [outer sep=-1pt]{$\cdot$}
(-8.72, 3.58) node(y2no2) [outer sep=-1pt]{$\cdot$}
(-3.33, 4.34) node(y3no5) [outer sep=-1pt]{$\cdot$}
;
\draw[->, bend right = 0, black, , font=\scriptsize](ug3) to node [sloped] {$\rho_{1}$} (y1g3);
\draw[->, bend right = 0, black, , font=\scriptsize](ug3) to node [sloped] {$\rho_{3}$} (y2g3);
\draw[->, bend right = -4.82, green, , font=\scriptsize](uno1) to node [sloped] {$\iota_0$} (x2yup);
\draw[->, bend right = -7.43, green, , font=\scriptsize](uno3) to node [sloped] {$\iota_0$} (x2xup);
\draw[->, bend right = -3.4, blue, , font=\scriptsize](x1xup) to node [sloped] {$\iota_1$} (vno5);
\draw[->, bend right = -3.75, blue, , font=\scriptsize](x1yup) to node [sloped] {$\iota_1$} (vno2);
\draw[->, bend right = 0, black, , font=\scriptsize](x2xup) to node [sloped] {$\rho_{1}$} (x1xup);
\draw[->, bend right = -3.4, blue, , font=\scriptsize](x2xup) to node [sloped] {$\iota_0$} (y3no5);
\draw[->, bend right = 0, black, , font=\scriptsize](x2yup) to node [sloped] {$\rho_{1}$} (x1yup);
\draw[->, bend right = -3.75, blue, , font=\scriptsize](x2yup) to node [sloped] {$\iota_0$} (y3no2);
\draw[->, bend right = 0, blue, , font=\scriptsize](x3xup) to node [sloped] {$\iota_1$} (y1g3);
\draw[->, bend right = 0, black, , font=\scriptsize](x3xup) to node [sloped] {$\rho_{2}$} (x2xup);
\draw[->, bend right = -3.4, blue, , font=\scriptsize](x3xup) to node [sloped] {$\iota_1$} (y2no5);
\draw[->, bend right = 0, blue, , font=\scriptsize](x3yup) to node [sloped] {$\iota_1$} (y1no4);
\draw[->, bend right = 0, black, , font=\scriptsize](x3yup) to node [sloped] {$\rho_{2}$} (x2yup);
\draw[->, bend right = -3.75, blue, , font=\scriptsize](x3yup) to node [sloped] {$\iota_1$} (y2no2);
\draw[->, bend right = -20, pink, , font=\scriptsize](x3yup) to node [sloped] {$\iota_1$~~~~**$\dagger$} (y1g2down);
\draw[->, bend right = -6.41, green, , font=\scriptsize](y1no1) to node [sloped] {$\iota_1$} (x1yup);
\draw[->, bend right = -8.97, green, , font=\scriptsize](y1no3) to node [sloped] {$\iota_1$} (x1xup);
\draw[->, bend right = 0, red, , font=\scriptsize](y2g2) to node [sloped] {$\rho_{23}$} (y2lno);
\draw[->, bend right = 0, black, , font=\scriptsize](y2g2) to node [sloped] {$\rho_{2}$~~~~**} (ubottomleft);
\draw[->, bend right = 0, black, , font=\scriptsize](y2g2up) to node [sloped] {$\iota_1$} (y2g2);
\draw[->, bend right = 0, red, , font=\scriptsize](y2g3) to node [sloped] {$\rho_{23}$} (y2g2);
\draw[->, bend right = -2.19, black, , font=\scriptsize](y2g3) to node [sloped] {$\rho_{2}$~~~~**$\dagger$} (ug2down);
\draw[->, bend right = 0, black, , font=\scriptsize](y2g3up) to node [sloped] {$\iota_1$} (y2g3);
\draw[->, bend right = 30, black, , font=\scriptsize](y2topright) to node [sloped] {$\rho_{2}$} (ug3);
\end{tikzpicture}
}
\caption{}\label{fig:g2d}
\end{subfigure}

\caption{Rightmost position of a row}
\end{figure}

\FloatBarrier
\subsubsection{Right full copy}\label{sec:rightfull}

\cref{fig:g3a} shows the part of the full copy in the tensor product. As we have done twice, we cancel the red arrows resulting in \cref{fig:g3b}. Note the arrows marked with $\#$ will be gone after we apply these cancellations in the rows above. Their existence does not interfere with the cancellation above, as they come out of targets of cancelable arrows. Such arrows coming from rows below would not have interfered here either. Disregarding those arrows, we arrive at \cref{fig:g3c}. Again the arrow marked with $**$ only exists if it goes into a $u$ in rightmost position, which in this case can only be the bottom row. \cref{fig:g3d} depicts the bottom row which has only one generator, with the potential $**$ arrow shown. The blue arrow comes from tensoring with $\rh{1}.$ Note there will not be any pink arrows from tensoring with potential $\rh{123}$, because this generator has the lowest Alexander grading.

The only arrows connecting to the full copy are those from the map $i:\mathbb{F}_2\to C$, which induces isomorphism in homology. The map $i$ must take the generator $e$ of $\mathbb{F}_2$ to some $\sum_jx_j\in C.$ The full picture of the full copy looks like \cref{fig:g3e}. The set of three generators on the left corresponds to $e$. Because we can simply treat the middle string of $\rh{23}$'s as we did in the simple case, a $\rh{23}$ coming out of $y_2$ is obviously the result of those cancellations. Only $x_1$ and $x_2$ are shown for simplicity. 

A map inducing isomorphism in homology in the opposite direction of $i$ must take $x_j\to e$ for some $j$ and all other generators of $C$ to 0. So in order to match with $KtD(\fl{C})$, we need one $\rh{23}$ from one of the $x_i$'s to $e$. We first cancel arrow marked by $\#1$ and arrive at \cref{fig:g3f}. Now cancel the arrow marked by $\#2$ and arrive at \cref{fig:g3g}. We then look to rid all the $v$'s in this picture. The set of $v$'s and the differentials among them make a chain complex $C_v$ over $\mathbb{Z}_2$ slightly different from $C$. We observe that if we perform the change of basis $x_1\to \sum_jx_j$ on $C$ and then the resulting complex minus generator $\sum_jx_j$ is isomorphic to $C_v$. Following the fact that $(C, \partial_z)$ is homotopy equivalent to $\mathbb{Z}_2$ and $\sum_ix_i$ is the generator of the homology, $C_v$ is homotopy equivalent to $\emptyset$. 
This means that there is a homotopy equivalence on the current stage of the module that exactly gets rid of all the $v$'s. 
A similarly constructed homotopy equivalence gets rid of all the $y_3$'s. The final result matches the full copy of $KtD(\fl{C})$, see \cref{fig:g3h}.

\begin{figure}

\begin{subfigure}{0.99\textwidth}
\centerline{

}
\caption{}\label{fig:g3a}
\end{subfigure}

\begin{subfigure}{0.99\textwidth}
\centerline{
%
}
\caption{}\label{fig:g3b}
\end{subfigure}

\begin{subfigure}{0.99\textwidth}
\centerline{
%
}
\caption{}\label{fig:g3c}
\end{subfigure}

\caption{The right full copy}
\end{figure}

\begin{figure}

\begin{subfigure}{0.99\textwidth}
\centerline{
%
}
\caption{}\label{fig:g3d}
\end{subfigure}

\begin{subfigure}{0.99\textwidth}
\centerline{
%
}
\caption{}\label{fig:g3e}
\end{subfigure}

\caption{The right full copy}
\end{figure}

\begin{figure}

\begin{subfigure}{0.99\textwidth}
\centerline{
%
}
\caption{}\label{fig:g3f}
\end{subfigure}

\begin{subfigure}{0.99\textwidth}
\centerline{
%
}
\caption{}\label{fig:g3g}
\end{subfigure}

\caption{The right full copy}
\end{figure}

\begin{figure}
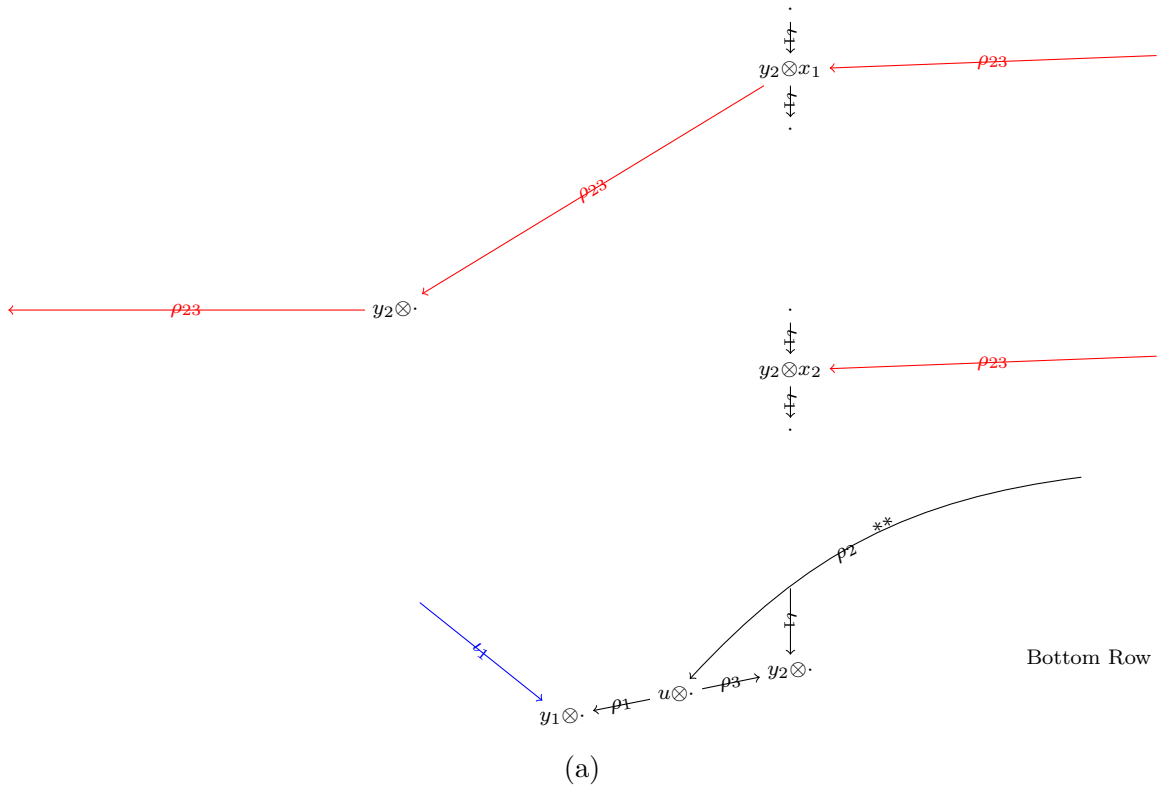


\begin{subfigure}{0.99\textwidth}
\centerline{
%
}
\caption{}\label{fig:g3h}
\end{subfigure}

\caption{The right full copy}
\end{figure}

\FloatBarrier
\subsubsection{Left-hand side middle part of rows in $V^1$}\label{sec:leftmiddle}

We switch to the left-hand side.
Recall that we assumed that $C$ is horizontally-simplified. This means generators in $(C, \partial_w)$ come in pairs except one generator, which we call $\xi_h$ to be consistent with the simplified case. Two generators in each pair are connected by one arrow. A typical pair $a\to b$ in $KtD(C)$ is displayed in \cref{fig:g4a}. The green arrow with $\rh{2}$ goes to $b\in V^0.$
Per \cref{alg2}, it always appear one column to the left of the left-most $a\to b$ arrow in $(C(\leq s),\partial_w)$ for some s. The dots are only place holders to demonstrate that the length of the arrow is $3.$

In the ``middle'' part of this long string, away from the left-most $a\to b$ arrow and the right-most full copy, the cancellation in the tensor product is just a simplified case of the process described in \cref{sec:rightmiddle}, see \cref{fig:g4b} and \cref{fig:g4c}. 

The unpaired $\xi_h$ constitutes a single string of $\rh{23}$'s in $KtD(C)$, whose middle part can be dealt with in the same manner. The result is a string of $\rh{23}$'s in the opposite direction, matching $KtD(\fl{C})$.

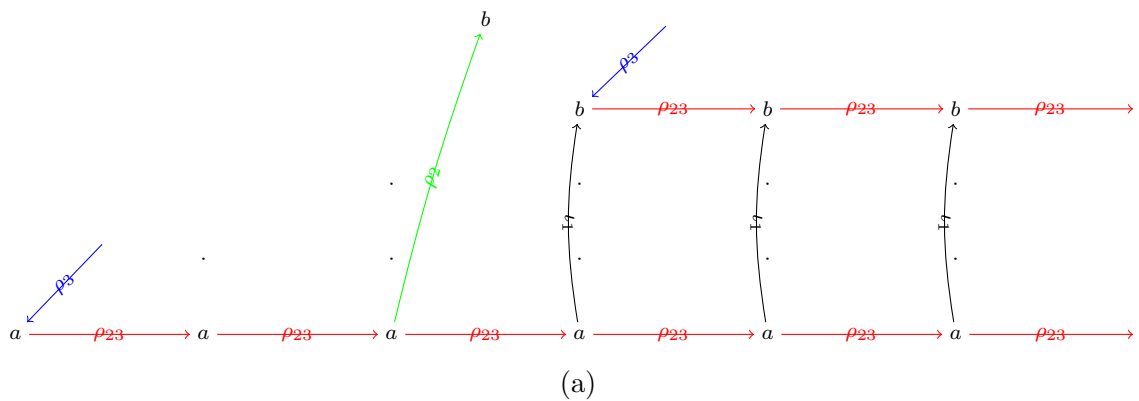
\begin{figure}[h]

\begin{subfigure}{0.99\textwidth}
\centerline{
\begin{tikzpicture}[node distance=2cm]
\path[font = \scriptsize]
(15.0, 1.0) node(g61) [outer sep=-1pt]{$a$}
(8.75, 5.2) node(a31) [outer sep=-1pt]{$b$}
(10.0, 2.0) node(g42) [outer sep=-1pt]{$\cdot$}
(17.5, 3.0) node(g73) [outer sep=-1pt]{$\cdot$}
(12.5, 2.0) node(g52) [outer sep=-1pt]{$\cdot$}
(10.0, 4.0) node(g44) [outer sep=-1pt]{$b$}
(5.0, 1.0) node(g21) [outer sep=-1pt]{$a$}
(12.5, 1.0) node(g51) [outer sep=-1pt]{$a$}
(2.5, 1.0) node(g11) [outer sep=-1pt]{$a$}
(7.5, 2.0) node(g32) [outer sep=-1pt]{$\cdot$}
(15.0, 4.0) node(g64) [outer sep=-1pt]{$b$}
(10.0, 3.0) node(g43) [outer sep=-1pt]{$\cdot$}
(7.5, 3.0) node(g33) [outer sep=-1pt]{$\cdot$}
(11.25, 5.2) node(a44) [outer sep=-1pt]{}
(3.75, 2.3) node(a11) [outer sep=-1pt]{}
(17.5, 4.0) node(g74) [outer sep=-1pt]{$\cdot$}
(12.5, 4.0) node(g54) [outer sep=-1pt]{$b$}
(7.5, 1.0) node(g31) [outer sep=-1pt]{$a$}
(5.0, 2.0) node(g22) [outer sep=-1pt]{$\cdot$}
(15.0, 2.0) node(g62) [outer sep=-1pt]{$\cdot$}
(15.0, 3.0) node(g63) [outer sep=-1pt]{$\cdot$}
(17.5, 1.0) node(g71) [outer sep=-1pt]{$\cdot$}
(17.5, 2.0) node(g72) [outer sep=-1pt]{$\cdot$}
(12.5, 3.0) node(g53) [outer sep=-1pt]{$\cdot$}
(10.0, 1.0) node(g41) [outer sep=-1pt]{$a$}
;
\draw[->, bend right = 0, blue, , font=\scriptsize](a11) to node [sloped] {$\rho_{3}$} (g11);
\draw[->, bend right = 0, blue, , font=\scriptsize](a44) to node [sloped] {$\rho_{3}$} (g44);
\draw[->, bend right = 0, red, , font=\scriptsize](g11) to node [sloped] {$\rho_{23}$} (g21);
\draw[->, bend right = 0, red, , font=\scriptsize](g21) to node [sloped] {$\rho_{23}$} (g31);
\draw[->, bend right = 0, red, , font=\scriptsize](g31) to node [sloped] {$\rho_{23}$} (g41);
\draw[->, bend right = -3.79, green, , font=\scriptsize](g31) to node [sloped] {$\rho_{2}$} (a31);
\draw[->, bend right = 0, red, , font=\scriptsize](g41) to node [sloped] {$\rho_{23}$} (g51);
\draw[->, bend right = -9.0, black, , font=\scriptsize](g41) to node [sloped] {$\iota_1$} (g44);
\draw[->, bend right = 0, red, , font=\scriptsize](g44) to node [sloped] {$\rho_{23}$} (g54);
\draw[->, bend right = 0, red, , font=\scriptsize](g51) to node [sloped] {$\rho_{23}$} (g61);
\draw[->, bend right = -9.0, black, , font=\scriptsize](g51) to node [sloped] {$\iota_1$} (g54);
\draw[->, bend right = 0, red, , font=\scriptsize](g54) to node [sloped] {$\rho_{23}$} (g64);
\draw[->, bend right = 0, red, , font=\scriptsize](g61) to node [sloped] {$\rho_{23}$} (g71);
\draw[->, bend right = -9.0, black, , font=\scriptsize](g61) to node [sloped] {$\iota_1$} (g64);
\draw[->, bend right = 0, red, , font=\scriptsize](g64) to node [sloped] {$\rho_{23}$} (g74);
\end{tikzpicture}
}
\caption{}\label{fig:g4a}
\end{subfigure}

\caption{Lefthand side middle part of rows in $V^1$}
\end{figure}

\begin{figure}

\begin{subfigure}{0.99\textwidth}
\centerline{
\begin{tikzpicture}[node distance=2cm]
\path[font = \scriptsize]
(11.67, 6.85) node(y3g21) [outer sep=-1pt]{$y_3$}
(9.89, 4.95) node(ug21) [outer sep=-1pt]{$u$}
(17.06, 7.8) node(vg31) [outer sep=-1pt]{$v$}
(9.89, 8.95) node(ug22) [outer sep=-1pt]{$u$}
(5.39, 4.95) node(ug11) [outer sep=-1pt]{$\cdot$}
(12.56, 11.8) node(vg22) [outer sep=-1pt]{$v$}
(4.5, 4.0) node(y1g11) [outer sep=-1pt]{$\cdot$}
(5.39, 8.95) node(ug12) [outer sep=-1pt]{$\cdot$}
(9.0, 4.0) node(y1g21) [outer sep=-1pt]{$y_1$}
(20.67, 6.85) node(y3g41) [outer sep=-1pt]{$\cdot$}
(12.56, 7.8) node(vg21) [outer sep=-1pt]{$v$}
(21.56, 7.8) node(vg41) [outer sep=-1pt]{$\cdot$}
(14.39, 4.95) node(ug31) [outer sep=-1pt]{$u$}
(17.06, 11.8) node(vg32) [outer sep=-1pt]{$v$}
(15.28, 5.9) node(y2g31) [outer sep=-1pt]{$y_2$}
(13.5, 8.0) node(y1g32) [outer sep=-1pt]{$y_1$}
(15.28, 9.9) node(y2g32) [outer sep=-1pt]{$y_2$}
(4.5, 8.0) node(y1g12) [outer sep=-1pt]{$\cdot$}
(13.5, 4.0) node(y1g31) [outer sep=-1pt]{$y_1$}
(10.78, 5.9) node(y2g21) [outer sep=-1pt]{$y_2$}
(16.17, 6.85) node(y3g31) [outer sep=-1pt]{$y_3$}
(20.67, 10.85) node(y3g42) [outer sep=-1pt]{$\cdot$}
(14.39, 8.95) node(ug32) [outer sep=-1pt]{$u$}
(9.0, 8.0) node(y1g22) [outer sep=-1pt]{$y_1$}
(16.17, 10.85) node(y3g32) [outer sep=-1pt]{$y_3$}
(11.67, 10.85) node(y3g22) [outer sep=-1pt]{$y_3$}
(21.56, 11.8) node(vg42) [outer sep=-1pt]{$\cdot$}
(10.78, 9.9) node(y2g22) [outer sep=-1pt]{$y_2$}
;
\draw[->, bend right = -1.06, red, , font=\scriptsize](ug11) to node [sloped] {$\iota_0$} (y3g21);
\draw[->, bend right = -1.06, red, , font=\scriptsize](ug12) to node [sloped] {$\iota_0$} (y3g22);
\draw[->, bend right = 0, black, , font=\scriptsize](ug21) to node [sloped] {$\rho_{1}$} (y1g21);
\draw[->, bend right = 0, black, , font=\scriptsize](ug21) to node [sloped] {$\rho_{3}$} (y2g21);
\draw[->, bend right = -1.06, red, , font=\scriptsize](ug21) to node [sloped] {$\iota_0$} (y3g31);
\draw[->, bend right = 0, black, , font=\scriptsize](ug21) to node [sloped] {$\iota_0$} (ug22);
\draw[->, bend right = 0, black, , font=\scriptsize](ug22) to node [sloped] {$\rho_{1}$} (y1g22);
\draw[->, bend right = 0, black, , font=\scriptsize](ug22) to node [sloped] {$\rho_{3}$} (y2g22);
\draw[->, bend right = -1.06, red, , font=\scriptsize](ug22) to node [sloped] {$\iota_0$} (y3g32);
\draw[->, bend right = 0, black, , font=\scriptsize](ug31) to node [sloped] {$\rho_{1}$} (y1g31);
\draw[->, bend right = 0, black, , font=\scriptsize](ug31) to node [sloped] {$\rho_{3}$} (y2g31);
\draw[->, bend right = -1.06, red, , font=\scriptsize](ug31) to node [sloped] {$\iota_0$} (y3g41);
\draw[->, bend right = 0, black, , font=\scriptsize](ug31) to node [sloped] {$\iota_0$} (ug32);
\draw[->, bend right = 0, black, , font=\scriptsize](ug32) to node [sloped] {$\rho_{1}$} (y1g32);
\draw[->, bend right = 0, black, , font=\scriptsize](ug32) to node [sloped] {$\rho_{3}$} (y2g32);
\draw[->, bend right = -1.06, red, , font=\scriptsize](ug32) to node [sloped] {$\iota_0$} (y3g42);
\draw[->, bend right = 0, black, , font=\scriptsize](vg21) to node [sloped] {$\iota_1$} (vg22);
\draw[->, bend right = 0, black, , font=\scriptsize](vg31) to node [sloped] {$\iota_1$} (vg32);
\draw[->, bend right = -4.7, red, , font=\scriptsize](y1g11) to node [sloped] {$\iota_1$} (vg21);
\draw[->, bend right = -4.7, red, , font=\scriptsize](y1g12) to node [sloped] {$\iota_1$} (vg22);
\draw[->, bend right = -4.7, red, , font=\scriptsize](y1g21) to node [sloped] {$\iota_1$} (vg31);
\draw[->, bend right = 0, black, , font=\scriptsize](y1g21) to node [sloped] {$\iota_1$} (y1g22);
\draw[->, bend right = -4.7, red, , font=\scriptsize](y1g22) to node [sloped] {$\iota_1$} (vg32);
\draw[->, bend right = -4.7, red, , font=\scriptsize](y1g31) to node [sloped] {$\iota_1$} (vg41);
\draw[->, bend right = 0, black, , font=\scriptsize](y1g31) to node [sloped] {$\iota_1$} (y1g32);
\draw[->, bend right = -4.7, red, , font=\scriptsize](y1g32) to node [sloped] {$\iota_1$} (vg42);
\draw[->, bend right = 0, black, , font=\scriptsize](y2g21) to node [sloped] {$\rho_{2}$} (y3g21);
\draw[->, bend right = 0, black, , font=\scriptsize](y2g21) to node [sloped] {$\iota_1$} (y2g22);
\draw[->, bend right = 0, black, , font=\scriptsize](y2g22) to node [sloped] {$\rho_{2}$} (y3g22);
\draw[->, bend right = 0, black, , font=\scriptsize](y2g31) to node [sloped] {$\rho_{2}$} (y3g31);
\draw[->, bend right = 0, black, , font=\scriptsize](y2g31) to node [sloped] {$\iota_1$} (y2g32);
\draw[->, bend right = 0, black, , font=\scriptsize](y2g32) to node [sloped] {$\rho_{2}$} (y3g32);
\draw[->, bend right = 0, black, , font=\scriptsize](y3g21) to node [sloped] {$\rho_{1}$} (vg21);
\draw[->, bend right = 0, black, , font=\scriptsize](y3g21) to node [sloped] {$\iota_0$} (y3g22);
\draw[->, bend right = 0, black, , font=\scriptsize](y3g22) to node [sloped] {$\rho_{1}$} (vg22);
\draw[->, bend right = 0, black, , font=\scriptsize](y3g31) to node [sloped] {$\rho_{1}$} (vg31);
\draw[->, bend right = 0, black, , font=\scriptsize](y3g31) to node [sloped] {$\iota_0$} (y3g32);
\draw[->, bend right = 0, black, , font=\scriptsize](y3g32) to node [sloped] {$\rho_{1}$} (vg32);
\end{tikzpicture}
}
\caption{}\label{fig:g4b}
\end{subfigure}

\begin{subfigure}{0.99\textwidth}
\centerline{
\begin{tikzpicture}[node distance=2cm]
\path[font = \scriptsize]
(15.28, 4.95) node(y2g32) [outer sep=-1pt]{$y_2$}
(19.78, 2.95) node(y2g41) [outer sep=-1pt]{$\cdot$}
(6.28, 4.95) node(y2g12) [outer sep=-1pt]{$\cdot$}
(19.78, 4.95) node(y2g42) [outer sep=-1pt]{$\cdot$}
(10.78, 2.95) node(y2g21) [outer sep=-1pt]{$y_2$}
(6.28, 2.95) node(y2g11) [outer sep=-1pt]{$\cdot$}
(15.28, 2.95) node(y2g31) [outer sep=-1pt]{$y_2$}
(10.78, 4.95) node(y2g22) [outer sep=-1pt]{$y_2$}
;
\draw[->, bend right = 0, black, , font=\scriptsize](y2g21) to node [sloped] {$\iota_1$} (y2g22);
\draw[->, bend right = 0, red, , font=\scriptsize](y2g21) to node [sloped] {$\rho_{23}$} (y2g11);
\draw[->, bend right = 0, red, , font=\scriptsize](y2g22) to node [sloped] {$\rho_{23}$} (y2g12);
\draw[->, bend right = 0, black, , font=\scriptsize](y2g31) to node [sloped] {$\iota_1$} (y2g32);
\draw[->, bend right = 0, red, , font=\scriptsize](y2g31) to node [sloped] {$\rho_{23}$} (y2g21);
\draw[->, bend right = 0, red, , font=\scriptsize](y2g32) to node [sloped] {$\rho_{23}$} (y2g22);
\draw[->, bend right = 0, red, , font=\scriptsize](y2g41) to node [sloped] {$\rho_{23}$} (y2g31);
\draw[->, bend right = 0, red, , font=\scriptsize](y2g42) to node [sloped] {$\rho_{23}$} (y2g32);
\end{tikzpicture}
}
\caption{}\label{fig:g4c}
\end{subfigure}

\caption{Lefthand side middle part of rows in $V^1$}
\end{figure}
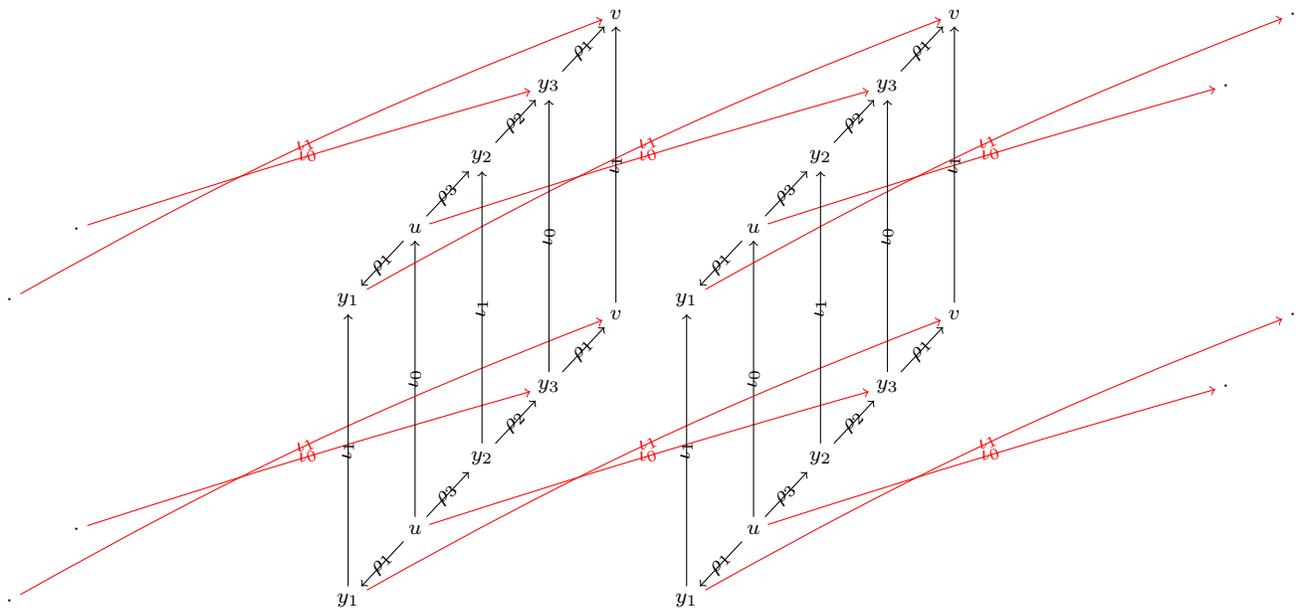
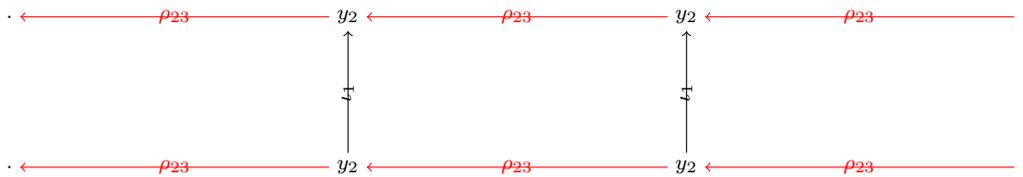

\FloatBarrier
\subsubsection{$V^0$}\label{sec:v0}
In this section we look at tensoring at $V^0$ where outgoing $\rh{1}$ and $\rh{3}$ connect to the left-most positions of the left-hand side and right-most positions of the right-hand side, see \cref{fig:g5a}. We first assume the arrow $a\to b$ has length more than 1. This means that on the left of \cref{fig:g5a}, the position where the green arrow comes out is not the left most position, where the blue $\rh{3}$ comes in. It also means that there won't be any other length one $\partial_w$ arrow connecting to $a$ or $b$, so there is not any pink $\rh{123}$'s to the right side of this diagram. 

Taking the tensor and keeping in mind all the previous cancellation we have done, especially in \cref{sec:rightend} and \cref{fig:g2d}, we arrive at \cref{fig:g5b}. We first cancel the two pair of red arrows in the same fashion we have always done. We also cancel the two blue arrows from $x_1 \to v$, then the two blue arrows from $x_2 \to y_3$, resulting in \cref{fig:g5c}. The cancelation of the eight arrows so far are simple and generated no new arrows. Note the green arrows are now gone. 

Next, we cancel the two pair of red arrows on the left and rearrange the $y_1$'s and $u$'s on the right side to the top, see \cref{fig:g5d}. Next, we cancel the two blue arrows $x_3\to y_1$ on the top, see \cref{fig:g5e}. 

Now the picture looks very much like $KtD(\fl{C})$ described in \cref{localchange}; For each generator, its $y_2$ takes its place. $\rh{23}$'s are reversed. $\partial_z$ and $\partial_w$ are intact. For each generator $x$, its $u$ at the right-most position in the right side takes the place of $x$ in $V^0$. $\rh{1}$ and $\rh{3}$ are switched. Let us take a closer look at the green arrows with $\rh{2}$ and marked by $*$ in this section. Say the $a$ has an incoming $c\to a$ in $\partial_z$. As seen in \cref{sec:rightmiddle}, the green arrow $*$ with $\rh{2}$ comes from $c$'s $y_2$ in one column to the right. In other words, if a downward $\partial_z$ arrows arrives at a generator $a$ at a right-most position from a generator $c$ above, then there is a $\rho_2$ arrow coming out of the $y_2$ to $c$'s right and going into the bottom $a$'s $u$, which takes the place of the old $a\in V^0$ now. This matches exactly to the green arrows of $KtD(\fl{C})$ described in \cref{localchange}. The green arrows marked by $**$ exist for the exact same reasons, except from the origin's point of view. They exist only when they are going into a rightmost position, which must be the rightmost position of one row below, as they themselves are at rightmost positions. We again observe that these green arrows never interacted with the cancellation done in this section, so we are safe to say that we the cancellations generalize to all applicable part. 

The case depicted in the diagrams is that $a\to b$ has length 2, the case that it has length more than 2 is similar. 

\begin{figure}

\centerline{
\begin{tikzpicture}[node distance=2cm]
\path[font = \scriptsize]
(14.4, 3.8) node(b6up) [outer sep=-1pt]{}
(5.4, 3.0) node(b3) [outer sep=-1pt]{$b$}
(7.2, 1.0) node(a4) [outer sep=-1pt]{}
(7.2, 3.0) node(b4) [outer sep=-1pt]{}
(16.2, 3.0) node(b7) [outer sep=-1pt]{$b$}
(12.6, 6.0) node(topb) [outer sep=-1pt]{$b$}
(12.6, 1.0) node(a5) [outer sep=-1pt]{}
(14.4, 1.0) node(a6) [outer sep=-1pt]{$a$}
(18.0, 3.8) node(b8up) [outer sep=-1pt]{}
(3.6, 1.0) node(a2) [outer sep=-1pt]{$a$}
(1.8, 1.0) node(a1) [outer sep=-1pt]{$a$}
(16.2, 3.8) node(b7up) [outer sep=-1pt]{}
(9.0, 6.0) node(topa) [outer sep=-1pt]{$a$}
(14.4, 1.8) node(a6up) [outer sep=-1pt]{}
(16.2, 2.2) node(b7down) [outer sep=-1pt]{}
(18.0, 3.0) node(b8) [outer sep=-1pt]{$b$}
(5.4, 1.0) node(a3) [outer sep=-1pt]{$a$}
(14.4, 2.2) node(b6down) [outer sep=-1pt]{}
(14.4, 3.0) node(b6) [outer sep=-1pt]{$b$}
(12.6, 3.0) node(b5) [outer sep=-1pt]{}
;
\draw[->, bend right = 0, red, , font=\scriptsize](a1) to node [sloped] {$\rho_{23}$} (a2);
\draw[->, bend right = 0, red, , font=\scriptsize](a2) to node [sloped] {$\rho_{23}$} (a3);
\draw[->, bend right = 10, green, , font=\scriptsize](a2) to node [sloped] {$\rho_{2}$} (topb);
\draw[->, bend right = 0, red, , font=\scriptsize](a3) to node [sloped] {$\rho_{23}$} (a4);
\draw[->, bend right = 0, black, , font=\scriptsize](a3) to node [sloped] {$\iota_1$} (b3);
\draw[->, bend right = 0, red, , font=\scriptsize](a5) to node [sloped] {$\rho_{23}$} (a6);
\draw[->, bend right = 0, black, , font=\scriptsize](a6up) to node [sloped] {$\iota_1$} (a6);
\draw[->, bend right = 0, red, , font=\scriptsize](b3) to node [sloped] {$\rho_{23}$} (b4);
\draw[->, bend right = 0, red, , font=\scriptsize](b5) to node [sloped] {$\rho_{23}$} (b6);
\draw[->, bend right = 0, black, , font=\scriptsize](b6) to node [sloped] {$\iota_1$} (b6down);
\draw[->, bend right = 0, red, , font=\scriptsize](b6) to node [sloped] {$\rho_{23}$} (b7);
\draw[->, bend right = 0, black, , font=\scriptsize](b6up) to node [sloped] {$\iota_1$} (b6);
\draw[->, bend right = 0, black, , font=\scriptsize](b7) to node [sloped] {$\iota_1$} (b7down);
\draw[->, bend right = 0, red, , font=\scriptsize](b7) to node [sloped] {$\rho_{23}$} (b8);
\draw[->, bend right = 0, black, , font=\scriptsize](b7up) to node [sloped] {$\iota_1$} (b7);
\draw[->, bend right = 0, black, , font=\scriptsize](b8up) to node [sloped] {$\iota_1$} (b8);
\draw[->, bend right = 10, blue, , font=\scriptsize](topa) to node [sloped] {$\rho_{3}$} (a1);
\draw[->, bend right = -6.34, blue, , font=\scriptsize](topa) to node [sloped] {$\rho_{1}$} (a6);
\draw[->, bend right = 10, blue, , font=\scriptsize](topb) to node [sloped] {$\rho_{3}$} (b3);
\draw[->, bend right = -6.74, blue, , font=\scriptsize](topb) to node [sloped] {$\rho_{1}$} (b8);
\end{tikzpicture}
}

\caption{$V^0$}\label{fig:g5a}
\end{figure}

\clearpage
\newgeometry{margin=2cm}
\begin{landscape}
\begin{figure}
\centering
\begin{tikzpicture}[node distance=2cm]
\path[font = \scriptsize]
(18.2, 1.9) node(y1a6) [outer sep=-1pt]{$y_1$$\otimes$$a$}
(14.76, 9.5) node(vb4) [outer sep=-1pt]{$\cdot$}
(10.18, 3.8) node(y2a3) [outer sep=-1pt]{$y_2$$\otimes$$a$}
(11.07, 4.75) node(y3a3) [outer sep=-1pt]{$y_3$$\otimes$$a$}
(19.09, 2.85) node(ua6) [outer sep=-1pt]{$u$$\otimes$$a$}
(14.89, 15.2) node(x2topa) [outer sep=-1pt]{$x_2$$\otimes$$a$}
(14.98, 3.8) node(a4half) [outer sep=-1pt]{$\cdot$}
(11.07, 8.55) node(y3b3) [outer sep=-1pt]{$y_3$$\otimes$$b$}
(12.98, 3.8) node(y2a4) [outer sep=-1pt]{$\cdot$}
(13.87, 8.55) node(y3b4) [outer sep=-1pt]{$\cdot$}
(21.98, 5.8) node(aextray2) [outer sep=-1pt]{$y_2$}
(6.49, 2.85) node(ua2) [outer sep=-1pt]{$u$$\otimes$$a$}
(9.16, 5.7) node(va2) [outer sep=-1pt]{$v$$\otimes$$a$}
(17.18, 3.8) node(y2a5) [outer sep=-1pt]{$\cdot$}
(12.98, 7.6) node(y2b4) [outer sep=-1pt]{$\cdot$}
(8.4, 1.9) node(y1a3) [outer sep=-1pt]{$y_1$$\otimes$$a$}
(7.38, 3.8) node(y2a2) [outer sep=-1pt]{$y_2$$\otimes$$a$}
(19.6, 14.25) node(x3topb) [outer sep=-1pt]{$x_3$$\otimes$$b$}
(16.98, -0.2) node(aextrau) [outer sep=-1pt]{$u$}
(15.78, 16.15) node(x1topa) [outer sep=-1pt]{$x_1$$\otimes$$a$}
(24.69, 6.65) node(ub8) [outer sep=-1pt]{$u$$\otimes$$b$}
(19.98, 3.8) node(y2a6) [outer sep=-1pt]{$y_2$$\otimes$$a$}
(2.8, 1.9) node(y1a1) [outer sep=-1pt]{$y_1$$\otimes$$a$}
(14.98, 7.6) node(b4half) [outer sep=-1pt]{$\cdot$}
(22.78, 7.6) node(y2b7) [outer sep=-1pt]{$y_2$$\otimes$$b$}
(11.96, 9.5) node(vb3) [outer sep=-1pt]{$v$$\otimes$$b$}
(14.76, 5.7) node(va4) [outer sep=-1pt]{$\cdot$}
(9.29, 6.65) node(ub3) [outer sep=-1pt]{$u$$\otimes$$b$}
(4.58, 3.8) node(y2a1) [outer sep=-1pt]{$y_2$$\otimes$$a$}
(5.6, 1.9) node(y1a2) [outer sep=-1pt]{$y_1$$\otimes$$a$}
(9.29, 2.85) node(ua3) [outer sep=-1pt]{$u$$\otimes$$a$}
(6.36, 5.7) node(va1) [outer sep=-1pt]{$v$$\otimes$$a$}
(27.58, 9.6) node(bextray2) [outer sep=-1pt]{$y_2$}
(23.8, 5.7) node(y1b8) [outer sep=-1pt]{$y_1$$\otimes$$b$}
(5.47, 4.75) node(y3a1) [outer sep=-1pt]{$y_3$$\otimes$$a$}
(25.58, 7.6) node(y2b8) [outer sep=-1pt]{$y_2$$\otimes$$b$}
(14.0, 14.25) node(x3topa) [outer sep=-1pt]{$x_3$$\otimes$$a$}
(10.18, 7.6) node(y2b3) [outer sep=-1pt]{$y_2$$\otimes$$b$}
(19.98, 7.6) node(y2b6) [outer sep=-1pt]{}
(8.27, 4.75) node(y3a2) [outer sep=-1pt]{$y_3$$\otimes$$a$}
(13.87, 4.75) node(y3a4) [outer sep=-1pt]{$\cdot$}
(21.38, 16.15) node(x1topb) [outer sep=-1pt]{$x_1$$\otimes$$b$}
(20.49, 15.2) node(x2topb) [outer sep=-1pt]{$x_2$$\otimes$$b$}
(11.96, 5.7) node(va3) [outer sep=-1pt]{$v$$\otimes$$a$}
(3.69, 2.85) node(ua1) [outer sep=-1pt]{$u$$\otimes$$a$}
(22.78, 9.12) node(y2b7up) [outer sep=-1pt]{$\cdot$}
(25.58, 9.12) node(y2b8up) [outer sep=-1pt]{$\cdot$}
(8.4, 5.7) node(y1b3) [outer sep=-1pt]{$y_1$$\otimes$$b$}
(22.78, 6.08) node(y2b7down) [outer sep=-1pt]{$\cdot$}
(19.98, 5.32) node(y2a6up) [outer sep=-1pt]{$\cdot$}
(22.58, 3.6) node(bextrau) [outer sep=-1pt]{$u$}
;
\draw[->, bend right = 0, red, , font=\scriptsize](a4half) to node [sloped] {$\rho_{23}$} (y2a4);
\draw[->, bend right = 40, green, , font=\scriptsize](aextray2) to node [sloped] {$\rho_{2}$~~~~~~~*} (ua6);
\draw[->, bend right = 0, red, , font=\scriptsize](b4half) to node [sloped] {$\rho_{23}$} (y2b4);
\draw[->, bend right = 40, green, , font=\scriptsize](bextray2) to node [sloped] {$\rho_{2}$~~~~~~~*} (ub8);
\draw[->, bend right = 0, black, , font=\scriptsize](ua1) to node [sloped] {$\rho_{1}$} (y1a1);
\draw[->, bend right = 0, black, , font=\scriptsize](ua1) to node [sloped] {$\rho_{3}$} (y2a1);
\draw[->, bend right = 0, red, , font=\scriptsize](ua1) to node [sloped] {$\iota_0$} (y3a2);
\draw[->, bend right = 0, black, , font=\scriptsize](ua2) to node [sloped] {$\rho_{1}$} (y1a2);
\draw[->, bend right = 0, black, , font=\scriptsize](ua2) to node [sloped] {$\rho_{3}$} (y2a2);
\draw[->, bend right = 0, red, , font=\scriptsize](ua2) to node [sloped] {$\iota_0$} (y3a3);
\draw[->, bend right = -20, green, , font=\scriptsize](ua2) to node [sloped] {$\iota_0$} (x2topb);
\draw[->, bend right = 0, black, , font=\scriptsize](ua3) to node [sloped] {$\rho_{1}$} (y1a3);
\draw[->, bend right = 0, black, , font=\scriptsize](ua3) to node [sloped] {$\rho_{3}$} (y2a3);
\draw[->, bend right = 0, red, , font=\scriptsize](ua3) to node [sloped] {$\iota_0$} (y3a4);
\draw[->, bend right = -6.33, black, , font=\scriptsize](ua3) to node [sloped] {$\iota_0$} (ub3);
\draw[->, bend right = 0, black, , font=\scriptsize](ua6) to node [sloped] {$\rho_{1}$} (y1a6);
\draw[->, bend right = 0, black, , font=\scriptsize](ua6) to node [sloped] {$\rho_{3}$} (y2a6);
\draw[->, bend right = 0, black, , font=\scriptsize](ub3) to node [sloped] {$\rho_{1}$} (y1b3);
\draw[->, bend right = 0, black, , font=\scriptsize](ub3) to node [sloped] {$\rho_{3}$} (y2b3);
\draw[->, bend right = 0, red, , font=\scriptsize](ub3) to node [sloped] {$\iota_0$} (y3b4);
\draw[->, bend right = 0, black, , font=\scriptsize](ub8) to node [sloped] {$\rho_{1}$} (y1b8);
\draw[->, bend right = 0, black, , font=\scriptsize](ub8) to node [sloped] {$\rho_{3}$} (y2b8);
\draw[->, bend right = 0, black, , font=\scriptsize](va3) to node [sloped] {$\iota_1$} (vb3);
\draw[->, bend right = 30, blue, , font=\scriptsize](x1topa) to node [sloped] {$\iota_1$} (va1);
\draw[->, bend right = 30, blue, , font=\scriptsize](x1topb) to node [sloped] {$\iota_1$} (vb3);
\draw[->, bend right = 0, black, , font=\scriptsize](x2topa) to node [sloped] {$\rho_{1}$} (x1topa);
\draw[->, bend right = 30, blue, , font=\scriptsize](x2topa) to node [sloped] {$\iota_0$} (y3a1);
\draw[->, bend right = 0, black, , font=\scriptsize](x2topb) to node [sloped] {$\rho_{1}$} (x1topb);
\draw[->, bend right = 30, blue, , font=\scriptsize](x2topb) to node [sloped] {$\iota_0$} (y3b3);
\draw[->, bend right = -30, blue, , font=\scriptsize](x3topa) to node [sloped] {$\iota_1$} (y1a6);
\draw[->, bend right = 0, black, , font=\scriptsize](x3topa) to node [sloped] {$\rho_{2}$} (x2topa);
\draw[->, bend right = 30, blue, , font=\scriptsize](x3topa) to node [sloped] {$\iota_1$} (y2a1);
\draw[->, bend right = -30, blue, , font=\scriptsize](x3topb) to node [sloped] {$\iota_1$} (y1b8);
\draw[->, bend right = 0, black, , font=\scriptsize](x3topb) to node [sloped] {$\rho_{2}$} (x2topb);
\draw[->, bend right = 30, blue, , font=\scriptsize](x3topb) to node [sloped] {$\iota_1$} (y2b3);
\draw[->, bend right = -5.13, red, , font=\scriptsize](y1a1) to node [sloped] {$\iota_1$} (va2);
\draw[->, bend right = -5.13, red, , font=\scriptsize](y1a2) to node [sloped] {$\iota_1$} (va3);
\draw[->, bend right = -20, green, , font=\scriptsize](y1a2) to node [sloped] {$\iota_1$} (x1topb);
\draw[->, bend right = -5.13, red, , font=\scriptsize](y1a3) to node [sloped] {$\iota_1$} (va4);
\draw[->, bend right = -6.33, black, , font=\scriptsize](y1a3) to node [sloped] {$\iota_1$} (y1b3);
\draw[->, bend right = -5.13, red, , font=\scriptsize](y1b3) to node [sloped] {$\iota_1$} (vb4);
\draw[->, bend right = 0, black, , font=\scriptsize](y2a1) to node [sloped] {$\rho_{2}$} (y3a1);
\draw[->, bend right = 0, black, , font=\scriptsize](y2a2) to node [sloped] {$\rho_{2}$} (y3a2);
\draw[->, bend right = 0, black, , font=\scriptsize](y2a3) to node [sloped] {$\rho_{2}$} (y3a3);
\draw[->, bend right = 0, black, , font=\scriptsize](y2a3) to node [sloped] {$\iota_1$} (y2b3);
\draw[->, bend right = 0, black, , font=\scriptsize](y2a4) to node [sloped] {$\rho_{2}$} (y3a4);
\draw[->, bend right = 0, red, , font=\scriptsize](y2a6) to node [sloped] {$\rho_{23}$} (y2a5);
\draw[->, bend right = 30, green, , font=\scriptsize](y2a6) to node [sloped] {$\rho_{2}$~~~**} (aextrau);
\draw[->, bend right = 0, black, , font=\scriptsize](y2a6up) to node [sloped] {$\iota_1$} (y2a6);
\draw[->, bend right = 0, black, , font=\scriptsize](y2b3) to node [sloped] {$\rho_{2}$} (y3b3);
\draw[->, bend right = 0, black, , font=\scriptsize](y2b4) to node [sloped] {$\rho_{2}$} (y3b4);
\draw[->, bend right = 0, black, , font=\scriptsize](y2b7) to node [sloped] {$\iota_1$} (y2b7down);
\draw[->, bend right = 0, red, , font=\scriptsize](y2b7) to node [sloped] {$\rho_{23}$} (y2b6);
\draw[->, bend right = 0, black, , font=\scriptsize](y2b7up) to node [sloped] {$\iota_1$} (y2b7);
\draw[->, bend right = 0, red, , font=\scriptsize](y2b8) to node [sloped] {$\rho_{23}$} (y2b7);
\draw[->, bend right = 30, green, , font=\scriptsize](y2b8) to node [sloped] {$\rho_{2}$~~~**} (bextrau);
\draw[->, bend right = 0, black, , font=\scriptsize](y2b8up) to node [sloped] {$\iota_1$} (y2b8);
\draw[->, bend right = 0, black, , font=\scriptsize](y3a1) to node [sloped] {$\rho_{1}$} (va1);
\draw[->, bend right = 0, black, , font=\scriptsize](y3a2) to node [sloped] {$\rho_{1}$} (va2);
\draw[->, bend right = 0, black, , font=\scriptsize](y3a3) to node [sloped] {$\rho_{1}$} (va3);
\draw[->, bend right = 0, black, , font=\scriptsize](y3a3) to node [sloped] {$\iota_0$} (y3b3);
\draw[->, bend right = 0, black, , font=\scriptsize](y3a4) to node [sloped] {$\rho_{1}$} (va4);
\draw[->, bend right = 0, black, , font=\scriptsize](y3b3) to node [sloped] {$\rho_{1}$} (vb3);
\draw[->, bend right = 0, black, , font=\scriptsize](y3b4) to node [sloped] {$\rho_{1}$} (vb4);
\end{tikzpicture}
\caption{$V^0$}\label{fig:g5b}
\end{figure}
\end{landscape}
\clearpage
\restoregeometry

\clearpage
\newgeometry{margin=2cm}
\begin{landscape}
\begin{figure}

\centering
\begin{tikzpicture}[node distance=2cm]
\path[font = \scriptsize]
(27.58, 9.6) node(bextray2) [outer sep=-1pt]{$y_2$}
(3.69, 2.85) node(ua1) [outer sep=-1pt]{$u$$\otimes$$a$}
(11.96, 5.7) node(va3) [outer sep=-1pt]{$v$$\otimes$$a$}
(19.98, 5.32) node(y2a6up) [outer sep=-1pt]{$\cdot$}
(10.18, 7.6) node(y2b3) [outer sep=-1pt]{$y_2$$\otimes$$b$}
(22.78, 6.08) node(y2b7down) [outer sep=-1pt]{$\cdot$}
(22.58, 3.6) node(bextrau) [outer sep=-1pt]{$u$}
(22.78, 9.12) node(y2b7up) [outer sep=-1pt]{$\cdot$}
(19.6, 14.25) node(x3topb) [outer sep=-1pt]{$x_3$$\otimes$$b$}
(25.58, 9.12) node(y2b8up) [outer sep=-1pt]{$\cdot$}
(8.27, 4.75) node(y3a2) [outer sep=-1pt]{$y_3$$\otimes$$a$}
(5.6, 1.9) node(y1a2) [outer sep=-1pt]{$y_1$$\otimes$$a$}
(18.2, 1.9) node(y1a6) [outer sep=-1pt]{$y_1$$\otimes$$a$}
(14.98, 3.8) node(a4half) [outer sep=-1pt]{$\cdot$}
(14.0, 14.25) node(x3topa) [outer sep=-1pt]{$x_3$$\otimes$$a$}
(19.09, 2.85) node(ua6) [outer sep=-1pt]{$u$$\otimes$$a$}
(11.07, 4.75) node(y3a3) [outer sep=-1pt]{$y_3$$\otimes$$a$}
(12.98, 3.8) node(y2a4) [outer sep=-1pt]{$\cdot$}
(12.98, 7.6) node(y2b4) [outer sep=-1pt]{$\cdot$}
(2.8, 1.9) node(y1a1) [outer sep=-1pt]{$y_1$$\otimes$$a$}
(25.58, 7.6) node(y2b8) [outer sep=-1pt]{$y_2$$\otimes$$b$}
(19.98, 3.8) node(y2a6) [outer sep=-1pt]{$y_2$$\otimes$$a$}
(9.16, 5.7) node(va2) [outer sep=-1pt]{$v$$\otimes$$a$}
(6.49, 2.85) node(ua2) [outer sep=-1pt]{$u$$\otimes$$a$}
(19.98, 7.6) node(y2b6) [outer sep=-1pt]{}
(21.98, 5.8) node(aextray2) [outer sep=-1pt]{$y_2$}
(10.18, 3.8) node(y2a3) [outer sep=-1pt]{$y_2$$\otimes$$a$}
(14.98, 7.6) node(b4half) [outer sep=-1pt]{$\cdot$}
(22.78, 7.6) node(y2b7) [outer sep=-1pt]{$y_2$$\otimes$$b$}
(7.38, 3.8) node(y2a2) [outer sep=-1pt]{$y_2$$\otimes$$a$}
(4.58, 3.8) node(y2a1) [outer sep=-1pt]{$y_2$$\otimes$$a$}
(16.98, -0.2) node(aextrau) [outer sep=-1pt]{$u$}
(17.18, 3.8) node(y2a5) [outer sep=-1pt]{$\cdot$}
(23.8, 5.7) node(y1b8) [outer sep=-1pt]{$y_1$$\otimes$$b$}
(24.69, 6.65) node(ub8) [outer sep=-1pt]{$u$$\otimes$$b$}
;
\draw[->, bend right = 0, red, , font=\scriptsize](a4half) to node [sloped] {$\rho_{23}$} (y2a4);
\draw[->, bend right = 40, green, , font=\scriptsize](aextray2) to node [sloped] {$\rho_{2}$~~~~~~~*} (ua6);
\draw[->, bend right = 0, red, , font=\scriptsize](b4half) to node [sloped] {$\rho_{23}$} (y2b4);
\draw[->, bend right = 40, green, , font=\scriptsize](bextray2) to node [sloped] {$\rho_{2}$~~~~~~~*} (ub8);
\draw[->, bend right = 0, black, , font=\scriptsize](ua1) to node [sloped] {$\rho_{1}$} (y1a1);
\draw[->, bend right = 0, black, , font=\scriptsize](ua1) to node [sloped] {$\rho_{3}$} (y2a1);
\draw[->, bend right = 0, red, , font=\scriptsize](ua1) to node [sloped] {$\iota_0$} (y3a2);
\draw[->, bend right = 0, black, , font=\scriptsize](ua2) to node [sloped] {$\rho_{1}$} (y1a2);
\draw[->, bend right = 0, black, , font=\scriptsize](ua2) to node [sloped] {$\rho_{3}$} (y2a2);
\draw[->, bend right = 0, red, , font=\scriptsize](ua2) to node [sloped] {$\iota_0$} (y3a3);
\draw[->, bend right = 0, black, , font=\scriptsize](ua6) to node [sloped] {$\rho_{1}$} (y1a6);
\draw[->, bend right = 0, black, , font=\scriptsize](ua6) to node [sloped] {$\rho_{3}$} (y2a6);
\draw[->, bend right = 0, black, , font=\scriptsize](ub8) to node [sloped] {$\rho_{1}$} (y1b8);
\draw[->, bend right = 0, black, , font=\scriptsize](ub8) to node [sloped] {$\rho_{3}$} (y2b8);
\draw[->, bend right = -30, blue, , font=\scriptsize](x3topa) to node [sloped] {$\iota_1$} (y1a6);
\draw[->, bend right = 30, blue, , font=\scriptsize](x3topa) to node [sloped] {$\iota_1$} (y2a1);
\draw[->, bend right = -30, blue, , font=\scriptsize](x3topb) to node [sloped] {$\iota_1$} (y1b8);
\draw[->, bend right = 30, blue, , font=\scriptsize](x3topb) to node [sloped] {$\iota_1$} (y2b3);
\draw[->, bend right = -5.13, red, , font=\scriptsize](y1a1) to node [sloped] {$\iota_1$} (va2);
\draw[->, bend right = -5.13, red, , font=\scriptsize](y1a2) to node [sloped] {$\iota_1$} (va3);
\draw[->, bend right = 0, black, , font=\scriptsize](y2a2) to node [sloped] {$\rho_{2}$} (y3a2);
\draw[->, bend right = 0, black, , font=\scriptsize](y2a3) to node [sloped] {$\rho_{2}$} (y3a3);
\draw[->, bend right = 0, black, , font=\scriptsize](y2a3) to node [sloped] {$\iota_1$} (y2b3);
\draw[->, bend right = 0, red, , font=\scriptsize](y2a4) to node [sloped] {$\rho_{23}$} (y2a3);
\draw[->, bend right = 0, red, , font=\scriptsize](y2a6) to node [sloped] {$\rho_{23}$} (y2a5);
\draw[->, bend right = 30, green, , font=\scriptsize](y2a6) to node [sloped] {$\rho_{2}$~~~**} (aextrau);
\draw[->, bend right = 0, black, , font=\scriptsize](y2a6up) to node [sloped] {$\iota_1$} (y2a6);
\draw[->, bend right = 0, red, , font=\scriptsize](y2b4) to node [sloped] {$\rho_{23}$} (y2b3);
\draw[->, bend right = 0, black, , font=\scriptsize](y2b7) to node [sloped] {$\iota_1$} (y2b7down);
\draw[->, bend right = 0, red, , font=\scriptsize](y2b7) to node [sloped] {$\rho_{23}$} (y2b6);
\draw[->, bend right = 0, black, , font=\scriptsize](y2b7up) to node [sloped] {$\iota_1$} (y2b7);
\draw[->, bend right = 0, red, , font=\scriptsize](y2b8) to node [sloped] {$\rho_{23}$} (y2b7);
\draw[->, bend right = 30, green, , font=\scriptsize](y2b8) to node [sloped] {$\rho_{2}$~~~**} (bextrau);
\draw[->, bend right = 0, black, , font=\scriptsize](y2b8up) to node [sloped] {$\iota_1$} (y2b8);
\draw[->, bend right = 0, black, , font=\scriptsize](y3a2) to node [sloped] {$\rho_{1}$} (va2);
\draw[->, bend right = 0, black, , font=\scriptsize](y3a3) to node [sloped] {$\rho_{1}$} (va3);
\end{tikzpicture}
\caption{$V^0$}\label{fig:g5c}
\end{figure}
\end{landscape}
\clearpage
\restoregeometry

\begin{figure}

\begin{subfigure}{0.99\textwidth}
\centerline{
\begin{tikzpicture}[node distance=2cm]
\path[font = \scriptsize]
(8.15, 1.9) node(y2a3) [outer sep=-1pt]{$y_2$$\otimes$$a$}
(20.47, 4.56) node(y2b8up) [outer sep=-1pt]{$\cdot$}
(13.75, 1.9) node(y2a5) [outer sep=-1pt]{$\cdot$}
(13.59, -0.1) node(aextrau) [outer sep=-1pt]{$u$}
(14.48, 6.63) node(x3topb) [outer sep=-1pt]{$x_3$$\otimes$$b$}
(10.39, 1.9) node(y2a4) [outer sep=-1pt]{$\cdot$}
(20.47, 3.8) node(y2b8) [outer sep=-1pt]{$y_2$$\otimes$$b$}
(18.23, 3.8) node(y2b7) [outer sep=-1pt]{}
(10.39, 3.8) node(y2b4) [outer sep=-1pt]{$\cdot$}
(16.48, 6.63) node(y1b8) [outer sep=-1pt]{$y_1$$\otimes$$b$}
(8.15, 3.8) node(y2b3) [outer sep=-1pt]{$y_2$$\otimes$$b$}
(8.8, 7.13) node(y1a6) [outer sep=-1pt]{$y_1$$\otimes$$a$}
(15.99, 1.9) node(y2a6) [outer sep=-1pt]{$y_2$$\otimes$$a$}
(22.07, 4.8) node(bextray2) [outer sep=-1pt]{$y_2$}
(18.48, 6.63) node(ub8) [outer sep=-1pt]{$u$$\otimes$$b$}
(6.8, 7.13) node(x3topa) [outer sep=-1pt]{$x_3$$\otimes$$a$}
(17.59, 2.9) node(aextray2) [outer sep=-1pt]{$y_2$}
(15.99, 2.66) node(y2a6up) [outer sep=-1pt]{$\cdot$}
(11.99, 1.9) node(a4half) [outer sep=-1pt]{$\cdot$}
(3.67, 1.9) node(y2a1) [outer sep=-1pt]{$y_2$$\otimes$$a$}
(18.07, 1.8) node(bextrau) [outer sep=-1pt]{$u$}
(11.99, 3.8) node(b4half) [outer sep=-1pt]{$\cdot$}
(10.8, 7.13) node(ua6) [outer sep=-1pt]{$u$$\otimes$$a$}
(5.91, 1.9) node(y2a2) [outer sep=-1pt]{$y_2$$\otimes$$a$}
;
\draw[->, bend right = 0, red, , font=\scriptsize](a4half) to node [sloped] {$\rho_{23}$} (y2a4);
\draw[->, bend right = 40, green, , font=\scriptsize](aextray2) to node [sloped] {$\rho_{2}$~~~~~~~*} (ua6);
\draw[->, bend right = 0, red, , font=\scriptsize](b4half) to node [sloped] {$\rho_{23}$} (y2b4);
\draw[->, bend right = 40, green, , font=\scriptsize](bextray2) to node [sloped] {$\rho_{2}$~~~~~~~*} (ub8);
\draw[->, bend right = 0, black, , font=\scriptsize](ua6) to node [sloped] {$\rho_{1}$} (y1a6);
\draw[->, bend right = -3.18, black, , font=\scriptsize](ua6) to node [sloped] {$\rho_{3}$} (y2a6);
\draw[->, bend right = 0, black, , font=\scriptsize](ub8) to node [sloped] {$\rho_{1}$} (y1b8);
\draw[->, bend right = 0, black, , font=\scriptsize](ub8) to node [sloped] {$\rho_{3}$} (y2b8);
\draw[->, bend right = -30, blue, , font=\scriptsize](x3topa) to node [sloped] {$\iota_1$} (y1a6);
\draw[->, bend right = 30, blue, , font=\scriptsize](x3topa) to node [sloped] {$\iota_1$} (y2a1);
\draw[->, bend right = -30, blue, , font=\scriptsize](x3topb) to node [sloped] {$\iota_1$} (y1b8);
\draw[->, bend right = 30, blue, , font=\scriptsize](x3topb) to node [sloped] {$\iota_1$} (y2b3);
\draw[->, bend right = 0, red, , font=\scriptsize](y2a2) to node [sloped] {$\rho_{23}$} (y2a1);
\draw[->, bend right = 0, black, , font=\scriptsize](y2a3) to node [sloped] {$\iota_1$} (y2b3);
\draw[->, bend right = 0, red, , font=\scriptsize](y2a3) to node [sloped] {$\rho_{23}$} (y2a2);
\draw[->, bend right = 0, red, , font=\scriptsize](y2a4) to node [sloped] {$\rho_{23}$} (y2a3);
\draw[->, bend right = 0, red, , font=\scriptsize](y2a6) to node [sloped] {$\rho_{23}$} (y2a5);
\draw[->, bend right = 30, green, , font=\scriptsize](y2a6) to node [sloped] {$\rho_{2}$~~~**} (aextrau);
\draw[->, bend right = 0, black, , font=\scriptsize](y2a6up) to node [sloped] {$\iota_1$} (y2a6);
\draw[->, bend right = 0, red, , font=\scriptsize](y2b4) to node [sloped] {$\rho_{23}$} (y2b3);
\draw[->, bend right = 0, red, , font=\scriptsize](y2b8) to node [sloped] {$\rho_{23}$} (y2b7);
\draw[->, bend right = 30, green, , font=\scriptsize](y2b8) to node [sloped] {$\rho_{2}$~~~**} (bextrau);
\draw[->, bend right = 0, black, , font=\scriptsize](y2b8up) to node [sloped] {$\iota_1$} (y2b8);
\end{tikzpicture}
}
\caption{}\label{fig:g5d}
\end{subfigure}

\begin{subfigure}{0.99\textwidth}
\centerline{
\begin{tikzpicture}[node distance=2cm]
\path[font = \scriptsize]
(11.99, 1.9) node(a4half) [outer sep=-1pt]{$\cdot$}
(3.67, 1.9) node(y2a1) [outer sep=-1pt]{$y_2$$\otimes$$a$}
(10.39, 3.8) node(y2b4) [outer sep=-1pt]{$\cdot$}
(17.59, 2.9) node(aextray2) [outer sep=-1pt]{$y_2$}
(20.47, 3.8) node(y2b8) [outer sep=-1pt]{$y_2$$\otimes$$b$}
(8.15, 1.9) node(y2a3) [outer sep=-1pt]{$y_2$$\otimes$$a$}
(15.99, 2.66) node(y2a6up) [outer sep=-1pt]{$\cdot$}
(18.23, 3.8) node(y2b7) [outer sep=-1pt]{}
(13.75, 1.9) node(y2a5) [outer sep=-1pt]{$\cdot$}
(15.99, 1.9) node(y2a6) [outer sep=-1pt]{$y_2$$\otimes$$a$}
(13.59, -0.1) node(aextrau) [outer sep=-1pt]{$u$}
(5.91, 1.9) node(y2a2) [outer sep=-1pt]{$y_2$$\otimes$$a$}
(22.07, 4.8) node(bextray2) [outer sep=-1pt]{$y_2$}
(18.07, 1.8) node(bextrau) [outer sep=-1pt]{$u$}
(20.47, 4.56) node(y2b8up) [outer sep=-1pt]{$\cdot$}
(8.15, 3.8) node(y2b3) [outer sep=-1pt]{$y_2$$\otimes$$b$}
(10.39, 1.9) node(y2a4) [outer sep=-1pt]{$\cdot$}
(10.8, 7.13) node(ua6) [outer sep=-1pt]{$u$$\otimes$$a$}
(18.48, 6.63) node(ub8) [outer sep=-1pt]{$u$$\otimes$$b$}
(11.99, 3.8) node(b4half) [outer sep=-1pt]{$\cdot$}
;
\draw[->, bend right = 0, red, , font=\scriptsize](a4half) to node [sloped] {$\rho_{23}$} (y2a4);
\draw[->, bend right = 40, green, , font=\scriptsize](aextray2) to node [sloped] {$\rho_{2}$~~~~~~~*} (ua6);
\draw[->, bend right = 0, red, , font=\scriptsize](b4half) to node [sloped] {$\rho_{23}$} (y2b4);
\draw[->, bend right = 40, green, , font=\scriptsize](bextray2) to node [sloped] {$\rho_{2}$~~~~~~~*} (ub8);
\draw[->, bend right = -3.18, black, , font=\scriptsize](ua6) to node [sloped] {$\rho_{3}$} (y2a6);
\draw[->, bend right = 0, black, , font=\scriptsize](ua6) to node [sloped] {$\rho_{1}$} (y2a1);
\draw[->, bend right = 0, black, , font=\scriptsize](ub8) to node [sloped] {$\rho_{3}$} (y2b8);
\draw[->, bend right = -4.59, black, , font=\scriptsize](ub8) to node [sloped] {$\rho_{1}$} (y2b3);
\draw[->, bend right = 0, red, , font=\scriptsize](y2a2) to node [sloped] {$\rho_{23}$} (y2a1);
\draw[->, bend right = 0, black, , font=\scriptsize](y2a3) to node [sloped] {$\iota_1$} (y2b3);
\draw[->, bend right = 0, red, , font=\scriptsize](y2a3) to node [sloped] {$\rho_{23}$} (y2a2);
\draw[->, bend right = 0, red, , font=\scriptsize](y2a4) to node [sloped] {$\rho_{23}$} (y2a3);
\draw[->, bend right = 0, red, , font=\scriptsize](y2a6) to node [sloped] {$\rho_{23}$} (y2a5);
\draw[->, bend right = 30, green, , font=\scriptsize](y2a6) to node [sloped] {$\rho_{2}$~~~**} (aextrau);
\draw[->, bend right = 0, black, , font=\scriptsize](y2a6up) to node [sloped] {$\iota_1$} (y2a6);
\draw[->, bend right = 0, red, , font=\scriptsize](y2b4) to node [sloped] {$\rho_{23}$} (y2b3);
\draw[->, bend right = 0, red, , font=\scriptsize](y2b8) to node [sloped] {$\rho_{23}$} (y2b7);
\draw[->, bend right = 30, green, , font=\scriptsize](y2b8) to node [sloped] {$\rho_{2}$~~~**} (bextrau);
\draw[->, bend right = 0, black, , font=\scriptsize](y2b8up) to node [sloped] {$\iota_1$} (y2b8);
\end{tikzpicture}
}
\caption{}\label{fig:g5e}
\end{subfigure}

\caption{$V^0$}\label{fig:g5}
\end{figure}

\FloatBarrier
Now we tackle the most complicated case: a $a\to b$ arrow in $\partial_w$ with length 1, see \cref{fig:g6a}. In this case, $a$ at the left-most position has both the incoming $\rh{3}$ and the outgoing $\rh{2}$. $\rh{123}$ also appears as the length of the arrow is 1. We first assume $b$ has no outgoing length 1 $\partial_z$ arrows. So no downward going black arrows from the second $b$ from the right.

Now we take the tensor, with \cref{fig:g2d} in \cref{sec:rightend} in mind. All potential green arrows marked with $*$ appear again, except $b$'s $y_2$ doesn't have one going out as $b$ has no length one $\partial_z$ arrow. The pink $\rh{123}$ doesn't lead to any new arrows, as argued in \cref{sec:rightend}, see \cref{fig:g6b}. Then, cancel the middle two pairs of reds, yielding only the two desired $\rh{23}$'s. Then, cancel the two $x_1\to v$'s ``for free''. Then, cancel the two $x_2\to y_3$'s ``for free'', see \cref{fig:g6c}. Next, cancel the pair of reds on the left side, again put $u$'s on top. Then cancel the two $x_3\to y_1$'s. Now see \cref{fig:g6d}. It matches $KtD(\fl{C})$ in a similar fashion as \cref{fig:g5}.

\begin{figure}[h]
\centerline{
\begin{tikzpicture}[node distance=2cm]
\path[font = \scriptsize]
(1.8, 1.0) node(a1) [outer sep=-1pt]{$a$}
(12.6, 3.0) node(b5) [outer sep=-1pt]{$b$}
(14.4, 3.0) node(b6) [outer sep=-1pt]{$b$}
(3.6, 1.0) node(a2) [outer sep=-1pt]{$a$}
(5.4, 6.0) node(topa) [outer sep=-1pt]{$a$}
(9.0, 6.0) node(topb) [outer sep=-1pt]{$b$}
(5.4, 3.0) node(b3) [outer sep=-1pt]{}
(5.4, 1.0) node(a3) [outer sep=-1pt]{}
(12.6, 1.0) node(a5) [outer sep=-1pt]{$a$}
(12.6, 3.8) node(b5up) [outer sep=-1pt]{}
(10.8, 3.0) node(b4) [outer sep=-1pt]{}
(10.8, 1.0) node(a4) [outer sep=-1pt]{}
(12.6, 1.8) node(a5up) [outer sep=-1pt]{}
(14.4, 3.8) node(b6up) [outer sep=-1pt]{}
(3.6, 3.0) node(b2) [outer sep=-1pt]{$b$}
;
\draw[->, bend right = 0, red, , font=\scriptsize](a1) to node [sloped] {$\rho_{23}$} (a2);
\draw[->, bend right = 10, green, , font=\scriptsize](a1) to node [sloped] {$\rho_{2}$} (topb);
\draw[->, bend right = 0, red, , font=\scriptsize](a2) to node [sloped] {$\rho_{23}$} (a3);
\draw[->, bend right = 0, black, , font=\scriptsize](a2) to node [sloped] {$\iota_1$} (b2);
\draw[->, bend right = 0, red, , font=\scriptsize](a4) to node [sloped] {$\rho_{23}$} (a5);
\draw[->, bend right = 0, black, , font=\scriptsize](a5up) to node [sloped] {$\iota_1$} (a5);
\draw[->, bend right = 0, red, , font=\scriptsize](b2) to node [sloped] {$\rho_{23}$} (b3);
\draw[->, bend right = 0, red, , font=\scriptsize](b4) to node [sloped] {$\rho_{23}$} (b5);
\draw[->, bend right = 0, red, , font=\scriptsize](b5) to node [sloped] {$\rho_{23}$} (b6);
\draw[->, bend right = 0, black, , font=\scriptsize](b5up) to node [sloped] {$\iota_1$} (b5);
\draw[->, bend right = 0, black, , font=\scriptsize](b6up) to node [sloped] {$\iota_1$} (b6);
\draw[->, bend right = 10, blue, , font=\scriptsize](topa) to node [sloped] {$\rho_{3}$} (a1);
\draw[->, bend right = -3.38, blue, , font=\scriptsize](topa) to node [sloped] {$\rho_{1}$} (a5);
\draw[->, bend right = -7.4, pink, , font=\scriptsize](topa) to node [sloped] {$\rho_{123}$} (b6);
\draw[->, bend right = 10, blue, , font=\scriptsize](topb) to node [sloped] {$\rho_{3}$} (b2);
\draw[->, bend right = -6.74, blue, , font=\scriptsize](topb) to node [sloped] {$\rho_{1}$} (b6);
\end{tikzpicture}
}
\caption{$V^0$}\label{fig:g6a}
\end{figure}

\clearpage
\newgeometry{margin=2cm}
\begin{landscape}
\begin{figure}
\centering
\begin{tikzpicture}[node distance=2cm]
\path[font = \scriptsize]
(15.23, 7.54) node(y2b4) [outer sep=-1pt]{$\cdot$}
(21.53, 7.54) node(y2b6) [outer sep=-1pt]{$y_2$$\otimes$$b$}
(11.04, 8.39) node(y3b3) [outer sep=-1pt]{$\cdot$}
(20.38, 4.64) node(aextray2) [outer sep=-1pt]{$y_2$}
(6.83, 2.8) node(ua2) [outer sep=-1pt]{$u$$\otimes$$a$}
(10.51, 15.93) node(x1topa) [outer sep=-1pt]{$x_1$$\otimes$$a$}
(3.15, 1.95) node(y1a1) [outer sep=-1pt]{$y_1$$\otimes$$a$}
(12.51, 7.54) node(b3half) [outer sep=-1pt]{$\cdot$}
(15.38, -0.36) node(aextrau) [outer sep=-1pt]{$u$}
(10.51, 7.54) node(y2b3) [outer sep=-1pt]{$\cdot$}
(8.42, 9.24) node(vb2) [outer sep=-1pt]{$v$$\otimes$$b$}
(12.51, 3.64) node(a3half) [outer sep=-1pt]{$\cdot$}
(17.85, 2.8) node(ua5) [outer sep=-1pt]{$u$$\otimes$$a$}
(15.75, 14.24) node(x3topb) [outer sep=-1pt]{$x_3$$\otimes$$b$}
(6.3, 5.85) node(y1b2) [outer sep=-1pt]{$y_1$$\otimes$$b$}
(18.38, 4.81) node(y2a5up) [outer sep=-1pt]{$\cdot$}
(10.51, 3.64) node(y2a3) [outer sep=-1pt]{$\cdot$}
(18.38, 7.54) node(y2b5) [outer sep=-1pt]{$y_2$$\otimes$$b$}
(23.53, 8.54) node(bextray2) [outer sep=-1pt]{$y_2$}
(11.57, 9.24) node(vb3) [outer sep=-1pt]{$\cdot$}
(21.0, 6.7) node(ub6) [outer sep=-1pt]{$u$$\otimes$$b$}
(4.21, 3.64) node(y2a1) [outer sep=-1pt]{$y_2$$\otimes$$a$}
(11.04, 4.49) node(y3a3) [outer sep=-1pt]{$\cdot$}
(11.57, 5.34) node(va3) [outer sep=-1pt]{$\cdot$}
(7.36, 3.64) node(y2a2) [outer sep=-1pt]{$y_2$$\otimes$$a$}
(7.89, 8.39) node(y3b2) [outer sep=-1pt]{$y_3$$\otimes$$b$}
(20.48, 5.85) node(y1b6) [outer sep=-1pt]{$y_1$$\otimes$$b$}
(18.38, 8.71) node(y2b5up) [outer sep=-1pt]{$\cdot$}
(18.38, 3.64) node(y2a5) [outer sep=-1pt]{$y_2$$\otimes$$a$}
(7.89, 4.49) node(y3a2) [outer sep=-1pt]{$y_3$$\otimes$$a$}
(9.98, 15.09) node(x2topa) [outer sep=-1pt]{$x_2$$\otimes$$a$}
(21.53, 8.71) node(y2b6up) [outer sep=-1pt]{$\cdot$}
(7.36, 7.54) node(y2b2) [outer sep=-1pt]{$y_2$$\otimes$$b$}
(6.83, 6.7) node(ub2) [outer sep=-1pt]{$u$$\otimes$$b$}
(15.23, 3.64) node(y2a4) [outer sep=-1pt]{$\cdot$}
(8.42, 5.34) node(va2) [outer sep=-1pt]{$v$$\otimes$$a$}
(4.74, 4.49) node(y3a1) [outer sep=-1pt]{$y_3$$\otimes$$a$}
(5.27, 5.34) node(va1) [outer sep=-1pt]{$v$$\otimes$$a$}
(9.45, 14.24) node(x3topa) [outer sep=-1pt]{$x_3$$\otimes$$a$}
(6.3, 1.95) node(y1a2) [outer sep=-1pt]{$y_1$$\otimes$$a$}
(16.28, 15.09) node(x2topb) [outer sep=-1pt]{$x_2$$\otimes$$b$}
(16.81, 15.93) node(x1topb) [outer sep=-1pt]{$x_1$$\otimes$$b$}
(3.68, 2.8) node(ua1) [outer sep=-1pt]{$u$$\otimes$$a$}
(17.33, 1.95) node(y1a5) [outer sep=-1pt]{$y_1$$\otimes$$a$}
;
\draw[->, bend right = 0, red, , font=\scriptsize](a3half) to node [sloped] {$\rho_{23}$} (y2a3);
\draw[->, bend right = 40, green, , font=\scriptsize](aextray2) to node [sloped] {$\rho_{2}$~~~*} (ua5);
\draw[->, bend right = 0, red, , font=\scriptsize](b3half) to node [sloped] {$\rho_{23}$} (y2b3);
\draw[->, bend right = 40, green, , font=\scriptsize](bextray2) to node [sloped] {$\rho_{2}$~~~*} (ub6);
\draw[->, bend right = 0, black, , font=\scriptsize](ua1) to node [sloped] {$\rho_{1}$} (y1a1);
\draw[->, bend right = 0, black, , font=\scriptsize](ua1) to node [sloped] {$\rho_{3}$} (y2a1);
\draw[->, bend right = 0, red, , font=\scriptsize](ua1) to node [sloped] {$\iota_0$} (y3a2);
\draw[->, bend right = -40, green, , font=\scriptsize](ua1) to node [sloped] {$\iota_0$} (x2topb);
\draw[->, bend right = 0, black, , font=\scriptsize](ua2) to node [sloped] {$\rho_{1}$} (y1a2);
\draw[->, bend right = 0, black, , font=\scriptsize](ua2) to node [sloped] {$\rho_{3}$} (y2a2);
\draw[->, bend right = 0, red, , font=\scriptsize](ua2) to node [sloped] {$\iota_0$} (y3a3);
\draw[->, bend right = 0, black, , font=\scriptsize](ua2) to node [sloped] {$\iota_0$} (ub2);
\draw[->, bend right = 0, black, , font=\scriptsize](ua5) to node [sloped] {$\rho_{1}$} (y1a5);
\draw[->, bend right = 0, black, , font=\scriptsize](ua5) to node [sloped] {$\rho_{3}$} (y2a5);
\draw[->, bend right = 0, black, , font=\scriptsize](ub2) to node [sloped] {$\rho_{1}$} (y1b2);
\draw[->, bend right = 0, black, , font=\scriptsize](ub2) to node [sloped] {$\rho_{3}$} (y2b2);
\draw[->, bend right = 0, red, , font=\scriptsize](ub2) to node [sloped] {$\iota_0$} (y3b3);
\draw[->, bend right = 0, black, , font=\scriptsize](ub6) to node [sloped] {$\rho_{1}$} (y1b6);
\draw[->, bend right = 0, black, , font=\scriptsize](ub6) to node [sloped] {$\rho_{3}$} (y2b6);
\draw[->, bend right = 0, black, , font=\scriptsize](va2) to node [sloped] {$\iota_1$} (vb2);
\draw[->, bend right = 0, black, , font=\scriptsize](va3) to node [sloped] {$\iota_1$} (vb3);
\draw[->, bend right = 40, blue, , font=\scriptsize](x1topa) to node [sloped] {$\iota_1$} (va1);
\draw[->, bend right = 40, blue, , font=\scriptsize](x1topb) to node [sloped] {$\iota_1$} (vb2);
\draw[->, bend right = 0, black, , font=\scriptsize](x2topa) to node [sloped] {$\rho_{1}$} (x1topa);
\draw[->, bend right = 40, blue, , font=\scriptsize](x2topa) to node [sloped] {$\iota_0$} (y3a1);
\draw[->, bend right = 0, black, , font=\scriptsize](x2topb) to node [sloped] {$\rho_{1}$} (x1topb);
\draw[->, bend right = 40, blue, , font=\scriptsize](x2topb) to node [sloped] {$\iota_0$} (y3b2);
\draw[->, bend right = -40, blue, , font=\scriptsize](x3topa) to node [sloped] {$\iota_1$} (y1a5);
\draw[->, bend right = 0, black, , font=\scriptsize](x3topa) to node [sloped] {$\rho_{2}$} (x2topa);
\draw[->, bend right = 40, blue, , font=\scriptsize](x3topa) to node [sloped] {$\iota_1$} (y2a1);
\draw[->, bend right = -40, blue, , font=\scriptsize](x3topb) to node [sloped] {$\iota_1$} (y1b6);
\draw[->, bend right = 0, black, , font=\scriptsize](x3topb) to node [sloped] {$\rho_{2}$} (x2topb);
\draw[->, bend right = 40, blue, , font=\scriptsize](x3topb) to node [sloped] {$\iota_1$} (y2b2);
\draw[->, bend right = -3.56, red, , font=\scriptsize](y1a1) to node [sloped] {$\iota_1$} (va2);
\draw[->, bend right = -40, green, , font=\scriptsize](y1a1) to node [sloped] {$\iota_1$} (x1topb);
\draw[->, bend right = -3.56, red, , font=\scriptsize](y1a2) to node [sloped] {$\iota_1$} (va3);
\draw[->, bend right = 0, black, , font=\scriptsize](y1a2) to node [sloped] {$\iota_1$} (y1b2);
\draw[->, bend right = -3.56, red, , font=\scriptsize](y1b2) to node [sloped] {$\iota_1$} (vb3);
\draw[->, bend right = 0, black, , font=\scriptsize](y2a1) to node [sloped] {$\rho_{2}$} (y3a1);
\draw[->, bend right = 0, black, , font=\scriptsize](y2a2) to node [sloped] {$\rho_{2}$} (y3a2);
\draw[->, bend right = 0, black, , font=\scriptsize](y2a2) to node [sloped] {$\iota_1$} (y2b2);
\draw[->, bend right = 0, black, , font=\scriptsize](y2a3) to node [sloped] {$\rho_{2}$} (y3a3);
\draw[->, bend right = 0, black, , font=\scriptsize](y2a3) to node [sloped] {$\iota_1$} (y2b3);
\draw[->, bend right = 0, red, , font=\scriptsize](y2a5) to node [sloped] {$\rho_{23}$} (y2a4);
\draw[->, bend right = 30, green, , font=\scriptsize](y2a5) to node [sloped] {$\rho_{2}$~~~*} (aextrau);
\draw[->, bend right = 0, black, , font=\scriptsize](y2a5up) to node [sloped] {$\iota_1$} (y2a5);
\draw[->, bend right = 0, black, , font=\scriptsize](y2b2) to node [sloped] {$\rho_{2}$} (y3b2);
\draw[->, bend right = 0, black, , font=\scriptsize](y2b3) to node [sloped] {$\rho_{2}$} (y3b3);
\draw[->, bend right = 0, red, , font=\scriptsize](y2b5) to node [sloped] {$\rho_{23}$} (y2b4);
\draw[->, bend right = 0, black, , font=\scriptsize](y2b5up) to node [sloped] {$\iota_1$} (y2b5);
\draw[->, bend right = 0, red, , font=\scriptsize](y2b6) to node [sloped] {$\rho_{23}$} (y2b5);
\draw[->, bend right = 0, black, , font=\scriptsize](y2b6up) to node [sloped] {$\iota_1$} (y2b6);
\draw[->, bend right = 0, black, , font=\scriptsize](y3a1) to node [sloped] {$\rho_{1}$} (va1);
\draw[->, bend right = 0, black, , font=\scriptsize](y3a2) to node [sloped] {$\rho_{1}$} (va2);
\draw[->, bend right = 0, black, , font=\scriptsize](y3a2) to node [sloped] {$\iota_0$} (y3b2);
\draw[->, bend right = 0, black, , font=\scriptsize](y3a3) to node [sloped] {$\rho_{1}$} (va3);
\draw[->, bend right = 0, black, , font=\scriptsize](y3a3) to node [sloped] {$\iota_0$} (y3b3);
\draw[->, bend right = 0, black, , font=\scriptsize](y3b2) to node [sloped] {$\rho_{1}$} (vb2);
\draw[->, bend right = 0, black, , font=\scriptsize](y3b3) to node [sloped] {$\rho_{1}$} (vb3);
\end{tikzpicture}
\caption{$V^0$}\label{fig:g6b}
\end{figure}
\end{landscape}
\clearpage
\restoregeometry

\clearpage
\newgeometry{margin=2cm}
\begin{landscape}
\begin{figure}

\centering
\begin{tikzpicture}[node distance=2cm]
\path[font = \scriptsize]
(9.45, 14.24) node(x3topa) [outer sep=-1pt]{$x_3$$\otimes$$a$}
(17.85, 2.8) node(ua5) [outer sep=-1pt]{$u$$\otimes$$a$}
(21.53, 8.71) node(y2b6up) [outer sep=-1pt]{$\cdot$}
(7.89, 4.49) node(y3a2) [outer sep=-1pt]{$y_3$$\otimes$$a$}
(12.51, 7.54) node(b3half) [outer sep=-1pt]{$\cdot$}
(23.53, 8.54) node(bextray2) [outer sep=-1pt]{$y_2$}
(7.36, 7.54) node(y2b2) [outer sep=-1pt]{$y_2$$\otimes$$b$}
(21.0, 6.7) node(ub6) [outer sep=-1pt]{$u$$\otimes$$b$}
(15.23, 3.64) node(y2a4) [outer sep=-1pt]{$\cdot$}
(15.38, -0.36) node(aextrau) [outer sep=-1pt]{$u$}
(3.15, 1.95) node(y1a1) [outer sep=-1pt]{$y_1$$\otimes$$a$}
(12.51, 3.64) node(a3half) [outer sep=-1pt]{$\cdot$}
(10.51, 3.64) node(y2a3) [outer sep=-1pt]{$\cdot$}
(3.68, 2.8) node(ua1) [outer sep=-1pt]{$u$$\otimes$$a$}
(15.23, 7.54) node(y2b4) [outer sep=-1pt]{$\cdot$}
(18.38, 7.54) node(y2b5) [outer sep=-1pt]{$y_2$$\otimes$$b$}
(20.48, 5.85) node(y1b6) [outer sep=-1pt]{$y_1$$\otimes$$b$}
(21.53, 7.54) node(y2b6) [outer sep=-1pt]{$y_2$$\otimes$$b$}
(15.75, 14.24) node(x3topb) [outer sep=-1pt]{$x_3$$\otimes$$b$}
(20.38, 4.64) node(aextray2) [outer sep=-1pt]{$y_2$}
(7.36, 3.64) node(y2a2) [outer sep=-1pt]{$y_2$$\otimes$$a$}
(4.21, 3.64) node(y2a1) [outer sep=-1pt]{$y_2$$\otimes$$a$}
(10.51, 7.54) node(y2b3) [outer sep=-1pt]{$\cdot$}
(18.38, 4.81) node(y2a5up) [outer sep=-1pt]{$\cdot$}
(17.33, 1.95) node(y1a5) [outer sep=-1pt]{$y_1$$\otimes$$a$}
(18.38, 8.71) node(y2b5up) [outer sep=-1pt]{$\cdot$}
(18.38, 3.64) node(y2a5) [outer sep=-1pt]{$y_2$$\otimes$$a$}
(8.42, 5.34) node(va2) [outer sep=-1pt]{$v$$\otimes$$a$}
;
\draw[->, bend right = 0, red, , font=\scriptsize](a3half) to node [sloped] {$\rho_{23}$} (y2a3);
\draw[->, bend right = 40, green, , font=\scriptsize](aextray2) to node [sloped] {$\rho_{2}$~~~*} (ua5);
\draw[->, bend right = 0, red, , font=\scriptsize](b3half) to node [sloped] {$\rho_{23}$} (y2b3);
\draw[->, bend right = 40, green, , font=\scriptsize](bextray2) to node [sloped] {$\rho_{2}$~~~*} (ub6);
\draw[->, bend right = 0, black, , font=\scriptsize](ua1) to node [sloped] {$\rho_{1}$} (y1a1);
\draw[->, bend right = 0, black, , font=\scriptsize](ua1) to node [sloped] {$\rho_{3}$} (y2a1);
\draw[->, bend right = 0, red, , font=\scriptsize](ua1) to node [sloped] {$\iota_0$} (y3a2);
\draw[->, bend right = 0, black, , font=\scriptsize](ua5) to node [sloped] {$\rho_{1}$} (y1a5);
\draw[->, bend right = 0, black, , font=\scriptsize](ua5) to node [sloped] {$\rho_{3}$} (y2a5);
\draw[->, bend right = 0, black, , font=\scriptsize](ub6) to node [sloped] {$\rho_{1}$} (y1b6);
\draw[->, bend right = 0, black, , font=\scriptsize](ub6) to node [sloped] {$\rho_{3}$} (y2b6);
\draw[->, bend right = -40, blue, , font=\scriptsize](x3topa) to node [sloped] {$\iota_1$} (y1a5);
\draw[->, bend right = 40, blue, , font=\scriptsize](x3topa) to node [sloped] {$\iota_1$} (y2a1);
\draw[->, bend right = -40, blue, , font=\scriptsize](x3topb) to node [sloped] {$\iota_1$} (y1b6);
\draw[->, bend right = 40, blue, , font=\scriptsize](x3topb) to node [sloped] {$\iota_1$} (y2b2);
\draw[->, bend right = -3.56, red, , font=\scriptsize](y1a1) to node [sloped] {$\iota_1$} (va2);
\draw[->, bend right = 0, black, , font=\scriptsize](y2a2) to node [sloped] {$\rho_{2}$} (y3a2);
\draw[->, bend right = 0, black, , font=\scriptsize](y2a2) to node [sloped] {$\iota_1$} (y2b2);
\draw[->, bend right = 0, black, , font=\scriptsize](y2a3) to node [sloped] {$\iota_1$} (y2b3);
\draw[->, bend right = 0, red, , font=\scriptsize](y2a3) to node [sloped] {$\rho_{23}$} (y2a2);
\draw[->, bend right = 0, red, , font=\scriptsize](y2a5) to node [sloped] {$\rho_{23}$} (y2a4);
\draw[->, bend right = 30, green, , font=\scriptsize](y2a5) to node [sloped] {$\rho_{2}$~~~*} (aextrau);
\draw[->, bend right = 0, black, , font=\scriptsize](y2a5up) to node [sloped] {$\iota_1$} (y2a5);
\draw[->, bend right = 0, red, , font=\scriptsize](y2b3) to node [sloped] {$\rho_{23}$} (y2b2);
\draw[->, bend right = 0, red, , font=\scriptsize](y2b5) to node [sloped] {$\rho_{23}$} (y2b4);
\draw[->, bend right = 0, black, , font=\scriptsize](y2b5up) to node [sloped] {$\iota_1$} (y2b5);
\draw[->, bend right = 0, red, , font=\scriptsize](y2b6) to node [sloped] {$\rho_{23}$} (y2b5);
\draw[->, bend right = 0, black, , font=\scriptsize](y2b6up) to node [sloped] {$\iota_1$} (y2b6);
\draw[->, bend right = 0, black, , font=\scriptsize](y3a2) to node [sloped] {$\rho_{1}$} (va2);
\end{tikzpicture}
\caption{$V^0$}\label{fig:g6c}

\end{figure}
\end{landscape}
\clearpage
\restoregeometry

\begin{figure}[ht]
\centerline{
\begin{tikzpicture}[node distance=2cm]
\path[font = \scriptsize]
(5.89, 1.82) node(y2a2) [outer sep=-1pt]{$y_2$$\otimes$$a$}
(18.83, 4.27) node(bextray2) [outer sep=-1pt]{$y_2$}
(10.01, 1.82) node(a3half) [outer sep=-1pt]{$\cdot$}
(15.4, 6.62) node(ub6) [outer sep=-1pt]{$u$$\otimes$$b$}
(14.71, 3.77) node(y2b5) [outer sep=-1pt]{$y_2$$\otimes$$b$}
(16.31, 2.32) node(aextray2) [outer sep=-1pt]{$y_2$}
(10.01, 3.77) node(b3half) [outer sep=-1pt]{$\cdot$}
(7.96, 7.12) node(ua5) [outer sep=-1pt]{$u$$\otimes$$a$}
(12.19, 3.77) node(y2b4) [outer sep=-1pt]{$\cdot$}
(8.41, 1.82) node(y2a3) [outer sep=-1pt]{$\cdot$}
(17.23, 4.36) node(y2b6up) [outer sep=-1pt]{$\cdot$}
(14.71, 2.41) node(y2a5up) [outer sep=-1pt]{$\cdot$}
(8.41, 3.77) node(y2b3) [outer sep=-1pt]{$\cdot$}
(5.89, 3.77) node(y2b2) [outer sep=-1pt]{$y_2$$\otimes$$b$}
(17.23, 3.77) node(y2b6) [outer sep=-1pt]{$y_2$$\otimes$$b$}
(14.71, 1.82) node(y2a5) [outer sep=-1pt]{$y_2$$\otimes$$a$}
(3.37, 1.82) node(y2a1) [outer sep=-1pt]{$y_2$$\otimes$$a$}
(12.31, -0.18) node(aextrau) [outer sep=-1pt]{$u$}
(14.71, 4.36) node(y2b5up) [outer sep=-1pt]{$\cdot$}
(12.19, 1.82) node(y2a4) [outer sep=-1pt]{$\cdot$}
;
\draw[->, bend right = 0, red, , font=\scriptsize](a3half) to node [sloped] {$\rho_{23}$} (y2a3);
\draw[->, bend right = 40, green, , font=\scriptsize](aextray2) to node [sloped] {$\rho_{2}$~~~*} (ua5);
\draw[->, bend right = 0, red, , font=\scriptsize](b3half) to node [sloped] {$\rho_{23}$} (y2b3);
\draw[->, bend right = 40, green, , font=\scriptsize](bextray2) to node [sloped] {$\rho_{2}$~~~*} (ub6);
\draw[->, bend right = -8.79, black, , font=\scriptsize](ua5) to node [sloped] {$\rho_{3}$} (y2a5);
\draw[->, bend right = -1.85, black, , font=\scriptsize](ua5) to node [sloped] {$\rho_{1}$} (y2a1);
\draw[->, bend right = -1.53, black, , font=\scriptsize](ub6) to node [sloped] {$\rho_{3}$} (y2b6);
\draw[->, bend right = -3.18, black, , font=\scriptsize](ub6) to node [sloped] {$\rho_{1}$} (y2b2);
\draw[->, bend right = 0, black, , font=\scriptsize](y2a2) to node [sloped] {$\iota_1$} (y2b2);
\draw[->, bend right = 0, red, , font=\scriptsize](y2a2) to node [sloped] {$\rho_{23}$} (y2a1);
\draw[->, bend right = 0, black, , font=\scriptsize](y2a3) to node [sloped] {$\iota_1$} (y2b3);
\draw[->, bend right = 0, red, , font=\scriptsize](y2a3) to node [sloped] {$\rho_{23}$} (y2a2);
\draw[->, bend right = 0, red, , font=\scriptsize](y2a5) to node [sloped] {$\rho_{23}$} (y2a4);
\draw[->, bend right = 30, green, , font=\scriptsize](y2a5) to node [sloped] {$\rho_{2}$~~~*} (aextrau);
\draw[->, bend right = 0, black, , font=\scriptsize](y2a5up) to node [sloped] {$\iota_1$} (y2a5);
\draw[->, bend right = 0, red, , font=\scriptsize](y2b3) to node [sloped] {$\rho_{23}$} (y2b2);
\draw[->, bend right = 0, red, , font=\scriptsize](y2b5) to node [sloped] {$\rho_{23}$} (y2b4);
\draw[->, bend right = 0, black, , font=\scriptsize](y2b5up) to node [sloped] {$\iota_1$} (y2b5);
\draw[->, bend right = 0, red, , font=\scriptsize](y2b6) to node [sloped] {$\rho_{23}$} (y2b5);
\draw[->, bend right = 0, black, , font=\scriptsize](y2b6up) to node [sloped] {$\iota_1$} (y2b6);
\end{tikzpicture}
}
\caption{$V^0$}\label{fig:g6d}
\end{figure}

Now we consider the possibility of an outgoing length 1 arrows $b\to c$ in $\partial_z$. ($c$ can not be the same as $a$, because that $d^2=0$ and horizontal simplifiedness would force $C$ to be infinitely generated.) 
To ensure $d^2=0$, there must be a generator $d$ in one Alexander filtration level below $a$ to form this 1 by 1 box in $C$, as in \cref{fig:g7a}. We do not exclude the possibility of other incoming or outgoing vertical arrows at the generators. \cref{fig:g7b} shows the corresponding part in $KtD(C)$. \cref{fig:g7c} shows the tensor product before any cancellation are applied. The open-ended purple arrows are potential arrows. Now we apply the cancellation described in \cref{sec:rightend} on the right side. Then make some ``free'' cancelations: cancel all blue $x_1\to v$'s and then all $x_2\to y_3$'s, which generate no new arrows. Next, we cancel all the pairs of reds on the left-hand side. All those cancelations have been previously described in details, so we just show the final result in \cref{fig:g7d}. 

A few things to point out here: Arrows marked with $*$ (eight green and one pink) are those dangling arrows in \cref{fig:g2d}. They exist because of the potential $\partial_z$ differentials connecting to $a,b,c,d$ (purple ones in \cref{fig:g7a}). The two green arrows marked by $**$ are the same kind of dangling arrow whose two ends we finally see in one diagram, as they come from the arrows $b\to c$ and $a\to d$ in $\partial_z$. The pink arrow marked with $**$ is also a manifestation of of the pink dangling arrow in \cref{fig:g2d}.

Next, we move $y_1$'s and $u$'s to the top, and
cancel $x_3 \to y_1$ at $b$ in $V^0$, then that arrow at $a$ in $V^0$, then that arrow at $c$ in $V^0$, then that arrow at $d$ in $V^0$, see \cref{fig:g7f}. Now everything matches to $KtD(\fl{C})$ as described above, except we are missing two $\rh{123}$'s. For convenience, we label the $u$'s in $V^0$ with the generator it corresponds: $u^a, u^b, u^c, u^d$ and label the $y_2$'s on the left hand side with the row it is in and the its distance to the leftmost position: $y_2^{d,0}, y_2^{c,1}$, etc. $y_2$'s on the right hand side are similarly labeled as $y_2^{d,0,r}, y_2^{c,1,r}$, etc. Per \cref{localchange}, there must be two $\rh{123}$'s from $u^a$ to $y_2^{d,0}$ at the left-most position of $d$'s row on the left, and from $u^b$ to $y_2^{c,0}$. We also have an extra arrow from $u^a$ to $y_2^{c,0}$ (partially pink in the diagram). 

To fix this, we do a change of basis: $<u^a, u^b, y_2^{d,1}, y_2^{c,1}>\to <u^a+\rh{1}y_2^{d,1}, u^b+\rh{1}y_2^{c,1}, y_2^{d,1}, y_2^{c,1}>$, where $d$ and $c$ are circled and boxed respectively on the diagram. We examine the effect of this change of basis. $u^a$: $\rh{2}u^a =\rh{2}(u^a+\rh{1}y_2^{d,1})$, so its incoming dangling $\rh{2}$ arrow is unaffected. $\partial u^a=\rh{1}y_2^{a,0}+\rh{1}y_2^{c,0}+\rh{3}y_2^{a,0,r}$ and $\partial (u^a+\rh{1}y_2^{d,1})=\rh{1}y_2^{a,0}+\rh{1}y_2^{c,0}+\rh{3}y_2^{a,0,r}+\rh{1}(\rh{23}y_2^{d,0}+y_2^{c,0})=\rh{1}y_2^{a,0}+\rh{3}y_2^{a,0,r}+\rh{123}y_2^{d,0}$, which is exactly what we need. For $u^b$, its incoming dangling $\rh{2}$ arrow is also unaffected. $\partial u^b=\rh{1}y_2^{b,0}+\rh{3}y_2^{b,0,r}$ and $\partial (u^b+\rh{1}y_2^{c,1})= \rh{1}y_2^{b,0}+\rh{3}y_2^{b,0,r}+\rh{1}\rh{23}y_2^{c,0}=\rh{1}y_2^{b,0}+\rh{3}y_2^{b,0,r}+\rh{123}y_2^{c,0}$, which is also what we need. See \cref{fig:g7g} for the final product. Note the dangling partially pink arrow will function as the extra arrow $**$ in \cref{fig:g7f} for some other 1 by 1 box starting at $d$, should there exist one.

\begin{figure}[ht]
\centerline{

}
\caption{$V^0$}\label{fig:g7a}
\end{figure}

\begin{figure}[ht]
\centerline{
%
}
\caption{$V^0$}\label{fig:g7b}
\end{figure}

\clearpage
\newgeometry{margin=2cm}
\begin{landscape}
\begin{figure}
\centering
%
\caption{$V^0$}\label{fig:g7c}
\end{figure}
\end{landscape}
\clearpage
\restoregeometry

\clearpage
\newgeometry{margin=2cm}
\begin{landscape}
\begin{figure}
\centering
%
\caption{$V^0$}\label{fig:g7d}
\end{figure}
\end{landscape}
\clearpage
\restoregeometry

\clearpage
\newgeometry{margin=2cm}
\begin{landscape}
\begin{figure}
\centering
\begin{tikzpicture}[node distance=2cm]
\path[font = \scriptsize]
(19.0, 7.82) node(extra_y1) [outer sep=-2pt]{$``y_1"$}
(8.08, 9.02) node(y2a2) [outer sep=-2pt]{$y_2^{0}$$\otimes$$a$}
(26.8, 12.32) node(y2b8) [outer sep=-2pt]{$y_2^{0,r}\otimes b$}
(11.68, 6.82) node(y2c3) [outer sep=-2pt]{$y_2^{1}$$\otimes$$c$}
(11.68, 3.52) node(y2d3) [outer sep=-2pt]{$y_2^{2}$$\otimes$$d$}
(19.6, 3.52) node(y2d6) [outer sep=-2pt]{$y_2$$\otimes$$d$}
(10.8, 17.38) node(ud6) [outer sep=-2pt]{$u$$\otimes$$d$}
(22.6, 15.88) node(u_extra_b) [outer sep=-2pt]{$\cdot$}
(23.2, 7.48) node(y2a7up) [outer sep=-2pt]{$\cdot$}
(19.4, 2.22) node(y2_extra_d) [outer sep=-2pt]{$\cdot$}
(19.6, 12.98) node(y2b6up) [outer sep=-2pt]{$\cdot$}
(19.6, 8.36) node(y2c6down) [outer sep=-2pt]{$\cdot$}
(19.6, 4.18) node(y2d6up) [outer sep=-2pt]{$\cdot$}
(19.6, 6.16) node(y2a6down) [outer sep=-2pt]{$\cdot$}
(14.2, 12.32) node(y2b4) [outer sep=-2pt]{$\cdot$}
(23.2, 9.68) node(y2c7up) [outer sep=-2pt]{$\cdot$}
(19.6, 9.02) node(y2c6) [outer sep=-2pt]{$y_2$$\otimes$$c$}
(23.2, 9.02) node(y2c7) [outer sep=-2pt]{$y_2$$\otimes$$c$}
(11.68, 9.02) node(y2a3) [outer sep=-2pt]{$y_2^{1}$$\otimes$$a$}
(19.6, 9.68) node(y2c6up) [outer sep=-2pt]{$\cdot$}
(4.48, 3.52) node(y2d1) [outer sep=-2pt]{$y_2^{0}$$\otimes$$d$}
(17.2, 16.98) node(u_extra_a) [outer sep=-2pt]{$\cdot$}
(8.08, 6.82) node(y2c2) [outer sep=-2pt]{$y_2^{0}$$\otimes$$c$}
(23.2, 12.32) node(y2b7) [outer sep=-2pt]{$y_2$$\otimes$$b$}
(8.08, 3.52) node(y2d2) [outer sep=-2pt]{$y_2^{1}$$\otimes$$d$}
(21.6, 17.38) node(ub8) [outer sep=-2pt]{$u$$\otimes$$b$}
(16.2, 16.28) node(uc7) [outer sep=-2pt]{$u$$\otimes$$c$}
(19.6, 11.66) node(y2b6down) [outer sep=-2pt]{$\cdot$}
(23.2, 12.98) node(y2b7up) [outer sep=-2pt]{$\cdot$}
(17.08, 3.52) node(y2d5) [outer sep=-2pt]{$\cdot$}
(23.0, 5.52) node(y2_extra_a) [outer sep=-2pt]{$\cdot$}
(17.08, 6.82) node(y2a5) [outer sep=-2pt]{$\cdot$}
(16.2, 18.48) node(ua7) [outer sep=-2pt]{$u$$\otimes$$a$}
(11.8, 15.88) node(u_extra_d) [outer sep=-2pt]{$\cdot$}
(26.6, 11.02) node(y2_extra_b) [outer sep=-2pt]{$\cdot$}
(14.2, 6.82) node(y2c4) [outer sep=-2pt]{$\cdot$}
(11.68, 12.32) node(y2b3) [outer sep=-2pt]{$y_2^{0}$$\otimes$$b$}
(19.6, 6.82) node(y2a6) [outer sep=-2pt]{$y_2$$\otimes$$a$}
(14.2, 9.02) node(y2a4) [outer sep=-2pt]{$\cdot$}
(19.6, 12.32) node(y2b6) [outer sep=-2pt]{$y_2$$\otimes$$b$}
(23.2, 11.66) node(y2b7down) [outer sep=-2pt]{$\cdot$}
(17.2, 14.78) node(u_extra_c) [outer sep=-2pt]{$\cdot$}
(26.8, 12.98) node(y2b8up) [outer sep=-2pt]{$\cdot$}
(23.0, 7.72) node(y2_extra_c) [outer sep=-2pt]{$\cdot$}
(23.2, 6.82) node(y2a7) [outer sep=-2pt]{$y_2^{0,r}\otimes a$}
(17.08, 9.02) node(y2c5) [outer sep=-2pt]{$\cdot$}
(14.2, 3.52) node(y2d4) [outer sep=-2pt]{$\cdot$}
(19.6, 7.48) node(y2a6up) [outer sep=-2pt]{$\cdot$}
(17.08, 12.32) node(y2b5) [outer sep=-2pt]{$\cdot$}
;
\draw[->, bend right = 0, green, , font=\scriptsize](u_extra_a) to node [sloped] {$\rho_{2}$*} (ua7);
\draw[->, bend right = 0, green, , font=\scriptsize](u_extra_b) to node [sloped] {$\rho_{2}$*} (ub8);
\draw[->, bend right = 0, green, , font=\scriptsize](u_extra_c) to node [sloped] {$\rho_{2}$*} (uc7);
\draw[->, bend right = 0, green, , font=\scriptsize](u_extra_d) to node [sloped] {$\rho_{2}$*} (ud6);
\draw[->, bend right = -20, black, , font=\scriptsize](ua7) to node [sloped] {$\rho_{3}$} (y2a7);
\draw[->, bend right = 20, blue, , font=\scriptsize](ua7) to node [sloped] {$\rho_{1}$} (y2a2);
\draw[->, bend right = 20, blue, , font=\scriptsize](ua7) to node [sloped] {$\rho_{1}$~~**} (y2c2);
\draw[->, bend right = -20, black, , font=\scriptsize](ub8) to node [sloped] {$\rho_{3}$} (y2b8);
\draw[->, bend right = 20, blue, , font=\scriptsize](ub8) to node [sloped] {$\rho_{1}$} (y2b3);
\draw[->, bend right = -20, black, , font=\scriptsize](uc7) to node [sloped] {$\rho_{3}$} (y2c7);
\draw[->, bend right = 20, blue, , font=\scriptsize](uc7) to node [sloped] {$\rho_{1}$} (y2c2);
\draw[->, bend right = -20, black, , font=\scriptsize](ud6) to node [sloped] {$\rho_{3}$} (y2d6);
\draw[->, bend right = 20, blue, , font=\scriptsize](ud6) to node [sloped] {$\rho_{1}$} (y2d1);
\draw[->, bend right = -20, pink, , font=\scriptsize](ud6) to node [sloped] {$\rho_{1}$} (extra_y1);
\draw[->, bend right = 0, black, , font=\scriptsize](y2a3) to node [sloped] {$\iota_1$} (y2b3);
\draw[->, bend right = 0, red, , font=\scriptsize](y2a3) to node [sloped] {$\rho_{23}$} (y2a2);
\draw[->, bend right = 0, red, , font=\scriptsize](y2a4) to node [sloped] {$\rho_{23}$} (y2a3);
\draw[->, bend right = 0, purple, , font=\scriptsize](y2a6) to node [sloped] {$\iota_1$} (y2a6down);
\draw[->, bend right = 9.0, black, , font=\scriptsize](y2a6) to node [sloped] {$\iota_1$} (y2d6);
\draw[->, bend right = 0, red, , font=\scriptsize](y2a6) to node [sloped] {$\rho_{23}$} (y2a5);
\draw[->, bend right = 0, purple, , font=\scriptsize](y2a6up) to node [sloped] {$\iota_1$} (y2a6);
\draw[->, bend right = 0, red, , font=\scriptsize](y2a7) to node [sloped] {$\rho_{23}$} (y2a6);
\draw[->, bend right = 30, green, , font=\scriptsize](y2a7) to node [sloped] {$\rho_{2}$~**} (ud6);
\draw[->, bend right = -10, green, , font=\scriptsize](y2a7) to node [sloped] {$\rho_{2}$~*} (y2_extra_a);
\draw[->, bend right = 0, purple, , font=\scriptsize](y2a7up) to node [sloped] {$\iota_1$} (y2a7);
\draw[->, bend right = 0, red, , font=\scriptsize](y2b4) to node [sloped] {$\rho_{23}$} (y2b3);
\draw[->, bend right = 0, purple, , font=\scriptsize](y2b6) to node [sloped] {$\iota_1$} (y2b6down);
\draw[->, bend right = 9.0, black, , font=\scriptsize](y2b6) to node [sloped] {$\iota_1$} (y2c6);
\draw[->, bend right = 0, red, , font=\scriptsize](y2b6) to node [sloped] {$\rho_{23}$} (y2b5);
\draw[->, bend right = 0, purple, , font=\scriptsize](y2b6up) to node [sloped] {$\iota_1$} (y2b6);
\draw[->, bend right = 0, purple, , font=\scriptsize](y2b7) to node [sloped] {$\iota_1$} (y2b7down);
\draw[->, bend right = 9.0, black, , font=\scriptsize](y2b7) to node [sloped] {$\iota_1$} (y2c7);
\draw[->, bend right = 0, red, , font=\scriptsize](y2b7) to node [sloped] {$\rho_{23}$} (y2b6);
\draw[->, bend right = 0, purple, , font=\scriptsize](y2b7up) to node [sloped] {$\iota_1$} (y2b7);
\draw[->, bend right = 0, red, , font=\scriptsize](y2b8) to node [sloped] {$\rho_{23}$} (y2b7);
\draw[->, bend right = 30, green, , font=\scriptsize](y2b8) to node [sloped] {$\rho_{2}$~**} (uc7);
\draw[->, bend right = -10, green, , font=\scriptsize](y2b8) to node [sloped] {$\rho_{2}$~*} (y2_extra_b);
\draw[->, bend right = 0, purple, , font=\scriptsize](y2b8up) to node [sloped] {$\iota_1$} (y2b8);
\draw[->, bend right = 0, red, , font=\scriptsize](y2c3) to node [sloped] {$\rho_{23}$} (y2c2);
\draw[->, bend right = 0, red, , font=\scriptsize](y2c4) to node [sloped] {$\rho_{23}$} (y2c3);
\draw[->, bend right = 0, purple, , font=\scriptsize](y2c6) to node [sloped] {$\iota_1$} (y2c6down);
\draw[->, bend right = 0, red, , font=\scriptsize](y2c6) to node [sloped] {$\rho_{23}$} (y2c5);
\draw[->, bend right = 0, purple, , font=\scriptsize](y2c6up) to node [sloped] {$\iota_1$} (y2c6);
\draw[->, bend right = 0, red, , font=\scriptsize](y2c7) to node [sloped] {$\rho_{23}$} (y2c6);
\draw[->, bend right = -10, green, , font=\scriptsize](y2c7) to node [sloped] {$\rho_{2}$~*} (y2_extra_c);
\draw[->, bend right = 0, purple, , font=\scriptsize](y2c7up) to node [sloped] {$\iota_1$} (y2c7);
\draw[->, bend right = 0, black, , font=\scriptsize](y2d2) to node [sloped] {$\iota_1$} (y2c2);
\draw[->, bend right = 0, red, , font=\scriptsize](y2d2) to node [sloped] {$\rho_{23}$} (y2d1);
\draw[->, bend right = 0, black, , font=\scriptsize](y2d3) to node [sloped] {$\iota_1$} (y2c3);
\draw[->, bend right = 0, red, , font=\scriptsize](y2d3) to node [sloped] {$\rho_{23}$} (y2d2);
\draw[->, bend right = 0, red, , font=\scriptsize](y2d4) to node [sloped] {$\rho_{23}$} (y2d3);
\draw[->, bend right = 0, red, , font=\scriptsize](y2d6) to node [sloped] {$\rho_{23}$} (y2d5);
\draw[->, bend right = -10, green, , font=\scriptsize](y2d6) to node [sloped] {$\rho_{2}$~*} (y2_extra_d);
\draw[->, bend right = 0, purple, , font=\scriptsize](y2d6up) to node [sloped] {$\iota_1$} (y2d6);
\draw (8.082,3.5171200000000002) circle (0.5);
\draw (11.182,7.31712) rectangle (12.182,6.31712);
\end{tikzpicture}
\caption{$V^0$}\label{fig:g7f}
\end{figure}
\end{landscape}
\clearpage
\restoregeometry

\clearpage
\newgeometry{margin=2cm}
\begin{landscape}
\begin{figure}
\centering
\begin{tikzpicture}[node distance=2cm]
\path[font = \scriptsize]
(8.08, 6.82) node(y2c2) [outer sep=-2pt]{$y_2^{0}$$\otimes$$c$}
(16.2, 18.48) node(ua7) [outer sep=-2pt]{$u$$\otimes$$a$$+\rho_1$$y_2^{1}$$\otimes$$d$}
(19.6, 8.36) node(y2c6down) [outer sep=-2pt]{$\cdot$}
(17.08, 12.32) node(y2b5) [outer sep=-2pt]{$\cdot$}
(17.2, 14.78) node(u_extra_c) [outer sep=-2pt]{$\cdot$}
(19.6, 3.52) node(y2d6) [outer sep=-2pt]{$y_2$$\otimes$$d$}
(26.6, 11.02) node(y2_extra_b) [outer sep=-2pt]{$\cdot$}
(14.2, 6.82) node(y2c4) [outer sep=-2pt]{$\cdot$}
(11.68, 12.32) node(y2b3) [outer sep=-2pt]{$y_2^{0}$$\otimes$$b$}
(23.2, 11.66) node(y2b7down) [outer sep=-2pt]{$\cdot$}
(19.6, 9.02) node(y2c6) [outer sep=-2pt]{$y_2$$\otimes$$c$}
(19.6, 12.98) node(y2b6up) [outer sep=-2pt]{$\cdot$}
(11.68, 3.52) node(y2d3) [outer sep=-2pt]{$y_2^{2}$$\otimes$$d$}
(11.68, 9.02) node(y2a3) [outer sep=-2pt]{$y_2^{1}$$\otimes$$a$}
(19.6, 12.32) node(y2b6) [outer sep=-2pt]{$y_2$$\otimes$$b$}
(11.68, 6.82) node(y2c3) [outer sep=-2pt]{$y_2^{1}$$\otimes$$c$}
(8.08, 9.02) node(y2a2) [outer sep=-2pt]{$y_2^{0}$$\otimes$$a$}
(16.2, 16.28) node(uc7) [outer sep=-2pt]{$u$$\otimes$$c$}
(19.6, 6.16) node(y2a6down) [outer sep=-2pt]{$\cdot$}
(19.4, 2.22) node(y2_extra_d) [outer sep=-2pt]{$\cdot$}
(23.2, 6.82) node(y2a7) [outer sep=-2pt]{$y_2^{0,r}\otimes a$}
(23.2, 12.98) node(y2b7up) [outer sep=-2pt]{$\cdot$}
(19.6, 7.48) node(y2a6up) [outer sep=-2pt]{$\cdot$}
(19.0, 7.82) node(extra_y1) [outer sep=-2pt]{$``y_1"$}
(19.6, 6.82) node(y2a6) [outer sep=-2pt]{$y_2$$\otimes$$a$}
(17.08, 6.82) node(y2a5) [outer sep=-2pt]{$\cdot$}
(17.08, 9.02) node(y2c5) [outer sep=-2pt]{$\cdot$}
(26.8, 12.98) node(y2b8up) [outer sep=-2pt]{$\cdot$}
(17.08, 3.52) node(y2d5) [outer sep=-2pt]{$\cdot$}
(19.6, 11.66) node(y2b6down) [outer sep=-2pt]{$\cdot$}
(23.0, 7.72) node(y2_extra_c) [outer sep=-2pt]{$\cdot$}
(11.8, 15.88) node(u_extra_d) [outer sep=-2pt]{$\cdot$}
(22.6, 15.88) node(u_extra_b) [outer sep=-2pt]{$\cdot$}
(26.8, 12.32) node(y2b8) [outer sep=-2pt]{$y_2^{0,r}\otimes b$}
(23.2, 9.02) node(y2c7) [outer sep=-2pt]{$y_2$$\otimes$$c$}
(23.0, 5.52) node(y2_extra_a) [outer sep=-2pt]{$\cdot$}
(19.6, 4.18) node(y2d6up) [outer sep=-2pt]{$\cdot$}
(14.2, 3.52) node(y2d4) [outer sep=-2pt]{$\cdot$}
(21.6, 17.38) node(ub8) [outer sep=-2pt]{$u$$\otimes$$b$$+\rho_1$$y_2^{1}$$\otimes$$c$}
(14.2, 9.02) node(y2a4) [outer sep=-2pt]{$\cdot$}
(17.2, 16.98) node(u_extra_a) [outer sep=-2pt]{$\cdot$}
(23.2, 7.48) node(y2a7up) [outer sep=-2pt]{$\cdot$}
(19.6, 9.68) node(y2c6up) [outer sep=-2pt]{$\cdot$}
(23.2, 12.32) node(y2b7) [outer sep=-2pt]{$y_2$$\otimes$$b$}
(14.2, 12.32) node(y2b4) [outer sep=-2pt]{$\cdot$}
(8.08, 3.52) node(y2d2) [outer sep=-2pt]{$y_2^{1}$$\otimes$$d$}
(4.48, 3.52) node(y2d1) [outer sep=-2pt]{$y_2^{0}$$\otimes$$d$}
(10.8, 17.38) node(ud6) [outer sep=-2pt]{$u$$\otimes$$d$}
(23.2, 9.68) node(y2c7up) [outer sep=-2pt]{$\cdot$}
;
\draw[->, bend right = 0, green, , font=\scriptsize](u_extra_a) to node [sloped] {$\rho_{2}$*} (ua7);
\draw[->, bend right = 0, green, , font=\scriptsize](u_extra_b) to node [sloped] {$\rho_{2}$*} (ub8);
\draw[->, bend right = 0, green, , font=\scriptsize](u_extra_c) to node [sloped] {$\rho_{2}$*} (uc7);
\draw[->, bend right = 0, green, , font=\scriptsize](u_extra_d) to node [sloped] {$\rho_{2}$*} (ud6);
\draw[->, bend right = -20, black, , font=\scriptsize](ua7) to node [sloped] {$\rho_{3}$} (y2a7);
\draw[->, bend right = 20, blue, , font=\scriptsize](ua7) to node [sloped] {$\rho_{1}$} (y2a2);
\draw[->, bend right = 30, pink, , font=\scriptsize](ua7) to node [sloped] {$\rho_{123}$} (y2d1);
\draw[->, bend right = -20, black, , font=\scriptsize](ub8) to node [sloped] {$\rho_{3}$} (y2b8);
\draw[->, bend right = 20, blue, , font=\scriptsize](ub8) to node [sloped] {$\rho_{1}$} (y2b3);
\draw[->, bend right = 10, pink, , font=\scriptsize](ub8) to node [sloped] {$\rho_{123}$} (y2c2);
\draw[->, bend right = -20, black, , font=\scriptsize](uc7) to node [sloped] {$\rho_{3}$} (y2c7);
\draw[->, bend right = 20, blue, , font=\scriptsize](uc7) to node [sloped] {$\rho_{1}$} (y2c2);
\draw[->, bend right = -20, black, , font=\scriptsize](ud6) to node [sloped] {$\rho_{3}$} (y2d6);
\draw[->, bend right = 20, blue, , font=\scriptsize](ud6) to node [sloped] {$\rho_{1}$} (y2d1);
\draw[->, bend right = -20, pink, , font=\scriptsize](ud6) to node [sloped] {$\rho_{1}$} (extra_y1);
\draw[->, bend right = 0, black, , font=\scriptsize](y2a3) to node [sloped] {$\iota_1$} (y2b3);
\draw[->, bend right = 0, red, , font=\scriptsize](y2a3) to node [sloped] {$\rho_{23}$} (y2a2);
\draw[->, bend right = 0, red, , font=\scriptsize](y2a4) to node [sloped] {$\rho_{23}$} (y2a3);
\draw[->, bend right = 0, purple, , font=\scriptsize](y2a6) to node [sloped] {$\iota_1$} (y2a6down);
\draw[->, bend right = 9.0, black, , font=\scriptsize](y2a6) to node [sloped] {$\iota_1$} (y2d6);
\draw[->, bend right = 0, red, , font=\scriptsize](y2a6) to node [sloped] {$\rho_{23}$} (y2a5);
\draw[->, bend right = 0, purple, , font=\scriptsize](y2a6up) to node [sloped] {$\iota_1$} (y2a6);
\draw[->, bend right = 0, red, , font=\scriptsize](y2a7) to node [sloped] {$\rho_{23}$} (y2a6);
\draw[->, bend right = 30, green, , font=\scriptsize](y2a7) to node [sloped] {$\rho_{2}$~**} (ud6);
\draw[->, bend right = -10, green, , font=\scriptsize](y2a7) to node [sloped] {$\rho_{2}$~*} (y2_extra_a);
\draw[->, bend right = 0, purple, , font=\scriptsize](y2a7up) to node [sloped] {$\iota_1$} (y2a7);
\draw[->, bend right = 0, red, , font=\scriptsize](y2b4) to node [sloped] {$\rho_{23}$} (y2b3);
\draw[->, bend right = 0, purple, , font=\scriptsize](y2b6) to node [sloped] {$\iota_1$} (y2b6down);
\draw[->, bend right = 9.0, black, , font=\scriptsize](y2b6) to node [sloped] {$\iota_1$} (y2c6);
\draw[->, bend right = 0, red, , font=\scriptsize](y2b6) to node [sloped] {$\rho_{23}$} (y2b5);
\draw[->, bend right = 0, purple, , font=\scriptsize](y2b6up) to node [sloped] {$\iota_1$} (y2b6);
\draw[->, bend right = 0, purple, , font=\scriptsize](y2b7) to node [sloped] {$\iota_1$} (y2b7down);
\draw[->, bend right = 9.0, black, , font=\scriptsize](y2b7) to node [sloped] {$\iota_1$} (y2c7);
\draw[->, bend right = 0, red, , font=\scriptsize](y2b7) to node [sloped] {$\rho_{23}$} (y2b6);
\draw[->, bend right = 0, purple, , font=\scriptsize](y2b7up) to node [sloped] {$\iota_1$} (y2b7);
\draw[->, bend right = 0, red, , font=\scriptsize](y2b8) to node [sloped] {$\rho_{23}$} (y2b7);
\draw[->, bend right = 30, green, , font=\scriptsize](y2b8) to node [sloped] {$\rho_{2}$~**} (uc7);
\draw[->, bend right = -10, green, , font=\scriptsize](y2b8) to node [sloped] {$\rho_{2}$~*} (y2_extra_b);
\draw[->, bend right = 0, purple, , font=\scriptsize](y2b8up) to node [sloped] {$\iota_1$} (y2b8);
\draw[->, bend right = 0, red, , font=\scriptsize](y2c3) to node [sloped] {$\rho_{23}$} (y2c2);
\draw[->, bend right = 0, red, , font=\scriptsize](y2c4) to node [sloped] {$\rho_{23}$} (y2c3);
\draw[->, bend right = 0, purple, , font=\scriptsize](y2c6) to node [sloped] {$\iota_1$} (y2c6down);
\draw[->, bend right = 0, red, , font=\scriptsize](y2c6) to node [sloped] {$\rho_{23}$} (y2c5);
\draw[->, bend right = 0, purple, , font=\scriptsize](y2c6up) to node [sloped] {$\iota_1$} (y2c6);
\draw[->, bend right = 0, red, , font=\scriptsize](y2c7) to node [sloped] {$\rho_{23}$} (y2c6);
\draw[->, bend right = -10, green, , font=\scriptsize](y2c7) to node [sloped] {$\rho_{2}$~*} (y2_extra_c);
\draw[->, bend right = 0, purple, , font=\scriptsize](y2c7up) to node [sloped] {$\iota_1$} (y2c7);
\draw[->, bend right = 0, black, , font=\scriptsize](y2d2) to node [sloped] {$\iota_1$} (y2c2);
\draw[->, bend right = 0, red, , font=\scriptsize](y2d2) to node [sloped] {$\rho_{23}$} (y2d1);
\draw[->, bend right = 0, black, , font=\scriptsize](y2d3) to node [sloped] {$\iota_1$} (y2c3);
\draw[->, bend right = 0, red, , font=\scriptsize](y2d3) to node [sloped] {$\rho_{23}$} (y2d2);
\draw[->, bend right = 0, red, , font=\scriptsize](y2d4) to node [sloped] {$\rho_{23}$} (y2d3);
\draw[->, bend right = 0, red, , font=\scriptsize](y2d6) to node [sloped] {$\rho_{23}$} (y2d5);
\draw[->, bend right = -10, green, , font=\scriptsize](y2d6) to node [sloped] {$\rho_{2}$~*} (y2_extra_d);
\draw[->, bend right = 0, purple, , font=\scriptsize](y2d6up) to node [sloped] {$\iota_1$} (y2d6);
\end{tikzpicture}
\caption{$V^0$}\label{fig:g7g}
\end{figure}
\end{landscape}
\clearpage
\restoregeometry

\FloatBarrier
\subsubsection{Left hand side full copy}

The situation at left hand side full copy is simpler than the right hand side, as $C$ is horizontally simplified. Generators come in pairs except $\xi_h$, which is generator of the homology, see \cref{fig:g8a} for a pair $x\to y$ and $\xi_h$. The isomorphism-inducing map between the full copy and $\mathbb{F}_2$ must be a single arrow from $\xi_h$ to the generator of $\mathbb{F}_2$.

At a pair $x\to y$, we cancel the pairs of red arrows, which we have done so many times, see \cref{fig:g8b}. Then we cancel the vertical black arrows between $y_1$'s and $u$'s, see \cref{fig:g8c}. At the row of $\xi_h$, it is even simpler. It is now a single string of $\rh{23}$'s. The usual cancellation bring them into shape.

One last corner case is of the top row, where the full copy meets the left-most position of a row. We need to examine that the cancellation needs to be done in both places are compatible. It is clear that the top row generator has only incoming $\rh{3}$, but no outgoing $\rh{2}$. We further observe that the cancellation in \cref{sec:v0} only concern $y_3,v$ at the top row generator, whereas cancellation in this section remove $y_1, u$ there. This means the two sets of cancellation never interact, hence must be compatible.

\begin{figure}

\begin{subfigure}{0.99\textwidth}
\centerline{
\begin{tikzpicture}[node distance=2cm]
\path[font = \scriptsize]
(12.0, 1.0) node(right) [outer sep=-2pt]{$\cdot$}
(3.0, 0.0) node(y2) [outer sep=-2pt]{$y$}
(4.5, 0.0) node(y3) [outer sep=-2pt]{$y$}
(1.5, 1.0) node(x1) [outer sep=-2pt]{}
(4.5, 1.0) node(x3) [outer sep=-2pt]{$x$}
(7.5, -1.0) node(left) [outer sep=-2pt]{$\cdot$}
(9.0, -1.0) node(xih) [outer sep=-2pt]{$\xi_h$}
(1.5, 0.0) node(y1) [outer sep=-2pt]{}
(10.5, 1.0) node(dot) [outer sep=-2pt]{$\cdot(\mathbb{F}_2)$}
(3.0, 1.0) node(x2) [outer sep=-2pt]{$x$}
;
\draw[->, bend right = 0, red, , font=\scriptsize](dot) to node [sloped] {$\rho_{23}$} (right);
\draw[->, bend right = 0, red, , font=\scriptsize](left) to node [sloped] {$\rho_{23}$} (xih);
\draw[->, bend right = 0, red, , font=\scriptsize](x1) to node [sloped] {$\rho_{23}$} (x2);
\draw[->, bend right = 0, red, , font=\scriptsize](x2) to node [sloped] {$\rho_{23}$} (x3);
\draw[->, bend right = 0, red, , font=\scriptsize](xih) to node [sloped] {$\rho_{23}$} (dot);
\draw[->, bend right = 0, red, , font=\scriptsize](y1) to node [sloped] {$\rho_{23}$} (y2);
\draw[->, bend right = 0, black, , font=\scriptsize](y2) to node [sloped] {$\iota_1$} (x2);
\draw[->, bend right = 0, red, , font=\scriptsize](y2) to node [sloped] {$\rho_{23}$} (y3);
\draw[->, bend right = 0, black, , font=\scriptsize](y3) to node [sloped] {$\iota_1$} (x3);
\end{tikzpicture}
}
\caption{}\label{fig:g8a}
\end{subfigure}

\begin{subfigure}{0.49\textwidth}
\centerline{
\begin{tikzpicture}[node distance=2cm]
\path[font = \scriptsize]
(8.45, 3.32) node(ux3) [outer sep=-3pt]{$u$$\otimes$$x$}
(7.92, 2.76) node(y1x3) [outer sep=-3pt]{$y_1$$\otimes$$x$}
(7.4, 2.26) node(vy2) [outer sep=-3pt]{$v$$\otimes$$y$}
(8.45, 0.56) node(uy3) [outer sep=-3pt]{$u$$\otimes$$y$}
(2.64, 0.0) node(y1y1) [outer sep=-3pt]{$\cdot$}
(5.81, 0.56) node(uy2) [outer sep=-3pt]{$u$$\otimes$$y$}
(3.17, 3.32) node(ux1) [outer sep=-3pt]{$\cdot$}
(8.98, 1.13) node(y2y3) [outer sep=-3pt]{$y_2$$\otimes$$y$}
(6.87, 1.69) node(y3y2) [outer sep=-3pt]{$y_3$$\otimes$$y$}
(6.87, 4.45) node(y3x2) [outer sep=-3pt]{$y_3$$\otimes$$x$}
(9.51, 4.45) node(y3x3) [outer sep=-3pt]{$y_3$$\otimes$$x$}
(7.4, 5.02) node(vx2) [outer sep=-3pt]{$v$$\otimes$$x$}
(6.34, 1.13) node(y2y2) [outer sep=-3pt]{$y_2$$\otimes$$y$}
(3.17, 0.56) node(uy1) [outer sep=-3pt]{$\cdot$}
(7.92, 0.0) node(y1y3) [outer sep=-3pt]{$y_1$$\otimes$$y$}
(10.04, 5.02) node(vx3) [outer sep=-3pt]{$v$$\otimes$$x$}
(10.04, 2.26) node(vy3) [outer sep=-3pt]{$v$$\otimes$$y$}
(5.28, 0.0) node(y1y2) [outer sep=-3pt]{$y_1$$\otimes$$y$}
(6.34, 3.89) node(y2x2) [outer sep=-3pt]{$y_2$$\otimes$$x$}
(9.51, 1.69) node(y3y3) [outer sep=-3pt]{$y_3$$\otimes$$y$}
(2.64, 2.76) node(y1x1) [outer sep=-3pt]{$\cdot$}
(8.98, 3.89) node(y2x3) [outer sep=-3pt]{$y_2$$\otimes$$x$}
(5.81, 3.32) node(ux2) [outer sep=-3pt]{$u$$\otimes$$x$}
(5.28, 2.76) node(y1x2) [outer sep=-3pt]{$y_1$$\otimes$$x$}
;
\draw[->, bend right = -10, red, , font=\scriptsize](ux1) to node [right, sloped] {$\iota_0$} (y3x2);
\draw[->, bend right = 0, black, , font=\scriptsize](ux2) to node [sloped] {$\rho_{1}$} (y1x2);
\draw[->, bend right = 0, black, , font=\scriptsize](ux2) to node [sloped] {$\rho_{3}$} (y2x2);
\draw[->, bend right = -10, red, , font=\scriptsize](ux2) to node [right, sloped] {$\iota_0$} (y3x3);
\draw[->, bend right = 0, black, , font=\scriptsize](ux3) to node [sloped] {$\rho_{1}$} (y1x3);
\draw[->, bend right = 0, black, , font=\scriptsize](ux3) to node [sloped] {$\rho_{3}$} (y2x3);
\draw[->, bend right = -10, red, , font=\scriptsize](uy1) to node [right, sloped] {$\iota_0$} (y3y2);
\draw[->, bend right = 0, black, , font=\scriptsize](uy2) to node [sloped] {$\rho_{1}$} (y1y2);
\draw[->, bend right = 0, black, , font=\scriptsize](uy2) to node [sloped] {$\rho_{3}$} (y2y2);
\draw[->, bend right = -10, red, , font=\scriptsize](uy2) to node [right, sloped] {$\iota_0$} (y3y3);
\draw[->, bend right = 0, black, , font=\scriptsize](uy2) to node [sloped] {$\iota_0$} (ux2);
\draw[->, bend right = 0, black, , font=\scriptsize](uy3) to node [sloped] {$\rho_{1}$} (y1y3);
\draw[->, bend right = 0, black, , font=\scriptsize](uy3) to node [sloped] {$\rho_{3}$} (y2y3);
\draw[->, bend right = 0, black, , font=\scriptsize](uy3) to node [sloped] {$\iota_0$} (ux3);
\draw[->, bend right = 0, black, , font=\scriptsize](vy2) to node [sloped] {$\iota_1$} (vx2);
\draw[->, bend right = 0, black, , font=\scriptsize](vy3) to node [sloped] {$\iota_1$} (vx3);
\draw[->, bend right = -10, red, , font=\scriptsize](y1x1) to node [right, sloped] {$\iota_1$} (vx2);
\draw[->, bend right = -10, red, , font=\scriptsize](y1x2) to node [right, sloped] {$\iota_1$} (vx3);
\draw[->, bend right = -10, red, , font=\scriptsize](y1y1) to node [right, sloped] {$\iota_1$} (vy2);
\draw[->, bend right = -10, red, , font=\scriptsize](y1y2) to node [right, sloped] {$\iota_1$} (vy3);
\draw[->, bend right = 0, black, , font=\scriptsize](y1y2) to node [sloped] {$\iota_1$} (y1x2);
\draw[->, bend right = 0, black, , font=\scriptsize](y1y3) to node [sloped] {$\iota_1$} (y1x3);
\draw[->, bend right = 0, black, , font=\scriptsize](y2x2) to node [sloped] {$\rho_{2}$} (y3x2);
\draw[->, bend right = 0, black, , font=\scriptsize](y2x3) to node [sloped] {$\rho_{2}$} (y3x3);
\draw[->, bend right = 0, black, , font=\scriptsize](y2y2) to node [sloped] {$\rho_{2}$} (y3y2);
\draw[->, bend right = 0, black, , font=\scriptsize](y2y2) to node [sloped] {$\iota_1$} (y2x2);
\draw[->, bend right = 0, black, , font=\scriptsize](y2y3) to node [sloped] {$\rho_{2}$} (y3y3);
\draw[->, bend right = 0, black, , font=\scriptsize](y2y3) to node [sloped] {$\iota_1$} (y2x3);
\draw[->, bend right = 0, black, , font=\scriptsize](y3x2) to node [sloped] {$\rho_{1}$} (vx2);
\draw[->, bend right = 0, black, , font=\scriptsize](y3x3) to node [sloped] {$\rho_{1}$} (vx3);
\draw[->, bend right = 0, black, , font=\scriptsize](y3y2) to node [sloped] {$\rho_{1}$} (vy2);
\draw[->, bend right = 0, black, , font=\scriptsize](y3y2) to node [sloped] {$\iota_0$} (y3x2);
\draw[->, bend right = 0, black, , font=\scriptsize](y3y3) to node [sloped] {$\rho_{1}$} (vy3);
\draw[->, bend right = 0, black, , font=\scriptsize](y3y3) to node [sloped] {$\iota_0$} (y3x3);
\end{tikzpicture}
}
\caption{}\label{fig:g8b}
\end{subfigure}
\begin{subfigure}{0.49\textwidth}
\centerline{
\begin{tikzpicture}[node distance=2cm]
\path[font = \scriptsize]
(6.34, 3.89) node(y2x2) [outer sep=-3pt]{$y_2$$\otimes$$x$}
(3.7, 1.13) node(y2y1) [outer sep=-3pt]{$\cdot$}
(8.98, 3.89) node(y2x3) [outer sep=-3pt]{$y_2$$\otimes$$x$}
(7.92, 0.0) node(y1y3) [outer sep=-3pt]{$y_1$$\otimes$$y$}
(3.7, 3.89) node(y2x1) [outer sep=-3pt]{$\cdot$}
(8.45, 3.32) node(ux3) [outer sep=-3pt]{$u$$\otimes$$x$}
(6.34, 1.13) node(y2y2) [outer sep=-3pt]{$y_2$$\otimes$$y$}
(7.92, 2.76) node(y1x3) [outer sep=-3pt]{$y_1$$\otimes$$x$}
(8.98, 1.13) node(y2y3) [outer sep=-3pt]{$y_2$$\otimes$$y$}
(8.45, 0.56) node(uy3) [outer sep=-3pt]{$u$$\otimes$$y$}
;
\draw[->, bend right = 0, black, , font=\scriptsize](ux3) to node [sloped] {$\rho_{1}$} (y1x3);
\draw[->, bend right = 0, black, , font=\scriptsize](ux3) to node [sloped] {$\rho_{3}$} (y2x3);
\draw[->, bend right = 0, black, , font=\scriptsize](uy3) to node [sloped] {$\rho_{1}$} (y1y3);
\draw[->, bend right = 0, black, , font=\scriptsize](uy3) to node [sloped] {$\rho_{3}$} (y2y3);
\draw[->, bend right = 0, black, , font=\scriptsize](uy3) to node [sloped] {$\iota_0$} (ux3);
\draw[->, bend right = 0, black, , font=\scriptsize](y1y3) to node [sloped] {$\iota_1$} (y1x3);
\draw[->, bend right = 0, red, , font=\scriptsize](y2x2) to node [sloped] {$\rho_{23}$} (y2x1);
\draw[->, bend right = 0, red, , font=\scriptsize](y2x3) to node [sloped] {$\rho_{23}$} (y2x2);
\draw[->, bend right = 0, black, , font=\scriptsize](y2y2) to node [sloped] {$\iota_1$} (y2x2);
\draw[->, bend right = 0, red, , font=\scriptsize](y2y2) to node [sloped] {$\rho_{23}$} (y2y1);
\draw[->, bend right = 0, black, , font=\scriptsize](y2y3) to node [sloped] {$\iota_1$} (y2x3);
\draw[->, bend right = 0, red, , font=\scriptsize](y2y3) to node [sloped] {$\rho_{23}$} (y2y2);
\end{tikzpicture}
}
\caption{}\label{fig:g8bb}
\end{subfigure}

\begin{subfigure}{0.99\textwidth}
\centerline{
\begin{tikzpicture}[node distance=2cm]
\path[font = \scriptsize]
(2.0, 0.0) node(y1) [outer sep=-2pt]{}
(4.0, 1.5) node(x2) [outer sep=-2pt]{$y_2$}
(6.0, 0.0) node(y3) [outer sep=-2pt]{$y_2$}
(4.0, 0.0) node(y2) [outer sep=-2pt]{$y_2$}
(2.0, 1.5) node(x1) [outer sep=-2pt]{}
(6.0, 1.5) node(x3) [outer sep=-2pt]{$y_2$}
;
\draw[->, bend right = 0, red, , font=\scriptsize](x2) to node [sloped] {$\rho_{23}$} (x1);
\draw[->, bend right = 0, red, , font=\scriptsize](x3) to node [sloped] {$\rho_{23}$} (x2);
\draw[->, bend right = 0, red, , font=\scriptsize](y2) to node [sloped] {$\rho_{23}$} (y1);
\draw[->, bend right = 0, black, , font=\scriptsize](y2) to node [sloped] {$\iota_1$} (x2);
\draw[->, bend right = 0, red, , font=\scriptsize](y3) to node [sloped] {$\rho_{23}$} (y2);
\draw[->, bend right = 0, black, , font=\scriptsize](y3) to node [sloped] {$\iota_1$} (x3);
\end{tikzpicture}
}
\caption{}\label{fig:g8c}
\end{subfigure}

\caption{Left hand side full copy}
\end{figure}
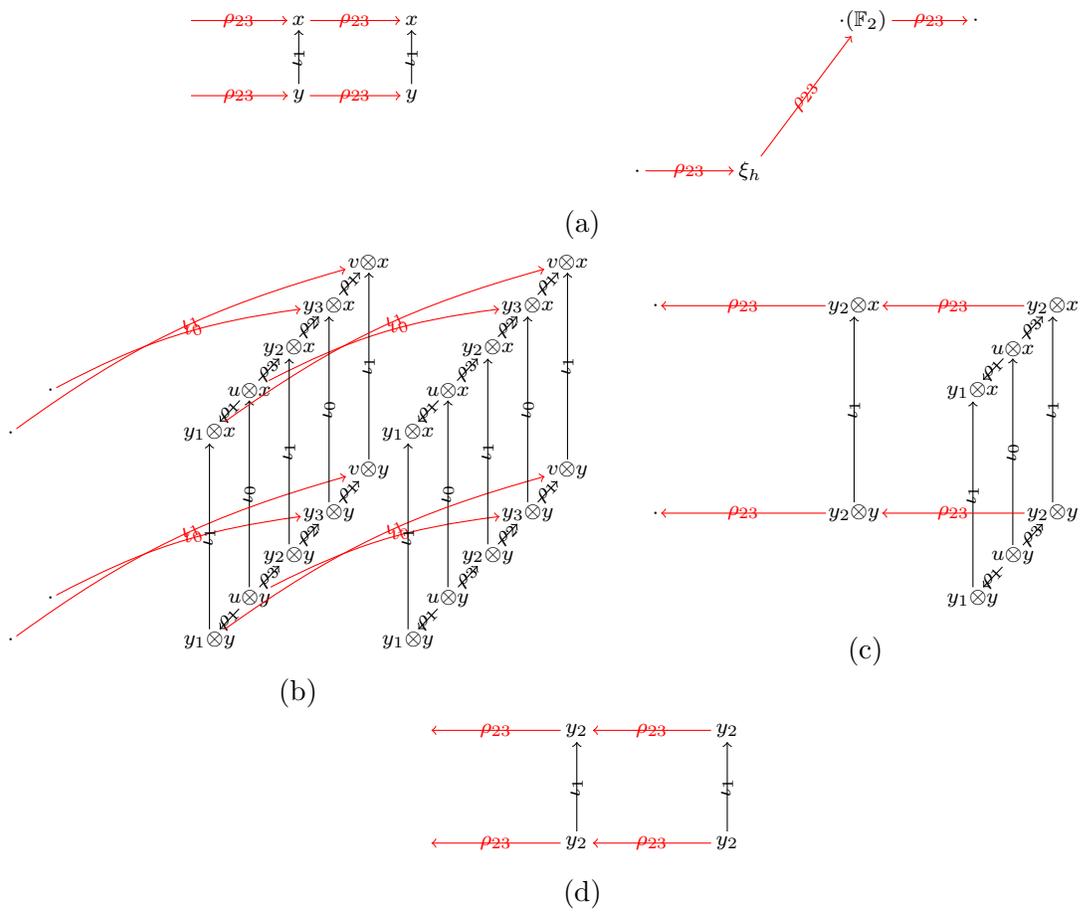

\subsubsection{Conclusion of proof}
Therefore, we have examined all parts of the modules and can conclude that tensoring with $H$ indeed transforms $KtD(C)$ to $KtD(\fl{C}).$ Hence \cref{flip} is proved.

\FloatBarrier

\end{document}